\documentclass[oneside,reqno]{amsart}
\usepackage[utf8]{inputenc}
\usepackage[margin=1.5in]{geometry}
\usepackage[title]{appendix}
\usepackage{amsthm}
\usepackage{url,color}
\usepackage{hyperref}
\usepackage{amsmath, amsfonts}
\usepackage{mathtools}
\usepackage{mathalfa}
\usepackage{amssymb}
\usepackage{upgreek}
\usepackage{mathrsfs}
\usepackage{relsize}

\usepackage[makeroom]{cancel}
\usepackage[mathscr]{eucal}
\numberwithin{equation}{section}
\newcommand{\blank}{\mathrel{\;\cdot\;}}

\newcommand{\dell}{\partial}

\newcommand{\ass}{\quad\hbox{as }}

\newcommand{\gradperp}{\nabla^{\perp}}

\newcommand{\R}{\mathbb R}
\newcommand{\nn}{\nabla} 
\newcommand{\ve}{\varepsilon} 
  
\newcommand{\be}{\begin{equation}} 
	\newcommand{\ee}{\end{equation}}

\newcommand{\bigO}{\mathcal{O}}
\newcommand{\twopartdef}[4]
{
	\left\{
	\begin{array}{ll}
		#1 & \text{if } #2 \\
		#3 & \text{if } #4
	\end{array}
	\right.
}

\DeclareMathOperator{\arccot}{arccot}

\DeclareMathOperator{\supp}{supp}

\DeclareMathOperator{\grad}{\nabla}

\DeclarePairedDelimiter\floor{\lfloor}{\rfloor}
\DeclareMathOperator{\sgn}{sgn}

\newtheorem{definition}{Definition}[section]

\newtheorem{theorem}[definition]{Theorem}

\newtheorem{corollary}[definition]{Corollary}

\newtheorem{lemma}[definition]{Lemma}

\newtheorem{remark}[definition]{Remark}

\usepackage[colorinlistoftodos]{todonotes}

\def\bcr{\begin{color}{red}}
	\def\bcb{\begin{color}{blue}}
		\def\ec{\end{color}}

\begin{document}
		
			\title[Vortex-pair solutions for incompressible Euler equations]{Asymptotic properties of  vortex-pair solutions for incompressible Euler equations in $\R^2$}
		
		\author[J. D\'avila]{Juan D\'avila}
		
		\address{J. D\'avila.- Department of Mathematical Sciences, University of Bath, Bath, Ba2 7AY, UK.}
		
		\email{jddb22@bath.ac.uk}
		
		\author[M. del Pino]{Manuel del Pino}
		
		\address{M. del Pino.- Department of Mathematical Sciences, University of Bath, Bath, Ba2 7AY, UK.}
		
		\email{mdp59@bath.ac.uk}
		
		\author[M.Musso]{Monica Musso}
		
		\address{M. Musso.- Department of Mathematical Sciences, University of Bath, Bath, Ba2 7AY, UK. }
		
		\email{mm2683@bath.ac.uk}
		
		\author[S. Parmeshwar ]{Shrish Parmeshwar }
		
		\address{S. Parmeshwar.- Department of Mathematical Sciences, University of Bath, Bath, Ba2 7AY, UK.}
		
		\email{sp667@bath.ac.uk}

		\thanks{}


		\maketitle
		
		\begin{abstract}
			We consider the problem of finding a solution to the incompressible Euler equations 
			$$ \omega_t + v\cdot \nn \omega = 0 \ \hbox{ in } \R^2 \times (0,\infty), \quad v(x,t) =  \frac 1{2\pi} \int_{\mathbb R^2} \frac {(y-x)^\perp} {|y-x|^2} \omega (y,t)\, dy \  $$
			that is close to a superposition of travelling vortices as $t\to \infty$. We employ a constructive approach by gluing classical travelling waves: at main order two vortex-antivortex pairs. Each pair consists of two vortices distance $2q$ from each other, with wave speed $c$, moving in the positive and negative $x_{1}$ directions respectively. More precisely, we find an initial condition that leads to a 4-vortex solution of the form 
			$$
			\omega (x,t)   =  \omega_0(x-cte_{1})   -  \omega_0 ( x+ cte_{1}) + o(1) \ \hbox{ as } t\to\infty
			$$
			where  $$ \omega_0( x ) =  \frac 1{\ve^{2}}  W \left ( \frac {x-q} \ve  \right ) -  \frac 1{\ve^{2}}W \left ( \frac {x+q} \ve  \right )  + o(1) \ \hbox{ as } \ve \to 0 $$
			and $W(y)$ is a certain fixed smooth profile, radially symmetric, positive in the unit disc zero outside. 
		\end{abstract}
		
		\maketitle
		
		\section{Introduction}\label{introduction-section}

		This paper deals with the long-time  behaviour of solutions $\omega(x,t)$ of the Euler equations for an incompressible, inviscid fluid in $\mathbb{R}^2$. In vorticity-stream formulation  these equations take the form 
		
		\be \label{2d-euler-vorticity-stream}  \left\{   
		\begin{aligned}
			\omega_{t}+\gradperp\Psi\cdot\nabla\omega&=0\ \ \ \ &&\text{in}\ \mathbb{R}^{2}\times (0,T),\\ \Psi= (-\Delta)^{-1}&\omega\ \  &&\text{in}\ \mathbb{R}^{2}\times (0,T), \\
			\omega(\cdot,0)= &\mathring{\omega}  &&\text{in}\ \mathbb{R}^{2}.
		\end{aligned} \right.  \ee
		The stream function $\Psi$ is determined as the inverse Laplacian applied to $\omega$, yielding:
		\begin{equation}
			\label{newtonian}
			\Psi(x,t)= (-\Delta)^{-1}\omega (x,t)   = \frac 1{2\pi} \int_{\mathbb R^2}\log \frac 1{|x-y|} \omega (y,t)\, dy    
		\end{equation}
		and the fluid velocity field  $v =  \nn^\perp \Psi = ( \Psi_{x_2}, -\Psi_{x_1})$  follows the Biot-Savart Law:
		$$
		v(x,t) =  \frac 1{2\pi} \int_{\mathbb R^2} \frac {(y-x)^\perp} {|y-x|^2} \omega (y,t)\, dy \ .
		$$
		The Cauchy problem \eqref{2d-euler-vorticity-stream} is well-posed in the weak sense in $L^1(\R^2)\cap L^\infty (\R^2)$ as deduced from the classical theories by Wolibner \cite{Wolibner1933} and Yudovich \cite{Yudovich1963}. Solutions are regular if the initial condition $\mathring{\omega}$ is so. See for instance Chapter 8 in Majda-Bertozzi's book \cite{majdabertozzi}. 
		
		\medskip
		We are interested in solutions of \eqref{2d-euler-vorticity-stream}   exhibiting nontrivial dynamics as $t\to \infty$. Specifically, we are interested in solutions with vorticity {\em concentrated} around a finite number of points $\xi_1(t), \ldots, \xi_N(t)$, to fix ideas, in the form 
		\be \label{ome}
		\omega_\ve (x,t) =  \sum_{j=1}^N   \frac{m_j} {\ve^2}  \Big [  W\left ( \frac {x-\xi_j (t) } \ve   \right ) + o(1) \Big ] ,
		\ee
		where $W(y)$ is a fixed radially symmetric profile with $\int_{\R^2} W(y) dy = 1$, $o(1) \to 0 $ as $\ve \to 0$ and $m_j \in \mathbb{R}$ are the circulations (or masses) of the vortices.     We have that 
		\begin{equation}\label{vortices-to-deltas}
			\omega_\ve(x,t)  \rightharpoonup  \omega_S(x,t) = \sum_{j=1}^N   {m_j} \delta_0  (  {x-\xi_j (t) } ) \ass \ve\to 0 ,
		\end{equation}
		where $\delta_0(y)$ is a Dirac mass at $0$.
		%
		Formally the trajectories $\xi_j(t)$ satisfy the Kirchhoff-Routh vortex system
		\begin{equation}\label{KR}
			\dot \xi_j (t) = -\sum_{\ell \not= j} \frac{m_\ell}{2\pi} \frac{(\xi_j(t) -\xi_\ell (t) )^\perp}{|\xi_j(t) -\xi_\ell (t)|^2}, \quad j=1, \ldots , N.
		\end{equation}
		In \cite{DDMW2020} the following was proven:  {\em Given a finite time $T>0$ and a solution $\left( \xi_1 , \ldots , \xi_N  \right)$ to  System \eqref{KR} with no collisions  in $[0,T]$,  namely with
			$$ \inf_{t\in [0,T],\, i\ne j} |\xi_i(t) -\xi_j(t)| >0 , $$ there exists a solution of \eqref{2d-euler-vorticity-stream} with the precise form \eqref{ome} and the specific choice  
			$W(y) =  \frac 1{\pi} \frac 1{(1+|y|^2)^2} $.} 
		
		\medskip
		In \cite{MP1993} it was previously shown that an initial condition supported on $\ve$-disks concentrated around the points $\xi_j (0)$ as bumps with masses $m_j$ lead to a solution which as $\ve \to 0$ satisfies the distributional convergence \eqref{vortices-to-deltas} but without the asymptotic description \eqref{ome} near the vortices.   
		
		\medskip
		The finiteness in time $T>0$  is essential to get any form of uniform estimates in $[0,T]$ as $\ve \to 0$ in the results in \cite{MP1993} and \cite{DDMW2020}. It is not clear that the solution as time evolves, for fixed $\ve$ will keep the form \eqref{ome} while trajectories of system \eqref{KR} generically do not collide, see Chapter 4 in \cite{MarchioroPulvirentiBook}. This paper deals with the construction of solutions to 
		the Euler equations \eqref{2d-euler-vorticity-stream} that keep 
		the form \eqref{ome} {\em at all times}.
		
		\medskip
		The problem of understanding long-time patterns for solutions of \eqref{2d-euler-vorticity-stream} still
		carries many mysteries, while specific conjectures have been raised. See for instance \cite{elgindimurraysaid} and its references.
		
		\medskip
		In the analysis of long-time dynamics, a fundamental class is that of {\em steady solutions}, namely solutions whose shape does not depend on time. Examples include radially symmetric decaying functions, which correspond to stationary solutions. Other examples are travelling and rotating waves. See for instance \cite{ao,HmidiMateu,smets-vanschaftingen,caowei} and references therein. The problem of finding such solutions is typically reduced to solving semilinear elliptic equations.
		
		\medskip
		The problem of finding time-dependent configurations that closely resemble the superposition of steady solutions, perhaps with the profile \eqref{ome}, is a challenging one, due to the emergence of coupling error terms in the main order, and our limited understanding of long-term behaviour of perturbations to steady solutions still remains . 
		
		\medskip
		We will consider a configuration of two vortex pairs travelling at constant, opposite speeds along the $x_1$-axis. 
		
		\medskip
		Let us first describe a single vortex pair.
		Let $e_1 =(1,0),\ e_2=(0,1)$, $q_\infty>0$, $M_{0}>0$.
		A point vortex pair solution to \eqref{KR} for $N=2$ is given by $\xi_1^{pair}(t)$, $\xi_2^{pair}(t)$ with $m_1=M_{0}=-m_2$ and 
		\begin{align}
			\label{xipair}
			\xi_1^{pair}(t) = q_\infty e_2 + c_\infty t e_1, \quad 
			\xi_2^{pair}(t) = -q_\infty e_2 + c_\infty t e_1, \quad 
			c_\infty=\frac{M_{0}}{4 \pi q_{\infty}}.
		\end{align}
		Associated to this solution, for all $\ve >0$ small enough, there exists an $\varepsilon$-concentrated travelling vortex pair solving \eqref{2d-euler-vorticity-stream} that can be described as
		\be  
		\nonumber
		\omega(x,t)  =    U_{\ve,q_{\infty}} (x -c_{\ve,\infty}  t e_1 )=\varepsilon^{-2} f_\varepsilon( \Psi (x) - c_{\varepsilon,\infty} x_2,x), 
		\ee
		for an appropriate nonlinearity $f_\varepsilon$, with $\Psi$ solving the elliptic problem
		\begin{align}
			-\Delta \Psi (x) =  \varepsilon^{-2} f_\varepsilon( \Psi (x) - c_{\varepsilon,\infty} x_2,x) .
			\label{stationary-solutions-elliptic-pde} 
		\end{align}
		in the upper half plane $x_2>0$, 
		and is even in $x_1$ and odd in $x_2$.
		The function $U_{\ve,q_{\infty}} $  can be written as
		\be 
		\nonumber
		U_{\ve,q_{\infty}} (x) = \frac{1}{\ve^2} 
		\left[ 
		\Gamma^\gamma_+\left ( \frac {x - q_\infty e_2 }{ \ve} \right ) -  \Gamma^\gamma_+ \left ( \frac {x + q_\infty e_2 }{ \ve} \right )   + \, \phi_{\ve , q_\infty} \left(\frac{x}{\varepsilon} \right) \right], \quad 
	\end{equation}
	where  $\Gamma(x)$ is an entire radial classical solution of the problem
	\begin{align}
		\Delta \Gamma+\Gamma^{\gamma}_{+}=0 \ \text{on}\ \mathbb{R}^{2} , \quad \supp{\Gamma} = B_1\left(0\right),\quad M\coloneqq\int_{\mathbb{R}^{2}}\Gamma^{\gamma}_{+}dx,
		\label{power-semilinear-problem-R2}
	\end{align}
	and
	$
	\| \phi_{\ve , q_\infty} \|_\infty = O(\ve^2 )$, $\mbox {supp} (\phi_{\ve , q_\infty}) \subseteq  B_{R\ve} (q_\infty e_{2} )   \cup B_{R\ve}  (- q_\infty e_{2} ) $ 
	for some fixed  $R>0$. 
	Here
	\begin{align*}
		\Gamma_{+}\coloneqq\max{\left\{\Gamma,0\right\}},\quad B_{r}\left(x\right)\coloneqq \left\{y\in\mathbb{R}^{2}\colon \left|x-y\right|\leq r\right\}.
	\end{align*}
	A solution of this type was first found in \cite{smets-vanschaftingen} via variational methods. In \cite{DDMPVP2023} the authors perform a fixed point construction that obtains the same vortex pair solution and also provides the fine information needed for our global-in-time two vortex configuration constructed in this paper. This information is summarized in Theorems \ref{vortexpair-theorem} and \ref{vortex-pair-properties-theorem}, as well as Remark \ref{support-of-nonlinearity-remark}. See \cite{ao,cao,caowei,norbury} for related results and \cite{burton,norbury,turkington} for earlier travelling solutions.

	\medskip
	
	There is a 4 point solution the Kirchhoff-Routh vortex system \eqref{KR}, that as $t\to\infty$ behaves as two vortex pairs travelling in opposite directions, and is given by 
	\begin{align}
		\label{ppair2}
		\begin{aligned}
			m_1 &= - m_2 =M, 
			\quad 
			\xi_{*{},1} (t) = \left(  p_{*{}}(t)  , q_{*{}}(t)  \right),  
			&\xi_{*{},2} (t)&= \left( p_{*{}}(t)  ,-q_{*{}} (t) \right), 
			\\
			m_3 &= -m_4 =M, \quad
			\xi_{*{},3} (t) = \left(  -p_{*{}}(t)  , -q_{*{}}(t)  \right),
			&\xi_{*{},4} (t)&= \left( -p_{*{}}(t)  ,q_{*{}} (t) \right), 
		\end{aligned}
	\end{align}
	where, as $t\to\infty$,
	\begin{align}
		p_{*{}}(t) = P_{0}+\frac{M}{4 \pi q_\infty} t + O\Bigl(\frac{1}{t}\Bigr), \quad q_{*{}}(t) = q_\infty + O\Bigl(\frac{1}{t^2}\Bigr) ,\label{p-q-star-asymptotics}
	\end{align}
	and $P_{0}$, $q_\infty >0$, with $M>0$  defined in \eqref{power-semilinear-problem-R2}. See Section~\ref{point-vortex-trajectory-section} for details.
	
	\medskip
	It is thus natural to ask whether an
	initial vorticity consisting of four compactly supported $\varepsilon$-concentrated regular functions, identical except for alternating 
	sign, located symmetrically in the four quadrants, generates a  time evolved vorticity that remains localised around the four points $\xi_{*{},j} (t)$, $j=1, \ldots , 4$, at all times $t $. 
	We state our result in the following form
	
	\begin{theorem}\label{teo1} Let $\gamma>18$, $q_\infty >0$. Then for all $T_0>0 $ large enough satisfying \eqref{K-T0-choices}, there is a constant $C>0$ such that for all $\varepsilon >0$ small enough, there exists a solution to \eqref{2d-euler-vorticity-stream}, $(\omega_\ve , \Psi_\ve)$, on the whole interval $[T_0,\infty)$, with
		\begin{equation}\label{vorticityexp}
			\omega_\ve (x,t) =  U_{\ve,q_{*}+q_{re}}\left(x-\left(p_{*}+p_{re}\right)e_{1}\right)-U_{\ve,q_{*}+q_{re}}\left(x+\left(p_{*}+p_{re}\right)e_{1}\right) + \frac{1}{\ve^2} \phi_\ve (x,t),
		\end{equation}
		where the point $\left(p_{*}(t),q_{*}(t)\right)$ is defined in \eqref{ppair2}, and $p_{re}(t)$ and $q_{re}(t)$ are functions $[T_{0},\infty)\rightarrow\mathbb{R}$, and satisfy, for all $t\in\left[T_{0},\infty\right)$, 
		\begin{align*}
			\left|t^{-1}p_{re}\left(t\right)\right|+\left|\dot{p}_{re}\left(t\right)\right|&\leq C\ve^{2}, \\
			\left|t^{2}q_{re}\left(t\right)\right|+\left|t^{3}\dot{q}_{re}\left(t\right)\right|&\leq C\ve^{2}. 
		\end{align*}
		The function $\phi_\ve\left(\cdot,t\right) $ in \eqref{vorticityexp} is continuous with a modulus of continuity given by
		\begin{align*}
			\left|\phi_{\varepsilon}\left(x,t\right)-\phi_{\varepsilon}\left(x',t\right)\right|\leq\frac{C\varepsilon^{1-\sigma}}{\left(\log{\left(\log{\left(\frac{\varepsilon^{\frac{6-\sigma}{2}}}{\left|x-x'\right|^{\frac{1}{2}}}\right)}\right)}\right)^{\left(1-\sigma\right)}}, \quad \text{for all } t >T_0,
		\end{align*}
		odd in both variables $x_1$ and $x_2$, and satisfies
		$$
		|\phi_\ve (x,t) | \leq C \left(\frac{\ve}{t}\right)^{1-\sigma } , \quad \|\phi_\ve (\cdot,t)\|_{L^{2}\left(Q\right)}\leq C \left(\frac{\ve}{t}\right)^{2},  \quad \text{for all } t >T_0,
		$$
		for any $\sigma >0$ small.
		If restricted to the first quadrant $Q=\{ (x_1 , x_2) \, : \, x_1 >0 , x_2 >0 \}$ it satisfies
		$$
		{\mbox {supp}} \, \phi_\ve (x,t) \subseteq B_{C\ve } (\left(p_{*}+p_{re},q_{*}+q_{re}\right)),  \quad \text{for all } t >T_0.
		$$
	\end{theorem}
	
	The functions $ \Psi_\ve (x, t+T_0)$,  $ \omega_\ve (x, t+T_0)$,  $t>0$ form a solution of \eqref{2d-euler-vorticity-stream} in $[0,\infty)$ and have the expected properties in the whole interval $[0,\infty)$. The condition $\gamma>18$ is a technical hypothesis required in the proof of Lemma \ref{L2-interior-a-priori-estimate-lemma}. 
	\begin{remark}\label{scattering-remark}
		An important point to note about the construction of the solution in Theorem \ref{teo1} is that it is found as a limit of solutions to \eqref{2d-euler-vorticity-stream}, indexed by a parameter $T$, on a time interval $\left[T_{0},T\right]$, with $T\to\infty$. These solutions have terminal data specified at $t=T$ rather than initial data at $t=T_{0}$. The motivation for this can be seen at the ODE level, where the symmetric $4$-point solution described in \eqref{ppair2}--\eqref{p-q-star-asymptotics} is naturally characterized by terminal data of the form $\lim_{t\to\infty}q(t)=q_{\infty}$. This is, in a sense, specifying which vortex pairs of the form \eqref{xipair} at $t=\infty$ the $4$ point solutions are perturbations of.
		
		\medskip
		Our desingularization of the $4$ point solution follows the same heuristic, wherein we specify the desingularized vortex pair, depending on $q_{\infty}$ (see Section \ref{tvp}) that the solution described in Theorem \ref{teo1}, is evolving towards. Thus, it is natural to specify data at $t=\infty$ and evolve backwards, and since $T$ informally plays the role of $t=\infty$ when we construct our solution on larger and larger time intervals, it is natural to specify data at $t=T$ and evolve backwards. Such a construction can also be seen for a closely related system to the $2$D Euler Equations, the Vlasov-Poisson equation, in \cite{cagmaf}.
		
		\medskip
		The strategy described in this remark is elaborated upon during Section \ref{scheme}, where we explain the scheme of the proof.
	\end{remark}
	\begin{remark}\label{global-remark}
		Global-in-time can refer to solutions that exist as $t\to\pm\infty$, that is on the whole time interval $\left(-\infty,\infty\right)$, and due to the time reversibility of the Euler equations, one can ask the question of whether this desingularization can be extended to all negative times. By a simple change of a variables, one can construct a solution on $\left(-\infty,-T_{0}\right]$ using the solution constructed in \ref{teo1}. 
		
		\medskip
		What remains is whether these desingularizations can be continued through the interval $\left[-T_{0},T_{0}\right]$ in a compatible way. In this interval the dynamics of the vortices are more complicated, especially since $T_{0}$ will be chosen so that the dynamics of the vortices on the right hand plane are suitably decoupled from the dynamics of the vortices on the left hand side plane enabling us to work in a perturbative regime suitable to carry out the procedure described in Remark \ref{scattering-remark}, see Section \ref{scheme}. 
		
		\medskip
		The ODE dynamics suggest that on $\left[-T_{0},T_{0}\right]$ the two approximate vortex pairs will approach each other as $t\downarrow0$, and then ``swap" pairs so that the point vortices in the two quadrants corresponding to the upper half plane form a new approximate vortex pair that travels in the positive $e_{2}$ direction, and asymptotically approaches a true vortex pair as $t\downarrow-\infty$, and similarly for the lower half plane for the negative $e_{2}$ direction. Desingularizing these dynamics on $\left(-\infty,\infty\right)$ would be an interesting problem, especially given the fundamentally non-monotonic nature of the conjectured dynamics on $\left[-T_{0},T_{0}\right]$.
	\end{remark}
	Our result gives a contribution to the understanding of the long-time behaviour of solutions to \eqref{2d-euler-vorticity-stream},  a problem that, as we remarked previously, remains largely open. Indeed, to our knowledge, this is one of the first constructions of global-in-time desingularized vortex configurations outside of steady states, and the first of greater regularity than that of vortex patches. 
	
	\medskip
	We construct specific initial conditions that allow us to exhibit the long-time behaviour shown in Theorem \ref{teo1}, but we do not claim that small perturbations of our initial data will preserve this structure. Physical experiments and numerical simulations suggest that most solutions evolve over infinite time towards simpler dynamics. Orbital stability for vortex pairs was proven in \cite{bnl}, and orbital stability for Lamb dipoles was established in \cite{abe-choi}. A more recent result \cite{ChoiJeongYao} establishes orbital stability for odd-odd quadrupole configurations similar to the one considered in this work. Regarding the long-time behaviour of solutions to the $2D$ Euler equation, the following informal conjecture has been formulated in \cite{sve}, \cite{sch}: as $t\to \pm \infty$ generic solutions experience loss of compactness. 
	Significant confirmations of this conjecture have emerged   near specific steady states. 
	Important stability results for shear flows have been found in \cite{Bedrossian,Ionescu, ionescu2,masmoudizhao}. Stability results for fast-decay, radial steady states have been found \cite{Bedrossian,Ionescu,choi-lim}. See also \cite{elgindimurraysaid} for discrete symmetry configurations. 
	
	\medskip
	Our result is consistent with the confinement result in \cite{ISG} and provides much finer qualitative information on the vorticity at all times $t$.
	Vorticity confinement results regard bounding  the expansion of the support for the vorticity, when initially this is compactly supported and in $L^\infty$. 
	Since the velocity is bounded, the radius of the support can grow at most linearly in time. When the initial vorticity has a definite sign, the radius of the support grows more slowly. We refer the reader to  \cite{M, ISG, serfati} and references therein.
	
	\medskip
	There are several results on long time dynamics for vortex patches. In \cite{Zbarsky}, a global solution to \eqref{2d-euler-vorticity-stream} consisting of vortex patches centered around an expanding self-similar solution to \eqref{KR} with $N=3$ was found. In \cite{Zbarsky} each patch may grow in size at a rate $t^{\frac{1}{4}^+}$, but they separate from each other at a rate $t^{\frac{1}{2}}$. This seems crucial in the analysis. 
	The existence of co-rotating and counter-rotating pair vortex patches is contained in \cite{HmidiMateu}. Logarithmic spiraling solutions were recently constructed in \cite{Jeongsaid}. For other results concerning regularity and long-term behaviour of vortex patches, we refer to \cite{ej1} and to the results they cite. We also mention the interesting recent papers 
	\cite{HassHmidiMasmoudi,HassHmidiRoulley} where solutions such as those in the form of two vortex pairs crossing each other in an alternate, periodic manner, and a periodic solution involving $4$ vortex trajectories moving with odd-odd symmetry in the unit disk, are constructed using KAM theory.  
	
	\medskip
	As a final point in this introduction, we consider in this paper a geometrically simple configuration that we nevertheless claim  contains many of the difficulties of more complex situations. In particular, one can look for more general travelling solutions of the form 
	\begin{align}
		\xi^{AM}_{j}=p^{AM}_{j}e_{1}+q^{AM}_{j}e_{2}+c^{AM}te_{1},\quad \mathring{\xi}^{AM}_{j}= p^{AM}_{j}e_{1}+q^{AM}_{j}e_{2},\label{xi-AM}
	\end{align}
	with masses $m^{AM}_{j}$, for $j=1,\dots,N$, some positive integer $N$. In that case, using \eqref{KR} one would have to solve the system
	\begin{align}
		c^{AM}e_{1}=-\sum_{l\neq j}\frac{m^{AM}_{l}}{2\pi}\frac{\left(\mathring{\xi}^{AM}_{j}-\mathring{\xi}^{AM}_{l}\right)^{\perp}}{\left|\mathring{\xi}^{AM}_{j}-\mathring{\xi}^{AM}_{l}\right|^{2}},\quad j=1,\dots,N.\label{xi-AM-system}
	\end{align}
	One can indeed find such solutions $\left(\xi^{AM}_{1+},\dots,\xi^{AM}_{k+},\xi^{AM}_{1-},\dots,\xi^{AM}_{k-}\right)$ with $\xi^{AM}_{j\pm}$ and $\mathring{\xi}^{AM}_{j\pm}$ defined as in \eqref{xi-AM}, masses $m_{j+}=M$, and $m_{j-}=-M$ for $j=1,\dots,k$, where $k=\frac{n\left(n+1\right)}{2}$, some positive integer $n$. Moreover, these arrays of travelling vortices are called symmetric if they also satisfy
	\begin{equation}\label{xi-AM-symmetric}
		\begin{aligned}
			\left(p^{AM}_{j+},q^{AM}_{j+}\right)&=\left(p^{AM}_{j-},-q^{AM}_{j-}\right),\quad j=1,\dots,k,\\
			\exists j_{0}, 2j_{0}+1\leq k\ \text{such that}\ \left(p^{AM}_{\left(2j-1\right)+},q^{AM}_{\left(2j-1\right)+}\right)&=\left(-p^{AM}_{\left(2j\right)+},q^{AM}_{\left(2j\right)+}\right),\quad j=1,\dots,j_{0},\\
			p^{AM}_{j+}&=0,\quad j=2j_{0}+1\dots,k.
		\end{aligned}
	\end{equation}
	There are solutions to \eqref{xi-AM-system}--\eqref{xi-AM-symmetric} related to the roots of the Adler-Moser polynomials, which were used to construct multi-vortex travelling waves to the Gross-Pitaevskii equation in \cite{LiuWei}. A closely related system was also used to find multi-vortex travelling wave solutions to the SQG system in \cite{ao}.
	
	\medskip
	In the same way that constructing the desingularized vortex pair for \eqref{2d-euler-vorticity-stream} in \cite{DDMPVP2023} using Lyapunov-Schmidt reduction followed an analogous strategy to the vortex pair construction in \cite{ao}, constructing desingularized travelling waves based on the symmetric arrays \eqref{xi-AM-system}--\eqref{xi-AM-symmetric} for \eqref{2d-euler-vorticity-stream} would closely follow the strategy for the analogous construction in \cite{ao}.
	
	\medskip
	Having constructed desingularized symmetric arrays of travelling vortices, the methods in this current work are robust enough to be able to desingularize vortex configurations consisting at main order of two of these symmetric arrays travelling away from each other at linear speed, analogously to how the 4 point solution \eqref{ppair2} relates to the pair \eqref{xipair}.
	
	\medskip
	We discuss the scheme of the proof in the next section.

	\section{Scheme of the proof}\label{scheme}
	
	The strategy we adopt to construct the  solution predicted by Theorem \ref{teo1}  is to first  produce a solution to the Euler equations \eqref{2d-euler-vorticity-stream} on a finite and arbitrarily large time interval $[T_0, T]$, for  $T \gg T_0$, for $T_{0}$ defined in \eqref{K-T0-choices}. This solution is built as a perturbation of the sum of two $\varepsilon$-concentrated vortex pairs travelling in opposite directions. The perturbation will be small and decay in time.  To conclude our argument we  apply the Ascoli-Arzel\'a Theorem for a sequence of times $(T_n)_n$, with $T_n\to \infty$ as $n \to \infty$ and show that the limit function gives the desired solution in the whole interval  $[T_0, \infty)$. We do this in Section \ref{conclusion}.
	
	\medskip
	The main step in our proof is the construction of the solution to the Euler equations \eqref{2d-euler-vorticity-stream} on $[T_0,T]$ as described above. As previously mentioned, the starting point for this construction is the basic $\ve$-concentrated travelling vortex pair of which we need a fine understanding. This precise information is contained in \cite{DDMPVP2023} and is summarized in Section \ref{tvp}. The  approximate solution is built upon the sum of two vortex pairs corrected by a relatively explicit function in order to produce a sufficiently small error. We describe this in Section \ref{appro}. Detailed proofs are contained in Section \ref{first-approximation-section}. We then solve the problem on $[T_0,T]$, as explained in Section \ref{cs}, with full details in Section \ref{constructing-full-solution-section}.

	\subsection{Travelling vortex pair}\label{tvp}
	Recall \eqref{stationary-solutions-elliptic-pde}, which tells us that to achieve a vortex pair with $\varepsilon$-concentrated vorticity, we want solve the semilinear elliptic PDE
	\begin{align}
		\Delta\Psi_{\varepsilon}+\varepsilon^{-2}f_{\varepsilon}(\Psi_{\varepsilon}-cx_{2},x)=0, \quad \omega_\ve = \varepsilon^{-2}f_{\varepsilon}(\Psi_{\varepsilon}-cx_{2},x) \label{2d-euler-travelling-semilinear-elliptic-equation-epsilon}
	\end{align}
	on $\mathbb{R}^{2}$, for a choice of power type nonlinearity  $f_\ve$ that we introduce next.
	
	\medskip
	In Section \ref{introduction-section} we introduced the function  $\Gamma$ solving \eqref{power-semilinear-problem-R2}, which we can write as
	\begin{align}
		\Gamma(x)=\nu\left(\left|x\right|\right),\ \text{if}\ \left|x\right|\leq 1,\quad \Gamma(x)=\nu'(1)\log{{\left|x\right|}},\ \text{if}\ \left|x\right|> 1,\label{Gamma-def}
	\end{align}
	where $\nu (|x|)$ is the unique positive, radial, ground state solving
	\begin{align}
		\Delta \nu+\nu^{\gamma}_{+}=0 \ \text{on}\ B_{1}\left(0\right),\quad \nu=0\ \text{on}\ \dell B_{1}\left(0\right).\label{nu-equation}
	\end{align}
	Recalling that 
	$M=\int_{\mathbb{R}^{2}}\Gamma^{\gamma}_{+}dx$
	a direct computation gives $M=2\pi\left|\nu'(1)\right|.$
	The quantity $\left|\nu'(1)\right|$, an important quantity which we call the reduced mass, will be labelled by $m$ so that
	\begin{align}
		m\coloneqq\frac{M}{2\pi}=\left|\nu'(1)\right|.\label{reduced-vortex-mass-definition}
	\end{align}
	For any $\varepsilon >0$, rescaling $\Gamma$ by $\varepsilon$ leaves it $C^1$ across the boundary of the ball radius $\varepsilon$. Thus, the rescaled $\Gamma$ solve the following family of semilinear elliptic equations, with associated mass
	\begin{equation}\label{Gamma-epsilon-elliptic-equation}
		\begin{aligned}
			\Delta\Gamma\left(\frac{x}{\varepsilon}\right)+\varepsilon^{-2}\left(\Gamma\left(\frac{x}{\varepsilon}\right)\right)^{\gamma}_{+}=0,\quad \varepsilon^{-2}\int_{\mathbb{R}^{2}}\left(\Gamma\left(\frac{x}{\varepsilon}\right)\right)^{\gamma}_{+}dx=\int_{B_{1}}\Gamma^{\gamma}_{+}\left(z\right)\  dz=M.
		\end{aligned}
	\end{equation}
	We take a travelling vortex pair solution using the rescaled $\Gamma$ and our family of $f_{\varepsilon}$ will be chosen using the power type nonlinearity in \eqref{Gamma-epsilon-elliptic-equation}. Concretely, let $q_\infty>0$. Then our solution has the form
	\begin{equation} \label{travelling-vortex-pair-ansatz}
		\begin{aligned}
			\Psi_{\varepsilon}(x)&=\mathring{\Psi}_{\varepsilon}(x)+\psi_{\frac{q_\infty}{\varepsilon}}\left(\frac{x}{\varepsilon}\right), \quad
			\mathring{\Psi}_{\varepsilon}(x)=\Gamma\left(\frac{x-q_\infty e_{2}}{\varepsilon}\right)-\Gamma\left(\frac{x+q_\infty e_{2}}{\varepsilon}\right),
		\end{aligned}
	\end{equation}
	with $f_{\varepsilon}$ defined by
	\begin{align}
		f_{\varepsilon}(\Psi_{\varepsilon}(x)-cx_{2},x)&=\left(\Psi_{\varepsilon}(x)-cx_{2}-\left|\log{\varepsilon}\right|\Omega\right)^{\gamma}_{+}\chi_{B_{s}\left(q_\infty e_{2}\right)}(x)\nonumber\\
		&-\left(-\Psi_{\varepsilon}(x)+cx_{2}-\left|\log{\varepsilon}\right|\Omega\right)^{\gamma}_{+}\chi_{B_{s}\left(-q_\infty e_{2}\right)}(x).\label{nonlinearity-definition}
	\end{align}
	Here $s>0$ is a fixed constant, and the $\chi$ are indicator functions. To find such a solution requires choosing our $c$ and $\Omega$ as a function of $q_\infty$. The quantity $\Omega$ is a parameter to adjust the mass of our vortex pair. 
	We demand that
	\begin{align}
		-\Gamma\left(\frac{2q_\infty e_{2}}{\varepsilon}\right)-cq_\infty-\left|\log{\varepsilon}\right|\Omega=0.\label{Omega-definition-1}
	\end{align}
	If we  suppose  $q_\infty\in\left(Q_{0}^{-1},Q_{0}\right)$, some $Q_{0}>0$ large and   $c=c(q_\infty,\varepsilon)$ has the bound $c\leq mQ_{0}$ for all $\varepsilon$ small enough, then 
	the support of $f_{\varepsilon}(\mathring{\Psi}_{\varepsilon}(x)-cx_{2},x)$ shrinks as $\varepsilon\to0$:   there exists $\varepsilon_{0}$ small enough such that for all $\varepsilon\leq\varepsilon_{0}$,  
	\begin{align}
		{\mbox {supp}} \, f_{\varepsilon}\left(\mathring{\Psi}_{\varepsilon}(x)-cx_{2},x\right) \subseteq B_{\rho_0\varepsilon}(q_\infty e_{2})\cup B_{\rho_0\varepsilon}(-q_\infty e_{2}),\quad \rho_0=2\left(1+\frac{s Q_{0}}{2}\right)e^{Q_{0}^{2}},\label{rho-0-definition}
	\end{align} 
	Next we note that $\psi_{\frac{q_\infty }{\varepsilon}}$ in \eqref{travelling-vortex-pair-ansatz} is a comparatively small error term. To main order, both our solution and nonlinearity look like an odd extension of the solution and nonlinearity defined in \eqref{Gamma-epsilon-elliptic-equation}. The solution $\psi_{\frac{q_\infty }{\varepsilon}}$ to our problem has the same oddness in the $x_{2}$ direction.
	
	\medskip
	The existence and properties of our $\varepsilon$-concentrated travelling vortex pair are summarized in the following result. The existence of such a solution can be seen in \cite{smets-vanschaftingen}, and the relevant properties of the solution claimed in the following theorem are shown in \cite{DDMPVP2023}.

	\begin{theorem}\label{vortexpair-theorem} Let $\gamma > 3$ in \eqref{power-semilinear-problem-R2} and let $q_\infty >0.$ There exist a constant $C>0$ such that for all $\varepsilon >0$ small, there is a constant $c>0$ and a function $\psi_{\frac{q_\infty}{\varepsilon}} (x)$, even in $x_1$, odd in $x_2$ such that $\Psi_\ve$ given by \eqref{travelling-vortex-pair-ansatz} with $\Omega$ as in \eqref{Omega-definition-1} solves  \eqref{2d-euler-travelling-semilinear-elliptic-equation-epsilon}, with $f_\varepsilon$ given by \eqref{nonlinearity-definition}. Moreover. there exists a function $g$ of regularity $C^{\floor{\gamma}-1}$ in its argument such that the speed $c$ is given by
		\begin{equation}\label{c-def}
			c= \frac{m}{2q_\infty} + g(q_{\infty}' ),\quad q_{\infty}'=\frac{q_\infty}{\varepsilon},\end{equation}
		and the following estimates hold true
		\begin{align}
			\| \psi_{q_{\infty}'} \|_{L^\infty (\mathbb{R}^2)} &+  \| \partial_{q_{\infty}'} \psi_{q'} \|_{L^\infty (\mathbb{R}^2)}+\| \partial_{q_{\infty}'}^{2} \psi_{q'} \|_{L^\infty (\mathbb{R}^2)} \leq C \ve^2,\label{psi-q-l-infinity-bounds} \\
			\varepsilon^{2}|g| &+ \varepsilon|\partial_{q_{\infty}'} g|+|\partial_{q_{\infty}'}^{2} g| \leq C \ve^4.\label{g-pointwise-bounds}
		\end{align}
	\end{theorem}
	
	\begin{remark}\label{support-of-nonlinearity-remark}
		The support of $f_{\varepsilon}\left(\Psi_{\varepsilon}(x)-cx_{2},x\right)$, with $\Psi_\varepsilon$ given by \eqref{travelling-vortex-pair-ansatz}, and $\psi_{\frac{q_\infty }{\varepsilon}}\in L^{\infty}(\mathbb{R}^{2})$ satisfying bounds \eqref{psi-q-l-infinity-bounds}, is such that
		\begin{align}
			\supp{f_{\varepsilon}\left(\Psi_{\varepsilon}(x)-cx_{2},x\right)}\subseteq B_{\rho\varepsilon} (q_\infty e_{2})\cup B_{\rho \varepsilon}(-q_\infty e_{2}).\label{vortex-pair-support-true}
		\end{align}
		for some $\rho>\rho_{0} >0$ independent of $\varepsilon$, $\rho_{0}$ defined in \eqref{rho-0-definition}.
		
		\medskip
		It will be crucial throughout this paper that the identity 
		\begin{align}
			\gradperp_{x}\left(\Psi_{\varepsilon}(x)-cx_{2}\right)\cdot\nabla_{x}\left(\varepsilon^{-2}f_{\varepsilon}(\Psi_{\varepsilon}(x)-cx_{2},x)\right)=0\label{gradperp-grad-identity-vortex-pair}
		\end{align}
		holds for $c$ and $q_{\infty}$ related as in \eqref{c-def}. This is not immediate from \eqref{nonlinearity-definition}, as there is an extra dependence on $x$ for $f_{\varepsilon}$ through the indicator functions $\chi_{B_{s}\left(\pm q_\infty e_{2}\right)}(x)$. However, from \eqref{vortex-pair-support-true}, we can see that these indicator functions will be $\equiv1$ on $B_{\rho\varepsilon} (\pm q_\infty e_{2})$ respectively, and that $f_{\varepsilon}$ will be $0$ outside of $B_{\rho\varepsilon} (q_\infty e_{2})\cup B_{\rho \varepsilon}(-q_\infty e_{2})$. This means that \eqref{gradperp-grad-identity-vortex-pair} will hold.
	\end{remark}
	Due to the fact that the indicator functions are identically $1$ on the support of the nonlinearity, we will slightly abuse notation and write $f_{\varepsilon}(\Psi_{\varepsilon}(x)-cx_{2},x)=f_{\varepsilon}(\Psi_{\varepsilon}(x)-cx_{2})$ throughout the rest of this paper. We also note that due to this property of the indicator functions being $\equiv1$ on the support of the nonlinearity, for small enough $\varepsilon>0$, and $f_{\varepsilon}$ defined in \eqref{nonlinearity-definition},
	\begin{align}
		f_{\varepsilon}(\Psi_{\varepsilon}(x)-cx_{2})\in\twopartdef{C^{\floor{\gamma},\gamma-\left[\gamma\right]}}{\gamma\notin \mathbb{Z}}{C^{\gamma-1}}{\gamma\in\mathbb{Z}}.\label{nonlinearity-regularity}
	\end{align}
	We also record some important properties of the vortex pair constructed in \cite{DDMPVP2023} here.
	\begin{theorem}\label{vortex-pair-properties-theorem}
		As in \eqref{travelling-vortex-pair-ansatz}, for $q_{\infty}'$ defined in \eqref{c-def}, let
		\begin{align*}
			\mathring{\Psi}_{\varepsilon}(x)=\Gamma\left(\frac{x-q_\infty e_{2}}{\varepsilon}\right)-\Gamma\left(\frac{x+q_\infty e_{2}}{\varepsilon}\right),\quad 
			\Psi_{\varepsilon}(x)=\mathring{\Psi}_{\varepsilon}(x)+\psi_{q_{\infty}'}\left(\frac{x}{\varepsilon}\right).
		\end{align*}
		Then in coordinates $y=\varepsilon^{-1}x$, we have:
		\begin{enumerate}
			\item For $\Omega$ defined in \eqref{Omega-definition-1} and $y\in B_{10\rho}\left(q_{\infty}'e_{2}\right)$, we have the following expansions,
			\begin{align*}
				&\mathring{\Psi}_{\varepsilon}(y)-c\varepsilon y_{2}-\left|\log{\varepsilon}\right|\Omega=\Gamma\left(\left|y-q_{\infty}'e_{2}\right|\right)-\varepsilon\left(y_{2}-q_{\infty}'e_{2}\right)g(q_{\infty}')\\
				&+\frac{m\varepsilon^{2}}{8 q_{\infty}^{2}}\left(y_{1}^{2}-\left(y_{2}-q_{\infty}'e_{2}\right)^{2}\right)\\
				&+\frac{m\varepsilon^{3}}{24q_{\infty}^{3}}\left(4\left(y_{2}-q_{\infty}'\right)^{3}-3\left(y_{2}-q_{\infty}'\right)\left|y-q_{\infty}'e_{2}\right|^{2}\right)+\bigO{\left(\varepsilon^{4}\right)},
			\end{align*}
			\begin{align*}
				&\Psi_{\varepsilon}(y)-c\varepsilon y_{2}-\left|\log{\varepsilon}\right|\Omega=\\
				&\Gamma\left(\left|y-q_{\infty}'e_{2}\right|\right)-\varepsilon\left(1-\frac{\varrho_{1}(\left|y-q_{\infty}'e_{2}\right|)}{\left|y-q_{\infty}'e_{2}\right|}\right)\left(y_{2}-q_{\infty}'e_{2}\right)g(q_{\infty}')\\
				&+\frac{m\varepsilon^{2}}{8 q_{\infty}^{2}}\left(1-\frac{\varrho_{2}(\left|y-q_{\infty}'e_{2}\right|)}{\left|y-q_{\infty}'e_{2}\right|^{2}}\right)\left(y_{1}^{2}-\left(y_{2}-q_{\infty}'e_{2}\right)^{2}\right)\\
				&+\frac{m\varepsilon^{3}}{24q_{\infty}^{3}}\left(1-\frac{\varrho_{3}(\left|y-q_{\infty}'e_{2}\right|)}{\left|y-q_{\infty}'e_{2}\right|^{3}}\right)\left(4\left(y_{2}-q_{\infty}'\right)^{3}-3\left(y_{2}-q_{\infty}'\right)\left|y-q_{\infty}'e_{2}\right|^{2}\right)+\bigO{\left(\varepsilon^{4}\right)},
			\end{align*}
			with $\varrho_{j}$ defined in \eqref{vortex-linearised-equation-rk-error-solution}, and an analogous statement holding on $B_{10\rho}\left(-q_{\infty}'e_{2}\right)$.
			\item We have
			\begin{align*}
				\|\Psi_{\varepsilon}(y)-c\varepsilon y_{2}-\left|\log{\varepsilon}\right|\Omega-\Gamma\left(y-q_{\infty}'e_{2}\right)\|_{C^{2}\left(B_{10\rho}\left(q_{\infty}'e_{2}\right)\right)}\leq C\varepsilon^{2},
			\end{align*}
			where the gradients are taken with respect to $y$. An analogous statement holds on $B_{10\rho}\left(-q_{\infty}'e_{2}\right)$.
			\item As $\left|y-q_{\infty}'e_{2}\right|\uparrow1$, we have
			\begin{align*}
				\Gamma\left(y-q_{\infty}'e_{2}\right)\sim \left(1-\left|y-q_{\infty}'e_{2}\right|\right)_{+}.
			\end{align*}
			An analogous statement holds on $B_{1}\left(-q_{\infty}'e_{2}\right)$.
		\end{enumerate}
	\end{theorem}
	The proofs for the first two properties are found in \cite{DDMPVP2023}, and the last property can be seen from \eqref{nu-equation}, the ODE in radial coordinates that $\nu$ solves on $B_{1}$. The corresponding behaviour of $f_{\varepsilon}\left(\Psi_{\varepsilon}(y)-c\varepsilon y_{2}\right)$ and its derivatives can then be inferred from \eqref{nonlinearity-definition}, the bounds \eqref{psi-q-l-infinity-bounds} and \eqref{g-pointwise-bounds} from Theorem \ref{vortexpair-theorem}, and Theorem \ref{vortex-pair-properties-theorem}. We have stated the estimates and expansions on the closed ball $B_{10\rho}\left(q'_{\infty}e_{2}\right)$ for concreteness.
	
	\medskip
	As a final remark in this section, the properties stated in Theorems \ref{vortexpair-theorem} and \ref{vortex-pair-properties-theorem}, as well as Remark \ref{support-of-nonlinearity-remark} hold true for time dependent $q(t)$ and $c(t)$ as long as $q(t)\in\left(Q_{0}^{-1},Q_{0}\right)$.
	\subsection{Modified vortex trajectories}

	From Theorem \ref{vortexpair-theorem} we learn that the speed of an $\varepsilon$-concentrated travelling vortex pair is the speed of an idealized travelling vortex pair $\frac{m}{2q_\infty}$ modified by a  correction $g$ of size $\ve^2$
	\begin{align*}
		c=\frac{m}{2q_\infty }+g\left(q_{\infty}'\right), \quad q_{\infty}'=\frac{q_\infty}{\varepsilon}.
	\end{align*}
	The idealized system \eqref{KR} for $N=4$ will need to be adjusted to incorporate this error term, and the idealized point vortex solution \eqref{ppair2} will require an additional term of magnitude $\varepsilon^2$. 
	
	\medskip
	Write the ODE system \eqref{KR} with $N=4$ for the points $\xi_{*{},j}$, $j=1,2,3,4$ as in \eqref{ppair2} as
	\begin{equation}\label{pkr-system}
		\frac{d}{dt}\begin{pmatrix}
			p_{*}  \\
			q_{*}
		\end{pmatrix}=\begin{pmatrix}
			\mathcal{F}\left(q_{*{}},p_{*{}}\right) \\
			-\mathcal{F}\left(p_{*{}},q_{*{}}\right)
		\end{pmatrix},\quad \mathcal{F}\left(x,y\right)\coloneqq \frac{m}{2}\left(\frac{1}{x}-\frac{x}{x^{2}+y^{2}}\right)=\frac{m}{2}\left(\frac{y^{2}}{x\left(x^{2}+y^{2}\right)}\right)
	\end{equation}
	Here, and for the rest of the paper, we define, for $p_{*}(t)+iq_{*}(t)=r(t)e^{i\vartheta(t)}$, 
	\begin{align}
		\lim_{t\to\infty}q_{*{}}(t)=q_{\infty}\in\left(Q_{0}^{-1},Q_{0}\right),\quad \vartheta(0)=\vartheta_{0}\in\left(0,\frac{\pi}{4}\right).\label{q-infinity-definition-full-dynamic-problem}
	\end{align}
	Then the modified vortex trajectories $\xi_0 (t) = \left( p_0 (t) , q_0 (t) \right)$ obey the modified system
	\begin{equation}\label{p0-q0-system}
		\begin{aligned}
			\dot{p}_{0}&=\mathcal{F}\left(q_{0},p_{0}\right)+g\left(q_{0}'\right),\quad
			\dot{q}_{0}=-\mathcal{F}\left(p_{0},q_{0}\right)
		\end{aligned}
	\end{equation}
	where $g$ is the $\ve^2$-correction to the speed of an idealized vortex pair, as described in Theorem \ref{vortexpair-theorem} and \eqref{c-def}. 
	In Section \ref{point-vortex-trajectory-section}, Theorem \ref{p0-q0-construction-theorem}, we construct a solution to the system \eqref{p0-q0-system} on $[T_{0},\infty)$ of the form
	\begin{equation}\label{xi0def}
		\xi_{0}(t)=\xi_{*{}}(t)+\overline{\xi}(t),\quad
		\overline{\xi}(t)=\left(\overline{p}(t),\overline{q}(t)\right),
	\end{equation}
	where $\overline{\xi}$ will be a small error term satisfying \eqref{points}.

	\subsection{Approximate solution on $[T_0,T]$}\label{appro}

	Our initial step in developing a two-vortex pair solution to the Euler equations \eqref{2d-euler-vorticity-stream} is to create a reasonably accurate first approximation within a finite time interval $\left[T_{0}, T\right]$. We treat $T$ as an arbitrarily large parameter, with the intention of eventually taking $T\to\infty$ to construct our complete solution over the interval $\left[T_{0},\infty\right)$. The approximation as well as the solution are odd in both spatial variables $x_1$ and $x_2$. 
	
	\medskip
	To describe the approximate solution  we first introduce $\xi(t)$, a time-dependent vector, moving in the first quadrant $x_1 \geq 0$, $x_2 \geq 0$. We assume it has the form  
	\begin{align}
		\xi(t)=(p(t),q(t)),\quad \xi(t)=\xi_{0}(t) +\hat \xi (t), \quad {\mbox {with}} \quad \hat \xi (t) =\xi_{1}(t)+\tilde{\xi}(t).\label{point-vortex-trajectory-function-decomposition-new}
	\end{align}
	The first term $\xi_{0}(t)=\left(p_{0}(t),q_{0}(t)\right)$ is explicit and given by \eqref{xi0def}. It will satisfy
	\begin{equation}\label{xi0-bound}
		\begin{aligned}
			\|t^{-1}p_{0}\|_{[T_{0},\infty)}+\|\dot{p}_{0}\|_{[T_{0},\infty)}+\|q_{0}\|_{[T_{0},\infty)}+\|t^{3}\dot{q}_{0}\|_{[T_{0},\infty)}\leq C_{1},
		\end{aligned} 
	\end{equation}
	where the norm is the $L^{\infty}$ norm on $[T_{0},\infty)$. The point $\xi_{1}(t)=\left(p_{1}(t),q_{1}(t)\right)$ will also be determined in the process of the construction of
	the approximation, and will be obtained as part of a solvability condition for an elliptic problem. It will satisfy
	\begin{equation}\label{xi1-bound}
		\begin{aligned}
			\|tp_{1}\|_{[T_0 , \infty)}+\|t^{2}\dot{p}_{1}\|_{[T_0 , \infty)}+\|t^{2}q_{1}\|_{[T_0 , \infty)}+\|t^{3}\dot{q}_{1}\|_{[T_0 , \infty)}\leq C\varepsilon^{3},
		\end{aligned}
	\end{equation}
	for some constant $C>0$. The point 
	$\tilde{\xi}(t)=\left(\tilde{p}(t),\tilde{q}(t)\right)$ is a free parameter to adjust at the end of our construction to obtain a true solution on $[T_0,T]$. For the
	moment, we ask that it is a continuous function in $[T_0, T ]$ for which $\dot{\tilde {\xi}}$ exists and such that
	\begin{equation}\label{tildexi-bound}
		\begin{aligned}
			\|t^{2}\tilde{p}\|_{\left[T_{0},T\right]}+\|t^{3}\dot{\tilde{p}}\|_{\left[T_{0},T\right]}+\|t^{3}\tilde{q}\|_{\left[T_{0},T\right]}+\|t^{4}\dot{\tilde{q}}\|_{\left[T_{0},T\right]}\leq \varepsilon^{4-\sigma},\ \tilde{p}\left(T\right)=\tilde{q}\left(T\right)=0,
		\end{aligned}
	\end{equation}
	for some $\sigma >0$ arbitrarily small, and where the norm is the $L^{\infty}$ norm on $\left[T_{0},T\right]$.
	
	\medskip
	Having introduced the point $\xi (t)$ in \eqref{point-vortex-trajectory-function-decomposition-new} and bounds \eqref{xi1-bound}-\eqref{tildexi-bound}, we let 
	$c(t)$ be dependent on $q(t)$ as in \eqref{c-def}, and $\Omega(t)$ as in \eqref{Omega-definition-1} so that
	\begin{align}
		c(t)=\frac{m}{2q(t)}+g\left(\frac{q(t)}{\varepsilon}\right),\ \Omega(t)=m+\left(\frac{m\log{2q(t)}-c(t)q(t)}{\left|\log{\varepsilon}\right|}\right).\label{c-q-Omega-time-dependent-relationship-new}
	\end{align}
	We start our construction with a stream function $\Psi_0 [\xi] (x,t)$ defined as the sum of the stream function of a vortex pair travelling to the right, denoted as $\Psi_R$, and one of a vortex pair travelling to the left, denoted by $\Psi_L$. We proceed in a similar way with the vorticity. We write
	\begin{equation} \label{psi-0-definition}
		\begin{aligned}
			\Psi_{0}[\xi](x,t)=\Psi_{R}(x,t)-\Psi_{L}(x,t),\quad
			\omega_{0}[\xi](x,t)=\varepsilon^{-2}\left(U_{R}(x,t)-U_{L}(x,t)\right),
		\end{aligned}
	\end{equation}
	with
	\begin{equation}\label{psi-r-definition}
		\begin{aligned}
			\Psi_{R}(x,t)&=\Gamma\left(\frac{x-pe_{1}-qe_{2}}{\varepsilon}\right)-\Gamma\left(\frac{x-pe_{1}+qe_{2}}{\varepsilon}\right)+\psi_{\frac{q}{\varepsilon}}\left(\frac{x-pe_{1}}{\varepsilon}\right),\\
			\Psi_{L}(x,t)&=\Gamma\left(\frac{x+pe_{1}-qe_{2}}{\varepsilon}\right)-\Gamma\left(\frac{x+pe_{1}+qe_{2}}{\varepsilon}\right)+\psi_{\frac{q}{\varepsilon}}\left(\frac{x+pe_{1}}{\varepsilon}\right),
		\end{aligned}
	\end{equation}
	and
	\begin{equation}\label{w-r-definition}
		\begin{aligned}
			U_{R}(x,t)
			=f_{\varepsilon}\left(\Psi_{R}-cx_{2}\right)&=\left(\Psi_{R}(x)-cx_{2}-\left|\log{\varepsilon}\right|\Omega\right)^{\gamma}_{+}\chi_{B_{s}\left(qe_{2}\right)}\left(x-pe_{1}\right)\\
			&-\left(-\Psi_{R}(x)+cx_{2}-\left|\log{\varepsilon}\right|\Omega\right)^{\gamma}_{+}\chi_{B_{s}\left(-qe_{2}\right)}\left(x-pe_{1}\right),\\
			U_{L}(x,t)
			=f_{\varepsilon}\left(\Psi_{L}+cx_{2}\right)&=\left(\Psi_{L}(x)+cx_{2}-\left|\log{\varepsilon}\right|\Omega\right)^{\gamma}_{+}\chi_{B_{s}\left(qe_{2}\right)}\left(x+pe_{1}\right)\\
			&-\left(-\Psi_{L}(x)-cx_{2}-\left|\log{\varepsilon}\right|\Omega\right)^{\gamma}_{+}\chi_{B_{s}\left(-qe_{2}\right)}\left(x+pe_{1}\right).
		\end{aligned}
	\end{equation}
	Here we recall the definition of 
	$\Gamma$ from \eqref{power-semilinear-problem-R2}, $\psi_{\frac{q}{\varepsilon}}$ from \eqref{travelling-vortex-pair-ansatz} with the definition of $\xi$ \eqref{point-vortex-trajectory-function-decomposition-new} in hand, and $c(t),\Omega(t)$ related to $q(t)$ by \eqref{c-q-Omega-time-dependent-relationship-new}. Moreover, by \eqref{nonlinearity-regularity} we have
	\begin{align}
		U_{R}(x,t),U_{L}(x,t)\in\twopartdef{C^{\floor{\gamma},\gamma-\left[\gamma\right]}}{\gamma\notin \mathbb{Z}}{C^{\gamma-1}}{\gamma\in\mathbb{Z}}.\label{UR-UL-regularity}
	\end{align}
	We introduce the Euler operators
	\begin{equation}\label{euler-operators}
		\begin{aligned}
			E_{1}\left[\Psi , \omega \right]&=\dell_{t}\omega(x,t)+\gradperp\Psi(x,t)\cdot\nabla\omega(x,t),\\
			E_{2}\left[\Psi, \omega\right]&=\Delta\Psi(x,t)+\omega(x,t).
		\end{aligned}
	\end{equation}
	The error of approximation for $\left(\Psi_0 [\xi] (x,t), \omega_0 [\xi] (x,t) \right)$ is an odd function in both variables $x_1$ and $x_2$ such that, when restricted to the first quadrant, for some constant $C>0$ it satisfies
	$$
	|E_1 \left[\Psi_0 [\xi], \omega_0 [\xi]\right] | \leq \frac{C\ve^{-2}}{t^{2} }, \quad {\mbox {supp}} \, \left ( E_1 \left[\Psi_0[\xi], \omega_0 [\xi]\right] \right) \subseteq B_{\mathcal{K}\ve} (p,q),
	$$
	where $\mathcal{K}$ is a fixed positive number independent of $\ve$.  By construction we have, for all $t\in\left[T_{0},T\right]$,
	\begin{align}
		-\Delta \Psi_{j}(x,t)=\varepsilon^{-2}U_{j}(x,t),\ \ \ j=L,R, \implies E_2  \left[\Psi_0 [\xi], \omega_0 [\xi] \right] =0.\label{Psi-U-equation}
	\end{align}
	In Section \ref{first-approximation-section} we prove there exists an improved approximation $(\Psi_* [\xi] , \omega_*[\xi]  )$ 
	such that the new error $E_1 \left[\Psi_* , \omega_* \right]$ is odd in $x_1$ and $x_2$, and  when restricted to the first quadrant, for some constant $C>0$ it satisfies, for $\mathcal{K}_1$ a fixed positive number independent of $\ve>0$,
	\begin{align*}
		&|E_1 \left[\Psi_* [\xi], \omega_* [\xi]\right] | \leq  \frac{C\ve}{t^{3}}, \quad {\mbox {supp}} \, \left ( E_1 \left[\Psi_* [\xi], \omega_* [\xi] \right] \right) \subseteq B_{\mathcal{K}_1 \ve} (p,q),\\
		&|E_2 \left[\Psi_* [\xi], \omega_* [\xi]\right] | \leq \frac{C\ve^{5}}{t^{5}}.
	\end{align*}
	A crucial step in the construction of the improved approximation is to recognize the very special structure of the principal parts of the linear terms in $E_1 \left[\Psi_0 [\xi] + \Psi , \omega_0 [\xi]+ \omega \right]$. This structure manifests on the right hand side of $\mathbb{R}^{2}$ in the form of a linear elliptic problem, as
	\begin{align}
		\gradperp\left(\Psi_{R}-cx_{2}-\left|\log{\varepsilon}\right|\Omega\left(t\right)\right)\cdot\nabla\left(\mathscr{L}\left(\psi\right)\right)=G,\label{linear-elliptic-strategy}
	\end{align}
	where $\mathscr{L}$ is the linearized operator around the vortex pair travelling to the right, for some right hand side $G$. By symmetry, there is a completely analogous structure on the left hand side of the plane. As it turns out, a key insight will be the fact that we can actually write $G$ as
	\begin{align}
		G=\gradperp\left(\Psi_{R}-cx_{2}-\left|\log{\varepsilon}\right|\Omega\left(t\right)\right)\cdot\nabla\tilde{G}+\hat{G},\label{linearized-operator-error-structure}
	\end{align}
	where $\tilde{G}$ has desirable properties, and $\hat{G}$ is small enough to deal with in the last step of the construction. These observations allow us to construct the first approximation by solving elliptic problems of the form
	\begin{align}
		\mathscr{L}\left(\psi\right)=\tilde{G}_{1}\label{vortex-pair-linear-elliptic-problem}
	\end{align}
	for some $\tilde{G}_{1}$ related to $\tilde{G}$ and the kernel of the linearized operator $\mathscr{L}$. Discovering this structure required  precise knowledge on how the vortex pair depends on $q(t)$, the motivation for \cite{DDMPVP2023}.
	
	\medskip
	This allows us to improve the initial approximation $(\Psi_0 [\xi], \omega_0 [\xi])$ solving linear elliptic problems. This can be done  provided the  point $\xi_1$ in  the decomposition of $\xi$ in \eqref{point-vortex-trajectory-function-decomposition-new} is properly adjusted and under the assumption  that $\tilde \xi$ satisfies \eqref{tildexi-bound}. The approximate vorticity $\omega_* [\xi]$ is found of the form
	\begin{equation}\label{omega*}
		\omega_* = \omega_0 + \ve^{-2} \left( \phi_R - \phi_L\right),\quad \| \phi_R \|_{L^\infty (\R^2 )} \leq  \frac{C\ve^2}{t^2}\quad \supp{\phi_{R}}=\supp{U_{R}},
	\end{equation}
	with $\phi_{R}$ odd in $x_{2}$ and an analogous statement holding for $\phi_L$ by symmetry.
	
	\medskip
	Our strategy \eqref{psi-0-definition}--\eqref{omega*} constructs a first approximation as perturbations of two vortex pairs moving away from each other, hence the need to solve elliptic problems involving the linearization around the vortex pair such as \eqref{vortex-pair-linear-elliptic-problem}. One could reasonably ask the question why the strategy is not constructing the first approximation as perturbations of four vortices, which are much simpler steady states. Indeed, linearizing around four vortices separately instead of two vortex pairs would allow us to solve the resulting linear elliptic problems by using ODE techniques for each Fourier mode, as is done in \cite{DDMW2020}.
	
	\medskip
	It is crucial to understand that such a strategy is impossible using our methods. To see this, note that the errors produced from such a strategy would be, at main order, the difference between approximating the vortex pair as two vortices and using the true solution of the vortex pair. By \eqref{psi-r-definition}, this is exactly $\psi_{\frac{q}{\varepsilon}}$ which is of size $\varepsilon^{2}$, but has no decay in time. Thus, integrating the error on $\left[T_{0},\infty\right)$ would be impossible. In this sense our construction genuinely treats the two translating vortex pairs as the ``building blocks" of our solution at main order, and cannot be interchangeably thought of as four vortices.
	
	\subsection{Construction of a solution}\label{cs}
	Next, let $K>0$ be a constant. Define cutoffs by
	\begin{align}
		\eta^{(R)}_{K}(x,t)\coloneqq\eta_{0}\left(\frac{\left|x-pe_{1}\right|}{Kt}\right),\quad
		\eta^{(L)}_{K}(x,t)\coloneqq\eta_{0}\left(\frac{\left|x+pe_{1}\right|}{Kt}\right)\label{cutoffs-definition-1}
	\end{align}
	where $\eta_{0}(r)$ is a cutoff that is $1$ on $r\leq1$ and $0$ on $r\geq2$.  From Remark \ref{support-of-nonlinearity-remark} we have
	\begin{align}
		\supp U_{R}\subset B_{\varepsilon \rho}\left(p,q\right)\cup B_{\varepsilon \rho}\left(p,-q\right),\label{first-approximation-main-order-vorticity-support}
	\end{align}
	for some absolute constant $\rho>0$. This is even with time dependence due to the size of $\xi_{0}$ we will construct in Section \ref{point-vortex-trajectory-section}, and the sizes of $\xi_{1}$ and $\tilde{\xi}$ from \eqref{xi1-bound}--\eqref{tildexi-bound}. A symmetric statement holds for $U_{L}$. We now choose $K$ and $T_{0}$ so that
	\begin{equation}\label{K-T0-choices}
		\begin{aligned}
			p_{*}\left(T_{0}\right)\geq 10q_{\infty},\ \ q_{*}\left(T_{0}\right)\leq 2q_{\infty},\ \ KT_{0}= 4q_{\infty},\ \ K\leq \frac{m}{5q_{\infty}},\ \ p_{*}\left(t\right)\geq \frac{mt}{3q_{\infty}},\ \ \forall t\geq T_{0}.
		\end{aligned}
	\end{equation}
	The choice of $K$ and $T_{0}$ in \eqref{K-T0-choices} ensures that at $t=T_{0}$ both $\eta^{(j)}_{K}\equiv1$ on $\supp{U_{j}}$ for $j=R,L$, and  $\eta^{(R)}_{K}\equiv0$ on $\supp{U_{L}}$ (and vice versa) hold true for small enough $\varepsilon>0$, due to Theorem \ref{p0-q0-construction-theorem} and bounds \eqref{xi1-bound}--\eqref{tildexi-bound}. 
	
	\medskip
	Then for $t>T_{0}$, $\eta^{(j)}_{K}\equiv1$ on $\supp{U_{j}}$ for $j=R,L$ still holds true because the radii of the balls centred at $\pm pe_{1}$ where the $\eta_{K}^{\left(j\right)}\equiv1$ is at least $KT_{0}= 4q_{\infty}$, where as the support of $U_{R}$ and $U_{L}$ are contained in balls of radius $q+\rho\varepsilon$ around $\pm pe_{1}$ respectively by \eqref{first-approximation-main-order-vorticity-support}. For $\varepsilon>0$ small enough, by \eqref{K-T0-choices}, Theorem \ref{p0-q0-construction-theorem}, \eqref{xi1-bound}--\eqref{tildexi-bound}, $4q_{\infty}>q+\rho\varepsilon$.
	
	\medskip
	Moreover, $\eta^{(R)}_{K}\equiv0$ on $\supp{U_{L}}$ (and vice versa) also holds for $t>T_{0}$, as the last two conditions in \eqref{K-T0-choices} mean that the supports of $\eta^{(j)}_{K}$, $j=R,L$ are, for $\varepsilon>0$ small enough, expanding at a strictly slower speed than they are moving away from each other, once again using Theorem \ref{p0-q0-construction-theorem}, \eqref{xi1-bound}--\eqref{tildexi-bound}. This, along with the fact that they were positively separated at $t=T_{0}$, and that we have already established $\supp{U_{j}}\subset \eta^{(j)}_{K}$, $j=R,L$ for $t\geq T_{0}$, means we have the desired property.
	\begin{remark}\label{K-T0-choices-remark}
		We note that by \eqref{pkr-asymptotics}, \eqref{K-T0-choices} is satisfied by enlarging $T_{0}$, and therefore shrinking $K$, as necessary. This can be done freely as long as $T_{0}$ remains a constant independent of $\varepsilon$.
	\end{remark}
	We look for a solution to the Euler equations
	\begin{equation}\label{eu-finiteinterval}
		E_1 [\Psi , \omega] (x,t) = E_2 [\Psi , \omega ] (x,t) = 0 \quad (x,t) \in \R^2 \times [T_0 , T]
	\end{equation}
	(see \eqref{euler-operators} for the definition of $E_1$ and $E_2$) of the form
	$$
	\Psi (x,t) = \Psi_* (x,t) + \psi_* (x,t), \quad \omega (x,t) = \omega_* (x,t) + \phi_* (x,t),
	$$
	where $\psi_*$ and $\phi_*$ are small corrections of the previously found first approximations. We decompose these remainders as follows
	\begin{align}
		\psi_{*}=\eta^{(R)}_{K}\psi_{*R}-\eta^{(L)}_{K}\psi_{*L}+\psi^{out}_*, \quad 
		\phi_{*}=\varepsilon^{-2}\left(\phi_{*R}-\phi_{*L}\right).\label{psi-star-phi-star-initial-forms}
	\end{align}
	The Euler equations become
	\begin{align*}
		E_{1}\left[\Psi , \omega \right]&=\varepsilon^{-4}\left(E_{*R1}\left(\phi_{*R},\psi_{*R},\psi_*^{out};\xi\right)-E_{*L1}\left(\phi_{*L},\psi_{*L},\psi_*^{out};\xi\right)\right) =0, \\\
		E_{2}\left[\Psi , \omega \right]&=E_{*2}^{out}\left(\psi_*^{in},\psi_*^{out};\xi\right)=0,\quad \psi_*^{in} = (\psi_{*R},\psi_{*L}),\quad \quad \phi_*^{in} = (\phi_{*R},\phi_{*L}),
	\end{align*}
	where, recalling \eqref{omega*}, for $j=R,L$
	\begin{align}
		E_{*j1}\left(\phi_{*j}, \psi_{*j}, \psi_*^{out}\right) &=\varepsilon^{2}\dell_{t}\left(U_j + \phi_j +\phi_{*j}\right)\nonumber\\
		&+\varepsilon^{2}\gradperp\left(\Psi_{*}+ \psi_{*j}+ \psi_*^{out}\right)\cdot\nabla\left(U_j + \phi_j +\phi_{*j}\right),\label{E-star-j1-initial-form}\\
		E_{*2}^{out} \left( \psi_*^{out},\phi_{*}^{in}, \psi_*^{in} \right)&=\Delta\psi_*^{out}+\left(\psi_{*R}\Delta\eta^{(R)}_{K}+2\nabla\eta^{(R)}_{K}\cdot\nabla\psi_{*R}\right)\nonumber\\
		&-\left(\psi_{*L}\Delta\eta^{(L)}_{K}+2\nabla\eta^{(L)}_{K}\cdot\nabla\psi_{*L}\right)+ E_2 [\Psi_*, \omega_*].\label{E-star-2-out-initial-form}
	\end{align}
	The
	inner-outer gluing method consists of finding $\psi_*^{in}$, $\psi_*^{out}$, $\phi_*^{in}$ to satisfy the following system of coupled equations for $j=R,L$:
	\begin{equation}\label{inj}
		\begin{aligned}
			E_{*j1} \left(\phi_{*j}, \psi_{*j}, \psi_*^{out}\right)&= 0 \quad {\mbox {in}} \quad \R^2 \times [T_0, T], \quad 
			\phi_{*j} (\cdot , T) = 0  \quad  {\mbox {in}} \quad \R^2 \\
			-\Delta \psi_{*j} & = \phi_{*j} \quad {\mbox {in}} \quad \R^2 \times [T_0, T]
		\end{aligned}
	\end{equation}
	\begin{equation}\label{out}
		E_{*2}^{out} \left(\psi_*^{out} , \phi_*^{in} , \psi_*^{in} \right) = 0 \quad {\mbox {in}} \quad \R^2_+ \times [T_0, T], \quad \psi_*^{out} (x_1 , 0 , t ) =0, \quad t \in [T_0 , T].
	\end{equation}

	A solution of the system \eqref{inj}-\eqref{out}, denoted as $(\phi_*^{in}, \psi_*^{in}, \psi_*^{out})$, can be used to construct a solution for equation \eqref{eu-finiteinterval} by a straightforward addition. To obtain the desired solution, we need to ensure the remainders $\psi_*$ and $\phi_*$ are suitably small. To ensure this, we specify further the form of $\psi_*$ and $\phi_*$ and, in $y$ coordinates, decompose into
	\begin{align}
		\phi_{*R}=\tilde{\phi}_{*R}+\alpha_{R}(t)U_{R},\quad \psi_{*R}=\tilde{\psi}_{*R}+\alpha_{R}\left(\Psi_{R}-c\varepsilon y_{2}\right),\quad \alpha_{R}(T)=0,\label{phi-star-initial-ansatz}
	\end{align}
	and similarly $\left(\phi_{*L},\psi_{*L}\right)$ with the odd reflection in the $y_{1}$ direction taken into account. We note that despite the growth at infinity for $\psi_{*R}$, this term is multiplied by a cutoff in the definition of $\psi_{*}$ in \eqref{psi-star-phi-star-initial-forms}. Plugging \eqref{phi-star-initial-ansatz} into \eqref{E-star-j1-initial-form} gives us, at main order, a transport-like equation for $\tilde{\phi}_{*R}$, with the higher order terms consisting of terms that are multiplied by either $U_{R}=f_{\varepsilon}\left(\Psi_{R}-c\varepsilon y_{2}\right)$, $f'_{\varepsilon}\left(\Psi_{R}-c\varepsilon y_{2}\right)$, $f''_{\varepsilon}\left(\Psi_{R}-c\varepsilon y_{2}\right)$, or $f'''_{\varepsilon}\left(\Psi_{R}-c\varepsilon y_{2}\right)$. This structure of the higher order terms is crucial, as along with the transport-like structure of the linear operator for $\tilde{\phi}_{*R}$, we are able to propagate compact support for $\tilde{\phi}_{*R}$, see \eqref{homotopic-operators-a-priori-estimates-7}--\eqref{E-prime-form} and Lemma \ref{phi-support-lemma}. 
	
	\medskip
	To obtain sufficient a-priori estimates, we require that $\tilde{\phi}_{*R}$ satisfies orthogonality conditions by adjusting its mass and centre of mass. This means first solving a projected version of \eqref{inj}, see \eqref{tilde-phi-star-R-initial-value-problem}, which then implies a solution to the original problem once we can enforce that the coefficients of projection are $0$, see \eqref{cRj-initial-value-problem}. This is enforced once $\alpha_{R},\tilde{\xi}_{1},\tilde{\xi}_{2}$ solve a system of ODEs given in \eqref{cRj-fixed-point-formulation-1}--\eqref{cRj-fixed-point-formulation-5}, and it is in this way that the mass and centre of mass parameters are adjusted. This also adjusts the position of the vortices via $\tilde{\xi}$. Finally, we will formulate the whole system as a fixed-point problem for a compact operator within a ball of an appropriate Banach space. We will then find a solution using a degree-theoretical argument, which entails establishing a priori estimates for a homotopical deformation of the problem into a linear one.
	
	\medskip
	Note that through the terminal data given in \eqref{inj} and subsequently \eqref{phi-star-initial-ansatz}, we are specifying that our solution at time $t=T$ is exactly the approximate solution constructed in Section \ref{first-approximation-section} using the elliptic methods described in Section \ref{appro}, and evolving backwards to $t=T_{0}$. This strategy necessitates integrating on $\left[t,T\right]$ when inverting the time derivative of the full linearised operator obtained from linearising the Euler equations around the approximate solution, as $T$ is the parameter we take to $\infty$ to construct our final solution. This requires sufficient time decay. Trying to specify initial data at $t=T_{0}$ and evolve forwards in time in \eqref{inj} would leave us with no chance of obtaining such decay, hence we would not be able to perform the limit procedure as $T\to\infty$.
	
	\medskip
	Also, in order to obtain sufficient a priori estimates on the linearised transport-like operator discussed above, we must obtain sufficient lower bounds on a certain quadratic form. In \cite{DDMW2020, gallay1, gallay2}, similar a priori estimates were required, and thus they first had to obtain a lower bound of the form 
	\begin{align}
		\int_{\mathbb{R}^{2}}\left(\frac{\Phi^{2}}{\mathfrak{K}}-\Phi\left(-\Delta\right)^{-1}\left(\Phi\right)\right)\geq C\int_{\mathbb{R}^{2}}\frac{\Phi^{2}}{\mathfrak{K}},\label{quadratic-form-discussion}
	\end{align}
	for some fixed profile $\mathfrak{K}$, some fixed constant $C>0$, and $\Phi$ a solution to their respective linearized operators. Crucially, in \cite{DDMW2020, gallay1, gallay2}, $\mathfrak{K}$ is positive on all of $\mathbb{R}^{2}$, so the quadratic form on the left hand side of \eqref{quadratic-form-discussion}, and the weighted $L^{2}$ norm on the right hand side both make sense with some decay assumptions on $\Phi$.
	
	\medskip
	However, the natural candidate for $\mathfrak{K}$ in our case would be $f_{\varepsilon}'$ defined in \eqref{f-prime-definition}. This profile has compact support, and in fact the same support as the vorticity $f_{\varepsilon}$ of the vortex pair introduced in Section \ref{tvp}. On the other hand a solution $\Phi$ to our linearized operator in general would not have support contained in $\supp{f_{\varepsilon}'}$.
	
	\medskip
	Thus a large part of Section \ref{constructing-full-solution-section}, and one of the most subtle parts of our construction, is devoted to first proving a local version of \eqref{quadratic-form-discussion}, up to errors that can be controlled, on a well chosen interior region of $\supp{f_{\varepsilon}'}$, where $f_{\varepsilon}'$ has an explicit positive lower bound in terms of $\varepsilon$ and $t$, see Section \ref{weight-section} and Lemma \ref{quadratic-form-estimate-lemma}. Then we use the transport structure of our linearized operator to obtain a priori estimates for solutions of the linearized operator on a well chosen exterior region, see Lemma \ref{L2-exterior-a-priori-estimate-lemma}. Combining this local estimate on the quadratic form and exterior estimate allows us to prove a priori estimates on our interior region, see Lemma \ref{L2-interior-a-priori-estimate-lemma}.
	
	\medskip
	Our a priori estimates finally allow us to use Arzel{\`a}-Ascoli, to extract a subsequence as $T\to\infty$ whose uniform limit solves \eqref{2d-euler-vorticity-stream} on $\left[T_{0},\infty\right)$, as desired.
	
	\medskip
	We can now also better describe $\xi_{re}=\left(p_{re},q_{re}\right)$ defined in Theorem \ref{teo1}. They take the form
	\begin{align*}
		p_{re}=\bar{p}+\hat{p},\quad q_{re}=\bar{q}+\hat{q},
	\end{align*}
	where $\left(\bar{p},\bar{q}\right)$ is defined in \eqref{xi0def}, and $\left(\hat{p},\hat{q}\right)$ defined in \eqref{point-vortex-trajectory-function-decomposition-new}. They satisfy bounds 
	\begin{equation}\label{points}
		\begin{aligned}
			&\| t^{-1} \bar{p} \|_{L^{\infty}[T_{0} , \infty]} + \| \dot{\bar{p}}  \|_{L^{\infty}[T_{0}, \infty]} +\| t^2 \bar{q} \|_{L^{\infty}[T_{0} , \infty]} + \| t^3 \dot{\bar{q}}  \|_{L^{\infty}[T_{0} , \infty]}\leq C \ve^2,\\
			&\| t {\hat p} \|_{L^{\infty}[T_0 , \infty]} + \| t^2 \dot{\hat p}  \|_{L^{\infty}[T_0 , \infty]} +\| t^2 {\hat q} \|_{L^{\infty}[T_0 , \infty]} + \| t^3 \dot{\hat q}  \|_{L^{\infty}[T_0 , \infty]}\leq C \ve^{3}.
		\end{aligned}
	\end{equation}
	The remainder of this paper will be dedicated to meticulously carrying out the steps mentioned above.

	\section{Point Vortex Trajectories for Two Vortex Pairs}\label{point-vortex-trajectory-section}
	Here we construct the main order point vortex dynamics $\xi_0 (t) = \left(p_0 (t) , q_0(t)\right)$. We start by solving \eqref{pkr-system} (see also \cite{saffman}). In polar coordinates $ p_{*{}}+iq_{*{}}=r(t)e^{i\vartheta(t)}$, this system becomes
	\begin{equation}\label{r-theta-system}
		\frac{\dot{r}}{r}=-2\cot{\left(2\vartheta\right)}\dot{\vartheta},
		\quad  \dot{\vartheta}=-\frac{m}{2r^{2}}.
	\end{equation}
	Solving for $r$, and then $\vartheta$, with data as in \eqref{q-infinity-definition-full-dynamic-problem}, our solution to \eqref{r-theta-system} is given by
	\begin{equation}\label{theta-solution-2}
		\begin{aligned}
			&r(t)=\frac{r_{0}\sin{2\vartheta_{0}}}{\sin{2\vartheta(t)}},\quad \vartheta(t)=\frac{1}{2}\arccot{\left(\cot{2\vartheta_{0}}+\frac{mt}{4q_{\infty}^{2}}\right)},\\ 
			&p_{*{}}(t)=\frac{q_{\infty}}{\sin{\vartheta(t)}}, \quad q_{*{}}(t)=\frac{q_{\infty}}{\cos{\vartheta(t)}},\quad r_{0}\sin{2\vartheta_{0}}=2q_{\infty}.
		\end{aligned}
	\end{equation}
	From \eqref{pkr-system}, we can deduce that, for $\|\cdot\|_{\infty}$, the $L^{\infty}$ norm on $[0,\infty)$ and an absolute constant $C_{1}>0$,
	\begin{align}
		\|(t+1)^{-1}p_{*{}}\|_{\infty}+\|\dot{p}_{*{}}\|_{\infty}+\|q_{*{}}\|_{\infty}+\|(t+1)^{3}\dot{q}_{*{}}\|_{\infty}\leq C_{1},\label{pkr-asymptotics}
	\end{align}
	For a generic $\varepsilon >0$ and $t$, the vortex trajectories 
	$
	\xi_0 (t) = \left( p_0 (t) , q_0 (t) \right)
	$ in \eqref{xi0def}
	obey the modified system \eqref{p0-q0-system}, where $g$ is the $\ve^2$-correction to the speed of an idealized vortex pair, as described in Theorem \ref{vortexpair-theorem} and \eqref{c-def}. 
	We construct the solution to the system \eqref{p0-q0-system} in the form \eqref{xi0def}, where $\overline{\xi}$ will be a small error term. To do this we need the linearisation of \eqref{pkr-system} around the solution \eqref{theta-solution-2}. It can be directly checked that this linearisation is given by
	\begin{align}
		\mathscr{L}\coloneqq \frac{d}{dt}+\mathscr{A}\left(\xi_{*{}}\right),\label{pkr-qkr-linearised operator-1}
	\end{align}
	where, in polar coordinates, using the ODE for $\vartheta$ in \eqref{r-theta-system} and the expression for $r$ in terms of $\vartheta$ and $q_{\infty}$ in \eqref{theta-solution-2}, we have 
	\begin{align}
		&\mathscr{A}\left(\xi_{*}\right)=\dot{\vartheta}\begin{pmatrix}
			\sin{2\vartheta} & -\frac{1}{\sin^{2}{\vartheta}}-\cos{2\vartheta} \\
			\frac{1}{\cos^{2}{\vartheta}}-\cos{2\vartheta} & -\sin{2\vartheta}
		\end{pmatrix}\nonumber\\
		&=-\frac{m}{8q_{\infty}^{2}}\begin{pmatrix}
			(\sin{2\vartheta})^{3} & -4\cos^{2}{\vartheta}-4\cos{2\vartheta}(\sin{2\vartheta})^{2} \\
			4(\sin{\vartheta})^{2}(1-(\cos{\vartheta})^{4}+(\sin{2\vartheta})^{2}) & -(\sin{2\vartheta})^{3}
		\end{pmatrix}.\label{matrix-A-polar-form}
	\end{align}
	The expressions in \eqref{theta-solution-2} along with the asymptotics in \eqref{p-q-star-asymptotics}, give us, as $t\to\infty$,
	\begin{align}
		\mathscr{A}\left(\xi_{*}\right)=\begin{pmatrix}
			\bigO{\left(\frac{q_{\infty}^{4}}{m^{2}t^{3}}\right)} & -\frac{m}{2q_{\infty}^{2}} \\
			\bigO{\left(\frac{q_{\infty}^{6}}{m^{3}t^{4}}\right)} & \bigO{\left(\frac{q_{\infty}^{4}}{m^{2}t^{3}}\right)}
		\end{pmatrix}.\label{expression-for-exp-minus-B}
	\end{align}
	Then finding a solution to \eqref{p0-q0-system} of the form $\xi_{0}=\xi_{*{}}+\overline{\xi}$ is formally equivalent to the problem
	\begin{align}
		\mathscr{L}\overline{\xi}=\mathscr{R}+\mathscr{M}_{g}(q_{*})\overline{\xi}+\mathscr{N}\left(\overline{\xi}\right),\quad \mathscr{R}=\begin{pmatrix}
			g\left(q_{*{}}'\right) \\
			0\\
		\end{pmatrix}, \quad 
		\mathscr{M}_{g}(q_{*})=\begin{pmatrix}
			0 & \frac{\dell_{q'}g\left(q_{*{}}'\right)}{\varepsilon} \\
			0 & 0
		\end{pmatrix},\label{p0-q0-fixed-point-problem-1}
	\end{align}
	and $\mathscr{N}\left(\overline{\xi}\right)$ are the remaining terms which are quadratic or higher in $\overline{\xi}$. To formulate \eqref{p0-q0-fixed-point-problem-1} as an integral equation for which we can use fixed point machinery, we define the operator $\mathscr{K}_{\left(0,\infty\right)}$ by
	$$
	\mathscr{K}_{\left(0,\infty\right)}\xi=\mathscr{K}_{\left(0,\infty\right)}\begin{pmatrix}
		\xi_{1} \\
		\xi_{2} \\
	\end{pmatrix}=\begin{pmatrix}
		\int_{T_{0}}^{t}\xi_{1} \\
		\int_{t}^{\infty}\xi_{2} \\
	\end{pmatrix}.$$
	Next, we define $\mathscr{X}_{\left(-1,2\right)}$ to be the space of functions $\xi=\left(\xi_{1},\xi_{2}\right):[T_{0},\infty)\to\mathbb{R}^{2}$ satisfying $\xi_{1}(T_{0})=0$ and
	$$
	\|\xi\|_{\mathscr{X}_{\left(-1,2\right)}}\coloneqq\|t^{-1}\xi_{1}\|_{[T_{0},\infty)}+\|t^{2}\xi_{2}\|_{[T_{0},\infty)}<\infty.$$
	Finally we define the operator $\mathscr{T}$ on $B$, the ball of radius $\varepsilon^{2}\left|\log{\varepsilon}\right|$ in $\mathscr{X}_{\left(-1,2\right)}$ by
	$$
	\mathscr{T}\xi\coloneqq \mathscr{K}_{\left(0,\infty\right)}\left(-\mathscr{A}(\xi_{*})\xi+\mathscr{R}+\mathscr{M}_{g}\xi+\mathscr{N}\left(\xi\right)\right).$$
	Given these definitions, we note that by the asymptotics \eqref{expression-for-exp-minus-B}, for $\varepsilon>0$ small enough, $\mathscr{T}:B\to B$ can be made a contraction, enlarging $T_{0}$ if necessary. Therefore we obtain the following result:
	\begin{theorem}\label{p0-q0-construction-theorem}
		For all $T_{0}>0$ large enough, there is a constant $C>0$ such that for $\varepsilon>0$ small enough, there is a unique solution $\xi_{0}=\xi_{*{}}+\overline{\xi}$ to \eqref{p0-q0-system} with $\overline{p}(T_{0})=0$ and $\lim_{t\to\infty}\overline{q}(t)=0$ in the ball $B$. This solution has the bound
		\begin{align}
			\|t^{-1}\overline{p}\|_{[T_{0},\infty)}+\|\dot{\overline{p}}\|_{[T_{0},\infty)}+\|t^{2}\overline{q}\|_{[T_{0},\infty)}+\|t^{3}\dot{\overline{q}}\|_{[T_{0},\infty)}\leq C\varepsilon^{2}.\label{pbar-bound}
		\end{align}
	\end{theorem}
	\section{First Approximation on $[T_{0},T]$}\label{first-approximation-section}
	\begin{remark}\label{first-approx-section-gamma-remark}
		For Sections \ref{first-approximation-section}, \ref{constructing-full-solution-section}, and \ref{conclusion}, we will always assume $\gamma>18$. This will only strictly be necessary rather than sufficient for Theorem \ref{L2-interior-a-priori-estimate-lemma}, and in many places in these sections milder assumptions on $\gamma$ would suffice, but for the sake of consistency and to avoid confusion, we will take $\gamma>18$ everywhere.
	\end{remark}
	Our first step in building a two vortex pair solution to the Euler equations \eqref{2d-euler-vorticity-stream} is to construct a good enough first approximation on a finite time interval $\left[T_{0},T\right]$ for $T_{0}$ defined in \eqref{K-T0-choices}. As discussed in Section~\ref{scheme},  $T_{0}$ will be some absolute constant large enough to ensure sufficient separation of the two vortex pairs, and $T$ is an arbitrarily large parameter that we will take $T\to\infty$ to construct our full solution on $\left[T_{0},\infty\right)$. To start, we wish to  calculate the size of $E_{1}\left[\omega_{*},\Psi_{*}\right]$ and $E_{2}\left[\omega_{*},\Psi_{*}\right]$, defined in \eqref{euler-operators} for well chosen $\omega_{*}$ and $\Psi_{*}$. First, let $c(t)$ and $\Omega(t)$ be dependent on $q(t)$ as in \eqref{c-q-Omega-time-dependent-relationship-new}. We will restrict ourselves to $\xi(t)$ of the form given in \eqref{point-vortex-trajectory-function-decomposition-new}. From the bounds on $\left(p,q\right)$ in \eqref{xi0-bound}--\eqref{tildexi-bound}, we also have the following estimates on the linearisation of certain functions of $\xi$ around $\xi_{0}$ that we will use repeatedly.
	\begin{lemma}\label{p-q-rational-function-linearisation-lemma}
		Let $\xi=\left(p,q\right)$ and $\xi_{0}=\left(p_{0},q_{0}\right)$ be as in \eqref{point-vortex-trajectory-function-decomposition-new}. Then on $\left[T_{0},T\right]$,
		\begin{align*}
			\frac{1}{p^{n}}=\frac{1}{p_{0}^{n}}+\bigO{\left(\frac{\varepsilon^{3}}{t^{n+2}}\right)},\quad \frac{1}{q^{n}}=\frac{1}{q_{0}^{n}}+\bigO{\left(\frac{\varepsilon^{3}}{t^{2}}\right)},\ \ n\in\mathbb{Z}_{\geq1},
		\end{align*}
		and $k\in\mathbb{Z}_{\geq1}$, $i+j\leq k$, we have
		\begin{align*}
			\frac{p^{i}q^{j}}{\left(p^{2}+q^{2}\right)^{k}}=\frac{p_{0}^{i}q_{0}^{j}}{\left(p_{0}^{2}+q_{0}^{2}\right)^{k}}+\bigO{\left(\frac{\varepsilon^{3}}{t^{2k-i+2}}\right)}.
		\end{align*}
	\end{lemma}
	Recalling the definition of $\Gamma$ from \eqref{power-semilinear-problem-R2}, $\psi_{\frac{q}{\varepsilon}}$ from \eqref{travelling-vortex-pair-ansatz} with the definition of $\xi$ \eqref{point-vortex-trajectory-function-decomposition-new} in hand, and $c(t),\Omega(t)$ related to $q(t)$ by \eqref{c-q-Omega-time-dependent-relationship-new}, we define $ \Psi_{0}, \omega_{0}, \Psi_{R}, \Psi_{L}, U_{R}, U_{L}$ as in \eqref{psi-0-definition}--\eqref{w-r-definition}.
	
	\medskip
	Next, recall cutoffs by $\eta^{R}_{K}$ and $\eta^{L}_{K}$ defined in \eqref{cutoffs-definition-1} along with the choices of $K$ and $T_{0}$ given in \eqref{K-T0-choices}. We also fix $T>0$, a large time satisfying $T\gg T_{0}$, and define $\psi_{R},\psi_{L}$ and $\phi_{R},\phi_{L}$, such that for $j=R,L$, where we formally set $(-1)^{R}=1$ and $(-1)^{L}=-1$, on $\mathbb{R}^{2}_{+}$, with boundary conditions $\psi_{j}=0$ on $\dell\mathbb{R}^{2}_{+}$,
	\begin{equation}\label{psi-R-phi-R-defs}
		\begin{aligned}
			&\psi_{j}=\psi_{j}\left(\frac{x-\left(-1\right)^{j}pe_{1}}{\varepsilon},t\right), \quad \phi_{j}=\phi_{j}\left(\frac{x-\left(-1\right)^{j}pe_{1}}{\varepsilon},t\right),\\
			&\Delta\psi_{j}(w,t)+\phi_{j}(w,t)=0
		\end{aligned}
	\end{equation}
	and then extend to $\mathbb{R}^{2}$ by odd symmetry. Given these definitions, we then set
	\begin{align}
		\Psi_{*}=\Psi_{0}+\eta^{(R)}_{K}\psi_{R}-\eta^{(L)}_{K}\psi_{L}+\psi^{out}(x,t),\quad \omega_{*}=\omega_{0}+\varepsilon^{-2}\left(\phi_{R}-\phi_{L}\right).\label{psi-star-definition}
	\end{align}
	From \eqref{psi-star-definition}, we can see that
	\begin{align}
		E_{1}\left[\omega_{*},\Psi_{*}\right]&=\varepsilon^{-4}\left(E_{R1}\left(\phi_{R},\psi_{R},\psi^{out};\xi\right)-E_{L1}\left(\phi_{L},\psi_{L},\psi^{out};\xi\right)\right),\label{first-approximation-inner-error}\\
		E_{2}\left[\omega_{*},\Psi_{*}\right]&=E_{2}^{out}\left(\psi^{in},\psi^{out};\xi\right),\quad \psi^{in}=\left(\psi_{R},\psi_{L}\right),\quad
		\phi^{in}=\left(\phi_{R},\phi_{L}\right),\label{psi-in-definition}
	\end{align}
	where for $j=R,L$
	\begin{align}
		E_{j1}&=\varepsilon^{2}\dell_{t}\left(U_{j}+\phi_{j}\right)+\varepsilon^{2}\gradperp\left(\Psi_{0}+(-1)^{j}\psi_{j}+\psi^{out}\right)\cdot\nabla\left(U_{j}+\phi_{j}\right),\label{first-approximation-inner-error-ER1-definition}\\
		E_{2}^{out}&=\Delta\psi^{out}+\sum_{j=R,L}\left(-1\right)^{j}\left[\psi_{j}\Delta\eta^{(j)}_{K}+2\nabla\eta^{(j)}_{K}\cdot\nabla\psi_{j}\right],\label{first-approximation-outer-error-E2out-definition}
	\end{align}
	We note that the specific forms of $E_{R1}$, $E_{L1}$, and $E_{2}^{out}$ are due to the fact that $\phi_{R}$ and $\phi_{L}$ will have the same supports as $U_{R}$ and $U_{L}$ respectively, and so due to the choice of $K$ and $T_{0}$ in \eqref{K-T0-choices}, $\eta^{(j)}_{K}\equiv1$ on $\supp{\phi_{j}}$ for $j=R,L$, and $\eta^{(R)}_{K}\equiv0$ on $\supp{\phi_{L}}$ and vice versa.
	We first prove a lemma that shows the structure of $E_{R1}$ and $E_{L1}$. Let
	\begin{align}
		y=\frac{x-pe_{1}}{\varepsilon},\ q'=\frac{q}{\varepsilon},\ p'=\frac{p}{\varepsilon},\label{change-of-coordinates-dynamic-problem}
	\end{align}
	and let
	\begin{align}
		f_{\varepsilon}'\coloneqq&\gamma\left(\Psi_{R}-c\varepsilon y_{2}-\left|\log{\varepsilon}\right|\Omega\right)^{\gamma-1}_{+}\chi_{B_{\frac{s}{\varepsilon}}\left(q'e_{2}\right)}(y)\nonumber\\
		&+\gamma\left(-\Psi_{R}+c\varepsilon y_{2}-\left|\log{\varepsilon}\right|\Omega\right)^{\gamma-1}_{+}\chi_{B_{\frac{s}{\varepsilon}}\left(-q'e_{2}\right)}(y),\label{f-prime-definition}
	\end{align}
	and finally define
	\begin{align}
		\left(f_{\varepsilon}'\right)^{\pm}&\coloneqq\gamma\left(\pm\Psi_{R}-c\varepsilon y_{2}\mp\left|\log{\varepsilon}\right|\Omega\right)^{\gamma-1}_{+}\chi_{B_{\frac{s}{\varepsilon}}\left(\pm q'e_{2}\right)}(y).\label{f-prime-plus-definition}
	\end{align}
	We note that with the explicit formulae \eqref{f-prime-definition}--\eqref{f-prime-plus-definition} in hand, we can see that, similarly to \eqref{gradperp-grad-identity-vortex-pair},
	\begin{align}
		\gradperp_{y}\left(\Psi_{R}-c\varepsilon y_{2}\right)\cdot\nabla_{y}\left(f_{\varepsilon}'\right)&=\gradperp_{y}\left(\Psi_{R}-c\varepsilon y_{2}\right)\cdot\nabla_{y}\left(\left(f_{\varepsilon}'\right)^{\pm}\right)=0.\label{gradperp-f-prime-0}
	\end{align}
	Moreover, from \eqref{nonlinearity-regularity} and \eqref{UR-UL-regularity}, we can infer that
	\begin{align}
		\left(f_{\varepsilon}'\right)^{\pm}\in\twopartdef{C^{\floor{\gamma}-1,\gamma-\left[\gamma\right]}}{\gamma\notin \mathbb{Z}}{C^{\gamma-2}}{\gamma\in\mathbb{Z}}.\label{f-prime-regularity}
	\end{align}
	\begin{lemma}\label{ER1-lemma}
		With change of coordinates given by \eqref{change-of-coordinates-dynamic-problem}, we can write $E_{R1}$, defined in \eqref{first-approximation-inner-error-ER1-definition}, as
		\begin{align}
			&E_{R1}=\varepsilon^{2}\dell_{t}\phi_{R}+\gradperp_{y}\psi_{R}\cdot\grad_{y}\phi_{R}-\gradperp_{y}\left(\Psi_{R}-c\varepsilon y_{2}\right)\cdot\grad_{y}\left(\Delta\psi_{R}+f_{\varepsilon}'\psi_{R}\right)\label{ER1-lemma-statement-1}\\
			&+\gradperp_{y}\left(\Psi_{R}-c\varepsilon y_{2}\right)\cdot\grad_{y}\left(f_{\varepsilon}'\left(\Psi_{L}+\varepsilon\left(\Dot{p}-c\right)\left(y_{2}-\sgn{\left(y_{2}\right)}q'\right)-\varepsilon\dot{q}\sgn{\left(y_{2}\right)}y_{1}\right)\right)\nonumber\\
			&+\gradperp_{y}\psi^{out}\cdot\grad_{y}\left(U_{R}+\phi_{R}\right)-\left[\gradperp_{y}\Psi_{L}+\varepsilon\left(\Dot{p}-c\right)e_{1}\right]\cdot\grad_{y}\phi_{R}\nonumber\\
			&+\dot{q}\left(\varepsilon\left(f_{\varepsilon}'\right)^{+}\mathcal{R}_{1}+\varepsilon\left(f_{\varepsilon}'\right)^{-}\mathcal{R}_{2}\right)\nonumber,
		\end{align}
		where $\mathcal{R}_{1}$ is a remainder term, $\mathcal{R}_{2}$ is its odd reflection in $y_{2}$, and for small enough $\varepsilon>0$, on the support of $U_{R}$ restricted to the upper half plane,
		\begin{align}
			\varepsilon \mathcal{R}_{1}=-\frac{m\varepsilon}{4\left(q'\right)^{3}}\left(\frac{\varrho_{2}(\left|y-q'e_{2}\right|)}{\left|y-q'e_{2}\right|^{2}}+1\right)\left(y_{1}^{2}-\left(y_{2}-q'\right)^{2}\right)+\bigO{\left(\varepsilon^{5}\right)}\label{ER1-lemma-statement-2},
		\end{align}
		where $\varrho_{2}$ is a radial function with respect to the coordinates $y-q'e_{2}$ defined in \eqref{vortex-linearised-equation-rk-error-solution}. An analogous statement holds for $\mathcal{R}_{2}$ in coordinates $y+q'e_{2}$ on the support of $U_{R}$ restricted to the lower plane. An analogous statement also holds for $E_{L1}$ in the appropriate coordinates on the support of $U_{L}$.
	\end{lemma}
	\begin{proof}
		Using the change of coordinates \eqref{change-of-coordinates-dynamic-problem}, we have
		\begin{align*}
			E_{R1}&=\varepsilon^{2}\dell_{t}\phi_{R}+\varepsilon\left(-\left(\Dot{p}-c+c\right)e_{1}\right)\cdot\grad_{y}\left(U_{R}+\phi_{R}\right)+\varepsilon\dot{q}\dell_{q'}\left(U_{R}+\phi_{R}\right)\nonumber\\
			&+\gradperp_{y}\left(\Psi_{R}-\Psi_{L}+\psi_{R}\right)\cdot\grad_{y}\left(U_{R}+\phi_{R}\right)+\gradperp_{y}\psi^{out}\cdot\grad_{y}\left(U_{R}+\phi_{R}\right).
		\end{align*}
		As $\left(U_{R},\Psi_{R}\right)$ is a travelling wave with velocity $ce_{1}$, we note that
		\begin{align}
			-\varepsilon ce_{1}\cdot\grad_{y}U_{R}+\gradperp_{y}\Psi_{R}\cdot\grad_{y}U_{R}=0.\label{first-error-cancellation-due-to-travelling-wave-solution}
		\end{align}
		Then \eqref{first-error-cancellation-due-to-travelling-wave-solution} allows us to obtain
		\begin{align}
			&E_{R1}=\varepsilon^{2}\dell_{t}\phi_{R}-\left[\gradperp_{y}\Psi_{L}+\varepsilon\left(\Dot{p}-c\right)e_{1}\right]\cdot\grad_{y}\phi_{R}+\gradperp_{y}\psi_{R}\cdot\grad_{y}\phi_{R}\label{first-approximation-error-calculation-3}\\
			&+\gradperp_{y}\psi^{out}\cdot\grad_{y}\left(U_{R}+\phi_{R}\right)\nonumber+\underbrace{\left[\gradperp_{y}\Psi_{R}\cdot\grad_{y}\phi_{R}+\gradperp_{y}\psi_{R}\cdot\grad_{y}U_{R}\right]-\varepsilon ce_{1}\cdot\grad_{y}\phi_{R}}_{\mathcal{A}_{1}}\nonumber\\
			&\underbrace{-\left[\gradperp_{y}\Psi_{L}+\varepsilon\left(\Dot{p}-c\right)e_{1}\right]\cdot\grad_{y}U_{R}+\varepsilon\dot{q}\dell_{q'}U_{R}}_{\mathcal{A}_{2}}\nonumber.
		\end{align}
		The first line on the right hand side of \eqref{first-approximation-error-calculation-3} are all terms on the right hand side of \eqref{ER1-lemma-statement-1}. We must show that the remaining terms on the right hand side of \eqref{first-approximation-error-calculation-3}, labelled $\mathcal{A}_{1}$ and $\mathcal{A}_{2}$, can be written as the remaining terms on the right hand side of \eqref{ER1-lemma-statement-1}. Note that
		\begin{align*}
			&\left[\gradperp_{y}\Psi_{R}\cdot\grad_{y}\phi_{R}+\gradperp_{y}\psi_{R}\cdot\grad_{y}U_{R}\right]-\varepsilon ce_{1}\cdot\grad_{y}\phi_{R}=\\
			&\gradperp_{y}\left(\Psi_{R}-c\varepsilon y_{2}\right)\cdot\grad_{y}\phi_{R}-\gradperp_{y}U_{R}\cdot\grad_{y}\psi_{R}.
		\end{align*}
		From \eqref{w-r-definition} and \eqref{nonlinearity-definition}, we know that $U_{R}=f_{\varepsilon}\left(\Psi_{R}-c\varepsilon y_{2}\right)$. Moreover, from \eqref{psi-R-phi-R-defs} we have, $-\Delta\psi_{R}=\phi_{R}$, so
		\begin{align*}
			\gradperp_{y}\left(\Psi_{R}-c\varepsilon y_{2}\right)\cdot\grad_{y}\phi_{R}&-\gradperp_{y}U_{R}\cdot\grad_{y}\psi_{R}\\
			&=-\gradperp_{y}\left(\Psi_{R}-c\varepsilon y_{2}\right)\cdot\grad_{y}\left(\Delta\psi_{R}\right)-f_{\varepsilon}'\gradperp_{y}\left(\Psi_{R}-c\varepsilon y_{2}\right)\cdot\grad_{y}\psi_{R}\\
			&=-\gradperp_{y}\left(\Psi_{R}-c\varepsilon y_{2}\right)\cdot\grad_{y}\left(\Delta\psi_{R}+f_{\varepsilon}'\psi_{R}\right),
		\end{align*}
		with $f_{\varepsilon}'$ defined in \eqref{f-prime-definition} and using \eqref{gradperp-f-prime-0}. Next, once again using \eqref{gradperp-f-prime-0}, we note that
		\begin{align*}
			-\left[\gradperp_{y}\Psi_{L}+\varepsilon\left(\Dot{p}-c\right)e_{1}\right]\cdot\grad_{y}U_{R}=\gradperp_{y}\left(\Psi_{R}-c\varepsilon y_{2}\right)\cdot\grad_{y}\left(f_{\varepsilon}'\left(\Psi_{L}+\varepsilon\left(\Dot{p}-c\right)y_{2}\right)\right).
		\end{align*}
		Therefore, we have
		\begin{align*}
			&-\left[\gradperp_{y}\Psi_{L}+\varepsilon\left(\Dot{p}-c\right)e_{1}\right]\cdot\grad_{y}U_{R}\\
			&=\gradperp_{y}\left(\Psi_{R}-c\varepsilon y_{2}\right)\cdot\grad_{y}\left(f_{\varepsilon}'\left(\Psi_{L}+\varepsilon\left(\Dot{p}-c\right)\left(y_{2}-\sgn{\left(y_{2}\right)}q'\right)\right)\right).
		\end{align*}
		Now we concentrate on $\varepsilon\dot{q}\dell_{q'}U_{R}$ which can be written as
		\begin{align}
			\varepsilon\Dot{q}\left(f_{\varepsilon}'\right)^{+}\dell_{q'}\left(\Psi_{R}-c\varepsilon y_{2}-\left|\log{\varepsilon}\right|\Omega\right)-\varepsilon\Dot{q}\left(f_{\varepsilon}'\right)^{-}\dell_{q'}\left(-\Psi_{R}+c\varepsilon y_{2}-\left|\log{\varepsilon}\right|\Omega\right).\label{first-approximation-error-calculation-8}
		\end{align}
		Due to the choice of $\Omega$, \eqref{c-q-Omega-time-dependent-relationship-new}, the definition of $\Gamma$, \eqref{Gamma-def}, and the definition of $m=-\nu'\left(1\right)$ from \eqref{reduced-vortex-mass-definition}, on the support of $\left(f_{\varepsilon}'\right)^{+}$, we can write $\Psi_{R}-c\varepsilon y_{2}-\left|\log{\varepsilon}\right|\Omega$ as
		\begin{align}
			\Gamma\left(y-q'e_{2}\right)+\frac{m}{2}\log{\left(1+\frac{y_{2}-q'}{q'}+\frac{\left|y-q'e_{2}\right|^{2}}{4\left(q'\right)^{2}}\right)}-c\varepsilon\left(y_{2}-q'\right)+\psi_{q'}(y,t).\label{first-approximation-error-calculation-9}
		\end{align}
		Thus we can see that, on the support of $\left(f_{\varepsilon}'\right)^{+}$,
		\begin{align}
			&\dell_{q'}\left(\Psi_{R}-c\varepsilon y_{2}-\left|\log{\varepsilon}\right|\Omega\right)=\nonumber\\
			&-\dell_{2}\Gamma\left(y-q'e_{2}\right)-\frac{m}{2}\frac{\frac{1}{q'}+\frac{y_{2}-q'}{2\left(q'\right)^{2}}}{1+\frac{y_{2}-q'}{q'}+\frac{\left|y-q'e_{2}\right|^{2}}{4\left(q'\right)^{2}}}+c\varepsilon-\dell_{2}\psi_{q'}\label{first-approximation-error-calculation-10}\\
			&+\frac{m}{2}\dell_{q'}\left(\log{\left(1+\frac{z_{2}}{q'}+\frac{\left|z\right|^{2}}{4\left(q'\right)^{2}}\right)}\right)\bigg\rvert_{z=y-q'e_{2}}-\dell_{q'}\left(c\right)\varepsilon\left(y_{2}-q'\right)+\left(\dell_{q'}\psi_{q'}+\dell_{2}\psi_{q'}\right)\nonumber.
		\end{align}
		By inspection, the first line on the right hand side of \eqref{first-approximation-error-calculation-10} is equal to
		\begin{align*}
			-\dell_{2}\left(\Gamma\left(y-q'e_{2}\right)-\Gamma\left(y+q'e_{2}\right)+\psi_{q'}(y,t)-c\varepsilon y_{2}-\left|\log{\varepsilon}\right|\Omega\right),
		\end{align*}
		and so on the support of $\left(f_{\varepsilon}'\right)^{+}$, we can write \eqref{first-approximation-error-calculation-10} as
		\begin{align*}
			-\dell_{2}\left(\Gamma\left(y-q'e_{2}\right)-\Gamma\left(y+q'e_{2}\right)+\psi_{q'}(y,t)-c\varepsilon y_{2}-\left|\log{\varepsilon}\right|\Omega\right)+\mathcal{R}_{1}(y,t),
		\end{align*}
		where $\mathcal{R}_{1}$ is the second line on the right hand side of \eqref{first-approximation-error-calculation-10}.
		
		\medskip
		Similarly, on the support of $\left(f_{\varepsilon}'\right)^{-}$, we have
		\begin{align*}
			&-\dell_{q'}\left(-\Psi_{R}+c\varepsilon y_{2}-\left|\log{\varepsilon}\right|\Omega\right)\\
			&=\dell_{2}\left(\Gamma\left(y-q'e_{2}\right)-\Gamma\left(y+q'e_{2}\right)+\psi_{q'}(y,t)-c\varepsilon y_{2}-\left|\log{\varepsilon}\right|\Omega\right)+\mathcal{R}_{2}(y,t),
		\end{align*}
		where $\mathcal{R}_{2}$ is the odd reflection of $\mathcal{R}_{1}$ in $y_{2}$. Thus
		\begin{align*}
			&\varepsilon\dot{q}\dell_{q'}U_{R}=-\varepsilon\dot{q}f_{\varepsilon}'\sgn{\left(y_{2}\right)}e_{2}\cdot\grad_{y}\left(\Psi_{R}-c\varepsilon y_{2}\right)+\varepsilon\dot{q}\left(f_{\varepsilon}'\right)^{+}\mathcal{R}_{1}+\varepsilon\dot{q}\left(f_{\varepsilon}'\right)^{-}\mathcal{R}_{2}\nonumber\\
			&=\gradperp_{y}\left(\Psi_{R}-c\varepsilon y_{2}\right)\cdot\grad_{y}\left(\varepsilon\dot{q}f_{\varepsilon}'\sgn{\left(y_{2}\right)}y_{1}\right)+\varepsilon\dot{q}\left(f_{\varepsilon}'\right)^{+}\mathcal{R}_{1}+\varepsilon\dot{q}\left(f_{\varepsilon}'\right)^{-}\mathcal{R}_{2}.
		\end{align*}
		It remains to show that $\varepsilon\mathcal{R}_{1}$ can be written in the form \eqref{ER1-lemma-statement-2}. The analogous statement for $\varepsilon\mathcal{R}_{2}$ will follow by symmetry. We first have
		\begin{align}
			\mathcal{R}_{1}&=-\frac{m}{2}\frac{\frac{y_{2}-q'}{\left(q'\right)^{2}}+\frac{\left|y-q'e_{2}\right|^{2}}{2\left(q'\right)^{3}}}{1+\frac{y_{2}-q'}{q'}+\frac{\left|y-q'e_{2}\right|^{2}}{4\left(q'\right)^{2}}}-\dell_{q'}\left(c\right)\varepsilon\left(y_{2}-q'\right)+\left(\dell_{q'}\psi_{q'}+\dell_{2}\psi_{q'}\right)\label{first-approximation-error-calculation-15}\\
			&=\mathcal{R}_{1,1}+\left(\dell_{q'}\psi_{q'}+\dell_{2}\psi_{q'}\right).\nonumber
		\end{align}
		On the support of $\left(f_{\varepsilon}'\right)^{+}$, $\mathcal{R}_{1,1}$ can be written as
		\begin{align*}
			&-\frac{m}{2\left(q'\right)^{2}}\left(y_{2}-q'\right)\left(1-\frac{y_{2}-q'}{q'}+\bigO{\left(\varepsilon^{2}\right)}\right)-\frac{m}{4\left(q'\right)^{3}}\left|y-q'e_{2}\right|^{2}\left(1+\bigO{\left(\varepsilon\right)}\right)\\
			&-\dell_{q'}\left(c\right)\varepsilon\left(y_{2}-q'\right),
		\end{align*}
		where we have used the fact that $q'=\bigO{\left(\varepsilon^{-1}\right)}$. This can then be written as
		\begin{align}
			&-\varepsilon\dell_{q'}g\left(q'\right)\left(y_{2}-q'\right)-\frac{m}{4\left(q'\right)^{3}}\left(y_{1}^{2}-\left(y_{2}-q'\right)^{2}\right)+\bigO{\left(\varepsilon^{4}\right)}\label{first-approximation-error-calculation-17},
		\end{align}
		and we note that by \eqref{g-pointwise-bounds} and \eqref{c-q-Omega-time-dependent-relationship-new}, the first term in \eqref{first-approximation-error-calculation-17} is actually $\bigO{\left(\varepsilon^{4}\right)}$ as well. Then, on the support of $\left(f_{\varepsilon}'\right)^{+}$, we have
		\begin{align}
			\varepsilon\mathcal{R}_{1,1}=-\frac{m\varepsilon}{4\left(q'\right)^{3}}\left(y_{1}^{2}-\left(y_{2}-q'\right)^{2}\right)+\bigO{\left(\varepsilon^{5}\right)}.\label{first-approximation-error-calculation-18}
		\end{align}
		To bound $\left(\dell_{q'}\psi_{q'}+\dell_{2}\psi_{q'}\right)$, we note that it solves the following equation on $\mathbb{R}^{2}_{+}$:
		\begin{align*}
			\Delta\left(\dell_{q'}\psi_{q'}+\dell_{2}\psi_{q'}\right)+\left(f_{\varepsilon}'\right)^{+}\left(\dell_{q'}\psi_{q'}+\dell_{2}\psi_{q'}\right)&=\mathcal{E},\\
			\left(\dell_{q'}\psi_{q'}+\dell_{2}\psi_{q'}\right)\left(y_{1},0\right)&=\dell_{2}\psi_{q'}\left(y_{1},0\right),
		\end{align*}
		and error term $\mathcal{E}$ given by
		\begin{align*}
			\mathcal{E}= \left(f_{\varepsilon}'\right)^{+}\left(2\dell_{2}\Gamma\left(y+q'e_{2}\right)+\dell_{q'}\left(c\right)\varepsilon y_{2}+\left|\log{\varepsilon}\right|\left(\dell_{q'}\Omega\right)+c\varepsilon\right).
		\end{align*}
		We write $\dell_{q'}\psi_{q'}+\dell_{2}\psi_{q'}=\beta_{1}+\beta_{2}$, where
		\begin{equation}\label{first-approximation-error-calculation-23}
			\begin{aligned}
				\Delta\beta_{1}=0,\ \ \beta_{1}\left(y_{1},0\right)=\dell_{2}\psi_{q'}\left(y_{1},0\right),\\
				\Delta\beta_{2}+\left(f_{\varepsilon}'\right)^{+}\beta_{2}=\mathcal{E}-\left(f_{\varepsilon}'\right)^{+}\beta_{1},\ \ \beta_{2}\left(y_{1},0\right)=0.
			\end{aligned}
		\end{equation}
		Using the representation formula on the upper half plane, we can write $\beta_{1}$ as
		\begin{align}
			\beta_{1}\left(y_{1},y_{2}\right)=\frac{1}{\pi}\int_{-\infty}^{\infty}\frac{y_{2}}{\left(z_{1}-y_{1}\right)^{2}+y_{2}^{2}}\dell_{2}\psi_{q'}\left(z_{1},0\right) dz_{1}.\label{first-approximation-error-calculation-25}
		\end{align}
		From the fact that the travelling vortex pair $\Psi_{\varepsilon}$ solves \ref{2d-euler-travelling-semilinear-elliptic-equation-epsilon} in $x$ coordinates, we know that $\psi_{q'}$ solves 
		\begin{align}
			&\Delta\psi_{q'}+V\psi_{q'}-E+N\left[\psi_{q'}\right]=0,\label{psi-q-elliptic-equation}\\
			&V=f_{\varepsilon}'\left(\mathring{\Psi}_{\varepsilon}-c\varepsilon y_{2}\right),\quad E=\Delta\mathring{\Psi}_{\varepsilon}+f_{\varepsilon}'\left(\mathring{\Psi}_{\varepsilon}-c\varepsilon y_{2}\right)\label{V-E-vortex-pair}
		\end{align}
		on the upper half plane, $\mathring{\Psi}_{\varepsilon}$ is defined in \eqref{travelling-vortex-pair-ansatz}, and $N$ is the nonlinearity we obtain upon changing rearranging \eqref{2d-euler-travelling-semilinear-elliptic-equation-epsilon} into \eqref{psi-q-elliptic-equation} after an appropriate change of coordinates.
		
		\medskip
		From \eqref{nonlinearity-definition}, Theorems \ref{vortexpair-theorem} and \ref{vortex-pair-properties-theorem}, and \eqref{first-approximation-main-order-vorticity-support}, we have that $V$, $E$, and $N$ have supports contained in $B_{2\rho}\left(q'e_{2}\right)$. As $\psi_{q'}$ is $\bigO{\left(\varepsilon^{2}\right)}$, $N$ is $\bigO{\left(\varepsilon^{4}\right)}$ on $B_{2\rho}\left(q'e_{2}\right)$. Thus, once again using the representation formula on the upper half plane, we have that
		\begin{align}
			\dell_{2}\psi_{q'}\left(z_{1},0\right)=\frac{1}{\pi}\int_{B_{2\rho}\left(q'e_{2}\right)}\frac{v_{2}}{\left(v_{1}-z_{1}\right)^{2}+v_{2}^{2}}\left(V\psi_{q'}-E+N[\psi_{q'}]\right) dv.\label{first-approximation-error-calculation-26}
		\end{align}
		One can see from Theorem \ref{vortex-pair-properties-theorem}, \eqref{f-prime-definition}, and \eqref{V-E-vortex-pair}, that on $ B_{2\rho}\left(q'e_{2}\right)$, for $v-q'e_{2}=re^{i\theta}$, we have
		\begin{align}
			&\left(\frac{v_{2}-q'}{\left(v_{1}-z_{1}\right)^{2}+\left(v_{2}\right)^{2}}+\frac{q'}{\left(v_{1}-z_{1}\right)^{2}+\left(v_{2}\right)^{2}}\right)\left(V\psi_{q'}-E+N[\psi_{q'}]\right)\label{first-approximation-error-calculation-30}\\
			&=\left(\bigO{\left(\varepsilon^{2}\right)}+\frac{q'}{z_{1}^{2}+\left(q'\right)^{2}}\right)\left(-\frac{m\gamma \varepsilon^{2}}{8q^{2}}r^{2}\Gamma^{\gamma-1}_{+}(r)\left(\frac{\varrho_{2}(r)}{r^{2}}-1\right)\cos{2\theta}\right)+\bigO{\left(\varepsilon^{4}\right)}\nonumber.
		\end{align}
		Then, using \eqref{first-approximation-error-calculation-30} as well as the fact that $\Gamma_{+}(r)$ has support $B_{1}(0)$, we can write \eqref{first-approximation-error-calculation-26} as
		\begin{align}
			\dell_{2}\psi_{q'}\left(z_{1},0\right)&=\frac{1}{\pi}\frac{q'}{z_{1}^{2}+\left(q'\right)^{2}}\int_{0}^{2\pi}\int_{0}^{1} F\left(r\right)\cos{2\theta}\ drd\theta+\bigO{\left(\varepsilon^{4}\right)}=\bigO{\left(\varepsilon^{4}\right)},\label{first-approximation-error-calculation-31}
		\end{align}
		Thus by \eqref{first-approximation-error-calculation-23}--\eqref{first-approximation-error-calculation-31} we have
		\begin{align*}
			\|\beta_{1}\|_{L^{\infty}\left(\mathbb{R}^{2}_{+}\right)}\leq C\varepsilon^{4}.
		\end{align*}
		To bound $\beta_{2}$ effectively, we first note that by \eqref{c-q-Omega-time-dependent-relationship-new}, we have
		\begin{align*}
			\mathcal{E}=\left(f_{\varepsilon}'\right)^{+}\left(2\dell_{q'}\Gamma\left(y+q'e_{2}\right)-2\dell_{q'}\Gamma\left(2q'e_{2}\right)+\dell_{q'}\left(c\right)\varepsilon\left(y_{2}-q'\right)\right),
		\end{align*}
		and we can use \eqref{first-approximation-error-calculation-9}--\eqref{first-approximation-error-calculation-17} to conclude that
		\begin{align*}
			\mathcal{E}=\left(f_{\varepsilon}'\right)^{+}\left(-\frac{m}{4\left(q'\right)^{3}}\left(y_{1}^{2}-\left(y_{2}-q'\right)^{2}\right)+\bigO{\left(\varepsilon^{4}\right)}\right).
		\end{align*}
		Next, since the operator $\Delta+\left(f_{\varepsilon}'\right)^{+}=\Delta+\gamma \Gamma_{+}^{\gamma-1}+\bigO{\left(\varepsilon^{2}\right)}$ by Theorem \ref{vortex-pair-properties-theorem} and \eqref{f-prime-plus-definition}, we can use Lemma \ref{vortex-linearised-equation-fourier-coefficients-behaviour-lemma} to deduce that
		\begin{align}
			\beta_{2}\left(y_{1},y_{2}\right)=-\frac{m}{4\left(q'\right)^{3}}\frac{\varrho_{2}(\left|y-q'e_{2}\right|)}{\left|y-q'e_{2}\right|^{2}}\left(y_{1}^{2}-\left(y_{2}-q'\right)^{2}\right)+\bigO{\left(\varepsilon^{4}\right)},\label{first-approximation-error-calculation-31a}
		\end{align}
		where $\varrho_{2}$ is defined in \eqref{vortex-linearised-equation-rk-error-solution}. Therefore, by \eqref{first-approximation-error-calculation-23}--\eqref{first-approximation-error-calculation-31a}, on the support of $\left(f_{\varepsilon}'\right)^{+}$ we have
		\begin{align}
			\varepsilon\left(\dell_{q'}\psi_{q'}+\dell_{2}\psi_{q'}\right)=-\frac{m\varepsilon}{4\left(q'\right)^{3}}\frac{\varrho_{2}(\left|y-q'e_{2}\right|)}{\left|y-q'e_{2}\right|^{2}}\left(y_{1}^{2}-\left(y_{2}-q'\right)^{2}\right)+\bigO{\left(\varepsilon^{5}\right)}.\label{first-approximation-error-calculation-40}
		\end{align}
		Combining \eqref{first-approximation-error-calculation-18} and \eqref{first-approximation-error-calculation-40}, we have, on the support of $\left(f_{\varepsilon}'\right)^{+}$
		\begin{align}
			\varepsilon \mathcal{R}_{1}&=-\frac{m\varepsilon}{4\left(q'\right)^{3}}\left(\frac{\varrho_{2}(\left|y-q'e_{2}\right|)}{\left|y-q'e_{2}\right|^{2}}+1\right)\left(y_{1}^{2}-\left(y_{2}-q'\right)^{2}\right)+\bigO{\left(\varepsilon^{5}\right)}.\label{first-approximation-error-calculation-41}
		\end{align}
		Due to symmetry, the exact same behaviour is true in coordinates $y+q'e_{2}$ for $\varepsilon\mathcal{R}_{2}$ on the support of $\left(f_{\varepsilon}'\right)^{-}$, as required.
	\end{proof}
	The calculations presented in \eqref{first-approximation-error-calculation-15}--\eqref{first-approximation-error-calculation-41} also give a bound which will be useful later on in the construction so we record it here separately for convenience. We have, for $y\in\mathbb{R}^{2}_{+}$, and $z=y-q'e_{2}$,
	\begin{align}
		&U_{R}=f_{\varepsilon}\left(\Psi_{R}-c\varepsilon y_{2}\right)=\left(\Psi_{R}-c\varepsilon y_{2}-\left|\log{\varepsilon}\right|\Omega\right)^{\gamma}_{+}\chi_{B_{\frac{s}{\varepsilon}}\left(q'e_{2}\right)}(y)\label{f-prime-q-variation-1}\\
		&=\left(\Gamma\left(z\right)+\frac{m}{2}\log{\left(1+\frac{z_{2}}{q'}+\frac{\left|z\right|^{2}}{4\left(q'\right)^{2}}\right)}-c\left(q\right)\varepsilon z_{2}+\psi_{q'}(z+q'e_{2})\right)^{\gamma}_{+}\chi_{B_{\frac{s}{\varepsilon}}\left(0\right)}(z),\nonumber
	\end{align}
	Accordingly, we define, for $z\in\mathbb{R}^{2}_{+}-q'e_{2}$,
	\begin{align}
		\mathscr{V}\left(z,\mathcal{Q}\right)&=\Psi_{R}\left(z+q'e_{2}\right)+\psi_{\mathcal{Q}'}\left(z+q'e_{2}\right)-c\left(\mathcal{Q}\right)\varepsilon\left(z_{2}+q'\right)-\left|\log{\varepsilon}\right|\Omega\left(\mathcal{Q}(t)\right)\nonumber\\
		&=\Gamma\left(z\right)+\frac{m}{2}\log{\left(1+\frac{z_{2}}{\mathcal{Q}'}+\frac{\left|z\right|^{2}}{4\left(\mathcal{Q}'\right)^{2}}\right)}-c\left(\mathcal{Q}\right)\varepsilon z_{2}+\psi_{\mathcal{Q}'}(z+\mathcal{Q}'e_{2}),\label{vortex-pair-variation-generalisation}
	\end{align}
	with $\Psi_{R}$ defined in \eqref{psi-R-phi-R-defs}, and we emphasise the dependence of $\Omega$ on the parameter $\mathcal{Q}$ via \eqref{c-q-Omega-time-dependent-relationship-new}.
	\begin{lemma}\label{vortex-pair-q-variation-lemma}
		For all $\varepsilon>0$ small enough, for all $z\in B_{10\rho}\left(0\right)$,
		\begin{align*}
			&\left|\dell_{q'}\mathscr{V}\left(z,q\right)\right|\leq C\varepsilon^{3}.
		\end{align*}
	\end{lemma}
	By symmetry, an analogous statement holds on the support of $U_{R}$ restricted to the lower half plane, and on the support of $U_{L}$, in the appropriate coordinates. As in Theorem \ref{vortex-pair-properties-theorem}, we fix the radius of the closed ball at $10\rho$ for concreteness.
	
	\medskip
	With Lemma \ref{vortex-pair-q-variation-lemma} in hand, we define some related quantities for convenience that we will use later on. For $k=0,1,2,\dots,\floor{\gamma}-1$, define for any function $F$ on $\mathbb{R}^{2}_{+}-q'e_{2}$,
	\begin{align}
		\mathscr{W}_{k}\left(F\right)=\left(\Pi_{j=0}^{k}\left(\gamma-j\right)\right)\left(F(z)\right)^{\gamma-k}_{+}\chi_{B_{\frac{s}{\varepsilon}}\left(0\right)}(z).\label{f-prime-vortex-pair-variation-relationship}
	\end{align}
	so that for $k=1,2,\dots,\floor{\gamma}-1$, on $\mathbb{R}^{2}_{+}$,
	\begin{equation}\label{f-prime-vortex-pair-variation-relationship-1}
		\begin{aligned}
			U_{R}\left(y,t\right)&=\left(f_{\varepsilon}\right)^{+}\left(y,t\right)=\mathscr{W}_{0}\left(\mathscr{V}\left(y-q'e_{2},q\right)\right),\\
			\left(f_{\varepsilon}^{\left(k\right)}\right)^{+}\left(y,t\right)&=\mathscr{W}_{k}\left(\mathscr{V}\left(y-q'e_{2},q\right)\right)\\
			&=\left(\Pi_{j=0}^{k}\left(\gamma-j\right)\right)\left(\mathscr{V}\left(y-q'e_{2},q\right)\right)^{\gamma-k}_{+}\chi_{B_{\frac{s}{\varepsilon}}\left(q'e_{2}\right)}(y)\\
			&=\left(\Pi_{j=0}^{k}\left(\gamma-j\right)\right)\left(\Psi_{R}-c\varepsilon y_{2}-\left|\log{\varepsilon}\right|\Omega\right)^{\gamma-k}_{+}\chi_{B_{\frac{s}{\varepsilon}}\left(q'e_{2}\right)}(y).
		\end{aligned}
	\end{equation}
	As in \eqref{f-prime-plus-definition}, $\left(f_{\varepsilon}^{\left(k\right)}\right)^{-}$ can be defined analogously to \eqref{f-prime-vortex-pair-variation-relationship-1}. Moreover, as with \eqref{gradperp-grad-identity-vortex-pair} and \eqref{gradperp-f-prime-0}, we have
	\begin{align}
		\gradperp_{y}\left(\Psi_{R}-c\varepsilon y_{2}\right)\cdot\nabla_{y}\left(\left(f_{\varepsilon}^{\left(k\right)}\right)^{+}\right)=0.\label{gradperp-f-j-prime-plus-0}
	\end{align}
	Given \eqref{f-prime-vortex-pair-variation-relationship}--\eqref{f-prime-vortex-pair-variation-relationship-1}, we define, for $k=0,1,2,\dots,\floor{\gamma}-1$,
	\begin{align}
		\tilde{\mathscr{W}}_{k}\left(\mathcal{Q}\right)=\tilde{\mathscr{W}}_{k}\left(z,\mathcal{Q}\right)=\mathscr{W}_{k}\left(\mathscr{V}\left(z,\mathcal{Q}\right)\right)\implies\ \tilde{\mathscr{W}}_{k}\left(q\right)=\tilde{\mathscr{W}}_{k}\left(z,q\right)=\left(f_{\varepsilon}^{\left(k\right)}\right)^{+}.\label{tilde-W-k-def}
	\end{align}
	As for \eqref{nonlinearity-regularity}, \eqref{UR-UL-regularity}, and \eqref{f-prime-regularity}, we have for $k=1,2,\dots,\floor{\gamma}-1$, on $\mathbb{R}^{2}_{+}$,
	\begin{align}
		\left(f_{\varepsilon}^{\left(k\right)}\right)^{\pm}\in\twopartdef{C^{\floor{\gamma}-k,\gamma-\left[\gamma\right]}}{\gamma\notin \mathbb{Z}}{C^{\gamma-k-1}}{\gamma\in\mathbb{Z}}.\label{f-k-prime-regularity}
	\end{align}
	We similarly define $\mathscr{Z}_{1}\left(z,\mathcal{Q}\right)$ and $\mathscr{Z}_{2}\left(z,\mathcal{Q}\right)$ by applying the change of variables $z=y-q'e_{2}$ to $\dell_{1}\Psi_{R}$ and $\dell_{q'}\left(\Psi_{R}-c\varepsilon y_{2}-\left|\log{\varepsilon}\right|\Omega\right)$ respectively:
	\begin{equation}\label{dq-Psi-R-variation}
		\begin{aligned}
			\mathscr{Z}_{1}\left(z,\mathcal{Q}\right)&=\dell_{1}\Gamma\left(z\right)+\frac{mz_{1}}{4\left(\mathcal{Q}'\right)^{2}\left(1+\frac{z_{2}}{\mathcal{Q}'}+\frac{\left|z\right|^{2}}{4\left(\mathcal{Q}'\right)^{2}}\right)}+\dell_{1}\psi_{\mathcal{Q}'}\left(z+\mathcal{Q}'e_{2}\right),\\
			\mathscr{Z}_{2}\left(z,\mathcal{Q}\right)&=-\dell_{2}\Gamma\left(z\right)-\frac{m\left(\frac{1}{\mathcal{Q}'}+\frac{2z_{2}}{4\left(\mathcal{Q}'\right)^{2}}\right)}{2\left(1+\frac{z_{2}}{\mathcal{Q}'}+\frac{\left|z\right|^{2}}{4\left(\mathcal{Q}'\right)^{2}}\right)}-\frac{m\left(\frac{z_{2}}{\left(\mathcal{Q}'\right)^{2}}+\frac{\left|z\right|^{2}}{2\left(\mathcal{Q}'\right)^{3}}\right)}{2\left(1+\frac{z_{2}}{\mathcal{Q}'}+\frac{\left|z\right|^{2}}{4\left(\mathcal{Q}'\right)^{2}}\right)}\\
			&-\dell_{\mathcal{Q}'}c\left(\mathcal{Q}\right)+\dell_{\mathcal{Q}'}\psi_{\mathcal{Q}'}\left(z+\mathcal{Q}'e_{2}\right).
		\end{aligned}
	\end{equation}
	Much like the support of $\left(f_{\varepsilon}'\right)^{+}$, for $\xi$ satisfying \eqref{xi0-bound}--\eqref{tildexi-bound}, $\mathscr{W}_{0}\left(\mathscr{V}\left(z,q_{0}\right)\right)$ is supported on the ball $B_{2\rho}\left(0\right)$, for $\rho>0$ defined in \eqref{first-approximation-main-order-vorticity-support}, with respect to the $z$ coordinate. We note that $\mathscr{W}_{0}$ and $\mathscr{Z}_{2}$ are even in $y_{1}$ and $\mathscr{Z}_{1}$ is odd in $y_{1}$.
	
	\medskip
	On $\left[T_{0},T\right]$, we have, for $j=0,1,2,\dots,\floor{\gamma}-2$
	\begin{align}
		&\left|\mathscr{W}_{j}\left(\mathscr{V}\left(z,\mathcal{Q}_{1}\right)\right)-\mathscr{W}_{j}\left(\mathscr{V}\left(z,\mathcal{Q}_{2}\right)\right)\right|\leq C\varepsilon^{2}\left|\mathcal{Q}_{1}-\mathcal{Q}_{2}\right|\mathscr{W}_{j+1}\left(\mathscr{V}\left(z,\mathcal{Q}_{1}\right)\right),\label{f-prime-q-variation-3}\\
		&\left|\mathscr{Z}_{1}\left(\mathcal{Q}_{1}\right)-\mathscr{Z}_{1}\left(\mathcal{Q}_{2}\right)\right|+\left|\mathscr{Z}_{2}\left(\mathcal{Q}_{1}\right)-\mathscr{Z}_{2}\left(\mathcal{Q}_{2}\right)\right|\leq C\varepsilon^{2}\left|\mathcal{Q}_{1}-\mathcal{Q}_{2}\right|.\label{f-prime-q-variation-4}
	\end{align}
	The proofs for both are immediate using Theorem \ref{vortexpair-theorem}, Lemma \ref{vortex-pair-q-variation-lemma}, as well as using the explicit formulae \eqref{f-prime-vortex-pair-variation-relationship},\eqref{f-prime-vortex-pair-variation-relationship-1}, and \eqref{dq-Psi-R-variation}.
	
	\medskip
	The following lemma shows that the top order terms from $\varepsilon\mathcal{R}_{1}$ and $\varepsilon\mathcal{R}_{2}$ have desirable structure as per the discussion in Section \ref{appro}.
	\begin{lemma}\label{correct-remainder-form-lemma}
		Let $\varepsilon\mathcal{R}_{1}$ be as in \eqref{ER1-lemma-statement-2}. Then for all $\varepsilon>0$ small enough, on the support of $U_{R}$ restricted to the upper half plane we have that $\varepsilon\mathcal{R}_{1}$ can be written as
		\begin{align}
			-\gradperp_{y}\left(\Psi_{R}-c\varepsilon y_{2}\right)\cdot\grad_{y}\left(\frac{m\varepsilon^{4}\left|y-q'e_{2}\right|}{8\Gamma'\left(\left|y-q'e_{2}\right|\right)q^{3}}\left(\frac{\varrho_{2}(\left|y-q'e_{2}\right|)}{\left|y-q'e_{2}\right|^{2}}+1\right)y_{1}\left(y_{2}-q'\right)\right)+\overline{\mathcal{R}}_{1},\label{correct-remainder-form-lemma-statement}
		\end{align}
		with $\overline{\mathcal{R}}_{1}=\bigO{\left(\varepsilon^{5}\right)}$. By symmetry, an analogous statement is true for $\varepsilon\mathcal{R}_{2}$ on the support of $U_{R}$ restricted to the lower half plane in coordinates $y+q'e_{2}$.
	\end{lemma}
	
	\begin{proof}
		On $\supp{\left(f_{\varepsilon}'\right)^{+}}$, which is the same support as $U_{R}$ restricted to the upper half plane by \eqref{f-prime-vortex-pair-variation-relationship}--\eqref{tilde-W-k-def}, we let $y-q'e_{2}=re^{i\theta}$, so that
		\begin{align*}
			\gradperp_{y}\left(\Gamma\left(\left|y-q'e_{2}\right|\right)\right)\cdot\grad_{y}=-\frac{\Gamma'\left(r\right)}{r}\dell_{\theta}.
		\end{align*}
		Then we see that for any function $h(y-q'e_{2})$ which we can write as $h(r,\theta)$ as a slight abuse of notation, we have
		\begin{align*}
			h(y-q'e_{2})&=h(r,\theta)=\dell_{\theta}\left(H(r,\theta)\right)=-\frac{\Gamma'\left(r\right)}{r}\dell_{\theta}\left(\frac{r}{\Gamma'\left(r\right)}H(r,\theta)\right)\\
			&=-\gradperp_{y}\left(\Gamma\left(\left|y-q'e_{2}\right|\right)\right)\cdot\grad_{y}G(y-q'e_{2}),\\
			H(r,\theta)&=\int_{0}^{\theta}h(r,s)ds,\quad G(y-q'e_{2})=\frac{r}{\Gamma'\left(r\right)}H(r,\theta).
		\end{align*}
		Next, due to Theorem \ref{vortex-pair-properties-theorem}, we know that on the support of $U_{R}$ restricted to the upper half plane,
		\begin{align*}
			-\gradperp_{y}\left(\Gamma\left(\left|y-q'e_{2}\right|\right)\right)\cdot\grad_{y}G(y-q'e_{2})&=-\gradperp_{y}\left(\Psi_{R}-c\varepsilon y_{2}\right)\cdot\grad_{y}G(y-q'e_{2})+\mathcal{R},\\
			\left|\mathcal{R}\right|&\leq C\varepsilon^{2}\left|\nabla_{y} G\right|.
		\end{align*}
		Recalling the definition of $\varepsilon\mathcal{R}_{1}$ in \eqref{ER1-lemma-statement-2}, we obtain \eqref{correct-remainder-form-lemma-statement} by letting
		\begin{align*}
			h(y-q'e_{2})=-\frac{m\varepsilon}{4\left(q'\right)^{3}}\left(\frac{\varrho_{2}(\left|y-q'e_{2}\right|)}{\left|y-q'e_{2}\right|^{2}}+1\right)\left(y_{1}^{2}-\left(y_{2}-q'\right)^{2}\right).
		\end{align*}
	\end{proof}
	For future reference we define, in radial coordinates,
	\begin{align} 
		\mathcal{M}_{2}(r,q)\coloneqq\frac{m\varepsilon^{4}r}{8\Gamma'\left(r\right)q^{3}}\left(\frac{\varrho_{2}(r)}{r^{2}}+1\right).\label{dq-mode-1-error}
	\end{align}
	A key tool in our construction of the first approximation is an understanding of how the quantity
	\begin{align*}
		\Psi_{L}+\varepsilon\left(\dot{p}-c\right)\left(y_{2}-q'\right)-\varepsilon\dot{q}y_{1}
	\end{align*}
	behaves on the support of $U_{R}$ restricted to the upper half plane. The next lemma shows that since $\left(p_{0},q_{0}\right)$ solves \eqref{p0-q0-system}, this quantity expands as a sum of operators in $\left(p_{1},q_{1}\right)$ and $\left(\tilde{p},\tilde{q}\right)$ with good structure, and error terms that have sufficient bounds to be able to construct our first approximation. Recall the definition of $\mathcal{F}$ from \eqref{pkr-system}, and define for $h,k:I\to\mathbb{R}^{2}$, for some time interval $I$,
	\begin{align}
		\mathcal{N}\left(h\right)\left[k\right]=\begin{pmatrix}
			-\mathcal{F}\left(h_{2}+k_{2},h_{1}+k_{1}\right)+\mathcal{F}\left(h_{2},h_{1}\right)-g\left(h_{2}'+k_{2}'\right)+g\left(h_{2}'\right) \\
			\mathcal{F}\left(h_{1}+k_{1},h_{2}+k_{2}\right)-\mathcal{F}\left(h_{1},h_{2}\right) \\
		\end{pmatrix}.\label{linearised-point-vortex-operator-shorthand}
	\end{align}
	\begin{lemma}\label{psi-L-correct-form-lemma}
		For all $\varepsilon>0$ small enough, and $t\geq T_{0}$ as in \eqref{K-T0-choices}, on the support of $U_{R}$ restricted to the upper half plane, we have that
		\begin{align}
			&\Psi_{L}+\varepsilon\left(\dot{p}-c\right)\left(y_{2}-q'\right)-\varepsilon\dot{q}y_{1}=\varepsilon\left(\mathcal{N}\left(\xi_{0}\right)\left[\xi_{1}\right]+\dot{\xi}_{1}\right)\cdot\left(y-q'e_{2}\right)^{\perp}\label{psi-L-correct-form-statement}\\
			&+\varepsilon\left(\mathcal{N}\left(\xi_{0}+\xi_{1}\right)\left[\tilde{\xi}\right]+\dot{\tilde{\xi}}\right)\cdot\left(y-q'e_{2}\right)^{\perp}+\mathcal{E}_{R1}\left(y-q'e_{2},\xi\right)+\overline{\mathcal{R}}_{2}\left(y,\xi\right)+\mathcal{C}_{1}\left(t\right),\nonumber
		\end{align}
		with $\mathcal{C}_{1}$ a time-dependent function given by \eqref{C1-def}, and
		\begin{align*}
			\mathcal{E}_{R1}\left(y-q'e_{2},\xi\right)&=\mathcal{E}_{2}\left(y-q'e_{2},\xi\right)+\mathcal{E}_{3}\left(y-q'e_{2},\xi\right)+\mathcal{E}_{4}\left(y-q'e_{2},\xi\right)\nonumber\\
			&+\frac{\varepsilon^{2}\Lambda_{01}\left(q'\right)}{p^{2}}\left(y_{2}-q'\right)+\frac{\varepsilon q\Lambda_{00}\left(q'\right)}{p^{2}}\left(y_{2}-q'\right),
		\end{align*}
		where $\Lambda_{01}=\bigO{\left(\varepsilon^{3}\right)}$, $\Lambda_{00}=\bigO{\left(\varepsilon^{4}\right)}$, and in coordinates $y-q'e_{2}=re^{i\theta}$, the $\mathcal{E}_{j}$ are linear combinations of $\left(r^{j}\sin{j\theta},r^{j}\cos{j\theta}\right)$ for $j=2,3,4$, given by the formulae \eqref{first-approximation-first-error-mode-2}, \eqref{first-approximation-first-error-mode-1-3}, and \eqref{first-approximation-first-error-mode-4} respectively. Finally,
		\begin{align}
			\left(f_{\varepsilon}'\right)^{+}\left(\left|\overline{\mathcal{R}}_{2}\right|+\left|\grad_{y}\overline{\mathcal{R}}_{2}\right|\right)\leq \frac{C\varepsilon^{5}\left(f_{\varepsilon}'\right)^{+}}{t^{3}}.\label{R2-error-estimate}
		\end{align}
		An analogous statement holds on the support of $\left(f_{\varepsilon}'\right)^{-}$ in coordinates $y+q'e_{2}$, as well as for the quantity $\Psi_{R}+\varepsilon\left(\dot{p}-c\right)\left(w_{2}-q'\right)-\varepsilon\dot{q}w_{1}$ on the support of $U_{L}$ restricted to the upper half plane in coordinates $\varepsilon w=x+pe_{1}$.
	\end{lemma}
	\begin{proof}
		First we note that on the support of $\left(f_{\varepsilon}'\right)^{+}$,
		\begin{align}
			\Psi_{L}&=-\frac{m}{2}\log{\left(1+\frac{\varepsilon y_{1}}{p}+\frac{\varepsilon^{2}\left|y-q'e_{2}\right|^{2}}{4p^{2}}\right)}\nonumber\\
			&+\frac{m}{2}\log{\left(1+\frac{\varepsilon\left(py_{1}+q\left(y_{2}-q'\right)\right)}{p^{2}+q^{2}}+\frac{\varepsilon^{2}\left|y-q'e_{2}\right|^{2}}{4\left(p^{2}+q^{2}\right)}\right)}\nonumber\\
			&-\frac{m}{2}\log{\left(\frac{p^{2}+q^{2}}{p^{2}}\right)}+\psi_{q'}\left(y+2p'e_{1}\right).\label{psi-L-correct-form-1}
		\end{align}
		We begin by concentrating on the first two terms on the right hand side of \eqref{psi-L-correct-form-1}. Letting $y-q'e_{2}=re^{i\theta}$, on the support of $\left(f_{\varepsilon}'\right)^{+}$, the first term on the right hand side can be written as
		\begin{align}
			-\frac{m}{2}\left(\frac{\varepsilon y_{1}}{p}-\frac{\varepsilon^{2}r^{2}\cos{\left(2\theta\right)}}{4p^{2}}-\frac{\varepsilon^{3}r^{3}\sin{\left(3\theta\right)}}{12p^{3}}\right)-\frac{m\varepsilon^{4}r^{4}\cos{\left(4\theta\right)}}{64p^{4}}+\mathcal{R}_{3}.\label{psi-L-correct-form-3}
		\end{align}
		where $\mathcal{R}_{3}=\bigO{\left(\varepsilon^{5}t^{-5}\right)}$. 
		
		\medskip
		Next we have that on the support of $\left(f_{\varepsilon}'\right)^{+}$, as with \eqref{psi-L-correct-form-3}, we can rewrite the second term on the right hand side of \eqref{psi-L-correct-form-1} as
		\begin{align*}
			&\frac{m\varepsilon\left(py_{1}+q\left(y_{2}-q'\right)\right)}{2\left(p^{2}+q^{2}\right)}-\frac{m\varepsilon^{2}r^{2}\left(\left(q^{2}-p^{2}\right)\cos{\left(2\theta\right)}+2pq\sin{\left(2\theta\right)}\right)}{8\left(p^{2}+q^{2}\right)^{2}}\\
			&+\frac{m\varepsilon^{3}r^{3}\sin{\left(3\theta\right)}}{24}\frac{\left(p^{3}-3pq^{2}\right)}{\left(p^{2}+q^{2}\right)^{3}}+\frac{m\varepsilon^{3}r^{3}\cos{\left(3\theta\right)}}{24}\frac{\left(q^{3}-3p^{2}q\right)}{\left(p^{2}+q^{2}\right)^{3}}\nonumber\\
			&-\frac{m\varepsilon^{4}r^{4}\cos{\left(4\theta\right)}}{64}\left(\frac{\left(p^{2}-q^{2}\right)^{2}-4p^{2}q^{2}}{\left(p^{2}+q^{2}\right)^{4}}\right)-\frac{m\varepsilon^{4}r^{4}\sin{\left(4\theta\right)}}{64}\left(\frac{pq\left(q^{2}-p^{2}\right)}{\left(p^{2}+q^{2}\right)^{4}}\right)+\mathcal{R}_{4},\nonumber
		\end{align*}
		where, again, $\mathcal{R}_{4}=\bigO{\left(\varepsilon^{5}t^{-5}\right)}$ on the support of $\left(f_{\varepsilon}'\right)^{+}$. 
		
		\medskip
		For $\psi_{q'}\left(y+2p'e_{1}\right)$, we use the fact that $\psi_{q'}$ solves \eqref{psi-q-elliptic-equation}, so that we can write 
		\begin{align}
			\psi_{q'}\left(y+2p'e_{1}\right)=\frac{1}{4\pi}\int_{\mathbb{R}_{+}^{2}}\log{\left(\frac{\left|z-\overline{\left(y+2p'e_{1}\right)}\right|^{2}}{\left|z-\left(y+2p'e_{1}\right)\right|^{2}}\right)}\left(V\psi_{q'}-E+N[\psi_{q'}]\right) dz.\label{psi-q-prime-left}
		\end{align}
		We can write 
		\begin{align*}
			\log{\left(\frac{\left|z-\overline{\left(y+2p'e_{1}\right)}\right|^{2}}{\left|z-\left(y+2p'e_{1}\right)\right|^{2}}\right)}=\log{\left(1+\frac{4\left(\left(z_{2}-q'\right)+q'\right)\left(y_{2}-q'\right)+4q'\left(\left(z_{2}-q'\right)+q'\right)}{\left(\left(z_{1}-y_{1}\right)+2p'\right)^{2}+\left(\left(z_{2}-q'\right)-\left(y_{2}-q'\right)\right)^{2}}\right)}.
		\end{align*}
		Then using \eqref{psi-q-elliptic-equation}, \eqref{travelling-vortex-pair-ansatz}, and Theorems \ref{vortexpair-theorem} and \ref{vortex-pair-properties-theorem}, we can write \eqref{psi-q-prime-left} as a linear combination of
		\begin{align}
			&\Lambda_{00}\left(q'\right)\coloneqq\frac{1}{4\pi}\int_{\mathbb{R}_{+}^{2}} \left(V\psi_{q'}-E+N[\psi_{q'}]\right) dz=\bigO{\left(\varepsilon^{4}\right)},\label{Lambda-ij-sizes}\\
			&\Lambda_{ij}\left(q'\right)\coloneqq\frac{1}{4\pi}\int_{\mathbb{R}_{+}^{2}} z_{1}^{i}\left(z_{2}-q'\right)^{j}\left(V\psi_{q'}-E+N[\psi_{q'}]\right) dz=\bigO{\left(\varepsilon^{3}\right)},\nonumber\\
			&\Lambda_{lk}\left(q'\right)=\frac{1}{4\pi}\int_{\mathbb{R}_{+}^{2}} z_{1}^{l}\left(z_{2}-q'\right)^{k}\left(V\psi_{q'}-E+N[\psi_{q'}]\right) dz=\bigO{\left(\varepsilon^{2}\right)},\nonumber
		\end{align}
		for $\left(i,j\right)\in\left\{\left(1,0\right), \left(0,1\right)\right\}$ and $\left(l,k\right)\in\left\{\left(1,1\right), \left(2,0\right), \left(0,2\right)\right\}$. Using \eqref{Lambda-ij-sizes}, we obtain that on the support of $\left(f_{\varepsilon}'\right)^{+}$,
		\begin{align*}
			&\psi_{q'}\left(y+2p'e_{1}\right)=\frac{\varepsilon^{2}\Lambda_{01}}{p^{2}}\left(y_{2}-q'\right)+\frac{\varepsilon q\Lambda_{00}}{p^{2}}\left(y_{2}-q'\right)\nonumber\\
			&+\frac{\varepsilon q\Lambda_{01}}{p^{2}}-\frac{\varepsilon^{2} q\Lambda_{11}}{p^{3}}+\frac{\varepsilon q^{3}\Lambda_{01}}{p^{4}}-2\frac{\varepsilon^{2}q^{3}\Lambda_{11}}{p^{5}}\\
			&+\frac{1}{4\pi}\int_{\mathbb{R}_{+}^{2}}\log{\left(1+\frac{q^{2}}{p^{2}}\frac{1}{1+\frac{\varepsilon z_{1}}{p}+\frac{\varepsilon^{2}\left|z-q'e_{2}\right|^{2}}{4p^{2}}}\right)}\left(V\psi_{q'}-E+N[\psi_{q'}]\right) dz+\mathcal{R}_{5}\nonumber
		\end{align*}
		where $\mathcal{R}_{5}=\bigO{\left(\varepsilon^{5}t^{-3}\right)}$. Therefore, we have that on the support of $\left(f_{\varepsilon}'\right)^{+}$,
		\begin{align}
			&\Psi_{L}+\varepsilon\left(\dot{p}-c\right)\left(y_{2}-q'\right)-\varepsilon\dot{q}y_{1}\nonumber\\
			&=\varepsilon\left(-\mathcal{F}\left(q,p\right)-g\left(q'\right)+\dot{p}\right)\left(y_{2}-q'\right)-\varepsilon\left(\mathcal{F}\left(p,q\right)+\dot{q}\right)y_{1}\nonumber\\
			&+\mathcal{E}_{2}+\mathcal{E}_{3}+\mathcal{E}_{4}+\mathcal{R}_{3}+\mathcal{R}_{4}+\mathcal{R}_{5}+\mathcal{C}_{1}+\frac{\varepsilon^{2}\Lambda_{01}}{p^{2}}\left(y_{2}-q'\right)+\frac{\varepsilon q\Lambda_{00}}{p^{2}}\left(y_{2}-q'\right),\label{psi-L-correct-form-9a}
		\end{align}
		where $\mathcal{C}_{1}(t)$ is defined as the time-dependent function
		\begin{align}
			\psi_{q'}\left(y+2p'e_{1}\right)-\frac{\varepsilon^{2}\Lambda_{01}}{p^{2}}\left(y_{2}-q'\right)+\frac{\varepsilon q\Lambda_{00}}{p^{2}}\left(y_{2}-q'\right)-\mathcal{R}_{5}-\frac{m}{2}\log{\left(\frac{p^{2}+q^{2}}{p^{2}}\right)},\label{C1-def}
		\end{align}
		and $\mathcal{F}$ is defined in \eqref{pkr-system}.
		
		\medskip
		Letting $y-q'e_{2}=re^{i\theta}$, $\mathcal{E}_{2}=\mathcal{E}_{2}\left(y-q'e_{2},p,q\right)$ is defined as
		\begin{align}
			&\frac{m\varepsilon^{2}r^{2}}{8}\left(-\frac{\cos{\left(2\theta\right)}}{p^{2}}+\frac{\left(p^{2}-q^{2}\right)\cos{\left(2\theta\right)}+2pq\sin{\left(2\theta\right)}}{\left(p^{2}+q^{2}\right)^{2}}\right)\label{first-approximation-first-error-mode-2}
		\end{align}
		$\mathcal{E}_{3}=\mathcal{E}_{3}\left(y-q'e_{2},p,q\right)$ is defined as
		\begin{align}
			&\frac{m\varepsilon^{3}r^{3}}{24}\left(\sin{\left(3\theta\right)}\left(\frac{1}{p^{3}}+\frac{\left(p^{3}-3pq^{2}\right)}{\left(p^{2}+q^{2}\right)^{3}}\right)+\cos{\left(3\theta\right)}\frac{\left(q^{3}-3p^{2}q\right)}{\left(p^{2}+q^{2}\right)^{3}}\right),\label{first-approximation-first-error-mode-1-3}
		\end{align}
		and $\mathcal{E}_{4}=\mathcal{E}_{4}\left(y-q'e_{2},p,q\right)$ is defined as
		\begin{align}
			&\frac{m\varepsilon^{4}r^{4}}{64}\left(-\frac{\cos{\left(4\theta\right)}}{p^{4}}-\cos{\left(4\theta\right)}\left(\frac{\left(p^{2}-q^{2}\right)^{2}-4p^{2}q^{2}}{\left(p^{2}+q^{2}\right)^{4}}\right)-\sin{\left(4\theta\right)}\left(\frac{pq\left(q^{2}-p^{2}\right)}{\left(p^{2}+q^{2}\right)^{4}}\right)\right).\label{first-approximation-first-error-mode-4}
		\end{align}
		We will use $\mathcal{E}_{j}\left(y-q'e_{2},p,q\right)$, $\mathcal{E}_{j}\left(y-q'e_{2},\xi\right)$, $\mathcal{E}_{j}\left(p,q\right)$, and $\mathcal{E}_{j}\left(\xi\right)$ interchangeably for $j=2,3,4$.
		
		\medskip
		Now, concentrating on the first two terms on the right hand side of \eqref{psi-L-correct-form-9a}, we recall that $\left(p_{0},q_{0}\right)$ solves the system \eqref{p0-q0-system}. Then, letting 
		\begin{align}
			\mathcal{E}_{R1}\left(y-q'e_{2},\xi\right)&=\mathcal{E}_{2}\left(y-q'e_{2},\xi\right)+\mathcal{E}_{3}\left(y-q'e_{2},\xi\right)+\mathcal{E}_{4}\left(y-q'e_{2},\xi\right)\nonumber\\
			&+\frac{\varepsilon^{2}\Lambda_{01}}{p^{2}}\left(y_{2}-q'\right)+\frac{\varepsilon q\Lambda_{00}}{p^{2}}\left(y_{2}-q'\right),\label{ER1-first-approximation-definition}
		\end{align}
		and defining
		\begin{align}
			\overline{\mathcal{R}}_{2}\left(y-q'e_{2},\xi\right)=\mathcal{R}_{3}\left(y-q'e_{2},\xi\right)+\mathcal{R}_{4}\left(y-q'e_{2},\xi\right)+\mathcal{R}_{5}\left(y-q'e_{2},\xi\right),\label{overline-R2-definition}
		\end{align}
		and comparing \eqref{psi-L-correct-form-9a} with \eqref{psi-L-correct-form-statement} keeping the definition \eqref{linearised-point-vortex-operator-shorthand} in mind, we obtain the desired identity.
	\end{proof}
	Note that $\mathcal{C}_{1}$ defined in \eqref{C1-def} does not play an important role in our construction as we will in reality be dealing with the quantity $\nabla_{y}\left(\Psi_{L}+\varepsilon\left(\dot{p}-c\right)\left(y_{2}-q'\right)-\varepsilon\dot{q}y_{1}\right)$, and the spatial derivatives will annihilate $\mathcal{C}_{1}$. We include it in the statement \eqref{psi-L-correct-form-statement} for completeness.
	
	\medskip
	Now we construct the first approximation. Recall the definitions of $\psi^{in}$, $\phi^{in}$, $E_{R1}$, $E_{L1}$, and $E_{2}^{out}$ from \eqref{psi-in-definition}--\eqref{first-approximation-outer-error-E2out-definition}.
	\begin{theorem}\label{first-approximation-construction-theorem}
		Let $t\in\left[T_{0},T\right]$ for $T_{0}$ as in \eqref{K-T0-choices} and $T$ arbitrarily large. Then there is a constant $C>0$ such that for all $\varepsilon>0$ small enough, there exists $\xi_{1}$ satisfying \eqref{xi1-bound} such that for any $\xi$ as in \eqref{point-vortex-trajectory-function-decomposition-new}, there exists functions $\psi_{jk}$, $\phi_{jk}$, $j=R,L$, $k=1,2$, $\psi^{out}$ such that $\phi_{Rk}$ has the same support as $U_{R}$, and
		\begin{align*}
			\|\psi_{R1}\|_{L^{\infty}\left(\mathbb{R}^{2}_{+}\right)}+\|\nabla\psi_{R1}\|_{L^{\infty}\left(\mathbb{R}^{2}_{+}\right)}+\|\phi_{R1}\|_{L^{\infty}\left(\mathbb{R}^{2}_{+}\right)}&\leq \frac{C\varepsilon^{2}}{t^{2}},\\
			\|\psi_{R2}\|_{L^{\infty}\left(\mathbb{R}^{2}_{+}\right)}+\|\nabla\psi_{R2}\|_{L^{\infty}\left(\mathbb{R}^{2}_{+}\right)}+\|\phi_{R2}\|_{L^{\infty}\left(\mathbb{R}^{2}_{+}\right)}&\leq \frac{C\varepsilon^{4}}{t^{2}},
		\end{align*}
		with analogous estimates for $\psi_{Lk}$, $\phi_{Lk}$, $k=1,2$ and $U_{L}$. Moreover,
		\begin{align*}
			\|\psi^{out}\|_{L^{\infty}\left(\mathbb{R}^{2}_{+}\right)}+\|\nabla\psi^{out}\|_{L^{\infty}\left(\mathbb{R}^{2}_{+}\right)}\leq\frac{C\varepsilon^{4}\left|\log{t}\right|}{t^{3}},
		\end{align*}
		such that if $\psi_{R}=\psi_{R1}+\psi_{R2}$, $\phi_{R}=\phi_{R1}+\phi_{R2}$ in \eqref{psi-star-definition}, and analogously for $\psi_{L}$ and $\phi_{L}$, we have, on $\mathbb{R}^{2}_{+}$, for some functions $\mathcal{R}^{*}_{k}=\mathcal{R}^{*}_{k}\left(y,\xi,\dot{\xi}\right)$, $k=1,2,3$, we have
		\begin{align*}
			E_{R1}&=-\varepsilon\left(\mathcal{N}\left(\xi_{0}+\xi_{1}\right)\left[\tilde{\xi}\right]+\dot{\tilde{\xi}}\right)\cdot\grad_{y}U_{R}+\left(f_{\varepsilon}'\right)^{+}\mathcal{R}^{*}_{1}+\left(f_{\varepsilon}''\right)^{+}\mathcal{R}^{*}_{2}+\left(f_{\varepsilon}'''\right)^{+}\mathcal{R}^{*}_{3},
		\end{align*}
		with,
		\begin{align*}
			\left|\left(f_{\varepsilon}'\right)^{+}\mathcal{R}^{*}_{1}\right|\leq\frac{C\varepsilon^{5}\left(f_{\varepsilon}'\right)^{+}}{t^{3}},\quad \left|\left(f_{\varepsilon}''\right)^{+}\mathcal{R}^{*}_{2}\right|\leq\frac{C\varepsilon^{5}\left(f_{\varepsilon}''\right)^{+}}{t^{3}},\quad \left|\left(f_{\varepsilon}'''\right)^{+}\mathcal{R}^{*}_{3}\right|\leq\frac{C\varepsilon^{5}}{t^{3}}\left(f_{\varepsilon}'''\right)^{+},
		\end{align*}
		and analogously for $E_{R1}$ on the lower half plane, as well as for $E_{L1}$. Finally,
		\begin{align*}
			E_{2}^{out}\left(\phi^{in},\psi^{in},\psi^{out};\xi\right)=\bigO{\left(\frac{\varepsilon^{5}}{t^{5}}\right)}.
		\end{align*}
	\end{theorem}
	\begin{remark}\label{grad-psi-R-psi-out-bounds-remark}
		Recalling \eqref{psi-R-phi-R-defs} and \eqref{psi-star-definition} respectively, we note that the gradient estimates stated in Theorem \ref{first-approximation-construction-theorem} are for $\nabla_{y}\psi_{R}\left(y,t\right)$, and $\nabla_{x}\psi^{out}\left(x,t\right)$ respectively.
	\end{remark}
	The construction will proceed as follows: we rewrite $E_{R1}$ from its form in Lemma \ref{ER1-lemma} and then construct an elliptic improvement. We then use this elliptic improvement to construct $\psi^{out}$. Then we analyze the errors created by the first elliptic improvement in order to both construct a second elliptic improvement, and finally construct $\xi_{1}$ to make sure the remaining errors are small enough.
	\begin{proof} 
		\textbf{Step I: Rewriting $E_{R1}$ Part A}
		
		\medskip
		We first decompose $\phi_{R}=\phi_{R1}+\phi_{R2}$ and $\psi_{R}=\psi_{R}=\psi_{R1}+\psi_{R2}$ so that
		\begin{align*}
			-\Delta \psi_{Rk}=\phi_{Rk},\ k=1,2,
		\end{align*}
		where the Poisson equations mean solving in $L^{\infty}\left(\mathbb{R}^2_{+}\right)$, with boundary condition given by $\psi_{Rk}\left(y_{1},0\right)\equiv0$. We then note that by , on the support of $\left(f_{\varepsilon}'\right)^{+}$, $\dot{q}\overline{\mathcal{R}}_{1}=\bigO{\left(\varepsilon^{5}t^{-3}\right)}$. Let
		\begin{align*}
			\mathcal{R}_{6}&=\mathcal{E}_{2}\left(\xi\right)-\mathcal{E}_{2}\left(\xi_{0}\right),\ \mathcal{R}_{7}=\mathcal{E}_{3}\left(\xi\right)-\mathcal{E}_{3}\left(\xi_{0}\right),\ \mathcal{R}_{8}=\mathcal{E}_{4}\left(\xi\right)-\mathcal{E}_{4}\left(\xi_{0}\right),\\
			\mathcal{R}_{9}&=\left[\left(\frac{\varepsilon^{2}\Lambda_{01}\left(q'\right)}{p^{2}}+\frac{\varepsilon q\Lambda_{00}\left(q'\right)}{p^{2}}\right)-\left(\frac{\varepsilon^{2}\Lambda_{01}\left(q_{0}'\right)}{p_{0}^{2}}+\frac{\varepsilon q_{0}\Lambda_{00}\left(q_{0}'\right)}{p_{0}^{2}}\right)\right]\left(y_{2}-q'\right)\\
			&+\overline{\mathcal{R}}_{2}\left(\xi\right)-\overline{\mathcal{R}}_{2}\left(\xi_{0}\right).
		\end{align*}
		Let $\overline{\mathcal{R}}_{3}=\mathcal{R}_{6}+\dots+\mathcal{R}_{9}$. Using \eqref{point-vortex-trajectory-function-decomposition-new}--\eqref{tildexi-bound}, \eqref{dq-mode-1-error}, \eqref{first-approximation-first-error-mode-2}--\eqref{overline-R2-definition}, \eqref{Lambda-ij-sizes}, and Lemmas \ref{p-q-rational-function-linearisation-lemma}, \ref{correct-remainder-form-lemma}, \eqref{g-pointwise-bounds} and \ref{vortex-pair-q-variation-lemma}, we have that
		\begin{align}
			&\left|-\dot{q}_{0}\left(f_{\varepsilon}'\right)^{+}\gradperp_{y}\left(\Psi_{R}-c\varepsilon y_{2}\right)\cdot\grad_{y}\left(\left(\mathcal{M}_{2}(q)-\mathcal{M}_{2}(q_{0})\right)y_{1}\left(y_{2}-q'\right)\right)\right|\nonumber\\
			&+\left|\left(f_{\varepsilon}'\right)^{+}\gradperp_{y}\left(\Psi_{R}-c\varepsilon y_{2}\right)\cdot\grad_{y}\left(\overline{\mathcal{R}}_{3}\right)\right|\nonumber+\left|\dot{q}\left(f_{\varepsilon}'\right)^{+}\overline{\mathcal{R}}_{1}\left(y,q\right)\right|\nonumber\\
			&+\left|-\left(\dot{q}_{1}+\dot{\tilde{q}}\right)\left(f_{\varepsilon}'\right)^{+}\gradperp_{y}\left(\Psi_{R}-c\varepsilon y_{2}\right)\cdot\grad_{y}\left(\mathcal{M}_{2}(q)y_{1}\left(y_{2}-q'\right)\right)\right|\leq \frac{C\varepsilon^{5}\left(f_{\varepsilon}'\right)^{+}}{t^{3}}.\label{first-approximation-construction-2a}
		\end{align}
		Letting $\overline{\mathcal{R}}_{4}$ be the sum of the error terms estimated in \eqref{first-approximation-construction-2a}, we can use Lemma \ref{psi-L-correct-form-lemma} to rewrite \eqref{ER1-lemma-statement-1}:
		\begin{align}
			&E_{R1}=\left(\varepsilon^{2}\dell_{t}\phi_{R}-\left[\gradperp_{y}\Psi_{L}+\varepsilon\left(\Dot{p}-c\right)e_{1}\right]\cdot\grad_{y}\phi_{R}\right)+\gradperp_{y}\psi_{R}\cdot\grad_{y}\phi_{R}\label{first-approximation-construction-8}\\
			&-\gradperp_{y}\left(\Psi_{R}-c\varepsilon y_{2}\right)\cdot\grad_{y}\left(\Delta\psi_{R}+\left(f_{\varepsilon}'\right)^{+}\psi_{R}-\left(f_{\varepsilon}'\right)^{+}\Sigma_{j=2}^{4}\mathcal{E}_{j}\left(\xi_{0}\right)-\left(f_{\varepsilon}'\right)^{+}\overline{\mathcal{R}}_{2}\left(\xi_{0}\right)\right)\nonumber\\
			&+\gradperp_{y}\left(\Psi_{R}-c\varepsilon y_{2}\right)\cdot\grad_{y}\left(\left(f_{\varepsilon}'\right)^{+}\left(\frac{\varepsilon^{2}\Lambda_{01}\left(q_{0}'\right)}{p_{0}^{2}}+\frac{\varepsilon q_{0}\Lambda_{00}\left(q_{0}'\right)}{p_{0}^{2}}\right)\left(y_{2}-q'\right)\right)\nonumber\\
			&+\gradperp_{y}\left(\Psi_{R}-c\varepsilon y_{2}\right)\cdot\grad_{y}\left(\varepsilon\left(f_{\varepsilon}'\right)^{+}\left(\mathcal{N}\left(\xi_{0}\right)\left[\xi_{1}\right]+\dot{\xi}_{1}\right)\cdot\left(y-q'e_{2}\right)^{\perp}\right)\nonumber\\
			&-\gradperp_{y}\left(\Psi_{R}-c\varepsilon y_{2}\right)\cdot\grad_{y}\left(\dot{q}_{0}\left(f_{\varepsilon}'\right)^{+}\mathcal{M}_{2}(q_{0})y_{1}\left(y_{2}-q'\right)\right)\nonumber\\
			&+\gradperp_{y}\left(\Psi_{R}-c\varepsilon y_{2}\right)\cdot\grad_{y}\left(\varepsilon\left(f_{\varepsilon}'\right)^{+}\left(\mathcal{N}\left(\xi_{0}+\xi_{1}\right)\left[\tilde{\xi}\right]+\dot{\tilde{\xi}}\right)\cdot\left(y-q'e_{2}\right)^{\perp}+\left(f_{\varepsilon}'\right)^{+}\overline{\mathcal{R}}_{4}\right)\nonumber\\
			&+\gradperp_{y}\psi^{out}\cdot\grad_{y}\left(U_{R}+\phi_{R1}+\phi_{R2}\right).\nonumber
		\end{align}
		
		\medskip
		\noindent \textbf{Step II: Constructing $\psi_{R1}$}
		
		\medskip
		We construct $\psi_{R1}$ on the upper half plane $\mathbb{R}^{2}_{+}$ with boundary condition $\psi_{R1}\left(y_{1},0\right)=0$ by solving the equation
		\begin{align}
			&\Delta\psi_{R1}+\left(f_{\varepsilon}'\right)^{+}\psi_{R1}=\left(f_{\varepsilon}'\right)^{+}\left(\mathcal{E}_{2}\left(\xi_{0}\right)+\mathcal{E}_{3}\left(\xi_{0}\right)+\mathcal{E}_{4}\left(\xi_{0}\right)\right)+\left(f_{\varepsilon}'\right)^{+}\overline{\mathcal{R}}_{2}\left(\xi_{0}\right)\label{psi-R1-equation}\\
			&+\left(f_{\varepsilon}'\right)^{+}\left(\frac{\varepsilon^{2}\Lambda_{01}\left(q_{0}'\right)}{p_{0}^{2}}+\frac{\varepsilon q_{0}\Lambda_{00}\left(q_{0}'\right)}{p_{0}^{2}}\right)\left(y_{2}-q'\right)-\dot{q}_{0}\left(f_{\varepsilon}'\right)^{+}\mathcal{M}_{2}(q_{0})y_{1}\left(y_{2}-q'\right)\nonumber\\
			&+\alpha_{1}\left(\xi\right)\left(f_{\varepsilon}'\right)^{+}\mathscr{Z}_{1}\left(q\right)+\alpha_{2}\left(\xi\right)\left(f_{\varepsilon}'\right)^{+}\mathscr{Z}_{2}\left(q\right),\nonumber
		\end{align}
		where we recall from \eqref{dq-Psi-R-variation} that $\dell_{1}\Psi_{R}=\mathscr{Z}_{1}\left(q\right)$ and $\dell_{q'}\left(\Psi_{R}-c\varepsilon y_{2}-\left|\log{\varepsilon}\right|\Omega\right)=\mathscr{Z}_{2}\left(q\right)$. The theory for this projected problem can be adapted by modifying the proofs of Lemmas 2.1--2.3, and Theorem 2.4 from \cite{DDMPVP2023}, from the linear operator $\Delta+V$, to the linear operator $\Delta+\left(f_{\varepsilon}'\right)^{+}$, with $V$ defined in \eqref{V-E-vortex-pair}. Thus $\psi_{R1}$ exists, and is unique. The theory of solutions adapted from \cite{DDMPVP2023}, as well as Lemma \ref{psi-L-correct-form-lemma}, imply that $\psi_{R1},\alpha_{1}$, and $\alpha_{2}$ have bounds
		\begin{align}
			\|\psi_{R1}\|_{L^{\infty}\left(\mathbb{R}^{2}_{+}\right)}+\|\nabla\psi_{R1}\|_{L^{\infty}\left(\mathbb{R}^{2}_{+}\right)}+\|\phi_{R1}\|_{L^{\infty}\left(\mathbb{R}^{2}_{+}\right)}+\left|\alpha_{1}\right|+\left|\alpha_{2}\right|\leq \frac{C\varepsilon^{2}}{t^{2}}.\label{psi-R1-bounds}
		\end{align}
		We can find better bounds for $\alpha_{1}$ and $\alpha_{2}$. Noting $\mathscr{Z}_{1}\left(q\right)$ and $\mathscr{Z}_{2}\left(q\right)$ are in the kernel of $\Delta + \left(f_{\varepsilon}'\right)^{+}$, we first integrate \eqref{psi-R1-equation} against $\mathscr{Z}_{1}(q)$ and obtain
		\begin{align}
			\alpha_{1}\left(\xi\right)\left(\int_{\mathbb{R}^{2}_{+}}\left(f_{\varepsilon}'\right)^{+}\mathscr{Z}_{1}^{2}\right)&=-\int_{\mathbb{R}^{2}_{+}}\left(f_{\varepsilon}'\right)^{+}\left(\mathcal{E}_{2}+\mathcal{E}_{3}+\mathcal{E}_{4}+\overline{\mathcal{R}}_{2}\right)\left(\xi_{0}\right)\mathscr{Z}_{1}\ dy\label{psi-R1-coefficient-of-projection-1-1}\\
			&-\int_{\mathbb{R}^{2}_{+}}\left(f_{\varepsilon}'\right)^{+}\left(\frac{\varepsilon^{2}\Lambda_{01}\left(q_{0}'\right)}{p_{0}^{2}}+\frac{\varepsilon q_{0}\Lambda_{00}\left(q_{0}'\right)}{p_{0}^{2}}\right)\left(y_{2}-q'\right)\mathscr{Z}_{1}\ dy\nonumber\\
			&+\int_{\mathbb{R}^{2}_{+}}\dot{q}_{0}\left(f_{\varepsilon}'\right)^{+}\mathcal{M}_{2}(q_{0})y_{1}\left(y_{2}-q'\right)\mathscr{Z}_{1}\ dy.\nonumber
		\end{align}
		Now, as $\Psi_{R}$ was constructed to be even in $y_{1}$, we have that $\left(f_{\varepsilon}'\right)^{+}\mathscr{Z}_{1}(q)$ is odd in $y_{1}$. Thus, noting the definition of $\mathcal{E}_{2}$, $\mathcal{E}_{3}$, and $\mathcal{E}_{4}$ in \eqref{first-approximation-first-error-mode-2}--\eqref{first-approximation-first-error-mode-4}, \eqref{psi-R1-coefficient-of-projection-1-1} becomes
		\begin{align}
			&\alpha_{1}\left(\xi\right)\left(\int_{\mathbb{R}^{2}_{+}}\left(f_{\varepsilon}'\right)^{+}\mathscr{Z}_{1}^{2}\right)=-\frac{m\varepsilon^{2}}{2}\int_{\mathbb{R}^{2}_{+}}\left(f_{\varepsilon}'\right)^{+}\mathscr{Z}_{1}\frac{p_{0}q_{0}y_{1}\left(y_{2}-q'\right)}{\left(p_{0}^{2}+q_{0}^{2}\right)^{2}} dy\label{psi-R1-coefficient-of-projection-1-2}\\
			&+\frac{m\varepsilon^{3}}{2}\int_{\mathbb{R}^{2}_{+}}\left(f_{\varepsilon}'\right)^{+}\mathscr{Z}_{1}\left(\frac{1}{12p_{0}^{3}}+\frac{\left(p_{0}^{3}-3p_{0}q_{0}^{2}\right)}{12\left(p_{0}^{2}+q_{0}^{2}\right)^{3}}\right)\left(4y_{1}^{3}-3\left|y-q'e_{2}\right|^{2}y_{1}\right) dy\nonumber\\
			&-\int_{\mathbb{R}^{2}_{+}}\left(f_{\varepsilon}'\right)^{+}\left(\mathcal{E}_{4}+\overline{\mathcal{R}}_{2}\right)\left(\xi_{0}\right)\mathscr{Z}_{1}\ dy+\int_{\mathbb{R}^{2}_{+}}\dot{q}_{0}\left(f_{\varepsilon}'\right)^{+}\mathcal{M}_{2}(q_{0})y_{1}\left(y_{2}-q'\right)\mathscr{Z}_{1}\ dy.\nonumber
		\end{align}
		From \eqref{dq-Psi-R-variation}, we see that
		\begin{align*}
			\dell_{1}\Psi_{R}=\dell_{1}\Gamma\left(\left|y-q'e_{2}\right|\right)+\bigO{\left(\varepsilon^{2}\right)}
		\end{align*}
		on the support of $\left(f_{\varepsilon}'\right)^{+}$, and note that the first term on the right hand side is mode $1$ in coordinates $y-q'e_{2}$. We can also infer from Theorem \ref{vortex-pair-properties-theorem} and \eqref{f-prime-plus-definition}, that $\left(f_{\varepsilon}'\right)^{+}$ is, to main order, a radial function in $y-q'e_{2}$, with a remainder of size $\varepsilon^{2}$. Combining these facts with \eqref{xi0-bound}, we obtain
		\begin{align}
			\left|\alpha_{1}\right|\leq \frac{C\varepsilon^{4}}{t^{3}}.\label{psi-R1-coefficient-of-projection-1-3}
		\end{align}
		Next, recalling \eqref{dq-Psi-R-variation}, we integrate \eqref{psi-R1-equation} against $\mathscr{Z}_{2}\left(q\right)$. Similarly to \eqref{psi-R1-coefficient-of-projection-1-1}--\eqref{psi-R1-coefficient-of-projection-1-2}, we use the fact that $\mathscr{Z}_{2}\left(q\right)$ is even in $y_{1}$ to obtain
		\begin{align}
			&\alpha_{2}\left(\xi\right)\left(\int_{\mathbb{R}^{2}_{+}}\left(f_{\varepsilon}'\right)^{+}\mathscr{Z}_{2}^{2}\right)=-\int_{\mathbb{R}^{2}_{+}}\left(f_{\varepsilon}'\right)^{+}\left(\mathcal{E}_{4}+\overline{\mathcal{R}}_{2}\right)\left(\xi_{0}\right)\mathscr{Z}_{1}\ dy\label{psi-R1-coefficient-of-projection-2-1}\\
			&-\frac{m\varepsilon^{2}}{2}\int_{\mathbb{R}^{2}_{+}}\left(f_{\varepsilon}'\right)^{+}\mathscr{Z}_{2}\left(\frac{1}{4p_{0}^{2}}-\frac{\left(p_{0}^{2}-q_{0}^{2}\right)}{4\left(p_{0}^{2}+q_{0}^{2}\right)^{2}}\right)\left(y_{1}^{2}-\left(y_{2}-q'\right)^{2}\right) dy\nonumber\\
			&-\frac{m\varepsilon^{3}}{2}\int_{\mathbb{R}^{2}_{+}}\left(f_{\varepsilon}'\right)^{+}\mathscr{Z}_{2}\frac{\left(q_{0}^{3}-3p_{0}^{2}q_{0}\right)}{12\left(p_{0}^{2}+q_{0}^{2}\right)^{3}}\left(3\left|y-q'e_{2}\right|^{2}\left(y_{2}-q'\right)-4\left(y_{2}-q'\right)^{3}\right) dy.\nonumber
		\end{align}
		The modal considerations when calculating the size of $\alpha_{1}$ in \eqref{psi-R1-coefficient-of-projection-1-3} hold for $\alpha_{2}$, and thus from \eqref{psi-R1-coefficient-of-projection-2-1}, we obtain that
		\begin{align}
			\left|\alpha_{2}\right|\leq\frac{C\varepsilon^{5}}{t^{2}}.\label{psi-R1-coefficient-of-projection-2-2}
		\end{align}
		Moreover, using the change of variables $z=y-q'e_{2}$ on the integrals on both sides of \eqref{psi-R1-coefficient-of-projection-2-1} as well as Lemma \ref{vortex-pair-q-variation-lemma}, \eqref{f-prime-q-variation-4} and bounds \eqref{xi0-bound}--\eqref{tildexi-bound}, we have that
		\begin{align}
			\left|\alpha_{1}\left(\xi\right)-\alpha_{1}\left(\xi_{0}\right)\right|+\left|\alpha_{2}\left(\xi\right)-\alpha_{2}\left(\xi_{0}\right)\right|\leq\frac{C\varepsilon^{5}}{t^{4}}.\label{psi-R1-coefficient-of-projection-2-3}
		\end{align}
		Let $\mathcal{R}_{12}$ be defined by
		\begin{align}
			\mathcal{R}_{12}=\left(\alpha_{1}\left(\xi\right)-\alpha_{1}\left(\xi_{0}\right)\right)\mathscr{Z}_{1}\left(q\right)+\left(\alpha_{2}\left(\xi\right)-\alpha_{2}\left(\xi_{0}\right)\right)\mathscr{Z}_{2}\left(q\right).\label{first-approximation-construction-10}
		\end{align}
		Then \eqref{gradperp-f-j-prime-plus-0}, \eqref{psi-R1-coefficient-of-projection-1-3}--\eqref{psi-R1-coefficient-of-projection-2-3} gives
		\begin{align}
			\left|\gradperp_{y}\left(\Psi_{R}-c\varepsilon y_{2}\right)\cdot\grad_{y}\left(\left(f_{\varepsilon}'\right)^{+}\mathcal{R}_{12}\right)\right|\leq\frac{C\varepsilon^{5}\left(f_{\varepsilon}'\right)^{+}}{t^{4}}.\label{R-12-bound}
		\end{align}
		Next, we see from the estimate for $\mathcal{E}_{4}$ from \eqref{first-approximation-first-error-mode-4}, the estimate for $\overline{\mathcal{R}}_{2}$ from \eqref{R2-error-estimate}, the estimates for $\varepsilon^{2}\Lambda_{01}p_{0}^{-2}$ and $\varepsilon\Lambda_{00}p_{0}^{-2}$ from \eqref{xi0-bound} and \eqref{Lambda-ij-sizes}, the estimate for $\mathcal{M}_{2}$ inferred from \eqref{dq-mode-1-error}, and the estimates for $\alpha_{1}$ and $\alpha_{2}$ from \eqref{psi-R1-coefficient-of-projection-1-3} and \eqref{psi-R1-coefficient-of-projection-2-2}, that \eqref{psi-R1-equation} can be written as 
		\begin{align}
			\Delta\psi_{R1}+\left(f_{\varepsilon}'\right)^{+}\psi_{R1}&=\left(f_{\varepsilon}'\right)^{+}\left(\mathcal{E}_{2}\left(\xi_{0}\right)+\mathcal{E}_{3}\left(\xi_{0}\right)\right)+\left(f_{\varepsilon}'\right)^{+}\tilde{\mathcal{E}},\label{psi-R1-equation-different-form}
		\end{align}
		where $\tilde{\mathcal{E}}$ is an error term of size $\varepsilon^{4}t^{-2}$ on the support of $\left(f_{\varepsilon}'\right)^{+}$. Next we know from Theorem \ref{vortex-pair-properties-theorem} and \eqref{f-prime-plus-definition} that $\left(f_{\varepsilon}'\right)^{+}=\gamma\Gamma^{\gamma-1}_{+}\left(\left|y-q'e_{2}\right|\right)+\bigO{\left(\varepsilon^{2}\right)}$. So define $\psi_{R1,2}$, $\psi_{R1,3}$ and $\psi_{R1,rem}$ such that, for $j=2,3$,
		\begin{subequations}\label{psi-R1-split-equations}
			\begin{align}
				&\Delta\psi_{R1,j}+\gamma\Gamma^{\gamma-1}_{+}\psi_{R1,j}=\gamma\Gamma^{\gamma-1}_{+}\mathcal{E}_{j}\left(\xi_{0}\right),\label{psi-R12-equation}\\
				&\Delta\psi_{R1,rem}+\left(f_{\varepsilon}'\right)^{+}\psi_{R1,rem}=\left(\left(f_{\varepsilon}'\right)^{+}-\gamma\Gamma^{\gamma-1}_{+}\right)\left(\sum_{j=2}^{3}\left(\mathcal{E}_{j}-\psi_{R1,j}\right)\right)+\left(f_{\varepsilon}'\right)^{+}\tilde{\mathcal{E}}.\label{psi-R1-rem-equation}
			\end{align}
		\end{subequations}
		Noting from \eqref{first-approximation-first-error-mode-2} and \eqref{first-approximation-first-error-mode-1-3} that, for $j=2,3$, $\mathcal{E}_{j}$ is a linear combination of mode $j$ terms in coordinates $y-q'e_{2}=re^{i\theta}$, we can apply Lemma \ref{vortex-linearised-equation-fourier-coefficients-behaviour-lemma} to solve both linear equations stated in \eqref{psi-R12-equation} on $\mathbb{R}^{2}$ and obtain the existence and uniqueness of $\psi_{R1,2}$ and $\psi_{R1,3}$. Moreover, with definition \eqref{vortex-linearised-equation-rk-error-solution} in hand, we can write
		\begin{align}
			\psi_{R1,j}\left(\left|y-q'e_{2}\right|\right)=\frac{\varrho_{j}\left(\left|y-q'e_{2}\right|\right)}{\left|y-q'e_{2}\right|^{j}}\mathcal{E}_{j}.\label{psi-R12-psi-R13-formulae}
		\end{align}
		We note that \eqref{psi-R12-psi-R13-formulae} and \eqref{variation-of-parameters} immediately imply that $\psi_{R1,j}$ and $\nabla\psi_{R1,j}$ are of size $\varepsilon^{j}t^{-j}$ on $\mathbb{R}^{2}$ for $j=2,3$. 
		
		\medskip
		Next, defining $\psi_{R1,rem}=\psi_{R1}-\psi_{R1,2}-\psi_{R1,3}$, from \eqref{psi-R1-equation-different-form} and \eqref{psi-R12-equation} we see that $\psi_{R1,rem}$ satisfies \eqref{psi-R1-rem-equation}. Moreover, we can once again adapt the theory from \cite{DDMPVP2023}, as well as use the fact that $\left(f_{\varepsilon}'\right)^{+}-\gamma\Gamma^{\gamma-1}_{+}=\bigO{\left(\varepsilon^{2}\right)}$ from Theorem \ref{vortex-pair-properties-theorem} and \eqref{f-prime-plus-definition}, along with the bounds for $\psi_{R1}$, $\tilde{\mathcal{E}}$, $\psi_{R1,2}$ and $\psi_{R1,3}$ to infer that
		\begin{align*}
			\|\psi_{R1,rem}\|_{L^{\infty}\left(\mathbb{R}^{2}_{+}\right)}+\|\nabla\psi_{R1,rem}\|_{L^{\infty}\left(\mathbb{R}^{2}_{+}\right)}\leq \frac{C\varepsilon^{4}}{t^{2}}.
		\end{align*}
		Now $-\Delta\psi_{R1}=\phi_{R1}$ so \eqref{psi-R1-equation-different-form} implies
		\begin{align}
			-\phi_{R1}=\left(f_{\varepsilon}'\right)^{+}\left(\mathcal{E}_{2}\left(\xi_{0}\right)+\mathcal{E}_{3}\left(\xi_{0}\right)-\psi_{R1,2}-\psi_{R1,3}\right)+\left(f_{\varepsilon}'\right)^{+}\left(\tilde{\mathcal{E}}-\psi_{R1,rem}\right).\label{phi-R1-formula}
		\end{align}
		Thus, on the support of $\left(f_{\varepsilon}'\right)^{+}$, for $y-q'e_{2}=re^{i\theta}$,
		\begin{align}
			\frac{\phi_{R1}}{\left(f_{\varepsilon}'\right)^{+}}&=\left(\frac{\varrho_{2}(r)}{r^{2}}-1\right)\mathcal{E}_{2}\left(\xi_{0}\right)+\left(\frac{\varrho_{3}(r)}{r^{3}}-1\right)\mathcal{E}_{3}\left(\xi_{0}\right)-\left(\tilde{\mathcal{E}}-\psi_{R1,rem}\right)\nonumber\\
			&=\hat{\varrho}_{2}(r)\mathcal{E}_{2}\left(\xi_{0}\right)+\hat{\varrho}_{3}(r)\mathcal{E}_{3}\left(\xi_{0}\right)-\left(\tilde{\mathcal{E}}-\psi_{R1,rem}\right),\label{first-approximation-construction-12}
		\end{align}
		where $\tilde{\mathcal{E}}-\psi_{R1,rem}=\bigO{\left(\varepsilon^{4}t^{-2}\right)}$ on the support of $\left(f_{\varepsilon}'\right)^{+}$, and we once again recall that $\varrho_{2},\varrho_{3}$ are defined in \eqref{vortex-linearised-equation-rk-error-solution}.
		
		\medskip
		\noindent \textbf{Step III: Constructing $\psi^{out}$}
		
		\medskip
		Having constructed $\psi_{R1}$, we can construct $\psi_{L1}$ by reflecting in $x_{1}$. Now we construct $\psi^{out}$ in the original $x$ variable. We obtain $\psi^{out}$ on $\mathbb{R}^{2}_{+}$ by solving the equation, for $(-1)^{R}=1$, and $(-1)^{L}=-1$
		\begin{align}
			-\Delta\psi^{out}=\sum_{j=R,L}\left(-1\right)^{j}\left[\psi_{j1}\Delta\eta^{(j)}_{K}+2\nabla\eta^{(j)}_{K}\cdot\nabla\psi_{j1}\right]\coloneqq\mathscr{S}(x,t),\quad \psi^{out}\rvert_{\dell\mathbb{R}^{2}_{+}}=0.\label{psi-out-equation}
		\end{align}
		Using $-\Delta\psi_{R1}=\phi_{R1}$ and \eqref{phi-R1-formula}, we can write
		\begin{align}
			\psi_{R1}(y,t)&=\frac{1}{2\pi}\int_{\mathbb{R}_{+}^{2}}\log{\left(\frac{\left|z-\bar{y}\right|}{\left|z-y\right|}\right)}\left(f_{\varepsilon}'\right)^{+}\left(\sum_{j=2}^{3}\left(\mathcal{E}_{j}-\psi_{R1,j}\right)+\tilde{\mathcal{E}}-\psi_{R1,rem}\right)dz.\label{psi-R1-greens-function-representation}
		\end{align}
		Due to the presence of $\left(f_{\varepsilon}'\right)^{+}$ in the integrand in \eqref{psi-R1-greens-function-representation}, and the derivatives of the cutoffs in \eqref{psi-out-equation} we can consider, by \eqref{cutoffs-definition-1}, \eqref{K-T0-choices}, Theorem \ref{p0-q0-construction-theorem}, \eqref{xi0-bound}--\eqref{tildexi-bound},
		\begin{align}
			\left|z-q'e_{2}\right|\leq \rho,\quad \left|y-q'e_{2}\right|\geq \frac{5}{8}Kt\varepsilon^{-1},\ \ 
			t\geq T_{0}.\label{psi-out-integrand region}
		\end{align}
		On this region, we have
		\begin{align*}
			\log{\left(\frac{\left|z-\bar{y}\right|}{\left|z-y\right|}\right)}&=\frac{1}{2}\log{\left(\left|\frac{z-q'e_{2}}{|y-q'e_{2}|}-\frac{\overline{y+q'e_{2}}}{|y-q'e_{2}|}\right|^{2}\right)}-\frac{1}{2}\log{\left(\left|\frac{z-q'e_{2}}{|y-q'e_{2}|}-\frac{y-q'e_{2}}{|y-q'e_{2}|}\right|^{2}\right)}\nonumber\\
			&\coloneqq\mathscr{L}_{1}\left(z-q'e_{2}\right)+\mathscr{L}_{2}\left(z-q'e_{2}\right),
		\end{align*}
		where
		\begin{align}
			&\mathscr{L}_{2}\left(z-q'e_{2}\right)=\frac{\left(z-q'e_{2}\right)\cdot\left(y-q'e_{2}\right)}{\left|y-q'e_{2}\right|^{2}}+\bigO{\left(\frac{1}{\left|y-q'e_{2}\right|^{2}}\right)}=\bigO{\left(\varepsilon t^{-1}\right)},\label{L2-definition}\\
			&\mathscr{L}_{1}\left(z-q'e_{2}\right)=\frac{1}{2}\log{\left(\frac{\left|y+q'e_{2}\right|^{2}}{\left|y-q'e_{2}\right|^{2}}-\frac{2\left(z_{1}y_{1}-\left(z_{2}-q'\right)\left(y_{2}+q'\right)\right)}{\left|y-q'e_{2}\right|^{2}}+\frac{\left|z-q'e_{2}\right|^{2}}{\left|y-q'e_{2}\right|^{2}}\right)}\nonumber\\
			&=\frac{1}{2}\log{\left(1-\frac{2\left(z_{1}y_{1}-\left(z_{2}-q'\right)\left(y_{2}+q'\right)\right)}{\left|y+q'e_{2}\right|^{2}}+\frac{\left|z-q'e_{2}\right|^{2}}{\left|y+q'e_{2}\right|^{2}}\right)}+\frac{1}{2}\log{\left(\frac{\left|y+q'e_{2}\right|^{2}}{\left|y-q'e_{2}\right|^{2}}\right)}\nonumber\\
			&=\bigO{\left(t^{-1}\right)}.\label{L1-definition}
		\end{align}
		Thus \eqref{psi-R1-greens-function-representation}--\eqref{L1-definition} along with Theorem \ref{vortex-pair-properties-theorem}, \eqref{f-prime-q-variation-1}--\eqref{tilde-W-k-def}, and \eqref{psi-R1-equation-different-form}--\eqref{first-approximation-construction-12} give, on \eqref{psi-out-integrand region},
		\begin{align}
			&\psi_{R1}\left(\frac{x-pe_{1}}{\varepsilon}\right)=\psi_{R1}\left(y\right)=\bigO{\left(\varepsilon^{4}t^{-3}\right)},\nonumber\\
			&\nabla_{x}\psi_{R1}\left(\frac{x-pe_{1}}{\varepsilon}\right)=\frac{1}{\varepsilon}\nabla_{y}\psi_{R1}\left(y\right)=\bigO{\left(\varepsilon^{4}t^{-4}\right)}.\label{psi-R1-psi-out-estimate}
		\end{align}
		Note that \eqref{cutoffs-definition-1} gives $\nabla\eta^{(R)}_{K},\Delta\eta^{(R)}_{K}$ are $\bigO{\left(t^{-1}\right)}$ and $\bigO{\left(t^{-2}\right)}$ respectively.  We also have completely analogously bounds for $\psi_{L1}$ and the derivatives of $\eta^{(L)}_{K}$. Therefore,
		\begin{align}
			\left|\mathscr{S}(x,t)\right|\leq \frac{C\varepsilon^{4}}{t^{5}}.\label{source-term-estimate}
		\end{align}
		Using this bound, and noting that $\mathscr{S}(x,t)$ is supported on two annuli of inner and outer radius $\sim t$, and thus of area $\sim t^{2}$, each of distance $\sim t$ away from the origin, we obtain
		\begin{align}
			\|\psi^{out}\|_{L^{\infty}\left(\mathbb{R}^{2}_{+}\right)}+\|\nabla\psi^{out}\|_{L^{\infty}\left(\mathbb{R}^{2}_{+}\right)}\leq\frac{C\varepsilon^{4}\left|\log{t}\right|}{t^{3}}.\label{first-approximation-construction-22}
		\end{align}
		We can also bound $\gradperp_{y}\psi^{out}\left(\varepsilon y+pe_{1}\right)\cdot\nabla_{y}U_{R}$ which appears on the right hand side of \eqref{first-approximation-construction-8}. First note that by chain rule:
		\begin{align}
			\gradperp_{y}\psi^{out}\left(\varepsilon y+pe_{1}\right)\cdot\nabla_{y}U_{R}=\varepsilon\gradperp_{y}\psi^{out}\left(\varepsilon y+pe_{1}\right)\cdot\nabla_{y}U_{R}.\label{psi-out-UR-estimate}
		\end{align}
		We have that $\psi^{out}\left(\varepsilon y+pe_{1}\right)$ can be written as
		\begin{align}
			&\frac{1}{4\pi}\int_{\supp{\mathscr{S}}}\log{\left(1-\frac{2\varepsilon\left(z-pe_{1}+qe_{2}\right)\cdot\overline{\left(y-q'e_{2}\right)}}{\left|z-pe_{1}+qe_{2}\right|^{2}}+\frac{\varepsilon^{2}\left|y-q'e_{2}\right|^{2}}{\left|z-pe_{1}+qe_{2}\right|^{2}}\right)}\mathscr{S}(z)dz\nonumber\\
			&+\frac{1}{4\pi}\int_{\supp{\mathscr{S}}}\log{\left(1-\frac{2\varepsilon\left(z-pe_{1}+qe_{2}\right)\cdot\left(y-q'e_{2}\right)}{\left|z-pe_{1}-qe_{2}\right|^{2}}+\frac{\varepsilon^{2}\left|y-q'e_{2}\right|^{2}}{\left|z-pe_{1}-qe_{2}\right|^{2}}\right)}\mathscr{S}(z)dz\nonumber\\
			&+\frac{1}{2\pi}\int_{\supp{\mathscr{S}}}\log{\left(\frac{\left|z-pe_{1}+qe_{2}\right|}{\left|z-pe_{1}-qe_{2}\right|}\right)}\mathscr{S}(z)dz,\label{psi-out-estimate-1}
		\end{align}
		with the expansion of the logarithm valid due to the choices of $K$ and $T$ in \eqref{K-T0-choices}, the support of $\mathscr{S}$, and the support of $U_{R}$ meaning we can consider the above integrand in the region
		\begin{align}
			\left|y-q'e_{2}\right|\leq \rho,\quad \left|z-pe_{1}\pm qe_{2}\right|\geq \frac{5}{8}Kt.\label{zpq-support}
		\end{align} 
		On \eqref{zpq-support},
		\begin{align*}
			\nabla_{y}\log{\left(1-\frac{2\varepsilon\left(z-pe_{1}\pm qe_{2}\right)\cdot\overline{\left(y-q'e_{2}\right)}}{\left|z-pe_{1}\pm qe_{2}\right|^{2}}+\frac{\varepsilon^{2}\left|y-q'e_{2}\right|^{2}}{\left|z-pe_{1}\pm qe_{2}\right|^{2}}\right)}=\bigO{\left(\varepsilon t^{-1}\right)},
		\end{align*}
		and noting that the gradient immediately kills the third term on the right hand side of \eqref{psi-out-estimate-1}, we have
		\begin{align}
			\left|\varepsilon\gradperp_{y}\psi^{out}\left(\varepsilon y+pe_{1}\right)\cdot\nabla_{y}\left(U_{R}(y)+\phi_{R1}\left(y\right)\right)\right|\leq \frac{C\varepsilon^{6}}{t^{4}}\left(\left(f_{\varepsilon}'\right)^{+}+\left(f_{\varepsilon}''\right)^{+}\right),\label{psi-out-bound-2}
		\end{align}
		where we have used \eqref{phi-R1-formula} and \eqref{psi-R1-bounds}.
		
		\medskip
		\noindent \textbf{Step IV: Bounding Errors from $\left(\psi_{R1},\phi_{R1}\right)$}
		
		\medskip
		We now analyze the errors produced from $\left(\psi_{R1},\phi_{R1}\right)$. From \eqref{psi-R1-equation}, we can rewrite \eqref{first-approximation-construction-8} as
		\begin{align}
			E_{R1}&=\left(\varepsilon^{2}\dell_{t}\phi_{R1}-\left[\gradperp_{y}\Psi_{L}+\varepsilon\left(\Dot{p}-c\right)e_{1}\right]\cdot\grad_{y}\phi_{R1}\right)+\gradperp_{y}\psi_{R}\cdot\grad_{y}\phi_{R}\label{first-approximation-construction-22a}\\
			&+\left(\varepsilon^{2}\dell_{t}\phi_{R2}-\left[\gradperp_{y}\Psi_{L}+\varepsilon\left(\Dot{p}-c\right)e_{1}\right]\cdot\grad_{y}\phi_{R2}\right)+\gradperp_{y}\psi^{out}\cdot\grad_{y}\phi_{R2}\nonumber\\
			&-\gradperp_{y}\left(\Psi_{R}-c\varepsilon y_{2}\right)\cdot\grad_{y}\left(\Delta\psi_{R2}+\left(f_{\varepsilon}'\right)^{+}\psi_{R2}\right)\nonumber\\
			&+\gradperp_{y}\left(\Psi_{R}-c\varepsilon y_{2}\right)\cdot\grad_{y}\left(\left(f_{\varepsilon}'\right)^{+}\left(\alpha_{1}\left(\xi_{0}\right)\mathscr{Z}_{1}\left(q\right)+\alpha_{2}\left(\xi_{0}\right)\mathscr{Z}_{2}\left(q\right)\right)\right)\nonumber\\
			&+\gradperp_{y}\left(\Psi_{R}-c\varepsilon y_{2}\right)\cdot\grad_{y}\left(\varepsilon\left(f_{\varepsilon}'\right)^{+}\left(\mathcal{N}\left(\xi_{0}\right)\left[\xi_{1}\right]+\dot{\xi}_{1}\right)\cdot\left(y-q'e_{2}\right)^{\perp}\right)\nonumber\\
			&+\gradperp_{y}\left(\Psi_{R}-c\varepsilon y_{2}\right)\cdot\grad_{y}\left(\varepsilon\left(f_{\varepsilon}'\right)^{+}\left(\mathcal{N}\left(\xi_{0}+\xi_{1}\right)\left[\tilde{\xi}\right]+\dot{\tilde{\xi}}\right)\cdot\left(y-q'e_{2}\right)^{\perp}\right)\nonumber\\
			&+\left(f_{\varepsilon}'\right)^{+}\gradperp_{y}\left(\Psi_{R}-c\varepsilon y_{2}\right)\cdot\grad_{y}\left(\overline{\mathcal{R}}_{4}+\mathcal{R}_{12}\right)+\left(f_{\varepsilon}'\right)^{+}\overline{\mathcal{R}}_{5}+\left(f_{\varepsilon}''\right)^{+}\overline{\mathcal{R}}_{6},\nonumber
		\end{align}
		with $\mathcal{R}_{12}$ being defined in \eqref{first-approximation-construction-10}, with bound \eqref{R-12-bound}, and 
		\begin{align*}
			\left(f_{\varepsilon}'\right)^{+}\overline{\mathcal{R}}_{5}+\left(f_{\varepsilon}''\right)^{+}\overline{\mathcal{R}}_{6}=\varepsilon\gradperp_{y}\psi^{out}\left(\varepsilon y+pe_{1}\right)\cdot\nabla_{y}\left(U_{R}(y)+\phi_{R1}\left(y\right)\right),
		\end{align*}
		and bound given by \eqref{psi-out-bound-2}. Using \eqref{first-approximation-construction-12} we first obtain, for $y-q'e_{2}=re^{i\theta}$,
		\begin{align}
			&\varepsilon^{2}\dell_{t}\phi_{R1}-\left[\gradperp_{y}\Psi_{L}+\varepsilon\left(\Dot{p}-c\right)e_{1}\right]\cdot\grad_{y}\phi_{R1}\label{first-approximation-construction-26}\\
			&=\left(f_{\varepsilon}'\right)^{+}\varepsilon^{2}\dell_{t}\left(\hat{\varrho}_{2}(r)\mathcal{E}_{2}\left(\xi_{0}\right)\right)-\left(f_{\varepsilon}'\right)^{+}\left[\gradperp_{y}\Psi_{L}+\varepsilon\left(\Dot{p}-c\right)e_{1}\right]\cdot\grad_{y}\left(\hat{\varrho}_{2}(r)\mathcal{E}_{2}\left(\xi_{0}\right)\right)\nonumber\\
			&+\hat{\varrho}_{2}(r)\mathcal{E}_{2}\left(\xi_{0}\right)\left(\varepsilon^{2}\dell_{t}\left(f_{\varepsilon}'\right)^{+}-\left[\gradperp_{y}\Psi_{L}+\varepsilon\left(\Dot{p}-c\right)e_{1}\right]\cdot\grad_{y}\left(f_{\varepsilon}'\right)^{+}\right)\nonumber\\
			&+\left(f_{\varepsilon}'\right)^{+}\overline{\mathcal{R}}_{7}+\left(f_{\varepsilon}''\right)^{+}\overline{\mathcal{R}}_{8},\nonumber
		\end{align}
		where, given the identities and bounds we have for $\mathcal{E}_{2}$, $\psi_{R1,2}$ $\mathcal{E}_{3}$, $\psi_{R1,3}$, $\tilde{\mathcal{E}}$, and $\psi_{R1,rem}$ from \eqref{first-approximation-first-error-mode-1-3}--\eqref{first-approximation-first-error-mode-4}, \eqref{psi-R1-equation}, \eqref{psi-R1-equation-different-form}--\eqref{first-approximation-construction-12}, we have the bounds
		\begin{align}
			\left|\left(f_{\varepsilon}'\right)^{+}\overline{\mathcal{R}}_{7}\right|\leq \frac{C\varepsilon^{5}\left(f_{\varepsilon}'\right)^{+}}{t^{3}},\quad \left|\left(f_{\varepsilon}''\right)^{+}\overline{\mathcal{R}}_{8}\right|\leq \frac{C\varepsilon^{5}\left(f_{\varepsilon}''\right)^{+}}{t^{3}}.\label{bar-R-7-8-bounds}
		\end{align}
		For the first term on the right hand side of \eqref{first-approximation-construction-26}, we have
		\begin{align*}
			&\left(f_{\varepsilon}'\right)^{+}\varepsilon^{2}\dell_{t}\left(\hat{\varrho}_{2}(r)\mathcal{E}_{2}\left(\xi_{0}\right)\right)\\
			&-\left(f_{\varepsilon}'\right)^{+}\left[\gradperp_{y}\Psi_{L}+\varepsilon\left(\Dot{p}-c\right)e_{1}\right]\cdot\grad_{y}\left(\hat{\varrho}_{2}(r)\mathcal{E}_{2}\left(\xi_{0}\right)\right)=\left(f_{\varepsilon}'\right)^{+}\varepsilon^{2}\hat{\varrho}_{2}\dot{\xi}_{0}\cdot\dell_{\xi_{0}}\mathcal{E}_{2}\left(\xi_{0}\right)\\
			&+\gradperp_{y}\left(\Psi_{L}+\varepsilon\left(\dot{p}-c\right)\left(y_{2}-q'\right)-\varepsilon\dot{q}y_{1}\right)\cdot\grad_{y}\left(\hat{\varrho}_{2}(r)\mathcal{E}_{2}\left(\xi_{0}\right)\right).
		\end{align*}
		Applying Lemma \ref{psi-L-correct-form-lemma} (since $\left(f_{\varepsilon}'\right)^{+}$ has the same support as $U_{R}$ by \eqref{f-prime-q-variation-1}), this becomes
		\begin{align}
			\left(f_{\varepsilon}'\right)^{+}\left[\varepsilon^{2}\hat{\varrho}_{2}\dot{\xi}_{0}\cdot\dell_{\xi_{0}}\mathcal{E}_{2}\left(\xi_{0}\right)+\gradperp_{y}\mathcal{E}_{2}\left(\xi_{0}\right)\cdot\grad_{y}\left(\hat{\varrho}_{2}(r)\mathcal{E}_{2}\left(\xi_{0}\right)\right)+\left(f_{\varepsilon}'\right)^{+}\overline{\mathcal{R}}_{9}\right],\label{first-approximation-construction-26-0}
		\end{align}
		with the bound
		\begin{align}
			\left|\left(f_{\varepsilon}'\right)^{+}\overline{\mathcal{R}}_{9}\right|\leq \frac{C\varepsilon^{5}\left(f_{\varepsilon}'\right)^{+}}{t^{3}}.\label{bar-9-bounds}
		\end{align}
		We now deal with the second term on the right hand side of \eqref{first-approximation-construction-26}. Note that, analogously to the proof of Lemma \ref{ER1-lemma}, specifically statement \eqref{ER1-lemma-statement-2}, and calculations \eqref{first-approximation-error-calculation-8}--\eqref{first-approximation-error-calculation-41}, we can write 
		\begin{align}
			\varepsilon^{2}\dell_{t}\left(f_{\varepsilon}'\right)^{+}&=\varepsilon\dot{q}\dell_{q'}\left(f_{\varepsilon}'\right)^{+}=\left(f_{\varepsilon}''\right)^{+}\gradperp_{y}\left(\varepsilon\dot{q}y_{1}\right)\cdot\nabla_{y}\left(\psi_{R}-c\varepsilon y_{2}\right)\nonumber\\
			&+\varepsilon\dot{q}\left(f_{\varepsilon}''\right)^{+}\dell_{q'}\mathscr{V}\left(z,q\right)\rvert_{z=y-q'e_{2}},\label{first-approximation-construction-26a}
		\end{align}
		where by Lemma \ref{vortex-pair-q-variation-lemma}, $\varepsilon\dell_{q'}\mathscr{V}\left(z,q\right)=\bigO{\left(\varepsilon^{4}\right)}$ on the support of $\left(f_{\varepsilon}''\right)^{+}$.
		Thus the second term on the right hand side of \eqref{first-approximation-construction-26} is equal to
		\begin{align}
			&-\left(f_{\varepsilon}''\right)^{+}\hat{\varrho}_{2}(r)\mathcal{E}_{2}\left(\xi_{0}\right)\gradperp_{y}\left(\Psi_{L}+\varepsilon\left(\Dot{p}-c\right)\left(y_{2}-q'\right)-\varepsilon\Dot{q}y_{1}\right)\cdot\grad_{y}\left(\Psi_{R}-c\varepsilon y_{2}\right)\nonumber\\
			&+\varepsilon\dot{q}\left(f_{\varepsilon}''\right)^{+}\hat{\varrho}_{2}(r)\mathcal{E}_{2}\left(\xi_{0}\right)\dell_{q'}\mathscr{V}\left(z,q\right)\rvert_{z=y-q'e_{2}},\label{first-approximation-construction-27}
		\end{align}
		Applying Lemma \ref{psi-L-correct-form-lemma} once again (this time noting $\left(f_{\varepsilon}''\right)^{+}$ has the same support as $U_{R}$ by \eqref{f-prime-q-variation-1}), 
		\begin{align}
			\gradperp_{y}\left(\Psi_{L}+\varepsilon\left(\dot{p}-c\right)\left(y_{2}-q'\right)-\varepsilon\dot{q}y_{1}\right)=\overline{\mathcal{R}}_{10}+\gradperp_{y}\mathcal{E}_{2}\left(\xi_{0}\right),\label{first-approximation-construction-27a}
		\end{align}
		where, using \eqref{xi0-bound}--\eqref{tildexi-bound}, \eqref{linearised-point-vortex-operator-shorthand}, and \eqref{first-approximation-first-error-mode-2}--\eqref{first-approximation-first-error-mode-4}, and Lemma \ref{vortex-pair-q-variation-lemma}
		\begin{align}
			\left|\varepsilon\dot{q}\left(f_{\varepsilon}''\right)^{+}\hat{\varrho}_{2}(r)\mathcal{E}_{2}\left(\xi_{0}\right)\dell_{q'}\mathscr{V}\left(z,q\right)\rvert_{z=y-q'e_{2}}\right|+\left|\left(f_{\varepsilon}''\right)^{+}\overline{\mathcal{R}}_{10}\right|\leq \frac{C\varepsilon^{5}\left(f_{\varepsilon}''\right)^{+}}{t^{3}}.\label{bar-R-10-bound}
		\end{align}
		Thus, using \eqref{first-approximation-construction-26}--\eqref{bar-R-10-bound}, we have
		\begin{align}
			&\varepsilon^{2}\dell_{t}\phi_{R1}-\left[\gradperp_{y}\Psi_{L}+\varepsilon\left(\Dot{p}-c\right)e_{1}\right]\cdot\grad_{y}\phi_{R1}\nonumber\\
			&=\left(f_{\varepsilon}'\right)^{+}\left[\varepsilon^{2}\hat{\varrho}_{2}(r)\dot{\xi}_{0}\cdot\dell_{\xi_{0}}\mathcal{E}_{2}\left(\xi_{0}\right)+\gradperp_{y}\mathcal{E}_{2}\left(\xi_{0}\right)\cdot\grad_{y}\left(\hat{\varrho}_{2}(r)\mathcal{E}_{2}\left(\xi_{0}\right)\right)\right]\nonumber\\
			&-\left(f_{\varepsilon}''\right)^{+}\hat{\varrho}_{2}(r)\mathcal{E}_{2}\left(\xi_{0}\right)\gradperp_{y}\mathcal{E}_{2}\left(\xi_{0}\right)\cdot\grad_{y}\Gamma(r)+\left(f_{\varepsilon}'\right)^{+}\overline{\mathcal{R}}_{11}+\left(f_{\varepsilon}''\right)^{+}\overline{\mathcal{R}}_{12}.\label{first-approximation-construction-27a0}
		\end{align}
		where
		\begin{align}
			\overline{\mathcal{R}}_{11}&=\overline{\mathcal{R}}_{7}+\overline{\mathcal{R}}_{9},\nonumber\\
			\overline{\mathcal{R}}_{12}&=\overline{\mathcal{R}}_{8}-\hat{\varrho}_{2}(r)\mathcal{E}_{2}\left(\xi_{0}\right)\left[\overline{\mathcal{R}}_{10}\cdot\grad_{y}\left(\Psi_{R}-c\varepsilon y_{2}\right)-\varepsilon\dot{q}\dell_{q'}\mathscr{V}\left(z,q\right)\rvert_{z=y-q'e_{2}}\right]\nonumber\\
			&-\hat{\varrho}_{2}(r)\mathcal{E}_{2}\left(\xi_{0}\right)\gradperp_{y}\mathcal{E}_{2}\left(\xi_{0}\right)\cdot\grad_{y}\left(\Psi_{R}-c\varepsilon y_{2}-\Gamma(r)\right),\label{bar-R-11-12-formulae}
		\end{align}
		and by Theorem \ref{vortex-pair-properties-theorem}, \eqref{first-approximation-first-error-mode-2}, \eqref{bar-R-7-8-bounds}, \eqref{bar-9-bounds}, and \eqref{bar-R-10-bound}, we have
		\begin{align}
			\left|\left(f_{\varepsilon}'\right)^{+}\overline{\mathcal{R}}_{11}\right|\leq \frac{C\varepsilon^{5}\left(f_{\varepsilon}'\right)^{+}}{t^{3}},\quad \left|\left(f_{\varepsilon}''\right)^{+}\overline{\mathcal{R}}_{12}\right|\leq \frac{C\varepsilon^{5}\left(f_{\varepsilon}''\right)^{+}}{t^{3}}.\label{bar-R-11-12-bounds}
		\end{align}
		Letting $\mathcal{E}_{2}\left(y-q'e_{2},\xi_{0}\right)=\varepsilon^{2}r^{2}\left(\beta_{1}(t)\sin{2\theta}+\beta_{2}(t)\cos{2\theta}\right)$, we obtain
		\begin{equation}\label{mode-4-formulae}
			\begin{aligned}
				&\varepsilon^{2}\hat{\varrho}_{2}(r)\dot{\xi}_{0}\cdot\dell_{\xi_{0}}\mathcal{E}_{2}\left(\xi_{0}\right)=\varepsilon^{4}\hat{\varrho}_{2}(r)\left(\tilde{\beta}_{1}(t)\sin{2\theta}+\tilde{\beta}_{2}(t)\cos{2\theta}\right),\\
				&\gradperp_{y}\mathcal{E}_{2}\left(\xi_{0}\right)\cdot\grad_{y}\left(\hat{\varrho}_{2}(r)\mathcal{E}_{2}\left(\xi_{0}\right)\right)=-\frac{\hat{\varrho}'_{2}(r)}{r}\varepsilon^{4}r^{4}\left(\left(\beta_{1}^{2}-\beta_{2}^{2}\right)\sin{4\theta}+2\beta_{1}\beta_{2}\cos{4\theta}\right),\\
				&-\hat{\varrho}_{2}(r)\mathcal{E}_{2}\left(\xi_{0}\right)\gradperp_{y}\mathcal{E}_{2}\left(\xi_{0}\right)\cdot\grad_{y}\Gamma(r)\\
				&=\hat{\varrho}_{2}(r)\frac{\Gamma'(r)}{r}\varepsilon^{4}r^{4}\left(\left(\beta_{1}^{2}-\beta_{2}^{2}\right)\sin{4\theta}+2\beta_{1}\beta_{2}\cos{4\theta}\right).
			\end{aligned}
		\end{equation}
		Evaluating \eqref{first-approximation-first-error-mode-2} at $\xi_{0}$ instead of $\xi$, and then differentiating with respect to $\xi_{0}$, we can see that $\varepsilon^{2}\hat{\varrho}_{2}(r)\dot{\xi}_{0}\cdot\dell_{\xi_{0}}\mathcal{E}_{2}\left(y-q'e_{2},\xi_{0}\right)$ is a sum of mode $2$ terms of size $\varepsilon^{4}t^{-3}$ on the support of $\left(f_{\varepsilon}'\right)^{+}$ with respect to coordinates $y-q'e_{2}=re^{i\theta}$. Next, comparing $\beta_{1}(t)$ and $\beta_{2}(t)$ to \eqref{first-approximation-first-error-mode-2} evaluated at $\xi_{0}$ instead of $\xi$, the last two terms in \eqref{mode-4-formulae} are a sum of mode $4$ terms of size $\varepsilon^{4}t^{-4}$ on the support of $\left(f_{\varepsilon}'\right)^{+}$ and $\left(f_{\varepsilon}''\right)^{+}$. Proceeding as in Lemma \ref{correct-remainder-form-lemma}, we can write \eqref{first-approximation-construction-27a0} as
		\begin{align}
			&\varepsilon^{2}\dell_{t}\phi_{R1}-\left[\gradperp_{y}\Psi_{L}+\varepsilon\left(\Dot{p}-c\right)e_{1}\right]\cdot\grad_{y}\phi_{R1}\nonumber\\
			&=\gradperp_{y}\left(\Psi_{R}-c\varepsilon y_{2}\right)\cdot\nabla_{y}\left(\left(f_{\varepsilon}'\right)^{+}\mathcal{E}_{2,2}\left(\xi_{0},\dot{\xi}_{0}\right)+\left(f_{\varepsilon}'\right)^{+}\mathcal{E}_{4,1}\left(\xi_{0}\right)+\left(f_{\varepsilon}''\right)^{+}\mathcal{E}_{4,2}\left(\xi_{0}\right)\right)\nonumber\\
			&+\left(f_{\varepsilon}'\right)^{+}\overline{\mathcal{R}}_{13}+\left(f_{\varepsilon}''\right)^{+}\overline{\mathcal{R}}_{14},\label{first-approximation-construction-27a00}
		\end{align}
		where, with respect to coordinates $y-q'e_{2}=re^{i\theta}$, $\mathcal{E}_{2,2}$ is a sum of mode $2$ terms of size $\varepsilon^{4}t^{-3}$ on the support of $\left(f_{\varepsilon}'\right)^{+}$, and the $\mathcal{E}_{4,j}$, $j=1,2$ are a sum of mode $4$ terms of size $\varepsilon^{4}t^{-4}$ on the support of $\left(f_{\varepsilon}'\right)^{+}$, and therefore on the support of $\left(f_{\varepsilon}''\right)^{+}$. Moreover, due to the strategy of applying Lemma \ref{correct-remainder-form-lemma}, as well as the identity \eqref{gradperp-f-j-prime-plus-0}, we have
		\begin{align*}
			\overline{\mathcal{R}}_{13}&=\overline{\mathcal{R}}_{11}-\gradperp_{y}\left(\Psi_{R}-c\varepsilon y_{2}-\Gamma(r)\right)\cdot\nabla_{y}\left(\mathcal{E}_{2,2}\left(\xi_{0},\dot{\xi}_{0}\right)+\mathcal{E}_{4,1}\left(\xi_{0}\right)\right)\\
			\overline{\mathcal{R}}_{14}&=\overline{\mathcal{R}}_{12}-\gradperp_{y}\left(\Psi_{R}-c\varepsilon y_{2}-\Gamma(r)\right)\cdot\nabla_{y}\mathcal{E}_{4,2}\left(\xi_{0}\right),
		\end{align*}
		and due to Theorem \ref{vortex-pair-properties-theorem} and \eqref{bar-R-11-12-bounds}, have bounds
		\begin{align}
			\left|\left(f_{\varepsilon}'\right)^{+}\overline{\mathcal{R}}_{13}\right|\leq \frac{C\varepsilon^{5}\left(f_{\varepsilon}'\right)^{+}}{t^{3}},\quad \left|\left(f_{\varepsilon}''\right)^{+}\overline{\mathcal{R}}_{14}\right|\leq \frac{C\varepsilon^{5}\left(f_{\varepsilon}''\right)^{+}}{t^{3}}.\label{bar-R-13-14-bounds}
		\end{align}
		We move on to $\gradperp_{y}\psi_{R1}\cdot\grad_{y}\phi_{R1}$. Once again using the decomposition for both $\psi_{R1}$ and $\phi_{R1}$ obtained in the calculations \eqref{psi-R1-equation-different-form}--\eqref{first-approximation-construction-12}, as well as using the same rearrangement of terms involving $\Psi_{R}-c\varepsilon y_{2}$ and $\Gamma(r)$ as in \eqref{first-approximation-construction-27a0}--\eqref{bar-R-11-12-formulae}, we can write
		\begin{align}
			&\gradperp_{y}\psi_{R1}\cdot\grad_{y}\phi_{R1}=\left(f_{\varepsilon}'\right)^{+}\gradperp_{y}\left(\hat{\varrho}_{2}(r)\mathcal{E}_{2}\left(\xi_{0}\right)\right)\cdot\nabla_{y}\left(\frac{\varrho_{2}(r)}{r^{2}}\mathcal{E}_{2}\left(\xi_{0}\right)\right)\nonumber\\
			&+\left(f_{\varepsilon}''\right)^{+}\frac{\varrho_{2}(r)}{r^{2}}\mathcal{E}_{2}\left(\xi_{0}\right)\gradperp_{y}\left(\hat{\varrho}_{2}(r)\mathcal{E}_{2}\left(\xi_{0}\right)\right)\cdot\nabla_{y}\Gamma(r)+\left(f_{\varepsilon}'\right)^{+}\overline{\mathcal{R}}_{15}+\left(f_{\varepsilon}''\right)^{+}\overline{\mathcal{R}}_{16},\label{first-approximation-construction-27a01}
		\end{align}
		where, once again, the formulae and bounds for $\mathcal{E}_{2}$, $\psi_{R1,2}$ $\mathcal{E}_{3}$, $\psi_{R1,3}$, $\tilde{\mathcal{E}}$, and $\psi_{R1,rem}$ that we obtain from \eqref{first-approximation-first-error-mode-1-3}--\eqref{first-approximation-first-error-mode-4}, \eqref{psi-R1-equation}, \eqref{psi-R1-equation-different-form}--\eqref{first-approximation-construction-12}, as well as Theorem \ref{vortex-pair-properties-theorem}, gives
		\begin{align}
			\left|\left(f_{\varepsilon}'\right)^{+}\overline{\mathcal{R}}_{15}\right|\leq \frac{C\varepsilon^{5}\left(f_{\varepsilon}'\right)^{+}}{t^{3}},\quad \left|\left(f_{\varepsilon}''\right)^{+}\overline{\mathcal{R}}_{16}\right|\leq \frac{C\varepsilon^{5}\left(f_{\varepsilon}''\right)^{+}}{t^{3}}.\label{bar-R-15-16-bounds}
		\end{align}
		Then, analogously to the calculations made in \eqref{mode-4-formulae}--\eqref{bar-R-13-14-bounds}, we can write 
		\begin{align}
			\gradperp_{y}\psi_{R1}\cdot\grad_{y}\phi_{R1}&=\gradperp_{y}\left(\Psi_{R}-c\varepsilon y_{2}\right)\cdot\nabla_{y}\left(\left(f_{\varepsilon}'\right)^{+}\mathcal{E}_{4,3}\left(\xi_{0}\right)+\left(f_{\varepsilon}''\right)^{+}\mathcal{E}_{4,4}\left(\xi_{0}\right)\right)\nonumber\\
			&+\left(f_{\varepsilon}'\right)^{+}\overline{\mathcal{R}}_{17}+\left(f_{\varepsilon}''\right)^{+}\overline{\mathcal{R}}_{18},\label{first-approximation-construction-27a02}
		\end{align}
		with $\mathcal{E}_{4,j}$, $j=3,4$ sums of mode of $4$ terms with respect to $y-q'e_{2}=re^{i\theta}$, of size $\varepsilon^{4}t^{-4}$ on the support of $\left(f_{\varepsilon}'\right)^{+}$ and $\left(f_{\varepsilon}''\right)^{+}$, and 
		\begin{align*}
			\overline{\mathcal{R}}_{17}&=\overline{\mathcal{R}}_{15}-\gradperp_{y}\left(\Psi_{R}-c\varepsilon y_{2}-\Gamma(r)\right)\cdot\nabla_{y}\mathcal{E}_{4,3}\left(\xi_{0}\right)\\
			\overline{\mathcal{R}}_{18}&=\overline{\mathcal{R}}_{16}-\gradperp_{y}\left(\Psi_{R}-c\varepsilon y_{2}-\Gamma(r)\right)\cdot\nabla_{y}\mathcal{E}_{4,4}\left(\xi_{0}\right),
		\end{align*}
		with \eqref{bar-R-15-16-bounds} and Theorem \ref{vortex-pair-properties-theorem} implying the bounds
		\begin{align}
			\left|\left(f_{\varepsilon}'\right)^{+}\overline{\mathcal{R}}_{17}\right|\leq \frac{C\varepsilon^{5}\left(f_{\varepsilon}'\right)^{+}}{t^{3}},\quad \left|\left(f_{\varepsilon}''\right)^{+}\overline{\mathcal{R}}_{18}\right|\leq \frac{C\varepsilon^{5}\left(f_{\varepsilon}''\right)^{+}}{t^{3}}.\label{bar-R-17-18-bounds}
		\end{align}
		
		\noindent \textbf{Step V: Constructing $\psi_{R2}$ and $\xi_{1}$}
		
		\medskip    
		In the final step of the construction of the first approximation, we construct $\psi_{R2}$, and subsequently construct $\xi_{1}$ as the solution to a system of ODEs that ensures the construction of $\psi_{R2}$ only produces sufficiently small error terms.
		
		\medskip
		We define $\psi_{R2}$ as a solution to 
		\begin{align}
			&\Delta\psi_{R2}+\left(f_{\varepsilon}'\right)^{+}\psi_{R2}=\varepsilon\left(f_{\varepsilon}'\right)^{+}\left(\mathcal{N}\left(\xi_{0}\right)\left[\xi_{1}\right]+\dot{\xi}_{1}\right)\cdot\left(y-q'e_{2}\right)^{\perp}\label{psi-R2-equation}\\
			&+\left(f_{\varepsilon}'\right)^{+}\left(\alpha_{1}\left(\xi_{0}\right)\mathscr{Z}_{1}\left(q\right)+\alpha_{2}\left(\xi_{0}\right)\mathscr{Z}_{2}\left(q\right)\right)+\left(f_{\varepsilon}'\right)^{+}\left(\alpha_{3}\left(\xi\right)\mathscr{Z}_{1}\left(q\right)+\alpha_{4}\left(\xi\right)\mathscr{Z}_{2}\left(q\right)\right)\nonumber\\
			&+\left(f_{\varepsilon}'\right)^{+}\left(\mathcal{E}_{2,2}\left(\xi_{0},\dot{\xi}_{0}\right)+\mathcal{E}_{4,1}\left(\xi_{0}\right)+\mathcal{E}_{4,3}\left(\xi_{0}\right)\right)+\left(f_{\varepsilon}''\right)^{+}\left(\mathcal{E}_{4,2}\left(\xi_{0}\right)+\mathcal{E}_{4,4}\left(\xi_{0}\right)\right)\nonumber
		\end{align}
		on $\mathbb{R}^{2}_{+}$ with boundary condition $\psi_{R2}\left(y_{1},0\right)\equiv0$. As in \eqref{psi-R1-equation}, modifying the proofs of Lemmas 2.1--2.3, and Theorem 2.4 from \cite{DDMPVP2023}, from the linear operator $\Delta+V$, to the linear operator $\Delta+\left(f_{\varepsilon}'\right)^{+}$, with $V$ defined in \eqref{V-E-vortex-pair}, we have existence, uniqueness, and the following bounds
		\begin{align}
			\|\psi_{R2}\|_{L^{\infty}\left(\mathbb{R}^{2}_{+}\right)}+\|\nabla\psi_{R2}\|_{L^{\infty}\left(\mathbb{R}^{2}_{+}\right)}+\|\phi_{R2}\|_{L^{\infty}\left(\mathbb{R}^{2}_{+}\right)}+\left|\alpha_{3}\right|+\left|\alpha_{4}\right|\leq \frac{C\varepsilon^{4}}{t^{2}}.\label{psi-R2-elliptic-estimates}
		\end{align}
		To obtain \eqref{psi-R2-elliptic-estimates} we note that the size of the first term on the right hand side of \eqref{psi-R2-equation} has size $\bigO{\left(\varepsilon^{4}t^{-2}\right)}$ on the support of $\left(f_{\varepsilon}'\right)^{+}$, coming from \eqref{pkr-system}, \eqref{linearised-point-vortex-operator-shorthand}, and \eqref{xi0-bound}--\eqref{tildexi-bound}. The terms on the right hand side involving $\alpha_{1}$ and $\alpha_{2}$ also have size $\bigO{\left(\varepsilon^{4}t^{-2}\right)}$ on the support of $\left(f_{\varepsilon}'\right)^{+}$, which can be seen from \eqref{psi-R1-coefficient-of-projection-1-1}--\eqref{psi-R1-coefficient-of-projection-2-3}. Finally, all of the terms on the last line of the right hand side of \eqref{psi-R2-equation} are $\bigO{\left(\varepsilon^{4}t^{-3}\right)}$ on the support of $\left(f_{\varepsilon}'\right)^{+}$ and $\left(f_{\varepsilon}''\right)^{+}$ respectively, which can be seen from \eqref{first-approximation-construction-27a00}--\eqref{first-approximation-construction-27a02}.
		
		\medskip
		As for \eqref{psi-R1-equation-different-form}--\eqref{first-approximation-construction-12}, we can use \eqref{psi-R2-equation} to write 
		\begin{align}
			\phi_{R2}=\left(f_{\varepsilon}'\right)^{+}\phi_{R2,1}+\left(f_{\varepsilon}''\right)^{+}\phi_{R2,2},\label{first-approximation-construction-40a0}
		\end{align}
		where, with respect to $y-q'e_{2}=re^{i\theta}$, $\phi_{R2,1}$ is to main order a sum of mode $1$, mode $2$, and mode $4$ terms of size $\bigO{\left(\varepsilon^{4}t^{-2}\right)}$ with a remainder of size $\bigO{\left(\varepsilon^{6}t^{-2}\right)}$ on the support of $\left(f_{\varepsilon}'\right)^{+}$. Similarly $\phi_{R2,2}$ is to main order a sum of mode $4$ terms of size $\bigO{\left(\varepsilon^{4}t^{-4}\right)}$ with a remainder of size $\bigO{\left(\varepsilon^{6}t^{-4}\right)}$ on the support of $\left(f_{\varepsilon}''\right)^{+}$. Given \eqref{psi-R2-elliptic-estimates}--\eqref{first-approximation-construction-40a0}, as well as the bounds for $\psi_{R1}$, $\phi_{R1}$, and $\psi^{out}$ in \eqref{psi-R1-bounds} and \eqref{first-approximation-construction-22}, we can argue as with \eqref{first-approximation-construction-26}--\eqref{bar-R-17-18-bounds}, and obtain, upon recalling the definition \eqref{f-prime-vortex-pair-variation-relationship-1},
		\begin{align}
			&\left(\varepsilon^{2}\dell_{t}\phi_{R2}-\left[\gradperp_{y}\Psi_{L}+\varepsilon\left(\Dot{p}-c\right)e_{1}\right]\cdot\grad_{y}\phi_{R2}\right)+\gradperp_{y}\psi_{R2}\cdot\grad_{y}\phi_{R1}\nonumber\\
			&+\gradperp_{y}\psi_{R1}\cdot\grad_{y}\phi_{R2}+\gradperp_{y}\psi_{R2}\cdot\grad_{y}\phi_{R2}+\varepsilon\gradperp_{y}\psi^{out}\left(\varepsilon y+pe_{1}\right)\cdot\nabla_{y}\phi_{R2}\nonumber\\
			&=\left(f_{\varepsilon}'\right)^{+}\overline{\mathcal{R}}_{19}+\left(f_{\varepsilon}''\right)^{+}\overline{\mathcal{R}}_{20}+\left(f_{\varepsilon}'''\right)^{+}\overline{\mathcal{R}}_{21}\label{first-approximation-construction-40aa}
		\end{align}
		where,
		\begin{align}
			&\left|\left(f_{\varepsilon}'\right)^{+}\overline{\mathcal{R}}_{j}\right|\leq \frac{C\varepsilon^{5}\left(f_{\varepsilon}'\right)^{+}}{t^{3}},\quad \left|\left(f_{\varepsilon}''\right)^{+}\overline{\mathcal{R}}_{20}\right|\leq \frac{C\varepsilon^{5}\left(f_{\varepsilon}''\right)^{+}}{t^{3}},\nonumber\\
			&\left|\left(f_{\varepsilon}'''\right)^{+}\overline{\mathcal{R}}_{21}\right|\leq \frac{C\varepsilon^{5}\left(f_{\varepsilon}'''\right)^{+}}{t^{3}}.\label{bar-R-19-20-21-bounds}
		\end{align}
		It is left to bound $\alpha_{3}$ and $\alpha_{4}$ effectively. To do this we need to construct appropriate $\xi_{1}=\left(p_{1},q_{1}\right)$ that solves a system of ODEs. 
		
		\medskip
		Recall the definitions of $\mathscr{W}_{j}$, $\tilde{\mathscr{W}}_{j}$ for $j=1,2,3$, as well as $\mathscr{Z}_{1}$, and $\mathscr{Z}_{2}$ from \eqref{f-prime-vortex-pair-variation-relationship}--\eqref{dq-Psi-R-variation}, including that the support of the $\mathscr{W}_{j}$, $\tilde{\mathscr{W}}_{j}$ are the same for $j=1,2,3$. Then, if $\xi_{1}=\left(p_{1},q_{1}\right)=\left(\xi_{1,1},\xi_{1,2}\right)$, $\xi_{1}$ will solve the system given by, for $(k,l)=(1,2)$ and $(2,1)$,
		\begin{align}
			&\varepsilon\left(\left(\mathcal{N}\left(\xi_{0}\right)\left[\xi_{1}\right]\right)_{k}+\dot{\xi}_{1,k}\right)\int_{B_{2\rho}\left(0\right)}\tilde{\mathscr{W}}_{1}\left(q_{0}\right)\mathscr{Z}_{l}\left(q_{0}\right)z_{l}\ dz\label{p1-equation}\\
			&+\int_{B_{2\rho}\left(0\right)}\tilde{\mathscr{W}}_{1}\left(q_{0}\right)\mathscr{Z}_{l}\left(q_{0}\right)\left(\mathcal{E}_{2,2}\left(\xi_{0},\dot{\xi}_{0}\right)+\mathcal{E}_{4,1}\left(\xi_{0}\right)+\mathcal{E}_{4,3}\left(\xi_{0}\right)\right)dz\nonumber\\
			&+\int_{B_{2\rho}\left(0\right)}\tilde{\mathscr{W}}_{2}\left(q_{0}\right)\mathscr{Z}_{l}\left(q_{0}\right)\left(\mathcal{E}_{4,2}\left(\xi_{0}\right)+\mathcal{E}_{4,4}\left(\xi_{0}\right)\right)dz\nonumber\\
			&+\alpha_{l}\left(\xi_{0}\right)\int_{B_{2\rho}\left(0\right)}\tilde{\mathscr{W}}_{1}\left(q_{0}\right)\mathscr{Z}_{l}^{2}\left(q_{0}\right) dz=0,\nonumber
		\end{align}
		The construction of the solution to the system above on $\left[T_{0},\infty\right)$ that satisfies \eqref{xi1-bound} proceeds similarly to Section \ref{point-vortex-trajectory-section}, and for $T_{0}>0$ large enough
		we can once again run a fixed point argument analogous to Theorem \ref{p0-q0-construction-theorem} so that for $\varepsilon>0$ small enough, we obtain a solution to \eqref{p1-equation} that satisfies bounds \eqref{xi1-bound}.
		
		\medskip
		With our $\xi_{1}$ in hand, we can compare the equations they solve in \eqref{p1-equation} to the identities obtained by integrating \eqref{psi-R2-equation} against $\mathscr{Z}_{2}\left(q\right)$ and $\mathscr{Z}_{1}\left(q\right)$ respectively. Then the bounds for the terms on the right hand side of \eqref{psi-R2-equation} as well as \eqref{f-prime-q-variation-3}--\eqref{f-prime-q-variation-4} give that
		\begin{align}
			\left|\alpha_{3}\right|+\left|\alpha_{4}\right|\leq\frac{C\varepsilon^{5}}{t^{4}}.\label{first-approximation-construction-41}
		\end{align}
		Recalling \eqref{first-approximation-construction-22a}, \eqref{first-approximation-construction-27a00}, \eqref{first-approximation-construction-27a02}, \eqref{psi-R2-equation}, and \eqref{first-approximation-construction-41}, let
		\begin{equation}\label{R-star-1-definition}
			\begin{aligned}
				\mathcal{R}^{*}_{1}&=\gradperp_{y}\left(\Psi_{R}-c\varepsilon y_{2}\right)\cdot\grad_{y}\left(\overline{\mathcal{R}}_{4}+\mathcal{R}_{12}\right)+\overline{\mathcal{R}}_{5}+\overline{\mathcal{R}}_{13}+\overline{\mathcal{R}}_{17}+\overline{\mathcal{R}}_{19}\nonumber\\
				&-\gradperp_{y}\left(\Psi_{R}-c\varepsilon y_{2}\right)\cdot\grad_{y}\left(\alpha_{3}\left(\xi\right)\mathscr{Z}_{1}(q)+\alpha_{4}\left(\xi\right)\mathscr{Z}_{2}(q)\right),\\
				\mathcal{R}^{*}_{2}&=\overline{\mathcal{R}}_{6}+\overline{\mathcal{R}}_{14}+\overline{\mathcal{R}}_{18}+\overline{\mathcal{R}}_{20},\quad \mathcal{R}^{*}_{3}=\overline{\mathcal{R}_{21}}.
			\end{aligned}
		\end{equation}
		Then \eqref{first-approximation-construction-26}--\eqref{first-approximation-construction-41} give that \eqref{first-approximation-construction-22a} can be written as
		\begin{align*}
			&E_{R1}=-\varepsilon\left(\mathcal{N}\left(\xi_{0}+\xi_{1}\right)\left[\tilde{\xi}\right]+\dot{\tilde{\xi}}\right)\cdot\grad_{y}U_{R}+\left(f_{\varepsilon}'\right)^{+}\mathcal{R}^{*}_{1}+\left(f_{\varepsilon}''\right)^{+}\mathcal{R}^{*}_{2}+\left(f_{\varepsilon}'''\right)^{+}\mathcal{R}^{*}_{3},
		\end{align*}
		where, by \eqref{first-approximation-construction-2a}, \eqref{R-12-bound}, \eqref{psi-out-bound-2}, \eqref{bar-R-7-8-bounds}, \eqref{bar-9-bounds}, \eqref{bar-R-10-bound}, \eqref{bar-R-11-12-bounds}, \eqref{bar-R-13-14-bounds}, \eqref{bar-R-15-16-bounds}, \eqref{bar-R-17-18-bounds}, \eqref{bar-R-19-20-21-bounds}, and \eqref{first-approximation-construction-41},
		\begin{align}
			&\left|\left(f_{\varepsilon}'\right)^{+}\mathcal{R}^{*}_{1}\right|\leq\frac{C\varepsilon^{5}\left(f_{\varepsilon}'\right)^{+}}{t^{3}},\quad \left|\left(f_{\varepsilon}''\right)^{+}\mathcal{R}^{*}_{2}\right|\leq\frac{C\varepsilon^{5}\left(f_{\varepsilon}''\right)^{+}}{t^{3}},\nonumber\\
			&\left|\left(f_{\varepsilon}'''\right)^{+}\mathcal{R}^{*}_{3}\right|\leq\frac{C\varepsilon^{5}}{t^{3}}\left(f_{\varepsilon}'''\right)^{+},\label{first-approximation-construction-42}
		\end{align}
		We have, for $k=1,2,3$,
		\begin{align}
			\mathcal{R}^{*}_{k}=\mathcal{R}^{*}_{k}\left(y,\xi,\dot{\xi}\right),
		\end{align}
		where we have suppressed any specific dependence on $\xi_{0}$ and $\dot{\xi}_{0}$ that manifest through the terms $\mathcal{E}_{j}\left(\xi_{0}\right)$, $j=2,3,4$ and $\dot{\xi}_{0}\cdot\dell_{\xi_{0}}\mathcal{E}_{2}\left(\xi_{0}\right)$ via the identities \eqref{first-approximation-construction-27a0}, \eqref{mode-4-formulae}, \eqref{first-approximation-construction-27a00}, \eqref{psi-R2-equation}, and \eqref{first-approximation-construction-40aa}. This is as $\xi_{0}$ and its time derivatives are known functions by the construction in Section \ref{point-vortex-trajectory-section}, whereas dependence on $\xi$ and $\dot{\xi}$ is dependence on the unknowns $\tilde{\xi}$ and $\dot{\tilde{\xi}}$ by \eqref{point-vortex-trajectory-function-decomposition-new}.
		
		\medskip
		Finally, due to \eqref{psi-R2-elliptic-estimates} and analogous bounds for $\psi_{L2}$, we can run the same argument as in Step III but for $\psi_{R2}$ and $\psi_{L2}$, and this time obtain that
		\begin{align}
			E_{2}^{out}=\sum_{j=R,L}\left(-1\right)^{j}\left[\psi_{j2}\Delta\eta^{(j)}_{K}+2\nabla\eta^{(j)}_{K}\cdot\nabla\psi_{j2}\right]=\bigO{\left(\frac{\varepsilon^{5}}{t^{5}}\right)}.\label{first-approximation-outer-error-E2out-final-bound}
		\end{align}
	\end{proof}
	As a final note in this section, finding a way to exclude the main order error terms with respect to $t$ when defining the ODE system that $\tilde{\xi}$ will solve will be crucial to obtaining the estimates for $\tilde{\xi}$ claimed in \eqref{tildexi-bound}. This will be done via orthogonality conditions for $\phi_{*R}$ defined in Section \ref{cs}. To that end, we define the following useful quantities which will be referred to repeatedly in the next Section. Recall \eqref{f-prime-vortex-pair-variation-relationship}--\eqref{tilde-W-k-def}, and define for $\left(i,j\right)\in\left\{\left(0,0\right),\left(1,0\right),\left(0,1\right)\right\}$,
	\begin{align}
		&\mathscr{Q}_{\left(i,j\right)}\left(\xi_{0}+\xi_{1},\dot{\xi}_{0}+\dot{\xi}_{1}\right)\nonumber\\
		&=\int_{\mathbb{R}^{2}_{+}-q'e_{2}}z_{1}^{i}z_{2}^{j}\left[\sum_{k=1}^{3}\tilde{\mathscr{W}}_{k}\left(q_{0}+q_{1}\right)\mathcal{R}^{*}_{k}\left(\xi_{0}+\xi_{0},\dot{\xi}_{0}+\dot{\xi}_{1}\right)\right]dz,\label{0-orthogonality-rhs}
	\end{align}
	where we have used the substitution $z=y-q'e_{2}$, and then frozen the dependence of the $\mathcal{R}^{*}_{k}$ on $\xi$, at $\xi_{0}+\xi_{1}$. Note that by Remark \ref{support-of-nonlinearity-remark}, \eqref{xi0-bound}--\eqref{tildexi-bound}, and \eqref{f-prime-vortex-pair-variation-relationship}--\eqref{tilde-W-k-def}, for $k=1,2,3$, $\tilde{\mathscr{W}}_{k}\left(q_{0}+q_{1}\right)$ has support uniformly bounded in time in $B_{2\rho}(0)$ in $z$ coordinates. By the bounds in \eqref{first-approximation-construction-42}, we have
	\begin{align}
		\left|\mathscr{Q}_{\left(i,j\right)}\right|\leq \frac{C\varepsilon^{5}}{t^{3}}.\label{orthogonality-rhs-bounds}
	\end{align}
	and by the bounds \eqref{xi0-bound}--\eqref{tildexi-bound}, for all $t\in\left[T_{0},T\right]$,
	\begin{align}
		\left|\mathscr{Q}_{\left(i,j\right)}-\int_{\mathbb{R}^{2}_{+}}y_{1}^{i}\left(y_{2}-q'\right)^{j}\left[\sum_{k=1}^{3}\left(f_{\varepsilon}^{\left(k\right)}\right)^{+}\mathcal{R}^{*}_{k}\left(y,\xi,\dot{\xi}\right)\right]dy\right|\leq \frac{C\varepsilon^{5}}{t^{4}}.\label{orthogonality-rhs-comparison-bounds}
	\end{align}
	\section{Constructing Solutions on $[T_{0},T]$}\label{constructing-full-solution-section}
	\begin{remark}\label{constructing-full-solution-gamma-remark}
		As in Remark \ref{first-approx-section-gamma-remark}, we recall here that in Sections \ref{constructing-full-solution-section} and \ref{conclusion}, $\gamma>18$.
	\end{remark}
	Having constructed the first approximation $\left(\omega_{*},\Psi_{*}\right)$ of a solution to \eqref{2d-euler-vorticity-stream} on $\left[T_{0},T\right]$ in Theorem \ref{first-approximation-construction-theorem}, we now look for a solution on the same time interval to \eqref{2d-euler-vorticity-stream} of the form
	\begin{align}
		\omega\left(x,t\right)=\omega_{*}\left(x,t\right)+\phi_{*}\left(x,t\right),\quad \Psi\left(x,t\right)=\Psi_{*}\left(x,t\right)+\psi_{*}\left(x,t\right).\label{omega-full-solution-decomposition}
	\end{align}
	We set
	\begin{align}
		\phi_{*}=\varepsilon^{-2}\left(\phi_{*R}-\phi_{*L}\right),\ \psi_{*}=\eta^{(R)}_{K}\psi_{*R}-\eta^{(L)}_{K}\psi_{*L}+\psi_{*}^{out},\label{psi-phi-star-cutoffs}
	\end{align}
	with $\eta^{(R)}_{K}$ and $\eta^{(L)}_{K}$ defined in \eqref{cutoffs-definition-1}. We define
	\begin{align*}
		\phi_{*}^{in}=\left(\phi_{*R},\phi_{*L}\right),\ \psi_{*}^{in}=\left(\psi_{*R},\psi_{*L}\right).
	\end{align*}
	Finally we again have $\xi(t)=\left(p(t),q(t)\right)$ as in \eqref{point-vortex-trajectory-function-decomposition-new} for $\xi_{0}$ solving \eqref{p0-q0-system}, $\xi_{1}$ solving \eqref{p1-equation}, and $\xi_{0},\xi_{1}$, and $\Tilde{\xi}$ satisfying \eqref{xi0-bound}--\eqref{tildexi-bound}. 
	
	\medskip
	Unlike Section \ref{first-approximation-section}, in order to reveal important structure of the full linearised operator, we will mainly work in coordinates of the form
	\begin{align}
		z=\frac{x-\xi\left(t\right)}{\varepsilon}=\frac{x-pe_{1}-qe_{2}}{\varepsilon}\label{change-of-coordinates-dynamic-problem-3}
	\end{align}
	for $\left(\phi_{*R},\psi_{*R}\right)$, and analogously for $\left(\phi_{*L},\psi_{*L}\right)$. However in reality the change of coordinates throughout Section \ref{constructing-full-solution-section} will rely on a homotopy parameter, see \eqref{homotopy-parameter-change-of-coordinates}. The majority of the construction of the full solution in this section will be done in $z$ coordinates. However, when emphasising the relationship between $\phi_{*}$ and $\psi_{*}$, that is
	\begin{align}
		-\Delta \psi_{*j}=\phi_{*j},\ \ \  j=R,L,\label{psi-star-phi-star-laplace-equation}
	\end{align}
	we will usually work in $y$ coordinates as in \eqref{change-of-coordinates-dynamic-problem} and take \eqref{psi-star-phi-star-laplace-equation} to mean satisfying the stated equation on the upper half plane with $\psi_{*R}\left(y_{1},0\right)\equiv0$, and analogously in the appropriate coordinates for $\psi_{*L}$. These instances will be clarified when they occur. Accordingly, let
	\begin{align}
		\mathbb{H}_{w}=\left\{\left(v_{1},v_{2}\right): v_{2}\geq w\right\}.\label{upper-half-plane-translate-definition}
	\end{align}
	Then $\mathbb{H}_{0}=\mathbb{R}_{+}^{2}$, and we will construct $\left(\phi_{*R}\left(z,t\right),\psi_{*R}\left(z,t\right)\right)$ on $\mathbb{H}_{-q'}$ such that $\supp \phi_{*R}\subset B_{3\rho}\left(0\right)$, $\rho>0$ defined in \eqref{first-approximation-main-order-vorticity-support}, $\psi_{*R}\left(z_{1},-q',t\right)=0$, and analogously for $\left(\phi_{*L},\psi_{*L}\right)$.
	
	\medskip
	We now concentrate on $\left(\phi_{*R}\left(z,t\right),\psi_{*R}\left(z,t\right)\right)$ with the understanding that $\left(\phi_{*L},\psi_{*L}\right)$ is obtained after reflection in an odd manner around $x_{1}=0$ in the original variables. We make a further decomposition of $\left(\phi_{*R}\left(z,t\right),\psi_{*R}\left(z,t\right)\right)$ of the form
	\begin{align}
		\phi_{*R}\left(z,t\right)=\Tilde{\phi}_{*R}\left(z,t\right)+\alpha_{R}(t)U_{R}\left(z,t\right),\label{real-phi-star-ansatz}
	\end{align}
	with $U_{R}$ defined in \eqref{w-r-definition}, and $\alpha_{R}$ a time-dependent parameter that will solve a certain ODE defined in \eqref{cRj-initial-value-problem}, and will satisfy 
	\begin{align}
		\|t^{3}\alpha_{R}\|_{\left[T_{0},T\right]}+\|t^{4}\dot{\alpha}_{R}\|_{\left[T_{0},T\right]}\leq \varepsilon^{3-\sigma},\ \ \alpha_{R}\left(T\right)=0,\label{alpha-bound}
	\end{align}
	for arbitrarily small $\sigma>0$. We will also impose three orthogonality conditions on $\Tilde{\phi}_{*R}$, fixing its total mass, and its centre of mass. The function $\tilde{\phi}_{*R}$ corresponding to a genuine solution to \eqref{2d-euler-vorticity-stream} on $\left[T_{0},T\right]$ will satisfy, for $\left(i,j\right)\in\left\{\left(0,0\right),\left(1,0\right),\left(0,1\right)\right\}$,
	\begin{align}
		\int_{\mathbb{H}_{-q'}}z_{1}^{i}z_{2}^{j}\Tilde{\phi}_{*R}\left(z,t\right)\ dz=\mathcal{J}_{\left(i,j\right)}(t)\coloneqq-\varepsilon^{-2}\int_{t}^{T}\mathscr{Q}_{\left(i,j\right)}(\tau) d\tau=\bigO{\left(\varepsilon^{3}t^{-2}\right)},\label{real-orthogonality-conditions}
	\end{align}
	where the functions $\mathscr{Q}_{\left(i,j\right)}$ were defined in \eqref{0-orthogonality-rhs}, $\mathcal{J}_{\left(i,j\right)}$ is shorthand that we will use repeatedly, and the estimate on the right hand side is from \eqref{orthogonality-rhs-bounds}. In reality, the orthogonality conditions we will work with throughout Section \ref{constructing-full-solution-section} will depend on a homotopy parameter, see \eqref{tilde-phi-star-R-0-mass-condition}.
	\begin{remark}
		The time-dependent parameter $\alpha_{R}$ appearing in \eqref{real-phi-star-ansatz}
		is a way of adjusting the mass, and will be used to enforce the mass condition, just as $\left(\Tilde{p},\Tilde{q}\right)$ will be used to enforce the centre of mass conditions.
	\end{remark}
	Having imposed the ansatz \eqref{real-phi-star-ansatz} for $\phi_{*R}$, we impose an ansatz for $\psi_{*R}$:
	\begin{align}
		\psi_{*R}\left(z,t\right)=\Tilde{\psi}_{*R}\left(z,t\right)+\alpha_{R}\left(t\right)\left(\Psi_{R}\left(z,t\right)-c\varepsilon\left(z_{2}+q'\right)\right),\label{real-psi-star-ansatz}
	\end{align}
	with $\Psi_{R}$ defined in \eqref{psi-r-definition}. By a slight abuse of notation, in $y$ coordinates this can be restated as
	\begin{align}
		\psi_{*R}\left(y,t\right)=\Tilde{\psi}_{*R}\left(y,t\right)+\alpha_{R}\left(t\right)\left(\Psi_{R}\left(y,t\right)-c\varepsilon y_{2}\right),\label{real-psi-star-ansatz-2}
	\end{align}
	where, in the upper half plane in $y$ variables, $-\Delta \Tilde{\psi}_{*R}=\Tilde{\phi}_{*R}$ with $\Tilde{\psi}_{*R}\left(y_{1},0\right)\equiv0$. That is,
	\begin{align}
		\Tilde{\psi}_{*R}\left(y,t\right)=\frac{1}{2\pi}\int_{\mathbb{R}_{+}^{2}}\log{\left(\frac{\left|v-\bar{y}\right|}{\left|v-y\right|}\right)}\Tilde{\phi}_{*R}\left(v,t\right) dv.\label{psi-star-r-green-function-representation-upper-half-plane}
	\end{align}
	We note here that by \eqref{Psi-U-equation} and a change of variables, that
	\begin{align*}
		-\Delta_{y}\left(\Psi_{R}\left(y,t\right)-c\varepsilon y_{2}\right)=U_{R}(y,t),
	\end{align*}
	with the corresponding statement holding in $z$ variables for \eqref{real-psi-star-ansatz} by translation. This fact, in addition to the fact that $U_{R}(y,t)=f_{\varepsilon}\left(\Psi_{R}-c\varepsilon y_{2}\right)$ in $y$ coordinates, is crucial to obtain the right structure in our computations \eqref{homotopic-operator-inner-error-E-star-R-definition-3}--\eqref{homotopic-operator-inner-error-main-definition-lower-order-terms}. We also note that whilst $\psi_{*R}$ grows linearly at infinity by \eqref{real-psi-star-ansatz}--\eqref{real-psi-star-ansatz-2}, it is multiplied by a cutoff in \eqref{psi-phi-star-cutoffs}.
	
	\medskip
	We let
	\be 
	\Tilde{\phi}_{*} =\left(\Tilde{\phi}_{*R},\Tilde{\phi}_{*L}\right), \quad \Tilde{\psi}_{*}=\left(\Tilde{\psi}_{*R},\Tilde{\psi}_{*L}\right).\label{tilde-psi-star-definition}
	\ee
	We finally note that by symmetry, we have the decomposition $\phi_{*L}=\Tilde{\phi}_{*L}-\alpha_{R}(t)U_{L}$
	in the appropriate coordinates, with $\Tilde{\phi}_{*L}$ satisfying the odd symmetry condition in $x_{1}$ with respect to $\Tilde{\phi}_{*R}$. Thus, we can drop the subscript on $\alpha_{R}$, and we shall simply work with
	$
	\alpha(t)=\alpha_{R}(t). $
	
	\medskip
	As a final note in this subsection, we remark that just like the orthogonality conditions \eqref{real-orthogonality-conditions} corresponding to those of the genuine solution, the true ansatz for $\phi_{*R}$ and $\psi_{*R}$ we work with throughout the rest of Section \ref{constructing-full-solution-section} are not quite \eqref{real-phi-star-ansatz}, \eqref{real-psi-star-ansatz} and \eqref{real-psi-star-ansatz-2}, but will depend on a homotopy parameter, see \eqref{phi-star-psi-star-homotopy}.
	
	\subsection{Homotopic Operators}\label{full-linearised-euler-operator-section}
	In this section we define, with motivation, operators depending on a homotopy parameter $\lambda\in\left[0,1\right]$. At $\lambda=0$, these operators will be linear, and at $\lambda=1$, their annihilation will imply a solution to the Euler system \eqref{2d-euler-vorticity-stream}. To begin with, we insert $\left(\omega,\Psi\right)$, defined in \eqref{omega-full-solution-decomposition} into $E_{1}$ and $E_{2}$ defined in \eqref{euler-operators}. Similarly to \eqref{first-approximation-inner-error}--\eqref{first-approximation-outer-error-E2out-definition}, we have, for $j=R,L$ and $(-1)^{R}=1$ and $(-1)^{L}=-1$,
	\begin{align}\nonumber 
		E_{1}\left[\omega,\Psi\right]&=\varepsilon^{-4}\sum_{j=R,L}\left(-1\right)^{j}E_{*j}\left(\phi_{j},\phi_{*j},\psi_{j},\psi_{*j},\psi^{out}, \psi_{*}^{out};\xi\right)\\
		\nonumber E_{2}\left[\omega,\Psi\right]&=E_{*}^{out}\left(\phi^{in},\phi^{in}_{*},\psi^{in},\psi^{in}_{*},\psi^{out}, \psi_{*}^{out};\xi\right),
	\end{align}
	where, for $E_{2}^{out}$ defined in \eqref{first-approximation-outer-error-E2out-final-bound},
	\begin{align} 
		E_{*j}&=\varepsilon^{2}\dell_{t}\left(U_{j}+\phi_{j}+\phi_{*j}\right)\nonumber\\
		&+\varepsilon^{2}\gradperp\left(\Psi_{0}+\left(-1\right)^{j}\left(\psi_{j}+\psi_{*j}\right)+\psi^{out}+\psi^{out}_{*}\right)\cdot\nabla\left(U_{j}+\phi_{j}+\phi_{*j}\right),\nonumber\\
		E_{*}^{out}&=\Delta\psi_{*}^{out}+\sum_{j=R,L}(-1)^{j}\left(\psi_{*j}\Delta\eta^{(j)}_{K}+2\nabla\eta^{(j)}_{K}\cdot\nabla\psi_{*j}\right)+E_{2}^{out},\label{homotopic-operator-inner-error-E-star-R-definition}
	\end{align}
	with $\psi_{j}$, $\phi_{j}$, $j=R,L$, and $\psi^{out}$ constructed in the proof of Theorem \ref{first-approximation-construction-theorem}, and $E_{2}^{out}$ given in \eqref{first-approximation-outer-error-E2out-final-bound}. We recall that the specific forms of $E_{R1}$, $E_{L1}$, and $E_{2}^{out}$ given in \eqref{first-approximation-inner-error-ER1-definition}--\eqref{first-approximation-outer-error-E2out-definition} are due to the choice of $K$ and $T_{0}$ in \eqref{K-T0-choices} implying $\eta^{(j)}_{K}\equiv1$ on $\supp{\phi_{j}}$ for $j=R,L$, $\eta^{(R)}_{K}\equiv0$ on $\supp{\phi_{L}}$, and vice versa. Due to \eqref{tildexi-bound} and \eqref{alpha-bound}, the supports of $\phi_{*R}$ and $\phi_{*L}$ will be close enough to the supports of $U_{R}$ and $U_{L}$ respectively to imply the same behaviour, hence the specific forms of $E_{*R}$, $E_{*L}$, and $E_{*}^{out}$ given in \eqref{homotopic-operator-inner-error-E-star-R-definition}.
	
	\medskip
	Regrouping terms for $E_{*R}$ defined in \eqref{homotopic-operator-inner-error-E-star-R-definition}, this becomes
	\begin{align}
		E_{*R}&=\varepsilon^{2}\dell_{t}\left(U_{R}+\phi_{R}\right)+\varepsilon^{2}\gradperp\left(\Psi_{0}+\psi_{R}+\psi^{out}\right)\cdot\nabla\left(U_{R}+\phi_{R}\right)\label{homotopic-operator-inner-error-E-star-R-definition-3}
		\\
		&+\varepsilon^{2}\dell_{t}\phi_{*R}+\varepsilon^{2}\gradperp\left(\Psi_{0}+\psi_{R}+\psi_{*R}+\psi^{out}+\psi^{out}_{*}\right)\cdot\nabla \phi_{*R}\nonumber\\
		&+\varepsilon^{2}\gradperp\left(\psi_{*R}+\psi^{out}_{*}\right)\cdot\nabla\left(U_{R}+\phi_{R}\right).\nonumber
	\end{align}
	Now, to define our homotopy parameter dependent operators, we introduce a change of coordinates $z_{\lambda}$ that itself depends on the homotopy parameter $\lambda$:
	\begin{align}
		z_{\lambda}=\frac{x-\xi_{0}\left(t\right)-\xi_{1}\left(t\right)-\lambda\tilde{\xi}\left(t\right)}{\varepsilon},\label{homotopy-parameter-change-of-coordinates}
	\end{align}
	where $\xi_{0}$ and $\xi_{1}$ were constructed in Section \ref{point-vortex-trajectory-section} and Theorem \ref{first-approximation-construction-theorem}, and $\tilde{\xi}$ satisfies \eqref{tildexi-bound}. We denote $\xi^{\left(\lambda\right)}\coloneqq\xi_{0}+\xi_{1}+\lambda\tilde{\xi}$, and in general any quantities that depend on $\xi^{\left(\lambda\right)}$ will be written with a $\lambda$ subscript. In these coordinates, the orthogonality conditions for $\tilde{\phi}_{*R}$ become 
	for $\left(i,j\right)\in\left\{\left(0,0\right),\left(1,0\right),\left(0,1\right)\right\}$,
	\begin{align}
		\int_{\mathbb{H}_{-q'}}z_{\lambda,1}^{i}z_{\lambda,2}^{j}\Tilde{\phi}_{*R}\left(z_{\lambda},t\right)\ dz_{\lambda}=\lambda\mathcal{J}_{\left(i,j\right)}(t)&\coloneqq-\lambda\varepsilon^{-2}\int_{t}^{T}\mathscr{Q}_{\left(i,j\right)}(\tau) d\tau\nonumber\\
		&=\bigO{\left(\lambda\varepsilon^{3}t^{-2}\right)}.\label{tilde-phi-star-R-0-mass-condition}
	\end{align}
	Using Theorem \ref{first-approximation-construction-theorem}, in $z_{\lambda}$ coordinates, \eqref{homotopic-operator-inner-error-E-star-R-definition-3} becomes
	\begin{align}
		&E_{*R}=-\varepsilon\left(\mathcal{N}\left(\xi_{0}+\xi_{1}\right)\left[\lambda\tilde{\xi}\right]+\lambda\dot{\tilde{\xi}}\right)\cdot\grad_{z_{\lambda}}U_{R}+\nabla_{z_{\lambda}}\left(U_{R}+\phi_{R}\right)\cdot\gradperp_{z_{\lambda}}\psi^{out}_{*}\label{homotopic-operator-inner-error-E-star-R-definition-5}\\
		&+\left(f_{\varepsilon}'\right)^{+}\mathcal{R}^{*}_{1}\left(z_{\lambda},\xi^{\left(\lambda\right)},\dot{\xi}^{\left(\lambda\right)}\right)+\left(f_{\varepsilon}''\right)^{+}\mathcal{R}^{*}_{2}\left(z_{\lambda},\xi^{\left(\lambda\right)},\dot{\xi}^{\left(\lambda\right)}\right)+\left(f_{\varepsilon}'''\right)^{+}\mathcal{R}^{*}_{3}\left(z_{\lambda},\xi^{\left(\lambda\right)},\dot{\xi}^{\left(\lambda\right)}\right)\nonumber\\
		&+\varepsilon^{2}\dell_{t}\phi_{*R}\left(z_{\lambda},t\right)+\gradperp_{z_{\lambda}}\psi_{*R}\cdot\nabla_{z_{\lambda}}\left(U_{R}+\phi_{R}\right)\nonumber\\
		&+\gradperp_{z_{\lambda}}\left(\Psi_{0}-\varepsilon \dot{p}_{\lambda}z_{\lambda,2}+\varepsilon \dot{q}_{\lambda}z_{\lambda,1}+\psi_{R}+\psi_{*R}+\psi^{out}+\psi^{out}_{*}\right)\cdot\nabla_{z_{\lambda}}\phi_{*R},\nonumber
	\end{align}
	where the $\mathcal{R}^{*}_{j}$ were constructed in the proof of Theorem \ref{first-approximation-construction-theorem}.
	
	\medskip
	We now state that $\phi_{*R}$ and $\psi_{*R}$ are of the form 
	\begin{align}
		\phi_{*R}\left(z_{\lambda},t\right)&=\Tilde{\phi}_{*R}\left(z_{\lambda},t\right)+\lambda\alpha_{R}(t)U_{R}\left(z_{\lambda},t\right),\nonumber\\
		\psi_{*R}\left(z_{\lambda},t\right)&=\Tilde{\psi}_{*R}\left(z_{\lambda},t\right)+\lambda\alpha_{R}\left(t\right)\left(\Psi_{R}\left(z_{\lambda},t\right)-c\varepsilon\left(z_{\lambda,2}+q'\right)\right),\label{phi-star-psi-star-homotopy}
	\end{align}
	with the appropriate adjustment when considering $\psi_{*R}$ as a function of $y$ instead, as in \eqref{real-psi-star-ansatz-2}. Then concentrating on the last three terms of \eqref{homotopic-operator-inner-error-E-star-R-definition-5}, these terms can be written as
	\begin{align}
		&\varepsilon^{2}\dell_{t}\tilde{\phi}_{*R}\left(z_{\lambda},t\right)+\gradperp_{z_{\lambda}}\tilde{\psi}_{*R}\cdot\nabla_{z_{\lambda}}\left(U_{R}+\phi_{R}\right)+\gradperp_{z_{\lambda}}\psi^{out}_{*}\cdot\nabla_{z_{\lambda}}\phi_{R}\label{homotopic-operator-inner-error-E-star-R-definition-7}\\
		&+\gradperp_{z_{\lambda}}\left(\Psi_{0}-\varepsilon \dot{p}_{\lambda}z_{\lambda,2}+\varepsilon \dot{q}_{\lambda}z_{\lambda,1}+\psi_{R}+\psi_{*R}+\psi^{out}+\psi^{out}_{*}\right)\cdot\nabla_{z_{\lambda}}\tilde{\phi}_{*R}\nonumber\\
		&+\lambda\varepsilon^{2}\left(\dot{\alpha}U_{R}+\alpha\dell_{t}U_{R}\right)+\alpha\gradperp_{z_{\lambda}}\left(\Psi_{R}-c_{\lambda}\varepsilon \left(z_{\lambda,2}+q'_{\lambda}\right)\right)\cdot\nabla_{z_{\lambda}}\phi_{R}\nonumber\\
		&+\alpha\gradperp_{z_{\lambda}}\left(\Psi_{0}-\varepsilon \dot{p}_{\lambda}z_{\lambda,2}+\varepsilon \dot{q}_{\lambda}z_{\lambda,1}+\psi_{R}+\psi_{*R}+\psi^{out}+\psi^{out}_{*}\right)\cdot\nabla_{z_{\lambda}}U_{R}\nonumber
	\end{align}
	We make the decomposition
	\begin{align}
		&\Psi_{0}-\varepsilon \dot{p}_{\lambda}z_{\lambda,2}+\varepsilon \dot{q}_{\lambda}z_{\lambda,1}+\psi_{R}+\psi_{*R}+\psi^{out}+\psi^{out}_{*}-\left|\log{\varepsilon}\right|\Omega_{\lambda}\label{homotopic-operator-inner-error-E-star-R-definition-10}\\
		&=\mathscr{V}\left(z_{\lambda},q_{\lambda}\right)\underbrace{-\left(\Psi_{L}\left(z_{\lambda},t\right)+\varepsilon\left(\dot{p}_{\lambda}-c_{\lambda}\right)z_{\lambda,2}-\varepsilon \dot{q}_{\lambda}z_{\lambda,1}\right)+\psi_{R}+\psi_{*R}+\psi^{out}+\psi^{out}_{*}}_{\mathfrak{a}_{R,\lambda}},\nonumber
	\end{align}
	wherein we note that whilst the main order term $\mathscr{V}\left(z_{\lambda},q_{\lambda}\right)$ defined in \eqref{vortex-pair-variation-generalisation} has a nonlinear dependence on $\lambda\Tilde{\xi}$ for $\lambda>0$, this dependence disappears at $\lambda=0$. By \eqref{f-prime-vortex-pair-variation-relationship}, $U_{R}\left(z_{\lambda},t\right)$ is a function of $\mathscr{V}\left(z_{\lambda},q_{\lambda}\right)$, so we define $\mathscr{E}_{R,\lambda}\left(\tilde{\phi}_{*R}, \lambda\alpha,\psi^{out}_{*},\lambda\Tilde{\xi}\right)$ by
	\begin{align}
		&\mathscr{E}_{R,\lambda}\coloneqq\varepsilon^{2}\dell_{t}\tilde{\phi}_{*R}\left(z_{\lambda},t\right)+\gradperp_{z_{\lambda}}\tilde{\psi}_{*R}\cdot\nabla_{z_{\lambda}}\left(U_{R}+\lambda\phi_{R}\right)+\lambda\varepsilon^{2}\dot{\alpha}\left(t\right)U_{R}\left(z_{\lambda},t\right)+\lambda\tilde{\mathscr{E}}_{R}\label{homotopic-operator-inner-error-main-definition}\\
		&+\gradperp_{z_{\lambda}}\left(\mathscr{V}\left(z_{\lambda},q_{\lambda}\right)+\lambda\mathfrak{a}_{R,\lambda}\right)\cdot\nabla_{z_{\lambda}}\tilde{\phi}_{*R}-\varepsilon\left(\mathcal{N}\left(\xi_{0}+\xi_{1}\right)\left[\lambda\tilde{\xi}\right]+\lambda\dot{\tilde{\xi}}\right)\cdot\grad_{z_{\lambda}}U_{R},\nonumber
	\end{align}
	where
	\begin{align}
		&\tilde{\mathscr{E}}_{R}\left(\tilde{\psi}_{*R}, \lambda\alpha,\psi^{out}_{*},\lambda\tilde{\xi}\right)=\gradperp_{z_{\lambda}}\psi^{out}_{*}\cdot\nabla_{z_{\lambda}}\left(U_{R}+\phi_{R}\right)+\varepsilon^{2}\alpha\dot{q}\dell_{q_{\lambda}}U_{R}\left(z_{\lambda},t\right)\label{homotopic-operator-inner-error-main-definition-lower-order-terms}\\
		&+\alpha\gradperp_{z_{\lambda}}\left(\Psi_{R}-c_{\lambda}\varepsilon \left(z_{\lambda,2}+q'_{\lambda}\right)\right)\cdot\nabla_{z_{\lambda}}\phi_{R}+\alpha\gradperp_{z_{\lambda}}\mathfrak{a}_{R,\lambda}\cdot\nabla_{z_{\lambda}}U_{R}\nonumber\\
		&+\left(f_{\varepsilon}'\right)^{+}\mathcal{R}^{*}_{1}\left(z_{\lambda},\xi^{\left(\lambda\right)}\right)+\left(f_{\varepsilon}''\right)^{+}\mathcal{R}^{*}_{2}\left(z_{\lambda},\xi^{\left(\lambda\right)}\right)+\left(f_{\varepsilon}'''\right)^{+}\mathcal{R}^{*}_{3}\left(z_{\lambda},\xi^{\left(\lambda\right)}\right).\nonumber
	\end{align}
	Here $\phi_{R}$ and $\psi_{R}$ were constructed in the proof of Theorem \ref{first-approximation-construction-theorem}, and $\Tilde{\psi}_{*R}$ is related to $\Tilde{\phi}_{*R}$ via \eqref{psi-star-r-green-function-representation-upper-half-plane}.
	
	\medskip
	We next define, for $(-1)^{R}=1$ and $(-1)^{L}=-1$,
	\begin{align} 
		\mathscr{E}_{\lambda}^{out}\left(\tilde{\psi}_{*}, \lambda\alpha,\psi^{out}_{*},\lambda\Tilde{\xi}\right)&=\Delta\psi_{*}^{out}+\lambda\sum_{j=R,L}(-1)^{j}\left(\psi_{*j}\Delta\eta^{(j)}_{K}+2\nabla\eta^{(j)}_{K}\cdot\nabla\psi_{*j}\right)\nonumber\\
		&+\lambda E_{2}^{out}\left(\lambda\tilde{\xi}\right),\label{homotopic-operator-outer-error-main-definition}
	\end{align}
	with $\Tilde{\psi}_{*}$ being defined in \eqref{tilde-psi-star-definition}, and $E_{2}^{out}$ being given in \eqref{first-approximation-outer-error-E2out-final-bound}.
	
	\medskip
	We note that at $\lambda=1$, we recover \eqref{homotopic-operator-inner-error-E-star-R-definition-5}:
	\begin{align}
		\mathscr{E}_{R,1}\left(\tilde{\phi}_{*R}, \lambda\alpha,\psi^{out}_{*},\lambda\Tilde{\xi}\right)&=E_{*R}\left(\phi_{R},\phi_{*R},\psi_{R},\psi_{*R},\psi^{out}, \psi_{*}^{out};\xi\right),\label{homotopic-operator-inner-error-lambda-1-equivalence}\\
		\mathscr{E}_{1}^{out}\left(\tilde{\psi}_{*}, \lambda\alpha,\psi^{out}_{*},\lambda\Tilde{\xi}\right)&=E_{*}^{out}\left(\phi^{in},\phi^{in}_{*},\psi^{in},\psi^{in}_{*},\psi^{out}, \psi_{*}^{out};\xi\right).\label{homotopic-operator-outer-error-lambda-1-equivalence}
	\end{align}
	Thus a solution to the system \eqref{2d-euler-vorticity-stream} is constructed if we find  $\left(\tilde{\phi}_{*R}, \alpha,\psi^{out}_{*},\Tilde{\xi}\right)$ that make the quantities stated in \eqref{homotopic-operator-inner-error-lambda-1-equivalence}--\eqref{homotopic-operator-outer-error-lambda-1-equivalence} equal to $0$. We will do this by a continuation argument that involves finding a priori estimates along the deformation parameter $\lambda$ for the equations one obtains by setting \eqref{homotopic-operator-inner-error-main-definition} and \eqref{homotopic-operator-outer-error-main-definition} equal to $0$, in addition imposing suitable bounds for the parameter functions for $t\in\left[T_{0},T\right]$.
	
	\medskip
	Then, given $\left(\lambda\alpha,\psi^{out}_{*},\lambda\Tilde{\xi}\right)$, we require that $\tilde{\phi}_{*R}$ satisfies the terminal value problem
	\begin{subequations}\label{tilde-phi-star-R-initial-value-problem}
		\begin{align}
			&\mathscr{E}_{R,\lambda}=c_{R0}\left(t\right)U_{R}\left(z_{\lambda},t\right)+c_{R1}\left(t\right)\left(f_{\varepsilon}'\right)^{+}\mathscr{Z}_{1}\left(q_{\lambda}\right)\nonumber\\
			&+c_{R2}\left(t\right)\left(f_{\varepsilon}'\right)^{+}\mathscr{Z}_{2}\left(q_{\lambda}\right)\ \ \text{in}\ \mathbb{H}_{-q'_{\lambda}}\times \left[T_{0},T\right],\label{tilde-phi-star-R-initial-value-problem-equation}\\
			&\tilde{\phi}_{*R}(\blank,T)=0\ \ \text{in}\ \mathbb{H}_{-q'_{\lambda}},\quad \tilde{\phi}_{*R}\left(z_{\lambda},t\right)=0\ \ \text{on}\ \dell\mathbb{H}_{-q_{\lambda}'}\times\left[T_{0},T\right],\label{tilde-phi-star-R-initial-value-problem-initial-data}
		\end{align}
	\end{subequations}
	where $\mathbb{H}_{-q'_{\lambda}}$ is defined in \eqref{upper-half-plane-translate-definition}, and we once again recall that $T_{0}$ is defined in \eqref{K-T0-choices} and $T$ is arbitrarily large. Upon integrating against $1$, $z_{\lambda,1}$, and $z_{\lambda,2}$ in turn, for $\tilde{\phi}_{*R}$ satisfying \eqref{tilde-phi-star-R-initial-value-problem}, we can derive explicit formulae for the $c_{Rj}$ that depend linearly on $\tilde{\phi}_{*R}$ using integration by parts, see Lemma \ref{projected-linear-transport-problem-coefficients-of-projection-a-priori-estimates-lemma}.
	
	\medskip
	Given a solution to \eqref{tilde-phi-star-R-initial-value-problem}, $\mathscr{E}_{R,\lambda}$ will be annihilated once we impose, for $j=0,1,2$, the terminal value problem
	\begin{align}
		c_{Rj}\left(\tilde{\phi}_{*R}, \lambda\alpha,\psi^{out}_{*},\lambda\Tilde{\xi},\lambda\right)=0\ \ \textrm{for all}\ t\in\left[T_{0},T\right],\label{cRj-initial-value-problem}
	\end{align}
	with $\Tilde{\xi}$  and $\alpha$ satisfying
	\begin{align}
		\|t^{2}\tilde{p}\|_{\left[T_{0},T\right]}+\|t^{3}\dot{\tilde{p}}\|_{\left[T_{0},T\right]}+\|t^{3}\tilde{q}\|_{\left[T_{0},T\right]}+\|t^{4}\dot{\tilde{q}}\|_{\left[T_{0},T\right]}&\leq \varepsilon^{4-\sigma},\nonumber\\
		\tilde{p}\left(T\right)=0,\ \ \tilde{q}\left(T\right)&=0,\nonumber\\
		\|t^{3}\alpha\|_{\left[T_{0},T\right]}+\|t^{4}\dot{\alpha}\|_{\left[T_{0},T\right]}&\leq \varepsilon^{3-\sigma},\nonumber\\
		\alpha\left(T\right)&=0,\label{p-tilde-bound-2}
	\end{align}
	for some small $\sigma>0$. We require that $\Tilde{\phi}_{*L}$ and $c_{Lj}$ solve analogous terminal value problems to \eqref{tilde-phi-star-R-initial-value-problem} and \eqref{cRj-initial-value-problem}.
	We also impose that $\psi^{out}_{*}$, in $x$ variables, solves the boundary value problem
	\begin{align}
		\mathscr{E}_{\lambda}^{out}\left(\tilde{\psi}_{*}, \lambda\alpha,\psi^{out}_{*},\lambda\Tilde{\xi}\right)&=0\ \ \text{in}\ \mathbb{R}^{2}_{+}\times \left[T_{0},T\right],\quad \psi^{out}_{*}\rvert_{\dell\mathbb{R}^{2}_{+}}= 0\ \forall t\in \left[T_{0},T\right]. \label{psi-out-star-boundary-value-problem}
	\end{align}
	To construct a solution to \eqref{2d-euler-vorticity-stream} via the above scheme, we consider the vector of parameter functions $\mathfrak{P}=\left(\tilde{\phi}_{*}, \lambda\alpha,\psi^{out}_{*},\lambda\Tilde{\xi}\right)$,  belonging to a Banach space $\left(\mathfrak{X},\|\cdot\|_{\mathfrak{X}}\right)$ and reformulate the equations \eqref{psi-star-r-green-function-representation-upper-half-plane}, \eqref{tilde-phi-star-R-initial-value-problem}--\eqref{psi-out-star-boundary-value-problem} as a fixed point problem of the form
	\begin{align}
		\mathcal{T}\left(\mathfrak{P},\lambda\right)=\mathfrak{P},\ \ \mathfrak{P}\in\mathscr{O},\label{full-solution-construction-fixed-point-formulation}
	\end{align}
	where $\mathscr{O}$ is a bounded open set in $\mathfrak{X}$, and $\mathcal{T}\left(\cdot,\lambda\right)$ is a homotopy of nonlinear compact operators with $\mathcal{T}\left(\cdot,0\right)$ being linear.
	
	\medskip
	For a suitable choice of $\mathfrak{X}$ and $\mathscr{O}$, we will show that for all $\lambda\in[0,1]$, no solution $\mathcal{Z}\in\dell\mathscr{O}$ exists. Existence of a solution to \eqref{full-solution-construction-fixed-point-formulation} at $\lambda=1$ will then follow by standard degree theory, which, as stated, will give a solution to \eqref{2d-euler-vorticity-stream}. The choices of $\mathfrak{X}$ and $\mathscr{O}$ will also give the desired properties we claimed the solution would possess in Theorem \ref{teo1}.
	
	\medskip
	We now move on to a priori estimates for the system  \eqref{psi-star-r-green-function-representation-upper-half-plane}, \eqref{tilde-phi-star-R-initial-value-problem}--\eqref{psi-out-star-boundary-value-problem}. 
	
	\subsection{Weighted $L^{2}$ Function Space}\label{weight-section}
	For all $\varepsilon>0$ small enough, the definitions and estimates in this subsection will hold for any $\lambda\in\left[0,1\right]$, and any $\tilde{\xi}$ satisfying \eqref{p-tilde-bound-2} in the coordinate transformation $x\to z_{\lambda}$ given in \eqref{homotopy-parameter-change-of-coordinates}. Thus we suppress dependence on $\lambda$ and write the definitions and estimates for a coordinate change $x\to z\in\mathbb{H}_{-q'}$, for and $q$ satisfying \eqref{point-vortex-trajectory-function-decomposition-new} and \eqref{p-tilde-bound-2}.
	
	\medskip
	When proving a priori estimates for some linearised operators, it will be convenient to have shorthand for several quantities we define in this section. Firstly, recall the definition of $\left(f_{\varepsilon}'\right)^{+}$ as a function of $y\in\mathbb{R}^{2}_{+}$ from \eqref{f-prime-plus-definition}, based on \eqref{nonlinearity-definition}, \eqref{f-prime-q-variation-1}--\eqref{tilde-W-k-def} and \eqref{f-prime-definition}. We think of $\left(f_{\varepsilon}'\right)^{+}$ as a function of $z=y-q'e_{2}$ via the definition \eqref{tilde-W-k-def}. Then let $\eta_{0}(r)$ once again be the smooth cutoff that is identically $1$ for $r\leq1$ and identically $0$ for $r\geq2$. Define the interior cutoff, $\eta_{\mathfrak{i}}$, and its support by,
	\begin{align}
		\eta_{\mathfrak{i}}\left(z,t\right)=1-\eta_{0}\left(\left(\frac{t}{\varepsilon}\right)^{\frac{\gamma-1}{2}}\left(f_{\varepsilon}'\right)^{+}\right),\quad \mathfrak{W}\coloneqq\supp{\eta_{\mathfrak{i}}}=\left\{\left(f_{\varepsilon}'\right)^{+}\geq\left(\frac{\varepsilon}{t}\right)^{\frac{\gamma-1}{2}}\right\},\label{weighted-L2-interior-cutoff}
	\end{align}
	Note that $\left(f_{\varepsilon}'\right)^{+}$ is a continuous non-negative function that is positive and size $\bigO{\left(1\right)}$ at $z=0$, and is $0$ outside the ball $B_{\rho}\left(0\right)$. Thus $\mathfrak{W}$ can be thought of as an ``interior" region of $\supp{\left(f_{\varepsilon}'\right)^{+}}$ where $\left(f_{\varepsilon}'\right)^{+}$ is strictly positive, with lower bounds chosen to prove ``interior" a-priori bounds in Lemma \ref{L2-interior-a-priori-estimate-lemma}. Consequently, $\supp{\eta_{\mathfrak{e}}}$ for $\eta_{\mathfrak{e}}$ defined below in \eqref{weighted-L2-exterior-cutoff} can correspondingly be thought of as the ``exterior" region of $\supp{\left(f_{\varepsilon}'\right)^{+}}$ where we can instead precisely control the size of $\left(f_{\varepsilon}'\right)^{+}$, and  we prove a corresponding ``exterior" a-priori estimate in Lemma \ref{L2-exterior-a-priori-estimate-lemma}.
	
	\medskip
	Thus, let $\eta_{1}(r)$ to be the smooth cutoff function that is identically $1$ for $r\leq3$ and identically $0$ for $r\geq4$. We also recall from \eqref{dq-Psi-R-variation}--\eqref{f-prime-q-variation-3} that $\left(f_{\varepsilon}'\right)^{+}$ is supported on $B_{\rho}\left(0\right)$, $\rho>0$ defined in \eqref{first-approximation-main-order-vorticity-support}, in $z$ coordinates. We let the exterior cutoff, $\eta_{\mathfrak{e}}$, be defined by
	\begin{align}
		\eta_{\mathfrak{e}}\left(z,t\right)=\twopartdef{\eta_{1}\left(\left(\frac{t}{\varepsilon}\right)^{\frac{\gamma-1}{2}}\left(f_{\varepsilon}'\right)^{+}\right)}{z\in\supp\left(f_{\varepsilon}'\right)^{+}}{\eta_{0}\left(\frac{\left|z\right|}{4\rho}\right)}{z\in \mathbb{H}_{-q'}\setminus \supp\left(f_{\varepsilon}'\right)^{+}}.\label{weighted-L2-exterior-cutoff}
	\end{align}
	We note the two functions that are glued to construct $\eta_{\mathfrak{e}}$ match at the boundary of their supports of definition as they are both identically $1$ on a region around said support. We note that $\eta_{\mathfrak{e}}$ is supported on the region
	\begin{align}
		\supp{\eta_{\mathfrak{e}}}=\left\{\left(f_{\varepsilon}'\right)^{+}\leq 4\left(\frac{\varepsilon}{t}\right)^{\frac{\gamma-1}{2}}\right\}.\label{weighted-L2-exterior-cutoff-support}
	\end{align}
	From \eqref{weighted-L2-interior-cutoff} and \eqref{weighted-L2-exterior-cutoff}, we can see that
	\begin{align}
		&\eta_{\mathfrak{i}}\equiv1,\quad \left(f_{\varepsilon}'\right)^{+}\geq 2\left(\frac{\varepsilon}{t}\right)^{\frac{\gamma-1}{2}}\nonumber\\
		&\eta_{\mathfrak{e}}\equiv1,\quad \left(f_{\varepsilon}'\right)^{+}\leq 3\left(\frac{\varepsilon}{t}\right)^{\frac{\gamma-1}{2}},\ z\in \supp{\left(f_{\varepsilon}'\right)^{+}}.\label{cutoffs-inequalities}
	\end{align}
	Thus on $\supp{\left(f_{\varepsilon}'\right)^{+}}$, since both $\eta_{\mathfrak{i}}$ and $\eta_{\mathfrak{e}}$ are bounded between $0$ and $1$ and at least one of them are identically $1$ on the two regions considered above whose union clearly covers $\supp{\left(f_{\varepsilon}'\right)^{+}}$, we have
	\begin{align*}
		1\leq\eta_{\mathfrak{i}}+\eta_{\mathfrak{e}}\leq2\ \ 
		\mathrm{on}\ \supp{\left(f_{\varepsilon}'\right)^{+}}.
	\end{align*}
	Then from \eqref{weighted-L2-exterior-cutoff}, we can also see that $\eta_{\mathfrak{e}}\equiv1$ on
	\begin{align*}
		\left(\left\{\left(f_{\varepsilon}'\right)^{+}\leq 3\left(\frac{\varepsilon}{t}\right)^{\frac{\gamma-1}{2}}\right\}\cap B_{2\rho}\left(0\right)\right)\cup \left(B_{4\rho}\left(0\right)\setminus B_{2\rho}\left(0\right)\right),
	\end{align*}
	where we recall that $\supp{\left(f_{\varepsilon}'\right)^{+}}\subset B_{2\rho}\left(0\right)$. This property of $\eta_{\mathfrak{e}}$, along with \eqref{cutoffs-inequalities} gives
	\begin{align*}
		1\leq\eta_{\mathfrak{i}}+\eta_{\mathfrak{e}}&\leq2\ \ 
		\mathrm{on}\ B_{4\rho}\left(0\right).
	\end{align*}
	Next, taking the gradient of $\eta_{\mathfrak{i}}$ gives
	\begin{align}
		\nabla_{z}\eta_{\mathfrak{i}}=-\left(\frac{t}{\varepsilon}\right)^{\frac{\gamma-1}{2}}\eta_{0}'\left(\left(\frac{t}{\varepsilon}\right)^{\frac{\gamma-1}{2}}\left(f_{\varepsilon}'\right)^{+}\right)\nabla_{z}\left(\left(f_{\varepsilon}'\right)^{+}\right).\label{grad-eta-i-0}
	\end{align}
	Recalling Theorem \ref{vortex-pair-properties-theorem}, and \eqref{f-prime-q-variation-1}--\eqref{tilde-W-k-def}, we have
	\begin{align*}
		\nabla_{z}\left(\left(f_{\varepsilon}'\right)^{+}\right)=\left(f_{\varepsilon}''\right)^{+}\nabla_{z}\left(\Psi_{R}-c\varepsilon \left(z_{2}+q'e_{2}\right)-\left|\log{\varepsilon}\right|\Omega\right)=\left(f_{\varepsilon}''\right)^{+}\nabla_{z}\left(\Psi_{R}-c\varepsilon z_{2}\right),
	\end{align*}
	whence we note that the $L^{\infty}$ norm of the gradient on the right hand side multiplying $\left(f_{\varepsilon}''\right)^{+}$ on $\mathbb{H}_{-q'}$ is bounded due to \eqref{vortexpair-theorem} and \eqref{psi-r-definition}, keeping in mind the change of coordinates \eqref{homotopy-parameter-change-of-coordinates}. Thus
	\begin{align}
		\left|\nabla_{z}\eta_{\mathfrak{i}}\right|\leq C\left|\eta_{0}'\left(\left(\frac{t}{\varepsilon}\right)^{\frac{\gamma-1}{2}}\left(f_{\varepsilon}'\right)^{+}\right)\right|\left(\frac{t}{\varepsilon}\right)^{\frac{\gamma-1}{2}}\left(f_{\varepsilon}''\right)^{+}.\label{grad-eta-i-00}
	\end{align}
	From \eqref{grad-eta-i-00}, we only have to worry about bounding $\nabla_{z}\eta_{\mathfrak{i}}$ when $\eta_{0}'\left(\left(\frac{t}{\varepsilon}\right)^{\frac{\gamma-1}{2}}\left(f_{\varepsilon}'\right)^{+}\right)$ is non-zero. Then noting that $\eta_{0}'(r)$ is both bounded, and only non-zero when $r\in\left(1,2\right)$, we have that $\eta_{0}'\left(\left(\frac{t}{\varepsilon}\right)^{\frac{\gamma-1}{2}}\left(f_{\varepsilon}'\right)^{+}\right)$ is non-zero exactly when
	\begin{align*}
		1<\left(\frac{t}{\varepsilon}\right)^{\frac{\gamma-1}{2}}\left(f_{\varepsilon}'\right)^{+}<2\iff\left(\frac{\varepsilon}{t}\right)^{\frac{\gamma-1}{2}}<\left(f_{\varepsilon}'\right)^{+}<2\left(\frac{\varepsilon}{t}\right)^{\frac{\gamma-1}{2}}.
	\end{align*}
	Then we note that by \eqref{f-prime-q-variation-1}--\eqref{tilde-W-k-def},
	\begin{align*}
		\left(f_{\varepsilon}'\right)^{+}\sim \left(\frac{\varepsilon}{t}\right)^{\frac{\gamma-1}{2}} \iff
		\left(f_{\varepsilon}''\right)^{+}\sim \left(\frac{\varepsilon}{t}\right)^{\frac{\gamma-2}{2}},
	\end{align*}
	Thus
	\begin{align}
		\left|\nabla_{z}\eta_{\mathfrak{i}}\right|\leq C\left(\frac{t}{\varepsilon}\right)^{\frac{1}{2}}.\label{grad-eta-i-bound}
	\end{align}
	A similar proof shows the same bound for $\nabla_{z}\eta_{\mathfrak{e}}$. To summarize, for $\eta_{\mathfrak{i}}$ and $\eta_{\mathfrak{e}}$, we have
	\begin{align}
		1\leq\eta_{\mathfrak{i}}+\eta_{\mathfrak{e}}\leq2\ \ 
		\mathrm{on}\ B_{4\rho}\left(0\right),\quad \left|\grad_{z}\eta_{\mathfrak{i}}\right|+\left|\grad_{z}\eta_{\mathfrak{e}}\right|\leq C\left(\frac{t}{\varepsilon}\right)^{\frac{1}{2}},\label{interior-exterior-cutoff-sum-bounds}
	\end{align}
	some absolute constant $C>0$. Moreover, since, for any space or time derivative $\dell$, $\dell\eta_{\mathfrak{i}}$ is supported on
	\begin{align*}
		\left\{\left(\frac{\varepsilon}{t}\right)^{\frac{\gamma-1}{2}}\leq\left(f_{\varepsilon}'\right)^{+}\leq 2\left(\frac{\varepsilon}{t}\right)^{\frac{\gamma-1}{2}}\right\},
	\end{align*}
	and since $\eta_{\mathfrak{e}}\equiv1$ on this region by \eqref{cutoffs-inequalities}, we have that
	\begin{align}
		\supp{\dell_{t}\eta_{\mathfrak{i}}}\cup\supp{\nabla_{z}\eta_{\mathfrak{i}}}\subset\left\{\eta_{\mathfrak{e}}=1\right\}\subset\supp{\eta_{\mathfrak{e}}}.\label{cutoff-derivative-support-inclusion}
	\end{align}
	Thus we have 
	\begin{align}
		\eta_{\mathfrak{i}}\leq\eta_{\mathfrak{e}},\quad \left|\dell_{t}\eta_{\mathfrak{i}}\right|\leq \frac{C\eta_{\mathfrak{e}}}{t},\quad \left|\nabla_{z}\eta_{\mathfrak{i}}\right|\leq C\eta_{\mathfrak{e}}\left(\frac{t}{\varepsilon}\right)^{\frac{1}{2}},\quad z\in\supp{\dell_{t}\eta_{\mathfrak{i}}}\cup\supp{\nabla_{z}\eta_{\mathfrak{i}}}.\label{derivative-of-inner-cutoff-outer-cutoff-comparison-bound-1}
	\end{align} 
	The first of these bounds is by \eqref{cutoff-derivative-support-inclusion}, and the fact that $0\leq\eta_{\mathfrak{i}}\leq1$. The second can be proved similarly to \eqref{grad-eta-i-0}--\eqref{grad-eta-i-bound} along with the use of \eqref{xi0-bound}--\eqref{xi1-bound} and Lemma \ref{vortex-pair-q-variation-lemma}. The third bound then combines \eqref{interior-exterior-cutoff-sum-bounds} and \eqref{cutoff-derivative-support-inclusion}.
	
	\medskip
	With the definitions \eqref{weighted-L2-interior-cutoff} and \eqref{weighted-L2-exterior-cutoff} in hand we define, for $1\leq r<\infty$, the space $\mathscr{Y}_{r}$ as all functions on $\mathbb{H}_{-q'}$ with
	\begin{align}
		\|\Phi\|_{\mathscr{Y}_{r}}\coloneqq \|\left(\left(f_{\varepsilon}'\right)^{+}\right)^{\frac{1}{r}-1}\eta_{\mathfrak{i}}\Phi\|_{L^{r}\left(\mathbb{H}_{-q'}\right)}+\left(\frac{t}{\varepsilon}\right)^{\left(\frac{\gamma-1}{2}\right)\left(1-\frac{1}{r}\right)}\|\eta_{\mathfrak{e}}\Phi\|_{L^{r}\left(\mathbb{H}_{-q'}\right)}<\infty.\label{weighted-Lr-norm-definition}
	\end{align}
	We have that if $\supp\Phi\subset B_{4\rho}\left(0\right)$, then by \eqref{cutoffs-inequalities} and \eqref{interior-exterior-cutoff-sum-bounds} we have
	\begin{align}
		&\|\Phi\|_{L^{r}\left(\mathbb{H}_{-q'}\right)}\leq \|\left(\eta_{\mathfrak{i}}+\eta_{\mathfrak{e}}\right)\Phi\|_{L^{r}\left(\mathbb{H}_{-q'}\right)}\leq \|\eta_{\mathfrak{i}}\Phi\|_{L^{r}\left(\mathbb{H}_{-q'}\right)}+\|\eta_{\mathfrak{e}}\Phi\|_{L^{r}\left(\mathbb{H}_{-q'}\right)}\nonumber\\
		&\leq C\|\tilde{\mathscr{W}}_{1}\left(q_{0}+q_{1}\right)^{\frac{1}{r}-1}\eta_{\mathfrak{i}}\Phi\|_{L^{r}\left(\mathbb{H}_{-q'}\right)}+\left(\frac{\varepsilon}{t}\right)^{\left(\frac{\gamma-1}{2}\right)\left(1-\frac{1}{r}\right)}\left(\frac{t}{\varepsilon}\right)^{\left(\frac{\gamma-1}{2}\right)\left(1-\frac{1}{r}\right)}\|\eta_{\mathfrak{e}}\Phi\|_{L^{r}\left(\mathbb{H}_{-q'}\right)}\nonumber\\
		&\leq C\|\Phi\|_{\mathscr{Y}_{r}},\label{weighted-Lr-norm-definition-1}
	\end{align}
	For $r=\infty$, we instead define, for $j=1,2,3$, the norm
	\begin{align}
		\|\Phi\|_{\mathscr{Y}_{\infty,j}}\coloneqq \|\left(\left(f_{\varepsilon}^{\left(j\right)}\right)^{+}\right)^{-1}\eta_{\mathfrak{i}}\Phi\|_{L^{\infty}\left(\mathbb{H}_{-q'}\right)}+\left(\frac{t}{\varepsilon}\right)^{\left(\frac{\gamma-j}{2}\right)}\|\eta_{\mathfrak{e}}\Phi\|_{L^{\infty}\left(\mathbb{H}_{-q'}\right)}<\infty,\label{weighted-L-infinity-norm-definition}
	\end{align}
	where we once again recall the definitions for $\left(f_{\varepsilon}^{\left(j\right)}\right)^{+}$ from \eqref{f-prime-q-variation-1}--\eqref{tilde-W-k-def}. As in \eqref{weighted-Lr-norm-definition-1}, we have for $\Phi$ such that $\supp\Phi\subset B_{4\rho}\left(0\right)$,
	\begin{align}
		\|\Phi\|_{L^{\infty}\left(\mathbb{H}_{-q'}\right)}&\leq C\|\left(\left(f_{\varepsilon}^{\left(j\right)}\right)^{+}\right)^{-1}\eta_{\mathfrak{i}}\Phi\|_{L^{\infty}\left(\mathbb{H}_{-q'}\right)}+\left(\frac{\varepsilon}{t}\right)^{\left(\frac{\gamma-j}{2}\right)}\left(\frac{t}{\varepsilon}\right)^{\left(\frac{\gamma-j}{2}\right)}\|\eta_{\mathfrak{e}}\Phi\|_{L^{\infty}\left(\mathbb{H}_{-q'}\right)}\nonumber\\
		&\leq C\|\Phi\|_{\mathscr{Y}_{\infty,j}}.\label{weighted-L-infinity-norm-bounds-L-infinity-norm}
	\end{align}
	On the support of $\eta_{i}$, we know that
	\begin{align*}
		\mathscr{V}\left(z,q\right)\geq \frac{1}{\gamma}\left(\frac{\varepsilon}{t}\right)^{\frac{1}{2}}.
	\end{align*}
	Then, on $\supp{\eta_{\mathfrak{i}}}$, for any $\beta>\frac{1}{2}$,
	\begin{align}
		\left(1+C\left(\frac{\varepsilon}{t}\right)^{\beta-\frac{1}{2}}\right)\mathscr{V}\left(q\right)\geq\mathscr{V}\left(q\right)+\bigO{\left(\left(\frac{\varepsilon}{t}\right)^{\beta}\right)}\geq\left(1-C\left(\frac{\varepsilon}{t}\right)^{\beta-\frac{1}{2}}\right)\mathscr{V}\left(q\right).\label{f-prime-L2-inner-weight-comparison-bound}
	\end{align}
	Similarly, using \eqref{weighted-L2-exterior-cutoff-support}, we have that on $\supp{\eta_{\mathfrak{e}}}$, for any $\beta>\frac{1}{2}$
	\begin{align}
		\mathscr{V}\left(z,q\right)+\bigO{\left(\left(\frac{\varepsilon}{t}\right)^{\beta}\right)}\leq C\left(\frac{\varepsilon}{t}\right)^{\frac{1}{2}}.\label{f-prime-L2-outer-weight-comparison-bound-1}
	\end{align}
	Now we record some inequalities related to the Poisson equation $-\Delta\Tilde{\psi}_{*R}=\Tilde{\phi}_{*R}$ that we will need later on.
	\subsection{The Poisson Equation}\label{poisson-equation-section}
	In this section, \ref{poisson-equation-section}, we work in $y$ coordinates defined in \eqref{change-of-coordinates-dynamic-problem}. Since $\Tilde{\phi}_{*R}$  and $\Tilde{\psi}_{*R}$ are related by \eqref{psi-star-r-green-function-representation-upper-half-plane}, we would like to prove effective bounds on $\bar{\psi}$ on $\mathbb{R}^{2}_{+}$ satisfying $\psi\left(y_{1},0\right)=0$ and related to $\bar{\phi}$ by $-\Delta\bar{\psi}=\bar{\phi}$, so that
	\begin{align}
		\bar{\psi}\left(y\right)=\frac{1}{2\pi}\int_{\mathbb{R}_{+}^{2}}\log{\left(\frac{\left|v-\bar{y}\right|}{\left|v-y\right|}\right)}\bar{\phi}\left(v\right) dv
		\label{bar-psi-green-function-representation-upper-half-plane}
	\end{align}
	For some fixed $\bar{\mathcal{J}}_{\left(0,0\right)}(t)=\bigO{\left(\varepsilon^{3}t^{-2}\right)}$ on $\left[T_{0},T\right]$, we make the assumptions
	\begin{align}
		\supp\bar{\phi}\subset B_{3\rho}\left(q'e_{2}\right),\ \ \ \|\bar{\phi}\|_{L^{2}\left(\mathbb{R}^{2}_{+}\right)}<\infty,\ \ \ \int_{\mathbb{R}^{2}_{+}}\bar{\phi}\left(v\right) dv=\bar{\mathcal{J}}_{\left(0,0\right)},\label{bar-phi-support-weighted-L2-space-and-mass-condition}
	\end{align}
	Note that \eqref{bar-phi-support-weighted-L2-space-and-mass-condition} implies for $1\leq r<\infty$ and $j=1,2,3$,
	\begin{align}\nonumber   \|\bar{\phi}\|_{L^{r}\left(\mathbb{R}^{2}_{+}\right)}\leq C\|\bar{\phi}\|_{\mathscr{Y}_{r}},\quad \|\bar{\phi}\|_{L^{\infty}\left(\mathbb{R}^{2}_{+}\right)}\leq C\|\bar{\phi}\|_{\mathscr{Y}_{\infty,j}}.
	\end{align}
	To help state the bounds we will need, we recall the definition of the H\"older seminorm of order $\beta$ for some $\beta\in\left(0,1\right)$:
	\begin{align}
		\nonumber
		\left[F\right]_{\beta}\left(y\right)\coloneqq\sup_{y_{1},y_{2}\in B_{1}\left(y\right)}\frac{\left|F\left(y_{1}\right)-F\left(y_{2}\right)\right|}{\left|y_{1}-y_{2}\right|^{\beta}}.
	\end{align}
	Then the following result holds.
	\begin{lemma}\label{poisson-equation-zero-mass-estimates-lemma}
		Let $\bar{\phi}$ satisfy \eqref{bar-phi-support-weighted-L2-space-and-mass-condition}, and let $\bar{\psi}$ depend on $\bar{\phi}$ by \eqref{bar-psi-green-function-representation-upper-half-plane}. Suppose further that $\bar{\phi}\in L^{r_{0}}\left(\mathbb{R}^{2}_{+}\right)$, some $r_{0}>2$. Then there is an absolute constant $C>0$ such that for all $y\in\mathbb{R}^{2}_{+}$ and $\infty> r_{1}>2$, and some $\beta\in\left(0,1\right)$ we have
		\begin{align}
			\left|\bar{\psi}\left(y\right)\right|&\leq C\left(\|\bar{\phi}\|_{L^{2}\left(\mathbb{R}^{2}_{+}\right)}+\left|\log{\varepsilon}\right|\left|\bar{\mathcal{J}}_{\left(0,0\right)}\right|\right),\label{poisson-equation-zero-mass-estimates-lemma-statement-1}
			\\    \|\nabla\bar{\psi}\left(y\right)\|_{L^{r_{1}}\left(\mathbb{R}^{2}_{+}\right)}&\leq C\left(\|\bar{\phi}\|_{L^{2}\left(\mathbb{R}^{2}_{+}\right)}+\left|\log{\varepsilon}\right|\left|\bar{\mathcal{J}}_{\left(0,0\right)}\right|\right),\label{poisson-equation-zero-mass-estimates-lemma-statement-2}
			\\
			\left|\nabla\bar{\psi}\left(y\right)\right|+\left[\nabla\bar{\psi}\right]_{\beta}\left(y\right)&\leq C\left(\|\bar{\phi}\|_{L^{r_{0}}\left(\mathbb{R}^{2}_{+}\right)}+\left|\log{\varepsilon}\right|\left|\bar{\mathcal{J}}_{\left(0,0\right)}\right|\right).
			\label{poisson-equation-zero-mass-estimates-lemma-statement-3}
		\end{align}
	\end{lemma}
	\begin{proof}
		To see this, first note that if $\left|y-q'e_{2}\right|<6\rho$, since $q'e_{2}=\bigO{\left(\varepsilon^{-1}\right)}$, we see that
		\begin{align}
			&\bar{\psi}\left(y\right)=\frac{1}{2\pi}\int_{B_{3\rho\left(q'e_{2}\right)}}\log{\left(\left|\frac{\left(v-q'e_{2}\right)-\left(\overline{y-q'e_{2}}\right)}{2q'}+e_{2}\right|\right)}\bar{\phi}(v) dv\nonumber\\
			&-\frac{1}{2\pi}\int_{B_{3\rho}\left(q'e_{2}\right)}\log{\left(\left|\left(v-q'e_{2}\right)-\left(y-q'e_{2}\right)\right|\right)}\bar{\phi}(v) dv\nonumber\\
			&+\frac{\log{2q'}}{2\pi}\int_{B_{3\rho}\left(q'e_{2}\right)}\bar{\phi}(v)dv.\label{hat-psi-identity-1}
		\end{align}
		For $\left|y-q'e_{2}\right|>6\rho$, note that the second component of $-\left(\overline{y-q'e_{2}}\right)+2q'e_{2}$ is $y_{2}+q'\geq q'=\bigO{\left(\varepsilon^{-1}\right)}$, so we instead have
		\begin{align}
			&\hat{\psi}\left(y\right)=\frac{1}{2\pi}\int_{B_{3\rho}\left(q'e_{2}\right)}\log{\left(\frac{\left|\left(v-q'e_{2}\right)-\left(\overline{y-q'e_{2}}\right)+2q'e_{2}\right|}{\left|-\left(\overline{y-q'e_{2}}\right)+2q'e_{2}\right|}\right)}\bar{\phi}(v) dv\label{hat-psi-identity-2}\\
			&-\frac{1}{2\pi}\int_{B_{3\rho}\left(q'e_{2}\right)}\log{\left(\frac{\left|\left(v-q'e_{2}\right)-\left(y-q'e_{2}\right)\right|}{\left|y-q'e_{2}\right|}\right)}\bar{\phi}(v) dv\nonumber\\
			&+\frac{1}{2\pi}\log{\left(\frac{\left|-\left(\overline{y-q'e_{2}}\right)+2q'e_{2}\right|}{\left|y-q'e_{2}\right|}\right)}\int_{B_{3\rho}\left(q'e_{2}\right)}\bar{\phi}(v)dv.\nonumber
		\end{align}
		Using the identity \eqref{hat-psi-identity-1} on the region $\left|y-q'e_{2}\right|<6\rho$ and \eqref{hat-psi-identity-2} on the region $\left|y-q'e_{2}\right|>6\rho$, it is straightforward to obtain the estimates.
	\end{proof}
	
	If estimates \eqref{poisson-equation-zero-mass-estimates-lemma-statement-1}, \eqref{poisson-equation-zero-mass-estimates-lemma-statement-2}, and \eqref{poisson-equation-zero-mass-estimates-lemma-statement-3} hold, and moreover $\bar{\phi}\in L^{\infty}\left(\mathbb{R}^{2}_{+}\right)$, we have the following interpolation estimate.
	\begin{lemma}Suppose $\bar{\phi}$ satisfies \eqref{bar-phi-support-weighted-L2-space-and-mass-condition}, and in addition, $\bar{\phi}\in L^{\infty}\left(\mathbb{R}^{2}_{+}\right)$. Then given $\sigma\in\left(0,1\right)$, there exist numbers $C_{\sigma}>0$ and $\beta\in\left(0,1\right)$ such that
		\begin{align}
			\left|\bar{\psi}\left(y\right)\right|+\left|\nabla\bar{\psi}\left(y\right)\right|+\left[\nabla\bar{\psi}\right]_{\beta}\left(y\right)\leq C_{\sigma}\left(\|\bar{\phi}\|_{L^{\infty}\left(\mathbb{R}^{2}_{+}\right)}^{\sigma}\|\bar{\phi}\|_{L^{2}\left(\mathbb{R}^{2}_{+}\right)}^{1-\sigma}+\left|\log{\varepsilon}\right|\left|\bar{\mathcal{J}}_{\left(0,0\right)}\right|\right).\label{poisson-equation-zero-mass-interpolation-estimate-lemma-statement}
		\end{align}
	\end{lemma}
	\subsection{Lower Bound on a Quadratic Form}
	We again consider functions $\bar{\phi}$ as in Section \ref{poisson-equation-section} satisfying \eqref{bar-psi-green-function-representation-upper-half-plane}--\eqref{bar-phi-support-weighted-L2-space-and-mass-condition}. In addition for fixed $\bar{\mathcal{J}}_{\left(i,j\right)}=\bigO{\left(\varepsilon^{3}t^{-2}\right)}$ on $\left[T_{0},T\right]$, assume
	\begin{align}
		\int_{\mathbb{R}^{2}_{+}}y_{1}\bar{\phi}\left(y\right) dy=\bar{\mathcal{J}}_{\left(1,0\right)},\ \ \ \int_{\mathbb{R}^{2}_{+}}\left(y_{2}-q'\right)\bar{\phi}\left(y\right) dy=\bar{\mathcal{J}}_{\left(0,1\right)}.\label{bar-phi-centre-of-mass-conditions}
	\end{align}
	If $\bar{\phi}$ satisfies the above conditions, we gain the following crucial estimate. Recall the definitions of $\eta_{\mathfrak{i}}$, $\eta_{\mathfrak{e}}$, and $\mathfrak{W}$ from \eqref{weighted-L2-interior-cutoff}--\eqref{weighted-L2-exterior-cutoff}. 
	\begin{lemma}\label{quadratic-form-estimate-lemma}
		Suppose $\bar{\phi}$ satisfies \eqref{bar-phi-support-weighted-L2-space-and-mass-condition}, \eqref{bar-phi-centre-of-mass-conditions}, and $\xi=\left(p,q\right)$ satisfies \eqref{xi0-bound}--\eqref{tildexi-bound}. Then for all $t\in\left[T_{0},T\right]$ there exists an absolute constant $C_{0}>0$ such that
		\begin{align}
			\int_{\mathbb{R}^{2}_{+}}\frac{\eta_{\mathfrak{i}}\bar{\phi}}{\left(f_{\varepsilon}'\right)^{+}}\left(\eta_{\mathfrak{i}}\bar{\phi}-\left(f_{\varepsilon}'\right)^{+}\left(-\Delta\right)^{-1}\left(\eta_{\mathfrak{i}}\bar{\phi}\right)\right)\geq C_{0}\int_{\mathbb{R}^{2}_{+}}\frac{\eta_{\mathfrak{i}}^{2}\left|\bar{\phi}\right|^{2}}{\left(f_{\varepsilon}'\right)^{+}}+\mathcal{R}_{1},\label{quadratic-form-estimate-lemma-statement-1}
		\end{align}
		with
		\begin{align}
			\mathcal{R}_{1}=\bigO\left(\left|\log{\varepsilon}\right|^{2}\left(\|\eta_{\mathfrak{e}}\bar{\phi}\|_{L^{2}\left(\mathbb{R}^{2}_{+}\right)}^{2}+\bar{\mathcal{J}}_{\left(0,0\right)}^{2}+\bar{\mathcal{J}}_{\left(1,0\right)}^{2}+\bar{\mathcal{J}}_{\left(0,1\right)}^{2}\right)\right).\label{quadratic-form-estimate-lemma-statement-2}
		\end{align}
	\end{lemma}
	\begin{proof}
		By \eqref{weighted-L2-interior-cutoff}, $\eta_{\mathfrak{i}}/\left(f_{\varepsilon}'\right)^{+}$ is in $L^{\infty}$. Moreover, since $\left(f_{\varepsilon}'\right)^{+}\geq \left(\varepsilon t^{-1}\right)^{\frac{\gamma-1}{2}}$ on $\supp{\eta_{\mathfrak{i}}}$, we know that $\supp{\eta_{\mathfrak{i}}}\subset \supp{\left(f_{\varepsilon}'\right)^{+}}$. Using this fact as well as \eqref{bar-phi-support-weighted-L2-space-and-mass-condition}, we have
		\begin{align}
			\frac{\eta_{\mathfrak{i}}\bar{\phi}}{\left(f_{\varepsilon}'\right)^{+}}\in\mathcal{H},\quad \frac{\eta_{\mathfrak{i}}\bar{\phi}}{\left(f_{\varepsilon}'\right)^{+}}=\sum_{j=0}^{\infty}\bar{\phi}_{j}e_{j}.\label{quadratic-form-estimate-proof-3}
		\end{align}
		as a function on $\supp{\left(f_{\varepsilon}'\right)^{+}}$. Here $\mathcal{H}$ is defined in Section \ref{vortex-pair-spectral-theory-appendix} on $\mathbb{R}^{2}$ and we instead take the equivalent definition by restricting to $\mathbb{R}^{2}_{+}$ and functions $\supp{\left(f_{\varepsilon}'\right)^{+}\to\mathbb{R}}$. Thus, we have expanded in the eigenbasis $\left(e_{j}\right)_{j\geq0}$ related to the operator $\mathfrak{T}\left(h\right)$ also defined in Section \ref{vortex-pair-spectral-theory-appendix}, also restricted to $\mathbb{R}^{2}_{+}$ with $0$ boundary conditions at $y_{2}=0$ in our case. Then
		\begin{align*}
			\int_{\mathbb{R}^{2}_{+}}\frac{\eta_{\mathfrak{i}}\bar{\phi}}{\left(f_{\varepsilon}'\right)^{+}}\left(\eta_{\mathfrak{i}}\bar{\phi}-\left(f_{\varepsilon}'\right)^{+}\left(-\Delta\right)^{-1}\left(\eta_{\mathfrak{i}}\bar{\phi}\right)\right)=\left(\frac{\eta_{\mathfrak{i}}\bar{\phi}}{\left(f_{\varepsilon}'\right)^{+}},\frac{\eta_{\mathfrak{i}}\bar{\phi}}{\left(f_{\varepsilon}'\right)^{+}}\right)_{\mathcal{H}}-\left(\frac{\eta_{\mathfrak{i}}\bar{\phi}}{\left(f_{\varepsilon}'\right)^{+}},\mathfrak{T}\left(\frac{\eta_{\mathfrak{i}}\bar{\phi}}{\left(f_{\varepsilon}'\right)^{+}}\right)\right)_{\mathcal{H}},
		\end{align*}
		where $\left(\blank, \blank\right)_{\mathcal{H}}$ is the inner product associated to $\mathcal{H}$. Then in the eigenbasis $\left(e_{j}\right)$ with associated eigenvalues $\left(\mu_{j}\left(\varepsilon\right)\right)$, we have
		\begin{align}
			\int_{\mathbb{R}^{2}_{+}}\frac{\eta_{\mathfrak{i}}\bar{\phi}}{\left(f_{\varepsilon}'\right)^{+}}\left(\eta_{\mathfrak{i}}\bar{\phi}-\left(f_{\varepsilon}'\right)^{+}\left(-\Delta\right)^{-1}\left(\eta_{\mathfrak{i}}\bar{\phi}\right)\right)=\left(1-\mu_{0}\right)\left|\bar{\phi}_{0}\right|^{2}+\sum_{j=3}^{\infty}\left(1-\mu_{j}\right)\left|\bar{\phi}_{j}\right|^{2},\label{quadratic-form-estimate-proof-5}
		\end{align}
		where we have used the fact that $\mu_{1}=\mu_{2}=1$ for all $\varepsilon>0$. From Section \ref{vortex-pair-spectral-theory-appendix} we also have that $e_{0}$ has eigenvalue $\mu_{0}\left(\varepsilon\right)\sim\left|\log{\varepsilon}\right|$ as $\varepsilon\to0$. From Lemma \ref{eigenvalue-estimates-lemma} we have that
		\begin{align}
			\sum_{j=3}^{\infty}\left(1-\mu_{j}\right)\left|\bar{\phi}_{j}\right|^{2}\geq \left(1-\delta\right)\sum_{j=3}^{\infty}\left|\bar{\phi}_{j}\right|^{2},\label{quadratic-form-estimate-proof-6}
		\end{align}
		some absolute constant $\delta\in\left(0,1\right)$. For $j=0$, using the orthogonality of the eigenbasis, we have
		\begin{align}
			&\mu_{0}\bar{\phi}_{0}=\int_{\mathbb{R}^{2}_{+}}\left(f_{\varepsilon}'\right)^{+}\left(-\Delta\right)^{-1}\left(\left(f_{\varepsilon}'\right)^{+}\frac{\eta_{\mathfrak{i}}\bar{\phi}}{\left(f_{\varepsilon}'\right)^{+}}\right)e_{0}=\int_{\mathbb{R}^{2}_{+}}\left(f_{\varepsilon}'\right)^{+}\left(-\Delta\right)^{-1}\left(\eta_{\mathfrak{i}}\bar{\phi}\right)e_{0}\label{quadratic-form-estimate-proof-7}\\
			&=\int_{\mathbb{R}^{2}_{+}}\left(f_{\varepsilon}'\right)^{+}\left(-\Delta\right)^{-1}\left(\bar{\phi}\right)e_{0}+\int_{\mathbb{R}^{2}_{+}}\left(f_{\varepsilon}'\right)^{+}\left(\left(-\Delta\right)^{-1}\left(\eta_{\mathfrak{i}}\bar{\phi}\right)-\left(-\Delta\right)^{-1}\left(\bar{\phi}\right)\right)e_{0}.\nonumber
		\end{align}
		The first term on the second line of \eqref{quadratic-form-estimate-proof-7} has the bound
		\begin{align}
			\left|\int_{\supp{\left(f_{\varepsilon}'\right)^{+}}}\left(f_{\varepsilon}'\right)^{+}\left(-\Delta\right)^{-1}\left(\bar{\phi}\right)e_{0}\right|&\leq C\|\left(-\Delta\right)^{-1}\bar{\phi}\|_{L^{\infty}\left(\mathbb{R}^{2}_{+}\right)}\nonumber\\
			&\leq C\left(\|\bar{\phi}\|_{L^{2}\left(\mathbb{R}^{2}_{+}\right)}+\left|\log{\varepsilon}\right|\left|\bar{\mathcal{J}}_{\left(0,0\right)}\right|\right),\label{QF-estimate-1}
		\end{align}
		where $C>0$ is an absolute constant. Here we have used the compact support of $\left(f_{\varepsilon}'\right)^{+}$, as well as \eqref{spectral-theory-appendix-orthogonality-relations} for the weighted $L^{2}$ norm of $e_{0}$. Finally we have used the mass condition for $\bar{\phi}$ and Lemma \ref{poisson-equation-zero-mass-estimates-lemma}.
		
		\medskip
		From here, we use the fact that $\supp \bar{\phi}\subset B_{3\rho}\left(q'e_{2}\right)$ alongside \eqref{interior-exterior-cutoff-sum-bounds} to obtain
		\begin{align}
			\|\bar{\phi}\|_{L^{2}\left(\mathbb{R}^{2}_{+}\right)}\leq \|\eta_{\mathfrak{i}}\bar{\phi}\|_{L^{2}\left(\mathbb{R}^{2}_{+}\right)}+\|\eta_{\mathfrak{e}}\bar{\phi}\|_{L^{2}\left(\mathbb{R}^{2}_{+}\right)}&\leq C\left(\left(\int_{\mathbb{R}^{2}_{+}}\frac{\eta_{\mathfrak{i}}^{2}\left|\bar{\phi}\right|^{2}}{\left(f_{\varepsilon}'\right)^{+}}\right)^{\frac{1}{2}}+\|\eta_{\mathfrak{e}}\bar{\phi}\|_{L^{2}\left(\mathbb{R}^{2}_{+}\right)}\right),\label{QF-estimate-2}
		\end{align}
		where we have used the fact that $\left(f_{\varepsilon}'\right)^{+}$ is strictly positive on $\supp{\eta_{\mathfrak{i}}}$, as well as its global boundedness. Thus, using \eqref{quadratic-form-estimate-proof-3}, \eqref{QF-estimate-1}--\eqref{QF-estimate-2}, and Parseval's identity with respect to the eigenbasis $\left(e_{j}\right)_{j\geq0}$, we obtain the bound 
		\begin{align}
			\left|\int_{\supp{\left(f_{\varepsilon}'\right)^{+}}}\left(f_{\varepsilon}'\right)^{+}\left(-\Delta\right)^{-1}\left(\bar{\phi}\right)e_{0}\right|&\leq C\left(\left(\sum_{j=0}^{\infty}\left|\bar{\phi}_{j}\right|^{2}\right)^{\frac{1}{2}}+\|\eta_{\mathfrak{e}}\bar{\phi}\|_{L^{2}\left(\mathbb{R}^{2}_{+}\right)}\right)\nonumber\\
			&+C\left|\log{\varepsilon}\right|\left|\bar{\mathcal{J}}_{\left(0,0\right)}\right|.\label{quadratic-form-estimate-proof-9-0}
		\end{align}
		For the second term on the right hand side of \eqref{quadratic-form-estimate-proof-7}, as the support of $\bar{\phi}$ is contained in $B_{4\rho}\left(q'e_{2}\right)$, using \eqref{interior-exterior-cutoff-sum-bounds} gives that $\left|1-\eta_{\mathfrak{i}}\right|\leq \eta_{\mathfrak{e}}$ on $\supp{\bar{\phi}}$. This fact gives
		\begin{align}
			\left|\left(-\Delta\right)^{-1}\left(\left(1-\eta_{\mathfrak{i}}\right)\bar{\phi}\right)\right|\leq C\left|\log{\varepsilon}\right|\|\eta_{\mathfrak{e}}\bar{\phi}\|_{L^{2}\left(\mathbb{R}^{2}_{+}\right)}.\label{quadratic-form-estimate-proof-9}
		\end{align}
		Squaring both sides of \eqref{quadratic-form-estimate-proof-7}, combining \eqref{quadratic-form-estimate-proof-9-0} and \eqref{quadratic-form-estimate-proof-9}, and dividing by $\mu_{0}$, we obtain
		\begin{align}
			\mu_{0}\left|\bar{\phi}_{0}\right|^{2}\leq \frac{C}{\mu_{0}}\sum_{j=0}^{\infty}\left|\bar{\phi}_{j}\right|^{2}+\bigO\left(\left|\log{\varepsilon}\right|^{2}\left(\|\eta_{\mathfrak{e}}\bar{\phi}\|_{L^{2}\left(\mathbb{R}^{2}_{+}\right)}^{2}+\bar{\mathcal{J}}_{\left(0,0\right)}^{2}\right)\right).\label{quadratic-form-estimate-proof-11}
		\end{align}
		Next, we use the centre of mass condition with respect to $y_{1}$ given in \eqref{bar-phi-centre-of-mass-conditions} to get
		\begin{align}
			\bar{\mathcal{J}}_{\left(1,0\right)}&=\int_{\mathbb{R}^{2}_{+}}\left(f_{\varepsilon}'\right)^{+}y_{1}\frac{\eta_{\mathfrak{i}}\bar{\phi}}{\left(f_{\varepsilon}'\right)^{+}}\ dy+\int_{\mathbb{R}^{2}_{+}}\left(1-\eta_{\mathfrak{i}}\right)y_{1}\bar{\phi}\ dy\nonumber\\
			&=\bar{\phi}_{1}\int_{\mathbb{R}^{2}_{+}}\left(f_{\varepsilon}'\right)^{+}y_{1}e_{1}\ dy+\int_{\mathbb{R}^{2}_{+}}\left(f_{\varepsilon}'\right)^{+}y_{1}\left(\sum_{j\neq1,2}\bar{\phi}_{j}e_{j}\right)\ dy +\bigO\left(\|\eta_{\mathfrak{e}}\bar{\phi}\|_{L^{2}\left(\mathbb{R}^{2}_{+}\right)}\right),
			\nonumber
		\end{align}
		where we have again used \eqref{bar-phi-support-weighted-L2-space-and-mass-condition}, \eqref{interior-exterior-cutoff-sum-bounds}, and \eqref{quadratic-form-estimate-proof-3}, and the fact that $e_{2}$ is even in $y_{1}$. By \eqref{dq-Psi-R-variation},
		\begin{align}
			\nonumber\int_{\mathbb{R}^{2}_{+}}\left(f_{\varepsilon}'\right)^{+}y_{1}e_{1}\ dy\geq c_{0}>0,
		\end{align}
		some absolute constant $c_{0}$. Thus, we have
		\begin{align}
			c_{1}\left(\left|\bar{\phi}_{1}\right|^{2}+\left|\bar{\phi}_{2}\right|^{2}\right)+\bigO\left(\|\eta_{\mathfrak{e}}\bar{\phi}\|_{L^{2}\left(\mathbb{R}^{2}_{+}\right)}^{2}+\bar{\mathcal{J}}_{\left(1,0\right)}^{2}+\bar{\mathcal{J}}_{\left(0,1\right)}^{2}\right)\leq \sum_{j\neq 1,2}\left|\bar{\phi}_{j}\right|^{2},\label{quadratic-form-estimate-proof-14}
		\end{align}
		as the argument for $\bar{\phi}_{2}$ using the centre of mass condition with respect to $y_{2}-q'$ in \eqref{bar-phi-centre-of-mass-conditions} is completely analogous. Using \eqref{quadratic-form-estimate-proof-6}, \eqref{quadratic-form-estimate-proof-11}, and \eqref{quadratic-form-estimate-proof-14}, and the fact that $\mu_{0}\sim\left|\log{\varepsilon}\right|$ as $\varepsilon\to0$, we obtain
		\begin{align*} 
			&\sum_{j\neq1,2}\left(1-\mu_{j}\right)\left|\bar{\phi}_{j}\right|^{2}\geq\frac{\min{\{c_{1},1\}}}{2}\left(1-\delta-\left(1+\frac{2}{c_{1}}\right)\frac{C}{\left|\log{\varepsilon}\right|}\right)\sum_{j=0}^{\infty}\left|\bar{\phi}_{j}\right|^{2}+\mathcal{R}_{1},
			\nonumber
		\end{align*}
		where $\mathcal{R}_{1}$ satisfies \eqref{quadratic-form-estimate-lemma-statement-2}. This, along with \eqref{quadratic-form-estimate-proof-5}, gives \eqref{quadratic-form-estimate-lemma-statement-1} for $\varepsilon>0$ small enough, as required.
	\end{proof}
	\subsection{Weighted $L^{2}$ A Priori Estimates}
	With Lemma \ref{quadratic-form-estimate-lemma} in hand, we wish to work with an operator that has structure compatible with the integrand on the left hand side of the inequality \eqref{quadratic-form-estimate-lemma-statement-1} that we proved in Lemma \ref{quadratic-form-estimate-lemma}. If we consider the last three terms on the right hand side of \eqref{homotopic-operator-inner-error-E-star-R-definition-5}, with dependence on $\lambda$ suppressed, given by
	\begin{align}
		\varepsilon^{2}\dell_{t}\tilde{\phi}_{*R}\left(z,t\right)&+\gradperp_{z}\tilde{\psi}_{*R}\cdot\nabla_{z}\left(U_{R}+\phi_{R}\right)\nonumber\\
		&+\gradperp_{z}\left(\Psi_{0}-\varepsilon \dot{p}z_{2}+\varepsilon \dot{q}z_{1}+\psi_{R}+\psi_{*R}+\psi^{out}+\psi^{out}_{*}\right)\cdot\nabla_{z}\tilde{\phi}_{*R},\label{linear-operator-structure-motivation-1}
	\end{align}
	from cancellations of certain higher order terms, which can be seen from Lemmas \ref{psi-L-correct-form-lemma}, \ref{first-approximation-construction-theorem}, and the structure of $\phi_{R1}/\left(f_{\varepsilon}'\right)^{+}$ which we recall from \eqref{first-approximation-construction-12}, on $\mathbb{H}_{-q'}$ we can write \eqref{linear-operator-structure-motivation-1} schematically as
	\begin{align}
		\varepsilon^{2}\dell_{t}\tilde{\phi}_{*R}\left(z,t\right)+\gradperp_{z}\left(\mathscr{V}\left(q\right)+a_{*}+a\right)\cdot\nabla_{z}\left(\tilde{\phi}_{*R}-f_{\varepsilon}'\left(\mathscr{V}\left(q\right)+a_{*}\right)\tilde{\psi}_{*R}\right)+E\label{linear-operator-structure-motivation-8}    
	\end{align}
	where $a_{*}$ and $a$ are certain functions which will have good bounds in spacetime and space respectively, and $E$ are all the quadratic and small linear terms with respect to the unknowns that can be treated as errors in a fixed point formulation.
	
	\medskip
	Recalling that the homotopic operators \eqref{homotopic-operator-inner-error-main-definition} and \eqref{homotopic-operator-outer-error-main-definition} were defined to have no nonlinear dependence on the unknowns at $\lambda=0$, we first wish to study solutions to a suitable linear (in $\tilde{\phi}_{*R}$) problem. With \eqref{linear-operator-structure-motivation-1}--\eqref{linear-operator-structure-motivation-8} in hand, we will derive a priori estimates for the system
	\begin{subequations}\label{linear-transport-operator-system}
		\begin{align}
			&\varepsilon^{2}\phi_{t}+\gradperp_{z_{b}}\left(\mathscr{V}\left(q_{b}\right)+a_{*}+a\right)\cdot\nabla_{z_{b}}\left(\phi-\mathscr{W}_{1}\left(\mathscr{V}\left(q_{b}\right)+a_{*}\right)\psi\right)\nonumber\\
			&+E=0\ \ \text{in}\ \mathbb{H}_{-q'_{b}}\times \left[T_{0},T\right],\label{linear-transport-equation}\\
			&\phi(\blank,T)=0\ \ \text{in}\ \mathbb{H}_{-q'_{b}},\quad \phi\left(z_{b},t\right)=0\ \ \text{on}\ \dell\mathbb{H}_{-q_{b}'}\times\left[T_{0},T\right],\label{linear-transport-initial-data}
		\end{align}
	\end{subequations}
	where we have used the coordinate transformation
	\begin{align}
		z_{b}=\frac{x-\xi_{0}\left(t\right)-\xi_{1}\left(t\right)-b\left(t\right)}{\varepsilon},\quad \xi^{\left(b\right)}\coloneqq\xi_{0}+\xi_{1}+b\label{linear-operator-change-of-coordinates-z-b}
	\end{align}
	with $b\left(t\right)$, $a\left(z_{b},t\right)$, and $a_{*}\left(z_{b},t\right)$ are some given functions. Moreover, $\psi$ is given by
	\begin{align}
		\psi\left(z,t\right)=\frac{1}{2\pi}\int_{\mathbb{H}_{-q'_{b}}}\log{\left(\frac{\left|v-\overline{\left(z_{b}+q'_{b}e_{2}\right)}\right|}{\left|v-\left(z_{b}+q'_{b}e_{2}\right)\right|}\right)}\phi\left(v,t\right) dv,\label{linear-transport-operator-psi-phi-representation-z-variables}
	\end{align}
	and we impose that for some $\mathcal{J}_{\left(i,j\right)}^{\left(b\right)}(t)=\bigO{\left(\varepsilon^{3}t^{-2}\right)}$ and $\mathcal{Q}_{\left(i,j\right)}^{\left(b\right)}(t)=\bigO{\left(\varepsilon^{5}t^{-3}\right)}$ for all $t\in\left[T_{0},T\right]$ and $\left(i,j\right)\in\{\left(0,0\right),\left(1,0\right),\left(0,1\right)\}$, 
	\begin{align}
		\int_{\mathbb{H}_{-q'_{b}}}z_{b,1}^{i}z_{b,2}^{j}\phi\left(z_{b},t\right)\ dz_{b}=\mathcal{J}_{\left(i,j\right)}^{\left(b\right)}=\varepsilon^{-2}\int_{t}^{T}\mathscr{Q}_{\left(i,j\right)}^{\left(b\right)}(\tau) d\tau.\label{linear-transport-operator-phi-mass-conditions}
	\end{align}
	On the functions $a_{*}\left(z_{b},t\right)$ and $a\left(z_{b},t\right)$, we assume
	\begin{equation}\label{linear-transport-operator-Laplacian-a-star-a-bounded}
		\begin{aligned}
			\left|\nabla_{z_{b}}\left(\mathscr{V}\left(q_{b}\right)+a_{*}+a\right)\right|+\left|\Delta_{z_{b}}\left(a_{*}+a\right)\right|&\leq C,\\
			\left(1+\left|z_{b}\right|\right)^{-2}\left(\left|a_{*}\right|+\left|a\right|\right)+\left(1+\left|z_{b}\right|\right)^{-1}\left(\left|\nabla_{z_{b}}a_{*}\right|+\left|\nabla_{z_{b}}a\right|\right)+\left|D^{2}_{z_{b}}a_{*}\right|&\leq C,
		\end{aligned}
	\end{equation}
	some constant $C>0$ independent of $\left(z_{b},t\right)\in\mathbb{H}_{-q'_{b}}\times\left[T_{0},T\right]$. We note that the $L^{\infty}$ bound for $\Delta_{z_{b}}\left(a_{*}+a\right)$ implies for all $t\in\left[T_{0},T\right]$, and $z_{1},z_{2}\in\mathbb{H}_{-q'_{b}}$, and an absolute constant $C>0$,
	\begin{align}
		\left|\nabla_{z_{b}}\left(a_{*}+a\right)\left(z_{1},t\right)-\nabla_{z_{b}}\left(a_{*}+a\right)\left(z_{2},t\right)\right|\leq C\left|z_{1}-z_{2}\right|\left|\log{\left|z_{1}-z_{2}\right|}\right|,\label{linear-transport-operator-a-log-lipschitz-bound}
	\end{align}
	We also assume that for some constants $C>0$ and $\nu>3/4$, we have, for all $t\in\left[T_{0},T\right]$, and for all $r_{1}>2$, in the ball $B_{10\rho}\left(0\right)$,
	\begin{align}
		&\left(1+\left|z_{b}\right|\right)^{-2}\left|a_{*}\right|+\left(1+\left|z_{b}\right|\right)^{-1}\left|\nabla_{z_{b}} a_{*}\right|+\left|D^{2}_{z_{b}}a_{*}\right|+t\left(1+\left|z_{b}\right|\right)^{-2}\left|\dell_{t}a_{*}\right|\leq \frac{C\varepsilon^{2}}{t^{2}},\label{linear-transport-operator-a-star-bounds}\\
		&\left(1+\left|z_{b}\right|\right)^{-5}\left|a\right|+\frac{\left(1+\left|z_{b}\right|\right)^{-4}}{t^{1-\nu}}\left|\nabla_{z_{b}}a\right|+\|\left(1+\left|\cdot\right|\right)^{-5}\nabla_{z_{b}} a\|_{L^{r_{1}}\left(\mathbb{H}_{-q_{b}'}\right)}\leq\frac{\varepsilon^{2+\nu}}{t^{2}}.\label{linear-transport-operator-a-bounds-0}
	\end{align}
	On $b=(b_1,b_2)$ we assume that
	\begin{align}
		\|t^{2}b_{1}\|_{\left[T_{0},T\right]}+\|t^{3}\dot{b}_{1}\|_{\left[T_{0},T\right]}+\|t^{3}b_{2}\|_{\left[T_{0},T\right]}+\|t^{4}\dot{b}_{2}\|_{\left[T_{0},T\right]}\leq \varepsilon^{3+\nu}.\label{b-1-bound}
	\end{align}
	Finally, we assume, for $\rho>0$ defined in \eqref{first-approximation-main-order-vorticity-support}.:
	\begin{align}
		\supp E\left(t\right)\subset B_{2\rho}\left(0\right),\quad t\in\left[T_{0},T\right],\quad 
		\supp \phi \left(t\right)\subset B_{3\rho}\left(0\right) , \quad t\in\left[T_{0},T\right].\label{linear-transport-operator-error-solution-supports}
	\end{align}
	\begin{remark}\label{10-rho-remark-2}
		As in Theorem \ref{vortex-pair-properties-theorem}, we state the bounds for \eqref{linear-transport-operator-a-star-bounds}--\eqref{linear-transport-operator-a-bounds-0} on a closed ball of radius $10\rho$ for concreteness.
	\end{remark}
	With these assumptions in mind, we now prove a priori estimates for solutions to \eqref{linear-transport-operator-system} satisfying \eqref{linear-transport-operator-psi-phi-representation-z-variables}--\eqref{linear-transport-operator-error-solution-supports} over the next two lemmas. We will suppress dependence on $b$ in the statement and proof, but we will point out whenever we use \eqref{b-1-bound}.
	\begin{lemma}\label{L2-exterior-a-priori-estimate-lemma}
		Let $E\left(z,t\right)\in\mathscr{Y}_{\infty,3}$, defined in \eqref{weighted-L-infinity-norm-definition}. Suppose $\phi$ solves \eqref{linear-transport-operator-system} with assumptions \eqref{linear-transport-operator-phi-mass-conditions}--\eqref{linear-transport-operator-error-solution-supports} holding, and suppose $\xi=\left(p,q\right)$ is of the form \eqref{linear-operator-change-of-coordinates-z-b}, with \eqref{xi0-bound}--\eqref{xi1-bound}, and \eqref{b-1-bound} satisfied. Then for all $\varepsilon>0$ small enough, and all $t\in\left[T_{0},T\right]$,
		\begin{align}
			\|\eta_{\mathfrak{e}}\phi\|_{L^{2}}&\leq \frac{C\varepsilon^{\frac{\gamma-7}{2}}}{t^{\frac{\gamma-3}{2}}}\int_{t}^{T}\left(\|E\|_{\mathscr{Y}_{\infty,3}}+\left(\int_{B_{2\rho}\left(0\right)}\frac{\eta_{\mathfrak{i}}^{2}\phi^{2}}{\left(f_{\varepsilon}'\right)^{+}}dz\right)^{\frac{1}{2}}+\left|\log{\varepsilon}\right|\left|\mathcal{J}_{\left(0,0\right)}\right|\right)d\tau,\label{L2-exterior-a-priori-estimate-lemma-statement-2}
		\end{align}
		where the $L^{2}$ norm in \eqref{L2-exterior-a-priori-estimate-lemma-statement-2} is the $L^{2}\left(\mathbb{H}_{-q'}\right)$ norm.
	\end{lemma}
	\begin{proof}
		Owing to the transport-like structure of \eqref{linear-transport-equation}, let $t\in\left[T_{0},T\right]$, and define characteristics $\bar{z}\left(\tau,t,z\right)$ by the ODE problem
		\begin{align}
			\frac{d\bar{z}}{d\tau}=\varepsilon^{-2}\gradperp_{\bar{z}}\left(\mathscr{V}\left(q\right)+a_{*}+a\right)\left(\bar{z}\left(\tau,t,z\right),\tau\right),\quad \tau\in\left[t,T\right],\quad \bar{z}\left(t,t,z\right)=z.\label{characteristics-ode}
		\end{align}
		The first thing to note is that due to \eqref{linear-transport-operator-Laplacian-a-star-a-bounded}--\eqref{linear-transport-operator-a-log-lipschitz-bound}, we have definiteness of the characteristics:
		\begin{align}
			\left|\bar{z}\left(\tau_{1},t,z\right)-\bar{z}\left(\tau_{2},t,z\right)\right|\leq C\varepsilon^{-2}\left|\tau_{1}-\tau_{2}\right|,\quad \tau_{1},\tau_{2}\in\left[t,T\right].\label{continuity-of-characteristics-0}
		\end{align}
		Next, as an application of Jacobi's formula, the right hand side of \eqref{characteristics-ode} being divergence free, and the initial condition give for the Jacobian $\mathscr{J}$ of the coordinate transformation $z\to\bar{z}$,
		\begin{align}
			\mathscr{J}\left(z,\bar{z}\right)=1.\label{jacobian}
		\end{align}
		We first wish to establish a bound on the size of the characteristics in the case when $z\in\supp{\phi}$. Given the assumption $\supp{\phi}\subset B_{3\rho}\left(0\right)$, \eqref{linear-transport-operator-error-solution-supports}, we claim that, for all $z\in B_{3\rho}\left(0\right)$ and $t\in\left[T_{0},T\right]$,
		\begin{align}
			\left|\bar{z}\left(\tau,t,z\right)\right|\leq 5\rho,\quad \forall \tau\in \left[t,T\right].\label{bar-z-5-rho-bound}
		\end{align}
		So let $z\in B_{3\rho}\left(0\right)$, and let $t\in\left[T_{0},T\right]$. By \eqref{continuity-of-characteristics-0}, we know that there is an interval $\left[t,\tau\right]\subset\left[t,T\right]$ where the bound stated in \eqref{bar-z-5-rho-bound} holds. Let $\left[t,\tau_{*}\right]$ be the largest such interval, and assume $\tau_{*}<T$. By continuity,
		\begin{align}
			\left|\bar{z}\left(\tau_{*},t,z\right)\right|=5\rho.\label{5-rho-equality}
		\end{align}
		Then we estimate $\left(\mathscr{V}\left(q\right)+a_{*}\right)\left(\bar{z}\left(\tau,t,z\right),\tau\right)$ for $\tau\in\left[t,\tau_{*}\right]$:
		\begin{align}
			&\left(\mathscr{V}\left(q\right)+a_{*}\right)\left(\bar{z}\left(\tau,t,z\right),\tau\right)-\left(\mathscr{V}\left(q\right)+a_{*}\right)\left(z,t\right)\nonumber\\
			&=\int_{t}^{\tau}\dell_{s}\left(\left(\mathscr{V}\left(q\right)+a_{*}\right)\left(\bar{z}\left(s,t,z\right),t\right)\right) ds\nonumber\\
			&=\int_{t}^{\tau}\left(\dot{q}\dell_{q}\mathscr{V}\left(q\right)+\dell_{s}a_{*}\right)\left(w,s\right)\rvert_{w=\bar{z}\left(s,t,z\right)}ds\nonumber\\
			&+\int_{t}^{\tau}\frac{d\bar{z}}{ds}\left(s,t,z\right)\cdot\nabla_{\bar{z}}\left(\mathscr{V}\left(q\right)+a_{*}\right)\left(\bar{z}\left(s,t,z\right),s\right)ds.\label{5-rho-proof-1}
		\end{align}
		For the second integral on the right hand side of \eqref{5-rho-proof-1}, we use \eqref{characteristics-ode} to obtain
		\begin{align*}
			&\int_{t}^{\tau}\frac{d\bar{z}}{ds}\left(s,t,z\right)\cdot\nabla_{\bar{z}}\left(\mathscr{V}\left(q\right)+a_{*}\right)\left(\bar{z}\left(s,t,z\right),s\right)ds\\
			&=\varepsilon^{-2}\int_{t}^{\tau}\gradperp_{\bar{z}}a\cdot\nabla_{\bar{z}}\left(\mathscr{V}\left(q\right)+a_{*}\right)\left(\bar{z}\left(s,t,z\right),s\right)ds
		\end{align*}
		Therefore, using \eqref{linear-transport-operator-Laplacian-a-star-a-bounded} and \eqref{linear-transport-operator-a-bounds-0}, we have the bound
		\begin{align}
			\left|\varepsilon^{-2}\int_{t}^{\tau}\gradperp_{\bar{z}}a\cdot\nabla_{\bar{z}}\left(\mathscr{V}\left(q\right)+a_{*}\right)\left(\bar{z}\left(s,t,z\right),s\right)ds\right|\leq C\left(1+5\rho\right)^{6}\left(\frac{\varepsilon}{t}\right)^{\nu}.\label{5-rho-proof-2}
		\end{align}
		For the first integral on the right hand side of \eqref{5-rho-proof-1}, we use Lemma \ref{vortex-pair-q-variation-lemma} and \eqref{xi0-bound}--\eqref{xi1-bound}, \eqref{b-1-bound} to bound $\dot{q}\dell_{q}\mathscr{V}\left(q\right)$, and \eqref{linear-transport-operator-a-star-bounds} to bound $\dell_{s}a_{*}$, and obtain
		\begin{align}
			\left|\int_{t}^{\tau}\left(\dot{q}\dell_{q}\mathscr{V}\left(q\right)+\dell_{s}a_{*}\right)\left(w,s\right)\rvert_{w=\bar{z}\left(s,t,z\right)}ds\right|\leq C\left(1+5\rho\right)^{2}\left(\frac{\varepsilon}{t}\right)^{2}.\label{5-rho-proof-3}
		\end{align}
		Bringing \eqref{5-rho-proof-1}--\eqref{5-rho-proof-3} together and using $\nu>3/4$, we obtain, for all $\tau\in\left[t,\tau_{*}\right]$,
		\begin{align}
			\left|\left(\mathscr{V}\left(q\right)+a_{*}\right)\left(\bar{z}\left(\tau,t,z\right),\tau\right)-\left(\mathscr{V}\left(q\right)+a_{*}\right)\left(z,t\right)\right|\leq C\left(1+5\rho\right)^{6}\left(\frac{\varepsilon}{t}\right)^{\frac{3}{4}}.\label{5-rho-proof-3a}
		\end{align}
		Applying this to $\tau=\tau_{*}$, and using $\left|z\right|\leq 3\rho<5\rho$, $\left|\bar{z}\left(\tau_{*},t,z\right)\right|=5\rho$, as well as \eqref{linear-transport-operator-a-star-bounds}, this becomes
		\begin{align}
			\left|\mathscr{V}\left(q\right)\left(\bar{z}\left(\tau_{*},t,z\right),\tau_{*}\right)-\mathscr{V}\left(q\right)\left(z,t\right)\right|\leq C\left(1+5\rho\right)^{6}\left(\frac{\varepsilon}{T_{0}}\right)^{\frac{3}{4}},\label{5-rho-proof-3b}
		\end{align}
		as $T_{0}\leq t$. From here, we can see that
		\begin{align}
			\mathscr{V}\left(q\right)\left(z,t\right)=\mathscr{V}\left(q\right)\left(\bar{z}\left(\tau_{*},t,z\right),\tau_{*}\right)-\left(\mathscr{V}\left(q\right)\left(\bar{z}\left(\tau_{*},t,z\right),\tau_{*}\right)-\mathscr{V}\left(q\right)\left(z,t\right)\right).\label{5-rho-proof-4}
		\end{align}
		We know that $\mathscr{V}\left(q\right)\left(z,t\right)\geq 0$ on $B_{\rho}\left(0\right)$ by Remark \ref{support-of-nonlinearity-remark} and \eqref{vortex-pair-variation-generalisation}, so by continuity, there is an absolute constant $C_{0}>0$ such that for $\left(z,t\right)\in B_{3\rho}\left(0\right)\times\left[T_{0},T\right]$,
		\begin{align}
			\mathscr{V}\left(q\right)\left(z,t\right)\geq -C_{0}\implies \mathscr{V}\left(q\right)\left(\bar{z}\left(\tau_{*},t,z\right),\tau_{*}\right)\geq -C_{0}-C\left(1+5\rho\right)^{6}\left(\frac{\varepsilon}{T_{0}}\right)^{\frac{3}{4}},\label{5-rho-proof-5}
		\end{align}
		whence for $\varepsilon>0$ small enough, we can again use continuity of $\mathscr{V}\left(q\right)$ to ensure that \eqref{5-rho-proof-5} implies $\left|\bar{z}\left(\tau_{*},t,z\right)\right|\leq 4\rho$, a contradiction to \eqref{5-rho-equality}, which establishes the bound \eqref{bar-z-5-rho-bound} as required.
		Proceeding exactly as in \eqref{5-rho-proof-1}--\eqref{5-rho-proof-3a}, for all $\tau_{1}\in\left[t,T\right]$, we have
		\begin{align}
			\left|\left(\mathscr{V}\left(q\right)+a_{*}\right)\left(\bar{z}\left(\tau_{1}\right),\tau_{1}\right)-\left(\mathscr{V}\left(q\right)+a_{*}\right)\left(z,t\right)\right|\leq C\left(\frac{\varepsilon}{t}\right)^{\frac{3}{4}}.\label{stream-function-level-sets-characteristics-estimate-4}
		\end{align}
		Then \eqref{linear-transport-operator-a-star-bounds}, \eqref{stream-function-level-sets-characteristics-estimate-4}, \eqref{f-prime-L2-outer-weight-comparison-bound-1}, and \eqref{f-prime-q-variation-1}--\eqref{tilde-W-k-def} give, for $z\in\supp{\eta_{\mathfrak{e}}}$ and for all $\tau\in\left[t,T\right]$,
		\begin{align}
			\left|\mathscr{W}_{j}\left(\mathscr{V}\left(q\right)+A\right)\left(\bar{z}\left(\tau\right),\tau\right)\right|&\leq C\left(\frac{\varepsilon}{t}\right)^{\left(\frac{\gamma-j}{2}\right)},\quad j=0,1,2,3,\quad A=0\ \text{or}\ a_{*}.\label{L2-exterior-a-priori-estimate-lemma-4a}
		\end{align}
		Using the characteristics defined by the system \eqref{characteristics-ode}, we can write a solution to \eqref{linear-transport-operator-system}, $\phi$, as
		\begin{align}\nonumber
			\phi\left(z,t\right)=\varepsilon^{-2}\int_{t}^{T}\left(-E\left(\bar{z}\right)+\gradperp_{\bar{z}}\left(\mathscr{V}\left(q\right)+a_{*}+a\right)\cdot\nabla_{\bar{z}}\left(\mathscr{W}_{1}\left(\mathscr{V}\left(q\right)+a_{*}\right)\psi\right)\left(\bar{z}\right)\right) d\tau.\nonumber
		\end{align}
		Now, because of \eqref{linear-transport-operator-error-solution-supports}, $E\left(\bar{z}\left(\tau\right),\tau\right)$ is only non zero when $\bar{z}\left(\tau\right)\in B_{2\rho}\left(0\right)$, and so \eqref{interior-exterior-cutoff-sum-bounds} applies. Then, using the assumption that $E\in\mathscr{Y}_{\infty,3}$, and proceeding as in the first inequality of \eqref{weighted-L-infinity-norm-bounds-L-infinity-norm} and then applying \eqref{L2-exterior-a-priori-estimate-lemma-4a}, we have, for all $\tau\in\left[t,T\right]$,
		\begin{align}
			\left|E\left(\bar{z}\left(\tau\right),\tau\right)\right|\leq C\left(\left(f_{\varepsilon}^{'''}\right)^{+}+\left(\frac{\varepsilon}{t}\right)^{\left(\frac{\gamma-3}{2}\right)}\right)\|E\|_{\mathscr{Y}_{\infty,3}}\leq C\left(\frac{\varepsilon}{t}\right)^{\left(\frac{\gamma-3}{2}\right)}\|E\|_{\mathscr{Y}_{\infty,3}}\left(\tau\right).
			\label{L2-exterior-a-priori-estimate-lemma-2}
		\end{align}
		Thus we have that
		\begin{align}\nonumber
			\left|\eta_{\mathfrak{e}}(z,t)\phi\left(z,t\right)\right| &\leq C\varepsilon^{-2}\int_{t}^{T}\left(\left(\frac{\varepsilon}{t}\right)^{\left(\frac{\gamma-3}{2}\right)}\|E\|_{\mathscr{Y}_{\infty,3}}+\left|\mathscr{W}_{1}\left(\mathscr{V}\left(q\right)+a_{*}\right)\right|\left|\nabla_{\bar{z}}\psi\right|\left(\bar{z}\right)\right) d\tau\nonumber\\
			&+C\varepsilon^{-2}\int_{t}^{T}\left|\mathscr{W}_{2}\left(\mathscr{V}\left(q\right)+a_{*}\right)\right|\left|\psi\right|\left(\bar{z}\left(\tau\right),\tau\right) d\tau,
			\nonumber
		\end{align}
		where we have bounded $\gradperp_{\bar{z}}\left(\mathscr{V}\left(q\right)+a_{*}+a\right)$ in $L^{\infty}$. Using \eqref{linear-transport-operator-phi-mass-conditions} to apply Lemma \ref{poisson-equation-zero-mass-estimates-lemma} and \eqref{L2-exterior-a-priori-estimate-lemma-4a} for $j=1,2$, as well as the Minkowski integral inequality, we obtain
		\begin{align}
			\|\eta_{\mathfrak{e}}\phi\|_{L^{2}}&\leq C\varepsilon^{-2}\left(\frac{\varepsilon}{t}\right)^{\left(\frac{\gamma-3}{2}\right)}\int_{t}^{T}\|E\|_{\mathscr{Y}_{\infty,3}} d\tau\nonumber\\
			&+C\varepsilon^{-2}\left(\frac{\varepsilon}{t}\right)^{\left(\frac{\gamma-2}{2}\right)}\int_{t}^{T}\left(\|\phi\|_{L^{2}\left(\mathbb{H}_{-q'}\right)}+\left|\log{\varepsilon}\right|\left|\mathcal{J}_{\left(0,0\right)}\right|\right)d\tau\nonumber\\
			&+C\varepsilon^{-2}\left(\frac{\varepsilon}{t}\right)^{\left(\frac{\gamma-1}{2}\right)}\int_{t}^{T}\left(\int_{B_{2\rho}(0)}\left|\nabla_{\bar{z}}\psi\right|^{2}\left(\bar{z}\left(\tau\right),\tau\right) d\bar{z}(\tau)\right)^{\frac{1}{2}}d\tau,\label{L2-exterior-a-priori-estimate-lemma-7}
		\end{align}
		where on the last term we use $\supp{\mathscr{W}_{1}\left(\mathscr{V}(q)+a_{*}\right)}\subset B_{2\rho}(0)$, and we have also used \eqref{jacobian} to change the variable of integration to $\bar{z}(\tau)$. The space integral in the last term on the right hand side can then by bounded by a constant times $\|\nabla_{z}\psi\|_{L^{r_{1}}}$, some $r_{1}>2$, using H{\"o}lder's inequality. Then \eqref{L2-exterior-a-priori-estimate-lemma-7} and another application of Lemma \ref{poisson-equation-zero-mass-estimates-lemma} gives
		\begin{align}
			&\left(\frac{t}{\varepsilon}\right)^{\left(\frac{\gamma-1}{4}\right)}\|\eta_{\mathfrak{e}}\phi\|_{L^{2}\left(\mathbb{H}_{-q'}\right)}\nonumber\\
			&\leq \frac{C\varepsilon^{\frac{\gamma-13}{4}}}{t^{\frac{\gamma-5}{4}}}\int_{t}^{T}\left(\|E\|_{\mathscr{Y}_{\infty,3}}+\|\phi\|_{L^{2}\left(\mathbb{H}_{-q'}\right)}+\left|\log{\varepsilon}\right|\left|\mathcal{J}_{\left(0,0\right)}\right|\right)d\tau.\label{L2-exterior-a-priori-estimate-lemma-17}
		\end{align}
		Let
		\begin{align}
			\left(\frac{t}{\varepsilon}\right)^{\left(\frac{\gamma-1}{4}\right)}\|\eta_{\mathfrak{e}}\phi\|_{L^{2}\left(\mathbb{H}_{-q'}\right)}=\alpha_{\mathfrak{e}}\left(t\right).\label{L2-exterior-a-priori-estimate-lemma-18}
		\end{align}
		Then, proceeding as in \eqref{weighted-Lr-norm-definition-1} for $r=2$ due to assumption on the support of $\phi$ \eqref{linear-transport-operator-error-solution-supports}, from \eqref{L2-exterior-a-priori-estimate-lemma-17} we have
		\begin{align*}
			\alpha_{\mathfrak{e}}\left(t\right)\leq \frac{C\varepsilon^{\frac{\gamma-13}{4}}}{t^{\frac{\gamma-5}{4}}}\int_{t}^{T}\left(\|E\|_{\mathscr{Y}_{\infty,3}}+\left(\int_{B_{2\rho}\left(0\right)}\frac{\eta_{\mathfrak{i}}^{2}\phi^{2}}{\left(f_{\varepsilon}'\right)^{+}}dz\right)^{\frac{1}{2}}+\left(\frac{\varepsilon}{\tau}\right)^{\frac{\gamma-1}{4}}\alpha_{\mathfrak{e}}+\left|\log{\varepsilon}\right|\left|\mathcal{J}_{\left(0,0\right)}\right|\right)d\tau.
		\end{align*}
		Thus, applying Gr\"{o}nwall, we obtain
		\begin{align}\nonumber
			\alpha_{\mathfrak{e}}\left(t\right)\leq \frac{C\varepsilon^{\frac{\gamma-13}{4}}}{t^{\frac{\gamma-5}{4}}}\int_{t}^{T}\left(\|E\|_{\mathscr{Y}_{\infty,3}}+\left(\int_{B_{2\rho}\left(0\right)}\frac{\eta_{\mathfrak{i}}^{2}\phi^{2}}{\left(f_{\varepsilon}'\right)^{+}}dz\right)^{\frac{1}{2}}+\left|\log{\varepsilon}\right|\left|\mathcal{J}_{\left(0,0\right)}\right|\right)d\tau,
		\end{align}
		as claimed in \eqref{L2-exterior-a-priori-estimate-lemma-statement-2}.
	\end{proof}
	Now we move on to the a priori estimate on the interior region where $\eta_{\mathfrak{i}}$ is supported.
	\begin{lemma}\label{L2-interior-a-priori-estimate-lemma}
		Let $E\left(z,t\right)\in\mathscr{Y}_{\infty,3}$, defined in \eqref{weighted-L-infinity-norm-definition}, and suppose $E$ has the bound
		\begin{align}
			\|E\|_{\mathscr{Y}_{\infty,3}}<\frac{\varepsilon^{5-\sigma_{0}}}{t^{3}},\label{L2-interior-a-priori-estimate-lemma-error-bound-statement}
		\end{align}
		for some small $\sigma_{0}>0$.
		
		\medskip
		Suppose further that $\phi$ solves \eqref{linear-transport-operator-system} with assumptions \eqref{linear-transport-operator-phi-mass-conditions}--\eqref{linear-transport-operator-error-solution-supports} holding, and suppose $\xi=\left(p,q\right)$ is of the form \eqref{linear-operator-change-of-coordinates-z-b}, with \eqref{xi0-bound}--\eqref{xi1-bound}, and \eqref{b-1-bound} satisfied.. Then for all $\varepsilon>0$ small enough, and all $T_{0}>0$ large enough, we have for all $t\in\left[T_{0},T\right]$,
		\begin{align}
			\left(\int_{B_{2\rho}\left(0\right)}\frac{\eta_{\mathfrak{i}}^{2}\phi^{2}}{\left(f_{\varepsilon}'\right)^{+}}dz\right)^{\frac{1}{2}}\leq\frac{C\varepsilon^{3-\sigma_{0}}\left|\log{\varepsilon}\right|^{2}}{t^{2}},\label{L2-interior-a-priori-estimate-lemma-statement}
		\end{align}
		for some absolute constant $C>0$.
	\end{lemma}
	\begin{proof}
		We define $\hat{\mathscr{W}}_{j}=\hat{\mathscr{W}}_{j}\left(z,t\right)\coloneqq\mathscr{W}_{j}\left(\mathscr{V}\left(q\right)+a_{*}\right)$, $j=0,1,2$. Take the equality \eqref{linear-transport-equation}, multiply by
		\begin{align}
			\eta_{\mathfrak{i}}G\coloneqq\left(\frac{\eta_{\mathfrak{i}}^{2}\phi}{\hat{\mathscr{W}}_{1}}-\eta_{\mathfrak{i}}\left(-\Delta\right)^{-1}\left(\eta_{\mathfrak{i}}\phi\right)\right),\label{L2-interior-a-priori-estimate-lemma-1}
		\end{align}
		and integrate over $\mathbb{H}_{-q'}$ to obtain
		\begin{align}
			&\varepsilon^{2}\int_{B_{2\rho}\left(0\right)}\frac{\eta_{\mathfrak{i}}^{2}\phi\ \dell_{t}\phi\ dz}{\hat{\mathscr{W}}_{1}}-\varepsilon^{2}\int_{B_{2\rho}\left(0\right)}\eta_{\mathfrak{i}}\dell_{t}\phi\left(-\Delta\right)^{-1}\left(\eta_{\mathfrak{i}}\phi\right)dz\label{L2-interior-a-priori-estimate-lemma-2}\\
			&+\int_{B_{2\rho}\left(0\right)}\eta_{\mathfrak{i}}\left(z,t\right)G\left(z,t\right)\ \gradperp_{z}\left(\mathscr{V}\left(q\right)+a_{*}+a\right)\cdot\nabla_{z}\left(\phi-\hat{\mathscr{W}}_{1}\psi\right)dz\nonumber\\
			&+\int_{B_{2\rho}\left(0\right)}\left(\frac{\eta_{\mathfrak{i}}^{2}\phi}{\hat{\mathscr{W}}_{1}}-\eta_{\mathfrak{i}}\left(-\Delta\right)^{-1}\left(\eta_{\mathfrak{i}}\phi\right)\right)E\ dz=0,\nonumber
		\end{align}
		where the region of integration is due to the support of $\eta_{\mathfrak{i}}$ given in \eqref{weighted-L2-interior-cutoff} with
		$\rho>0$ defined in \eqref{first-approximation-main-order-vorticity-support}. Moreover, by \ref{f-prime-L2-inner-weight-comparison-bound} and \eqref{linear-transport-operator-a-star-bounds}, we have, for $j=0,1,2,3$,
		\begin{align*}
			\hat{\mathscr{W}}_{j}\left(z,t\right)=\mathscr{W}_{j}\left(\mathscr{V}\left(q\right)+a_{*}\right)\geq C\left(\frac{\varepsilon}{t}\right)^{\frac{\gamma-j}{2}},\quad z\in\supp{\eta_{\mathfrak{i}}},
		\end{align*}
		which shows that the integrals in \eqref{L2-interior-a-priori-estimate-lemma-2} are well defined. 
		
		\medskip
		For the first term on the left hand side of \eqref{L2-interior-a-priori-estimate-lemma-2}, we have
		\begin{align}
			\varepsilon^{2}\int_{B_{2\rho}\left(0\right)}\frac{\eta_{\mathfrak{i}}^{2}\phi\ \dell_{t}\phi\ dz}{\hat{\mathscr{W}}_{1}}&=\frac{\varepsilon^{2}}{2}\frac{d}{dt}\left(\int_{B_{2\rho}\left(0\right)}\frac{\eta_{\mathfrak{i}}^{2}\phi^{2}\ dz}{\hat{\mathscr{W}}_{1}}\right)-\frac{\varepsilon^{2}}{2}\int_{B_{2\rho}\left(0\right)}\dell_{t}\left(\frac{\eta_{\mathfrak{i}}^{2}}{\hat{\mathscr{W}}_{1}}\right)\phi^{2}\ dz.\label{L2-interior-a-priori-estimate-lemma-3}
		\end{align}
		Then \eqref{weighted-L2-interior-cutoff}, \eqref{derivative-of-inner-cutoff-outer-cutoff-comparison-bound-1}, and \eqref{f-prime-L2-inner-weight-comparison-bound}, give
		\begin{align}
			\left|\eta_{\mathfrak{i}}\dell_{t}\eta_{\mathfrak{i}}\right|\left|\frac{1}{\hat{\mathscr{W}}_{1}}\right|\leq \frac{C\eta_{\mathfrak{e}}^{2}}{t}\left(\frac{t}{\varepsilon}\right)^{\frac{\gamma-1}{2}}.\label{L2-interior-a-priori-estimate-lemma-4}
		\end{align}
		Next, we note that on $\supp{\eta_{\mathfrak{i}}}$ that due to the definition of $\hat{\mathscr{W}}_{1}$ and \eqref{f-prime-q-variation-1}--\eqref{tilde-W-k-def},
		\begin{align*}
			\dell_{t}\left(\frac{1}{\hat{\mathscr{W}}_{1}}\right)&=\frac{\dell_{t}\left(\mathscr{V}\left(q\right)+a_{*}\right)}{\gamma\left(\gamma-1\right)\left(\mathscr{V}\left(q\right)+a_{*}\right)^{\gamma-1}_{+}\chi_{B_{\frac{s}{\varepsilon}}\left(0\right)}(z)}\frac{1}{\left(\mathscr{V}\left(q\right)+a_{*}\right)}\\
			&=\frac{\dell_{t}\left(\mathscr{V}\left(q\right)+a_{*}\right)}{\gamma\hat{\mathscr{W}}_{1}}\frac{1}{\left(\mathscr{V}\left(q\right)+a_{*}\right)}.
		\end{align*}
		Then, due to \eqref{linear-transport-operator-a-star-bounds}, and \eqref{f-prime-L2-inner-weight-comparison-bound}, we have
		\begin{align}
			\left|\eta_{\mathfrak{i}}^{2}\dell_{t}\left(\frac{1}{\hat{\mathscr{W}}_{1}}\right)\right|&\leq \frac{C\eta_{\mathfrak{i}}^{2}}{\hat{\mathscr{W}}_{0}}\left(\frac{\varepsilon}{t}\right)^{2}\leq \frac{C\eta_{\mathfrak{i}}^{2}}{\hat{\mathscr{W}}_{1}}\left(\frac{\varepsilon}{t}\right)^{\frac{3}{2}}.
			\nonumber \end{align}
		Thus
		\begin{align}
			\left|\int_{B_{2\rho}\left(0\right)}\dell_{t}\left(\frac{\eta_{\mathfrak{i}}^{2}}{\hat{\mathscr{W}}_{1}}\right)\phi^{2}\ dz\right|&\leq C\left(\frac{\varepsilon}{t}\right)^{\frac{3}{2}}\|\hat{\mathscr{W}}_{1}^{-\frac{1}{2}}\eta_{\mathfrak{i}}\phi\|_{L^{2}\left(\mathbb{H}_{-q'}\right)}^{2}+\frac{C\alpha_{\mathfrak{e}}^{2}\left(t\right)}{t},\label{L2-interior-a-priori-estimate-lemma-6}
		\end{align}
		where $\alpha_{\mathfrak{e}}$ is defined in \eqref{L2-exterior-a-priori-estimate-lemma-18}.
		
		\medskip
		Next we have the second term on the left hand side of \eqref{L2-interior-a-priori-estimate-lemma-2} which we can write as
		\begin{align}
			\frac{\varepsilon^{2}}{2}\frac{d}{dt}\left(\int_{B_{2\rho}\left(0\right)}\left(\eta_{\mathfrak{i}}\phi\right)\left(-\Delta\right)^{-1}\left(\eta_{\mathfrak{i}}\phi\right)dz\right)-\varepsilon^{2}\int_{B_{2\rho}\left(0\right)}\left(\phi\dell_{t}\eta_{\mathfrak{i}}\right)\left(-\Delta\right)^{-1}\left(\eta_{\mathfrak{i}}\phi\right)dz.\label{L2-interior-a-priori-estimate-lemma-7}
		\end{align}
		Similarly to \eqref{L2-interior-a-priori-estimate-lemma-4}, we can apply Young's inequality to obtain
		\begin{align}
			&\left|\int_{B_{2\rho}\left(0\right)}\left(\phi\dell_{t}\eta_{\mathfrak{i}}\right)\left(-\Delta\right)^{-1}\left(\eta_{\mathfrak{i}}\phi\right)dz\right|\leq \frac{C}{t}\|\eta_{\mathfrak{e}}\phi\|_{L^{2}\left(\mathbb{H}_{-q'}\right)}\|\left(-\Delta\right)^{-1}\left(\eta_{\mathfrak{i}}\phi\right)\|_{L^{\infty}\left(\mathbb{H}_{-q'}\right)}\label{L2-interior-a-priori-estimate-lemma-10}\\
			&\leq C\left(\frac{\left|\log{\varepsilon}\right|}{t}\right)^{2}\left(\frac{\varepsilon}{t}\right)^{\left(\frac{\gamma-1}{4}\right)}\alpha_{\mathfrak{e}}\left(t\right)^{2}+C\left(\frac{\varepsilon}{t}\right)^{\left(\frac{\gamma-1}{4}\right)}\|\hat{\mathscr{W}}_{1}^{-\frac{1}{2}}\eta_{\mathfrak{i}}\phi\|_{L^{2}\left(\mathbb{H}_{-q'}\right)}^{2}.\nonumber
		\end{align}
		Let $\bar{\eta}_{\mathfrak{i}}=1-\eta_{\mathfrak{i}}$. For the term on the second line on the left hand side of \eqref{L2-interior-a-priori-estimate-lemma-2}, we first note that since $\psi=\left(-\Delta\right)^{-1}\left(\phi\right)$,
		\begin{align}
			&\nabla_{z}\left(\hat{\mathscr{W}}_{1}\psi\right)=\nabla_{z}\left(\hat{\mathscr{W}}_{1}\left(-\Delta\right)^{-1}\left(\bar{\eta}_{\mathfrak{i}}\phi\right)\right)+\nabla_{z}\left(\hat{\mathscr{W}}_{1}\left(-\Delta\right)^{-1}\left(\eta_{\mathfrak{i}}\phi\right)\right).\label{L2-interior-a-priori-estimate-lemma-11}
		\end{align}
		Using \eqref{L2-interior-a-priori-estimate-lemma-11} we can rewrite the integral on the second line of \eqref{L2-interior-a-priori-estimate-lemma-2}, and obtain
		\begin{align}
			&\int_{B_{2\rho}\left(0\right)}\eta_{\mathfrak{i}}G\  \gradperp_{z}\left(\mathscr{V}\left(q\right)+a_{*}+a\right)\cdot\nabla_{z}\left(\phi-\hat{\mathscr{W}}_{1}\psi\right)dz\label{L2-interior-a-priori-estimate-lemma-12}\\
			&=\int_{B_{2\rho}\left(0\right)}\eta_{\mathfrak{i}}G\  \gradperp_{z}\left(\mathscr{V}\left(q\right)+a_{*}+a\right)\cdot\nabla_{z}\left(\hat{\mathscr{W}}_{1}G\right)dz+\mathcal{I},\nonumber
		\end{align}
		where, noting that $\hat{\mathscr{W}}_{1}$ is a function of $\mathscr{V}\left(q\right)+a_{*}$, and the definitions \eqref{f-prime-q-variation-1}--\eqref{tilde-W-k-def}, we have
		\begin{align}
			\mathcal{I}&=\int_{B_{2\rho}\left(0\right)}\left(-\frac{\eta_{\mathfrak{i}}^{2}\phi}{\hat{\mathscr{W}}_{1}}\left(-\Delta\right)^{-1}\left(\bar{\eta}_{\mathfrak{i}}\phi\right)+\eta_{\mathfrak{i}}\left(-\Delta\right)^{-1}\left(\eta_{\mathfrak{i}}\phi\right)\left(-\Delta\right)^{-1}\left(\bar{\eta}_{\mathfrak{i}}\phi\right)\right)\gradperp_{z}\left(a\right)\cdot\nabla_{z}\hat{\mathscr{W}}_{1}\ dz,\nonumber\\
			&+\int_{B_{2\rho}\left(0\right)}\left(-\eta_{\mathfrak{i}}^{2}\phi+\eta_{\mathfrak{i}}\left(-\Delta\right)^{-1}\left(\eta_{\mathfrak{i}}\phi\right)\right)\nabla_{z}\left(\left(-\Delta\right)^{-1}\left(\bar{\eta}_{\mathfrak{i}}\phi\right)\right)\cdot\gradperp_{z}\hat{\mathscr{W}}_{0}\ dz,\nonumber\\
			&+\int_{B_{2\rho}\left(0\right)}  \gradperp_{z}\left(\mathscr{V}\left(q\right)+a_{*}+a\right)\cdot\left(\eta_{\mathfrak{i}}G\phi\nabla_{z}\bar{\eta}_{\mathfrak{i}}+\frac{\phi^{2}}{2\hat{\mathscr{W}}_{1}}\nabla_{z}\left(\eta_{\mathfrak{i}}^{2}\bar{\eta}_{\mathfrak{i}}\right)\right) dz,\nonumber\\
			&-\int_{B_{2\rho}\left(0\right)}\frac{\eta_{\mathfrak{i}}^{2}\bar{\eta}_{\mathfrak{i}}\phi^{2}}{2\hat{\mathscr{W}}_{0}}\nabla_{z}\left(\mathscr{V}\left(q\right)+a_{*}\right)\cdot\gradperp_{z}\left(a\right)dz,\label{L2-interior-a-priori-estimate-lemma-13}
		\end{align}
		where we have used integration by parts for the last two terms. We begin to estimate \eqref{L2-interior-a-priori-estimate-lemma-12} with the first term on the right hand side, and note
		\begin{align}
			&\int_{B_{2\rho}\left(0\right)}\eta_{\mathfrak{i}}G\ \gradperp_{z}\left(\mathscr{V}\left(q\right)+a_{*}+a\right)\cdot\nabla_{z}\left(\hat{\mathscr{W}}_{1}G\right)dz\label{L2-interior-a-priori-estimate-lemma-17}\\
			&=-\frac{1}{2}\int_{B_{2\rho}\left(0\right)}\hat{\mathscr{W}}_{1}G^{2}\ \gradperp_{z}\left(\mathscr{V}\left(q\right)+a_{*}+a\right)\cdot\nabla_{z}\left(\eta_{\mathfrak{i}}\right)dz\nonumber\\
			&+\frac{1}{2}\int_{B_{2\rho}\left(0\right)}\eta_{\mathfrak{i}}\hat{\mathscr{W}}_{2}G^{2}\ \gradperp_{z}\left(a\right)\cdot\nabla_{z}\left(\mathscr{V}\left(q\right)+a_{*}\right)dz,\nonumber
		\end{align}
		where we use integration by parts. For the first term on the right hand side of \eqref{L2-interior-a-priori-estimate-lemma-17}, recalling the definition of $G$ in \eqref{L2-interior-a-priori-estimate-lemma-1}, we have
		\begin{align}
			&\left|\int_{B_{2\rho}\left(0\right)}\hat{\mathscr{W}}_{1}G^{2}\left(z,t\right)\ \gradperp_{z}\left(\mathscr{V}\left(q\right)+a_{*}+a\right)\cdot\nabla_{z}\left(\eta_{\mathfrak{i}}\left(z,t\right)\right)dz\right|\label{L2-interior-a-priori-estimate-lemma-18}\\
			&\leq C\int_{B_{2\rho}\left(0\right)}\frac{\eta_{\mathfrak{i}}^{2}\phi^{2}}{\hat{\mathscr{W}}_{1}}\left|\nabla_{z}\left(\eta_{\mathfrak{i}}\left(z,t\right)\right)\right|dz\nonumber\\
			&+C\left(\frac{\varepsilon}{t}\right)^{\frac{\gamma-1}{2}}\int_{B_{2\rho}\left(0\right)}\left|\left(-\Delta\right)^{-1}\left(\eta_{\mathfrak{i}}\phi\right)\right|^{2}\left|\nabla_{z}\left(\eta_{\mathfrak{i}}\left(z,t\right)\right)\right|dz.\nonumber
		\end{align}
		Then \eqref{cutoff-derivative-support-inclusion}, \eqref{derivative-of-inner-cutoff-outer-cutoff-comparison-bound-1}, \eqref{f-prime-L2-inner-weight-comparison-bound}, and \eqref{weighted-L2-exterior-cutoff-support} give us
		\begin{align}
			&\left|\int_{B_{2\rho}\left(0\right)}\hat{\mathscr{W}}_{1}G^{2}\left(z,t\right)\ \gradperp_{z}\left(\mathscr{V}\left(q\right)+a_{*}+a\right)\cdot\nabla_{z}\left(\eta_{\mathfrak{i}}\left(z,t\right)\right)dz\right|\label{L2-interior-a-priori-estimate-lemma-18a}\\
			&\leq C\left(\left(\frac{t}{\varepsilon}\right)^{\frac{1}{2}}\alpha_{\mathfrak{e}}\left(t\right)^{2}+\left|\log{\varepsilon}\right|^{2}\left(\frac{\varepsilon}{t}\right)^{\frac{\gamma-2}{2}}\|\hat{\mathscr{W}}_{1}^{-\frac{1}{2}}\eta_{\mathfrak{i}}\phi\|_{L^{2}\left(\mathbb{H}_{-q'}\right)}^{2}\right).\nonumber
		\end{align}
		For the second term on the right hand side of \eqref{L2-interior-a-priori-estimate-lemma-17}, using \eqref{linear-transport-operator-a-bounds-0}, \eqref{f-prime-L2-inner-weight-comparison-bound}, and \eqref{weighted-L2-interior-cutoff}, we have the bound
		\begin{align}
			\left|\eta_{\mathfrak{i}}\left(z,t\right)\hat{\mathscr{W}}_{2} \gradperp_{z}\left(a\right)\right|\leq \frac{C\varepsilon^{\frac{11}{4}}}{t^{\frac{7}{4}}}\frac{\eta_{\mathfrak{i}}\left(z,t\right)\hat{\mathscr{W}}_{1}}{\mathscr{V}\left(q\right)+a_{*}}&\leq\frac{C\varepsilon^{\frac{11}{4}}}{t^{\frac{7}{4}}}\frac{\eta_{\mathfrak{i}}\left(z,t\right)\hat{\mathscr{W}}_{1}}{\left(1-C\left(\frac{\varepsilon}{t}\right)^{\frac{3}{2}}\right)\mathscr{V}\left(q_{0}+q_{1}\right)}\label{L2-interior-a-priori-estimate-lemma-19}\\
			&\leq\frac{C\varepsilon^{\frac{9}{4}}}{t^{\frac{5}{4}}}\eta_{\mathfrak{i}}\left(z,t\right)\hat{\mathscr{W}}_{1}.\nonumber
		\end{align}
		Then, analogously to \eqref{L2-interior-a-priori-estimate-lemma-18}--\eqref{L2-interior-a-priori-estimate-lemma-18a}, we have
		\begin{align}
			&\left|\int_{B_{2\rho}\left(0\right)}\eta_{\mathfrak{i}}\left(z,t\right)\hat{\mathscr{W}}_{2}G^{2}\left(z,t\right)\ \gradperp_{z}\left(a\right)\cdot\nabla_{z}\left(\mathscr{V}\left(q\right)+a_{*}\right)dz\right|\label{L2-interior-a-priori-estimate-lemma-20}\\
			&\leq\frac{C\left|\log{\varepsilon}\right|^{2}\varepsilon^{\frac{9}{4}}}{t^{\frac{5}{4}}}\left(\|\hat{\mathscr{W}}_{1}^{-\frac{1}{2}}\eta_{\mathfrak{i}}\phi\|_{L^{2}\left(\mathbb{H}_{-q'}\right)}^{2}+\alpha_{\mathfrak{e}}\left(t\right)^{2}\right).\nonumber
		\end{align}
		The remainder $\mathcal{I}$ defined in \eqref{L2-interior-a-priori-estimate-lemma-13} can be similarly estimated term by term once we recall that $\left|\bar{\eta}_{\mathfrak{i}}\right|=\left|1-\eta_{\mathfrak{i}}\right|\leq \eta_{\mathfrak{e}}$ on $B_{4\rho}(0)$ by \eqref{interior-exterior-cutoff-sum-bounds}. Thus
		\begin{align}\nonumber
			\left|\mathcal{I}\right|\leq \frac{C\varepsilon^{\frac{9}{4}}\left|\log{\varepsilon}\right|^{2}}{t^{\frac{5}{4}}}\|\hat{\mathscr{W}}_{1}^{-\frac{1}{2}}\eta_{\mathfrak{i}}\phi\|_{L^{2}\left(\mathbb{H}_{-q'}\right)}^{2}+C\left(\frac{t}{\varepsilon}\right)^{\frac{1}{2}}\alpha_{\mathfrak{e}}\left(t\right)^{2}.
		\end{align}
		Finally, we estimate the last term on the left hand side of \eqref{L2-interior-a-priori-estimate-lemma-2} given by
		\begin{align*}
			&\int_{B_{2\rho}\left(0\right)}\left(\frac{\eta_{\mathfrak{i}}^{2}\phi}{\hat{\mathscr{W}}_{1}}-\eta_{\mathfrak{i}}\left(-\Delta\right)^{-1}\left(\eta_{\mathfrak{i}}\phi\right)\right)E\ dz\\
			&=\int_{B_{2\rho}\left(0\right)}\left[\left(\frac{\eta_{\mathfrak{i}}\phi}{\hat{\mathscr{W}}_{1}^{\frac{1}{2}}}\right)\left(\frac{\eta_{\mathfrak{i}}E}{\hat{\mathscr{W}}_{3}}\right)\left(\frac{\hat{\mathscr{W}}_{3}}{\hat{\mathscr{W}}_{1}^{\frac{1}{2}}}\right) -\left(-\Delta\right)^{-1}\left(\eta_{\mathfrak{i}}\phi\right)\left(\eta_{\mathfrak{i}}E\right)\right]dz,
		\end{align*}
		where the rewriting of the integrals on the right hand side are valid by the definitions \eqref{f-prime-q-variation-1}--\eqref{tilde-W-k-def}, and the fact that $\gamma>18>5$ means that $\hat{\mathscr{W}}_{3}\hat{\mathscr{W}}_{1}^{-\frac{1}{2}}$ is a bounded function. From here we again use \eqref{f-prime-L2-inner-weight-comparison-bound} along with the definition \eqref{weighted-L-infinity-norm-definition} and Young's inequality to obtain that
		\begin{align}
			\left|\int_{B_{2\rho}\left(0\right)}\left(\frac{\eta_{\mathfrak{i}}^{2}\phi}{\hat{\mathscr{W}}_{1}}-\eta_{\mathfrak{i}}\left(-\Delta\right)^{-1}\left(\eta_{\mathfrak{i}}\phi\right)\right)E\ dz\right|&\leq \frac{C\varepsilon^{2}}{\left|\log{\varepsilon}\right|t}\|\hat{\mathscr{W}}_{1}^{-\frac{1}{2}}\eta_{\mathfrak{i}}\phi\|_{L^{2}\left(\mathbb{H}_{-q'}\right)}^{2}\nonumber\\
			&+\frac{Ct}{\varepsilon^{2}}\left|\log{\varepsilon}\right|^{4}\|E\|_{\mathscr{Y}_{\infty,3}}^{2}.\label{L2-interior-a-priori-estimate-lemma-30}
		\end{align}
		Then, using the identities \eqref{L2-interior-a-priori-estimate-lemma-3}, \eqref{L2-interior-a-priori-estimate-lemma-7}, \eqref{L2-interior-a-priori-estimate-lemma-12}, as well as bounds \eqref{L2-interior-a-priori-estimate-lemma-6}, \eqref{L2-interior-a-priori-estimate-lemma-10}, \eqref{L2-interior-a-priori-estimate-lemma-18}, and \eqref{L2-interior-a-priori-estimate-lemma-20}--\eqref{L2-interior-a-priori-estimate-lemma-30}, we have
		\begin{align}
			&\frac{\varepsilon^{2}}{2}\frac{d}{dt}\left(\int_{B_{2\rho}\left(0\right)}\left(\frac{\eta_{\mathfrak{i}}^{2}\phi^{2}}{\hat{\mathscr{W}}_{1}}-\left(\eta_{\mathfrak{i}}\phi\right)\left(-\Delta\right)^{-1}\left(\eta_{\mathfrak{i}}\phi\right)\right)dz\right)\nonumber\\
			&\geq-\frac{C\varepsilon^{2}}{\left|\log{\varepsilon}\right|t}\|\hat{\mathscr{W}}_{1}^{-\frac{1}{2}}\eta_{\mathfrak{i}}\phi\|_{L^{2}\left(\mathbb{H}_{-q'}\right)}^{2}-\frac{Ct}{\varepsilon^{2}}\left|\log{\varepsilon}\right|^{4}\|E\|_{\mathscr{Y}_{\infty,3}}^{2}-C\left(\frac{t}{\varepsilon}\right)^{\frac{1}{2}}\alpha_{\mathfrak{e}}\left(t\right)^{2},
			\nonumber
		\end{align}
		Then integrating on $\left[t,T\right]$, and noting that $\phi\left(T\right)=0$ from \eqref{linear-transport-initial-data}, we have
		\begin{align}
			&\int_{B_{2\rho}\left(0\right)}\left(\frac{\eta_{\mathfrak{i}}^{2}\phi^{2}}{\hat{\mathscr{W}}_{1}}-\left(\eta_{\mathfrak{i}}\phi\right)\left(-\Delta\right)^{-1}\left(\eta_{\mathfrak{i}}\phi\right)\right)dz\leq\int_{t}^{T}C\varepsilon^{-2}\left(\frac{t}{\varepsilon}\right)^{\frac{1}{2}}\alpha_{\mathfrak{e}}\left(\tau\right)^{2}d\tau\label{L2-interior-a-priori-estimate-lemma-32}\\
			&+\int_{t}^{T}\frac{C}{\left|\log{\varepsilon}\right|\tau}\|\hat{\mathscr{W}}_{1}^{-\frac{1}{2}}\eta_{\mathfrak{i}}\phi\|_{L^{2}\left(\mathbb{H}_{-q'}\right)}^{2} d\tau+\int_{t}^{T}\frac{C\tau}{\varepsilon^{4}}\left|\log{\varepsilon}\right|^{4}\|E\|_{\mathscr{Y}_{\infty,3}}^{2}d\tau.\nonumber
		\end{align}
		Once again using \eqref{f-prime-L2-inner-weight-comparison-bound}, we have
		\begin{align}
			\left|\int_{B_{2\rho}\left(0\right)}\eta_{\mathfrak{i}}^{2}\phi^{2}\left(\frac{1}{\hat{\mathscr{W}}_{1}}-\frac{1}{\left(f_{\varepsilon}'\right)^{+}}\right)dz\right|\leq C\left(\frac{\varepsilon}{t}\right)^{\frac{3}{2}}\int_{B_{2\rho}\left(0\right)}\frac{\eta_{\mathfrak{i}}^{2}\phi^{2}}{\left(f_{\varepsilon}'\right)^{+}}dz.\label{L2-interior-a-priori-estimate-lemma-34}
		\end{align}
		Then applying Lemma \ref{quadratic-form-estimate-lemma} and \eqref{L2-interior-a-priori-estimate-lemma-34} to \eqref{L2-interior-a-priori-estimate-lemma-32}, we obtain
		\begin{align}
			\int_{B_{2\rho}\left(0\right)}\frac{\eta_{\mathfrak{i}}^{2}\phi^{2}}{\left(f_{\varepsilon}'\right)^{+}}dz+\mathcal{R}_{1}&\leq\int_{t}^{T}\frac{C}{\left|\log{\varepsilon}\right|\tau}\left(\int_{B_{2\rho}\left(0\right)}\frac{\eta_{\mathfrak{i}}^{2}\phi^{2}}{\left(f_{\varepsilon}'\right)^{+}}dz\right) d\tau\label{L2-interior-a-priori-estimate-lemma-35}\\
			&+\int_{t}^{T}\frac{C\tau}{\varepsilon^{4}}\left|\log{\varepsilon}\right|^{4}\|E\|_{\mathscr{Y}_{\infty,3}}^{2}ds\nonumber\\
			&+\int_{t}^{T}C\varepsilon^{-2}\left(\frac{t}{\varepsilon}\right)^{\frac{1}{2}}\alpha_{\mathfrak{e}}\left(\tau\right)^{2}d\tau.\nonumber
		\end{align}
		Then using Lemma \ref{L2-exterior-a-priori-estimate-lemma}, we have that
		\begin{align}
			&\left|\log{\varepsilon}\right|^{2}\|\eta_{\mathfrak{e}}\phi\|_{L^{2}\left(\mathbb{H}_{-q'}\right)}^{2}\label{L2-interior-a-priori-estimate-lemma-36}\\
			&\leq \left|\log{\varepsilon}\right|^{2}\frac{C\varepsilon^{\gamma-7}}{t^{\gamma-3}}\left(\int_{t}^{T}\left(\|E\|_{\mathscr{Y}_{\infty,3}}+\left(\int_{B_{2\rho}\left(0\right)}\frac{\eta_{\mathfrak{i}}^{2}\phi^{2}}{\left(f_{\varepsilon}'\right)^{+}}dz\right)^{\frac{1}{2}}+\left|\log{\varepsilon}\right|\left|\mathcal{J}_{\left(0,0\right)}\right|\right)d\tau\right)^{2}.\nonumber
		\end{align}
		Define $\alpha_{\mathfrak{i}}$ as
		\begin{align}\nonumber
			\alpha_{\mathfrak{i}}\left(t\right)^{2}=\int_{B_{2\rho}\left(0\right)}\frac{\eta_{\mathfrak{i}}^{2}\phi^{2}}{\left(f_{\varepsilon}'\right)^{+}}dz.
		\end{align}
		Using \eqref{L2-interior-a-priori-estimate-lemma-34}--\eqref{L2-interior-a-priori-estimate-lemma-36}, \eqref{f-prime-q-variation-3}--\eqref{f-prime-q-variation-4}, \eqref{f-prime-L2-inner-weight-comparison-bound}, Lemmas \ref{quadratic-form-estimate-lemma} and \ref{L2-exterior-a-priori-estimate-lemma}, and \eqref{L2-exterior-a-priori-estimate-lemma-18}, we have
		\begin{align}
			&\alpha_{\mathfrak{i}}\left(t\right)^{2}\leq \int_{t}^{T}\frac{C\tau}{\varepsilon^{4}}\left|\log{\varepsilon}\right|^{4}\|E\|_{\mathscr{Y}_{\infty,3}}^{2}d\tau+\left|\log{\varepsilon}\right|^{2}\frac{C\varepsilon^{\gamma-7}}{t^{\gamma-3}}\left(\int_{t}^{T}\|E\|_{\mathscr{Y}_{\infty,3}}d\tau\right)^{2}\label{L2-interior-a-priori-estimate-lemma-39}\\
			&+\int_{t}^{T}\frac{C\varepsilon^{\frac{\gamma-18}{2}}}{\tau^{\frac{\gamma-6}{2}}}\left(\int_{\tau}^{T}\|E\|_{\mathscr{Y}_{\infty,3}}d\tau'\right)^{2}d\tau+\left|\log{\varepsilon}\right|^{2}\sum_{i=0}^{2}\mathcal{J}_{\left(i,j\right)}^{2}\nonumber\\
			&+\left|\log{\varepsilon}\right|^{4}\frac{C\varepsilon^{\gamma-7}}{t^{\gamma-3}}\left(\int_{t}^{T}\left|\mathcal{J}_{\left(0,0\right)}\right|d\tau\right)^{2}+\left|\log{\varepsilon}\right|^{2}\int_{t}^{T}\frac{C\varepsilon^{\frac{\gamma-18}{2}}}{\tau^{\frac{\gamma-6}{2}}}\left(\int_{\tau}^{T}\left|\mathcal{J}_{\left(0,0\right)}\right|d\tau'\right)^{2}d\tau\nonumber\\
			&+\int_{t}^{T}\frac{C}{\left|\log{\varepsilon}\right|\tau}\alpha_{\mathfrak{i}}\left(\tau\right)^{2}d\tau+\left|\log{\varepsilon}\right|^{2}\frac{C\varepsilon^{\gamma-7}}{t^{\gamma-3}}\left(\int_{t}^{T}\alpha_{\mathfrak{i}}\left(\tau\right)d\tau\right)^{2}\nonumber\\
			&+\int_{t}^{T}\frac{C\varepsilon^{\frac{\gamma-18}{2}}}{\tau^{\frac{\gamma-6}{2}}}\left(\int_{\tau}^{T}\alpha_{\mathfrak{i}}\left(\tau'\right)d\tau'\right)^{2}d\tau.\nonumber
		\end{align}
		Now, using the bounds we have for the $\mathcal{J}_{\left(i,j\right)}$ defined in \eqref{linear-transport-operator-phi-mass-conditions}, as well as \eqref{L2-interior-a-priori-estimate-lemma-error-bound-statement}, we have for the first six terms on the right hand side of \eqref{L2-interior-a-priori-estimate-lemma-39} the bound
		\begin{align}
			&\left(1+\left|\log{\varepsilon}\right|^{-2}\left(\frac{\varepsilon}{t}\right)^{\gamma-3}+\frac{\varepsilon^{\frac{\gamma-10}{2}}\left|\log{\varepsilon}\right|^{-4}}{t^{\frac{\gamma-8}{2}}}\right)\frac{C\varepsilon^{6-2\sigma_{0}}\left|\log{\varepsilon}\right|^{4}}{t^{4}}\label{L2-interior-a-priori-estimate-lemma-40}\\
			&+\varepsilon^{2\sigma_{0}}\left(\left|\log{\varepsilon}\right|^{-2}+\frac{\varepsilon^{\gamma-7}}{t^{\gamma-5}}+\left|\log{\varepsilon}\right|^{-2}\frac{\varepsilon^{\frac{\gamma-18}{2}}}{t^{\frac{\gamma-12}{2}}}\right)\frac{C\varepsilon^{6-2\sigma_{0}}\left|\log{\varepsilon}\right|^{4}}{t^{4}}\leq \frac{2C\varepsilon^{6-2\sigma_{0}}\left|\log{\varepsilon}\right|^{4}}{t^{4}}\nonumber
		\end{align}
		for $\varepsilon>0$ small enough. We claim that for $\varepsilon>0$ small enough, and for all $t\in\left[T_{0},T\right]$,
		\begin{align}
			\alpha_{\mathfrak{i}}\left(t\right)^{2}<\frac{3C\varepsilon^{6-2\sigma_{0}}\left|\log{\varepsilon}\right|^{4}}{t^{4}}\label{L2-interior-a-priori-estimate-lemma-41}
		\end{align}
		on $\left[T_{0},T\right]$.
		
		\medskip
		Clearly \eqref{L2-interior-a-priori-estimate-lemma-41} is true for all $t$ close enough to $T$ as we know $\alpha_{\mathfrak{i}}\left(T\right)=0$. Then, observing the last term on the right hand side of \eqref{L2-interior-a-priori-estimate-lemma-39}, as long as $\gamma>18$ so that have positive powers of $\varepsilon$ multiplying every term on the right hand side of \eqref{L2-interior-a-priori-estimate-lemma-39}, a standard continuation argument moving back in time from $T$ to $T_{0}$ gives us \eqref{L2-interior-a-priori-estimate-lemma-41}, and therefore \eqref{L2-interior-a-priori-estimate-lemma-statement}, as required.
	\end{proof}
	Using Lemmas \ref{L2-exterior-a-priori-estimate-lemma} and \ref{L2-interior-a-priori-estimate-lemma}, we can get a direct estimate on solutions to the $L^{2}$ norm of solutions to \eqref{linear-transport-operator-system} satisfying assumptions \eqref{linear-transport-operator-phi-mass-conditions}--\eqref{linear-transport-operator-error-solution-supports}.
	\begin{corollary}\label{L2-a-priori-estimate-corollary}
		Suppose the assumptions of Lemmas \ref{L2-exterior-a-priori-estimate-lemma} and \ref{L2-interior-a-priori-estimate-lemma} hold. Then there is a constant $C>0$ such that for all $\varepsilon>0$ small enough, and all $T_{0}>0$ large enough, we have for all $t\in\left[T_{0},T\right]$,
		\begin{align} \nonumber
			\|\phi\|_{L^{2}\left(\mathbb{H}_{-q'}\right)}\leq\frac{C\varepsilon^{3-\sigma_{0}}\left|\log{\varepsilon}\right|^{2}}{t^{2}},
		\end{align}
		$\sigma_{0}>0$ defined in \eqref{L2-interior-a-priori-estimate-lemma-error-bound-statement}.
	\end{corollary}
	Moreover, we can use Lemma \ref{poisson-equation-zero-mass-estimates-lemma} alongside Lemmas \ref{L2-exterior-a-priori-estimate-lemma} and \ref{L2-interior-a-priori-estimate-lemma} to obtain the following estimates.
	\begin{corollary}\label{L-infinity-a-priori-estimate-corollary}
		Suppose the assumptions of Lemmas \ref{L2-exterior-a-priori-estimate-lemma} and \ref{L2-interior-a-priori-estimate-lemma} hold. Then for all $\varepsilon>0$ small enough, and all $T_{0}>0$ large enough, we have for all $t\in\left[T_{0},T\right]$,
		\begin{align}
			\left|\psi\right|+\left|\nabla\psi\right|+\left[\nabla\psi\right]_{\beta}\left(y\right)&\leq \frac{C\varepsilon^{3-\sigma_{1}}}{t^{2-\sigma_{1}}},\label{L-infinity-a-priori-estimate-corollary-statement-2}\\
			t\varepsilon^{-2}\|\phi\|_{L^{r_{*}}\left(\mathbb{H}_{-q'}\right)}+\|\phi\|_{L^{\infty}\left(\mathbb{H}_{-q'}\right)}&\leq \frac{C\varepsilon^{1-\sigma_{1}}}{t^{1-\sigma_{1}}}\label{L-infinity-a-priori-estimate-corollary-statement-4},
		\end{align}
		for some $2<r_{*}<\infty$, some absolute constant $C>0$, some small $\sigma_{1}>0$.
	\end{corollary}
	\begin{proof}
		Once again define $\hat{\mathscr{W}}_{j}=\hat{\mathscr{W}}_{j}\left(z,t\right)\coloneqq\mathscr{W}_{j}\left(\mathscr{V}\left(q\right)+a_{*}\right)$, $j=0,1,2$. We proceed as in Lemma \ref{L2-exterior-a-priori-estimate-lemma}, and use characteristics to represent $\phi$ as
		\begin{align}\nonumber
			\phi\left(z,t\right)=\varepsilon^{-2}\int_{t}^{T}\left(-E\left(\bar{z}\left(\tau\right),\tau\right)+\gradperp_{\bar{z}}\left(\mathscr{V}\left(q\right)+a_{*}+a\right)\cdot\nabla_{\bar{z}}\left(\hat{\mathscr{W}}_{1}\psi\right)\left(\bar{z}\left(\tau\right),\tau\right)\right) d\tau.\nonumber
		\end{align}
		To begin with, bounds for $\mathcal{J}_{\left(i,j\right)}$ defined in \eqref{linear-transport-operator-phi-mass-conditions}, as well as Corollary \ref{L2-a-priori-estimate-corollary} and Lemma \ref{poisson-equation-zero-mass-estimates-lemma} give the $L^{\infty}$ bound for $\psi$ in \eqref{L-infinity-a-priori-estimate-corollary-statement-2}.
		
		\medskip
		Then taking the $L^{r_{1}}$ norm of $\phi$, some $2<r_{1}<\infty$, we obtain, by Lemma \ref{poisson-equation-zero-mass-estimates-lemma},
		\begin{align*}
			\|\phi\|_{L^{r_{1}}\left(\mathbb{H}_{-q'}\right)}\leq C\varepsilon^{-2}\int_{t}^{T}\left(\|E\|_{\mathscr{Y}_{\infty,3}}+\|\phi\|_{L^{2}\left(\mathbb{H}_{-q'}\right)}+\left|\log{\varepsilon}\right|\left|\mathcal{J}_{\left(0,0\right)}\right|\right)d\tau,
		\end{align*}
		where the second inequality comes from Lemma \ref{poisson-equation-zero-mass-estimates-lemma}.
		
		\medskip
		Then we interpolate for some $2<r_{2}<r_{1}$, and obtain, for some arbitrarily small $\sigma_{2}>0$,
		\begin{align}
			\|\phi\|_{L^{r_{2}}\left(\mathbb{H}_{-q'}\right)}\leq C\|\phi\|_{L^{2}\left(\mathbb{H}_{-q'}\right)}^{1-\sigma_{2}}\|\phi\|_{L^{r_{1}}\left(\mathbb{H}_{-q'}\right)}^{\sigma_{2}}.\label{L-infinity-a-priori-estimate-corollary-3}
		\end{align}
		Then bounds for $\mathcal{J}_{\left(i,j\right)}$ defined in \eqref{linear-transport-operator-phi-mass-conditions}, \eqref{L-infinity-a-priori-estimate-corollary-3} and Lemma \ref{poisson-equation-zero-mass-estimates-lemma} give \eqref{L-infinity-a-priori-estimate-corollary-statement-4} and therefore \eqref{L-infinity-a-priori-estimate-corollary-statement-2}.
		
		\medskip
		Finally using \eqref{L-infinity-a-priori-estimate-corollary-3} alongside Lemma \ref{poisson-equation-zero-mass-estimates-lemma} once again, we obtain
		\begin{align}
			&\|\phi\|_{L^{\infty}\left(\mathbb{H}_{-q'}\right)}\nonumber\\
			&\leq C\varepsilon^{-2}\int_{t}^{T}\left(\|E\|_{\mathscr{Y}_{\infty,3}}+\|\phi\|_{L^{2}\left(\mathbb{H}_{-q'}\right)}+\|\phi\|_{L^{r_{2}}\left(\mathbb{H}_{-q'}\right)}+\left|\log{\varepsilon}\right|\left|\mathcal{J}_{\left(0,0\right)}\right|\right)d\tau.
		\end{align}
		which gives \eqref{L-infinity-a-priori-estimate-corollary-statement-4}, as required.
	\end{proof}
	\subsection{A Priori Estimates on a Projected Linear Problem}\label{projected-transport-problem-a-priori-estimates-section}
	In view of \eqref{tilde-phi-star-R-initial-value-problem}, we want to show a priori estimates on the coefficients of projection in the following linear problem
	\begin{subequations}\label{projected-linear-transport-operator-system}
		\begin{align}
			&\varepsilon^{2}\phi_{t}+\gradperp_{z_{b}}\left(\mathscr{V}\left(q\right)+a_{*}+a\right)\cdot\nabla_{z_{b}}\left(\phi-\hat{\mathscr{W}}_{1}\psi\right)+E=c_{R0}\left(t\right)U_{R}\label{projected-linear-transport-equation}\\
			&+\left(f_{\varepsilon}'\right)^{+}\left(c_{R1}\left(t\right)\mathscr{Z}_{1}\left(q_{b}\right)+c_{R2}\left(t\right)\mathscr{Z}_{2}\left(q_{b}\right)\right)\ \ \text{in}\ \mathbb{H}_{-q'_{b}}\times \left[T_{0},T\right],\nonumber\\
			&\phi(\blank,T)=0\ \ \text{in}\ \mathbb{H}_{-q'_{b}},\quad \phi\left(z_{b},t\right)=0\ \ \text{on}\ \dell\mathbb{H}_{-q_{b}'}\times\left[T_{0},T\right],\label{projected-linear-transport-initial-data}
		\end{align}
	\end{subequations}
	with the same assumptions on $\phi,a,a_{*},b$, and $E$ as in \eqref{linear-operator-change-of-coordinates-z-b}--\eqref{linear-transport-operator-error-solution-supports}.
	
	\medskip
	To obtain a priori estimates for the coefficients of projection in \eqref{projected-linear-transport-operator-system}, we also need an extra assumption on $a_{*}$ that we justify here. The $a_{*}$ we consider corresponds to the quantity 
	\begin{align}
		\lambda\left(-\mathcal{E}_{2}\left(z,\xi_{0}\right)+\psi_{R1,2}\left(z,\xi_{0}\right)\right)
		\label{a-star-extra-assumption-justification-2}
	\end{align}
	in the full nonlinear problem \eqref{tilde-phi-star-R-initial-value-problem}. Using \eqref{first-approximation-first-error-mode-2} and \eqref{psi-R12-equation} we compute that
	\begin{align}
		\Delta_{z}\left(-\mathcal{E}_{2}+\psi_{R1,2}\right)+\left(f_{\varepsilon}'\right)^{+}\left(-\mathcal{E}_{2}+\psi_{R1,2}\right)&=\left(\left(f_{\varepsilon}'\right)^{+}-\gamma\Gamma^{\gamma-1}_{+}\right)\left(-\mathcal{E}_{2}+\psi_{R1,2}\right)\nonumber\\
		&=\bigO{\left(\varepsilon^{4}t^{-2}\right)},\label{a-star-extra-assumption-justification-3}
	\end{align}
	where the estimate on the right hand side uses Theorem \ref{vortex-pair-properties-theorem}.
	
	\medskip
	With \eqref{a-star-extra-assumption-justification-3} in hand, and recalling Lemma \ref{vortex-pair-q-variation-lemma} and subsequently \eqref{f-prime-q-variation-1}--\eqref{f-prime-q-variation-4}, our additional assumption for $a_{*}$ is that, for some $\nu>3/4$,
	\begin{align}
		\|\Delta_{z}\left(\mathscr{V}\left(q\right)+a_{*}\right)+\mathscr{W}_{0}\left(\mathscr{V}\left(q\right)+a_{*}\right)\|_{C^{1}}\leq\frac{\varepsilon^{2+\nu}}{t^{2}},\label{a-star-extra-assumption-statement-1}    
	\end{align}
	Now we state and prove a priori estimates on $\left(c_{R0},c_{R1},c_{R2}\right)$. We once again suppress dependence of $z$ on $b$ as in Lemmas \ref{L2-exterior-a-priori-estimate-lemma} and \ref{L2-interior-a-priori-estimate-lemma}.
	\begin{lemma}\label{projected-linear-transport-problem-coefficients-of-projection-a-priori-estimates-lemma}
		Suppose $\phi$ solves \eqref{projected-linear-transport-operator-system} with assumptions \eqref{linear-operator-change-of-coordinates-z-b}--\eqref{linear-transport-operator-error-solution-supports}, as well as the size assumption \eqref{L2-interior-a-priori-estimate-lemma-error-bound-statement} on $E$ from Lemma \ref{L2-interior-a-priori-estimate-lemma}, and the extra assumptions on $a_{*}$, \eqref{a-star-extra-assumption-statement-1}. Then for all $\varepsilon>0$ small enough, and all $T_{0}>0$ large enough, we have that for all $t\in\left[T_{0},T\right]$ that $c_{R}=\left(c_{R0},c_{R1},c_{R2}\right)$ satisfies an equation of the form
		\begin{align}
			\begin{pmatrix}
				M+\bigO{\left(\varepsilon^{2}\right)} & 0 & \bigO{\left(\varepsilon^{2}\right)} \\
				0 & -M+\bigO{\left(\varepsilon^{2}\right)} & 0\\
				\bigO{\left(\varepsilon^{2}\right)} & 0 & M+\bigO{\left(\varepsilon^{2}\right)}
			\end{pmatrix}\begin{pmatrix}
				c_{R0}\\ c_{R1}\\ c_{R2}
			\end{pmatrix}=\begin{pmatrix}
				\mathfrak{F}_{0}\\\mathfrak{F}_{1}\\ \mathfrak{F}_{2}\end{pmatrix},\label{projected-linear-transport-problem-coefficients-of-projection-a-priori-estimates-lemma-linear-system-form}
		\end{align}
		and the right hand side has the form, for $(i,j)=(0,0),(1,0)$, and $(0,1)$
		\begin{align}
			\mathfrak{F}_{(i,j)}&=\mathscr{Q}_{(i,j)}+\int_{\mathbb{H}_{-q'}}z_{1}^{i}z_{2}^{j}E,\ dz+\bigO{\left(\frac{\varepsilon^{2+\nu}}{t^{2}}\right)}\|\phi\|_{L^{2}\left(\mathbb{H}_{-q'}\right)}+\bigO{\left(\frac{\varepsilon^{5+\nu}\left|\log{\varepsilon}\right|}{t^{4}}\right)},\label{projected-linear-transport-problem-coefficients-of-projection-a-priori-estimates-lemma-error-form-1}
		\end{align}
		where the $\mathscr{Q}_{\left(i,j\right)}$ are defined in \eqref{linear-transport-operator-phi-mass-conditions}, $M>0$ is an absolute constant defined in \eqref{power-semilinear-problem-R2}, and $\nu>3/4$.
	\end{lemma}
	\begin{proof}
		As above, define $\hat{\mathscr{W}}_{j}=\hat{\mathscr{W}}_{j}\left(z,t\right)\coloneqq\mathscr{W}_{j}\left(\mathscr{V}\left(q\right)+a_{*}\right)$, $j=0,1,2$. To begin with, we test \eqref{projected-linear-transport-equation} with 1, and integrate on $\mathbb{H}_{-q'}$. The first term on the left hand side gives us $\mathscr{Q}_{\left(0,0\right)}$ due to assumption \eqref{linear-transport-operator-phi-mass-conditions} and integration by parts means the second term on the left hand side disappears. The second term on the right hand side also disappears due to oddness in $z_{1}$. We obtain from \eqref{power-semilinear-problem-R2}, Theorem \ref{vortex-pair-properties-theorem}, and \eqref{f-prime-q-variation-1}--\eqref{dq-Psi-R-variation}
		\begin{align}
			\mathscr{Q}_{\left(0,0\right)}+\int_{\mathbb{H}_{-q'}}E\ dz&=c_{R0}\left(\int_{\mathbb{H}_{-q'}}U_{R}\ dz\right)+c_{R2}\left(\int_{\mathbb{H}_{-q'}}\left(f_{\varepsilon}'\right)^{+}\mathscr{Z}_{2}\left(q\right)dz\right)\nonumber\\
			&=c_{R0}\left(M+\bigO{\left(\varepsilon^{2}\right)}\right)+c_{R2}\left(\bigO{\left(\varepsilon^{2}\right)}\right).\label{projected-linear-transport-problem-coefficients-of-projection-a-priori-estimates-lemma-1}
		\end{align}
		Next, we test \eqref{projected-linear-transport-equation} with $z_{1}$, and integrate on $\mathbb{H}_{-q'}$. The first term on the left hand side gives $\mathscr{Q}_{\left(1,0\right)}$ due to \eqref{linear-transport-operator-phi-mass-conditions}. The first and third terms on the right hand side will disappear because multiplying by $z_{1}$ will make the integrands odd in $z_{1}$. We are left with
		\begin{align}
			&\mathscr{Q}_{\left(1,0\right)}+\int_{\mathbb{H}_{-q'}}z_{1}E\ dz+\int_{\mathbb{H}_{-q'}}z_{1}\gradperp_{z}\left(\mathscr{V}\left(q\right)+a_{*}+a\right)\cdot\nabla_{z}\left(\phi-\hat{\mathscr{W}}_{1}\psi\right)dz\label{projected-linear-transport-problem-coefficients-of-projection-a-priori-estimates-lemma-3}\\
			&=c_{R1}\left(\int_{\mathbb{H}_{-q'}}z_{1}\left(f_{\varepsilon}'\right)^{+}\mathscr{Z}_{1}\left(q\right)dz\right).\nonumber
		\end{align}
		Similarly to \eqref{projected-linear-transport-problem-coefficients-of-projection-a-priori-estimates-lemma-1}, we can write the factor multiplying $c_{R1}$ on the right hand side of \eqref{projected-linear-transport-problem-coefficients-of-projection-a-priori-estimates-lemma-3} as
		\begin{align}
			\int_{\mathbb{H}_{-q'}}z_{1}\left(f_{\varepsilon}'\right)^{+}\mathscr{Z}_{1}\left(q\right)dz&=\int_{\mathbb{H}_{-q'}}z_{1}\dell_{1}\left(\left(\Gamma\right)^{\gamma}_{+}\right)+\bigO{\left(\varepsilon^{2}\right)}=-M+\bigO{\left(\varepsilon^{2}\right)},\label{projected-linear-transport-problem-coefficients-of-projection-a-priori-estimates-lemma-4}
		\end{align}
		For the second term on the left hand side of \eqref{projected-linear-transport-problem-coefficients-of-projection-a-priori-estimates-lemma-3}, we use integration by parts to obtain
		\begin{align}
			&\int_{\mathbb{H}_{-q'}}z_{1}\gradperp_{z}\left(\mathscr{V}\left(q\right)+a_{*}+a\right)\cdot\nabla_{z}\left(\phi-\hat{\mathscr{W}}_{1}\psi\right)dz\nonumber\\
			&=\int_{\mathbb{H}_{-q'}}\dell_{z_{2}}\left(\mathscr{V}\left(q\right)+a_{*}\right)\left(\Delta_{z}\psi+\hat{\mathscr{W}}_{1}\psi\right)dz-\int_{\mathbb{H}_{-q'}}\dell_{z_{2}}\left(a\right)\left(\phi-\hat{\mathscr{W}}_{1}\psi\right)dz.\label{projected-linear-transport-problem-coefficients-of-projection-a-priori-estimates-lemma-5}
		\end{align}
		Here we used the fact that $-\Delta_{z}\psi=\phi$ to obtain the first term on the right hand side of \eqref{projected-linear-transport-problem-coefficients-of-projection-a-priori-estimates-lemma-5}. We have also used the fact that $\phi$, and $\hat{\mathscr{W}}_{1}$ have compact support to see that all boundary terms are $0$.
		
		\medskip
		Then for the second term on the right hand side of \eqref{projected-linear-transport-problem-coefficients-of-projection-a-priori-estimates-lemma-5}, we obtain, for some $r_{1}>2$,
		\begin{align}
			&\left|\int_{\mathbb{H}_{-q'}}\dell_{z_{2}}\left(a\right)\left(\phi-\hat{\mathscr{W}}_{1}\psi\right)dz\right|\label{projected-linear-transport-problem-coefficients-of-projection-a-priori-estimates-lemma-5a}\\
			&\leq C\|\nabla a\|_{L^{r_{1}}\left(\mathbb{H}_{-q'}\right)}\|\phi\|_{L^{r^{*}_{1}}\left(\mathbb{H}_{-q'}\right)}+\|\nabla a\|_{L^{r_{1}}\left(\mathbb{H}_{-q'}\right)}\|\psi\|_{L^{\infty}\left(\mathbb{H}_{-q'}\right)}\nonumber\\
			&\leq\bigO{\left(\frac{\varepsilon^{2+\nu}}{t^{2}}\right)}\|\phi\|_{L^{2}\left(\mathbb{H}_{-q'}\right)}+\bigO{\left(\frac{\varepsilon^{2+\nu}}{t^{2}}\right)}\left(\left|\log{\varepsilon}\right|\left|\mathcal{J}_{\left(0,0\right)}\right|\right)\nonumber\\
			&\leq\bigO{\left(\frac{\varepsilon^{2+\nu}}{t^{2}}\right)}\|\phi\|_{L^{2}\left(\mathbb{H}_{-q'}\right)}+\bigO{\left(\frac{\varepsilon^{5+\nu}\left|\log{\varepsilon}\right|}{t^{4}}\right)},\nonumber,
		\end{align}
		where the H\"older conjugate $r_{1}^{*}$ of $r_{1}$ satisfies $r^{*}_{1}<2$, and finally we use the bounds for $\mathcal{J}_{\left(i,j\right)}$ defined in \eqref{linear-transport-operator-phi-mass-conditions}, and \eqref{linear-transport-operator-a-bounds-0} as well as Lemma \ref{poisson-equation-zero-mass-estimates-lemma} to obtain the last line of \eqref{projected-linear-transport-problem-coefficients-of-projection-a-priori-estimates-lemma-5a}. 
		
		\medskip
		Concentrating now on the first term on the right hand side of \eqref{projected-linear-transport-problem-coefficients-of-projection-a-priori-estimates-lemma-5}, we first note that since we defined $\hat{\mathscr{W}}_{j}=\hat{\mathscr{W}}_{j}\left(z,t\right)\coloneqq\mathscr{W}_{j}\left(\mathscr{V}\left(q\right)+a_{*}\right)$, using \eqref{f-prime-q-variation-1}--\eqref{tilde-W-k-def} and integration by parts, we obtain 
		\begin{align}
			&\int_{\mathbb{H}_{-q'}}\dell_{z_{2}}\left(\mathscr{V}\left(q\right)+a_{*}\right)\left(\Delta_{z}\psi+\hat{\mathscr{W}}_{1}\psi\right)dz\label{projected-linear-transport-problem-coefficients-of-projection-a-priori-estimates-lemma-6}\\
			&=\int_{\mathbb{H}_{-q'}}\psi\dell_{z_{2}}\left(\Delta_{z}\left(\mathscr{V}\left(q\right)+a_{*}\right)+\hat{\mathscr{W}}_{0}\right)dz=\bigO{\left(\frac{\varepsilon^{2+\nu}}{t^{2}}\right)}\|\phi\|_{L^{2}\left(\mathbb{H}_{-q'}\right)}+\bigO{\left(\frac{\varepsilon^{5+\nu}\left|\log{\varepsilon}\right|}{t^{4}}\right)},\nonumber
		\end{align}
		where we have used Lemma \ref{poisson-equation-zero-mass-estimates-lemma} and \eqref{a-star-extra-assumption-statement-1} to obtain the right hand side of \eqref{projected-linear-transport-problem-coefficients-of-projection-a-priori-estimates-lemma-6}.
		
		\medskip
		Finally, we test \eqref{projected-linear-transport-equation} against $z_{2}$ and integrate. The first term on the left hand side gives $\mathscr{Q}_{\left(0,1\right)}$ due to assumption \eqref{linear-transport-operator-phi-mass-conditions}. The second term on the right hand side disappears because the integrand preserves its oddness in $z_{1}$ after multiplication by $z_{2}$. Analogously to \eqref{projected-linear-transport-problem-coefficients-of-projection-a-priori-estimates-lemma-4}--\eqref{projected-linear-transport-problem-coefficients-of-projection-a-priori-estimates-lemma-6}, we obtain
		\begin{align}
			&\mathscr{Q}_{\left(0,1\right)}+\int_{\mathbb{H}_{-q'}}z_{2}E\ dz+ \bigO{\left(\frac{\varepsilon^{2+\nu}}{t^{2}}\right)}\|\phi\|_{L^{2}\left(\mathbb{H}_{-q'}\right)}+\bigO{\left(\frac{\varepsilon^{5+\nu}\left|\log{\varepsilon}\right|}{t^{4}}\right)}\label{projected-linear-transport-problem-coefficients-of-projection-a-priori-estimates-lemma-7}\\
			&=c_{R0}\bigO{\left(\varepsilon^{2}\right)}+c_{R2}\left(M+\bigO{\left(\varepsilon^{2}\right)}\right).\nonumber
		\end{align}
		Noting that $\left(c_{R0},c_{R1},c_{R2}\right)$ satisfy the linear system given by \eqref{projected-linear-transport-problem-coefficients-of-projection-a-priori-estimates-lemma-1}, \eqref{projected-linear-transport-problem-coefficients-of-projection-a-priori-estimates-lemma-3}, and \eqref{projected-linear-transport-problem-coefficients-of-projection-a-priori-estimates-lemma-7}, and we have estimates on the coefficients and errors for the linear system given by \eqref{projected-linear-transport-problem-coefficients-of-projection-a-priori-estimates-lemma-1}, \eqref{projected-linear-transport-problem-coefficients-of-projection-a-priori-estimates-lemma-4}, \eqref{projected-linear-transport-problem-coefficients-of-projection-a-priori-estimates-lemma-7}, and \eqref{projected-linear-transport-problem-coefficients-of-projection-a-priori-estimates-lemma-7}, and \eqref{projected-linear-transport-problem-coefficients-of-projection-a-priori-estimates-lemma-5}--\eqref{projected-linear-transport-problem-coefficients-of-projection-a-priori-estimates-lemma-6}, and \eqref{projected-linear-transport-problem-coefficients-of-projection-a-priori-estimates-lemma-7} respectively, we obtain a linear system for $\left(c_{R0},c_{R1},c_{R2}\right)$ of the form \eqref{projected-linear-transport-problem-coefficients-of-projection-a-priori-estimates-lemma-linear-system-form}--\eqref{projected-linear-transport-problem-coefficients-of-projection-a-priori-estimates-lemma-error-form-1}, as required.
	\end{proof}
	
	\subsection{A Priori Estimates on Solutions to the Homotopic Operators}\label{phi-a-priori-estimates-section}
	We now obtain a priori estimates to solutions $\tilde{\phi}_{*R}$ solving \eqref{tilde-phi-star-R-initial-value-problem} given by
	\begin{align*}
		&\mathscr{E}_{R,\lambda}\left(\tilde{\phi}_{*R}, \lambda\alpha,\psi^{out}_{*},\lambda\Tilde{\xi}\right)=c_{R0}\left(t\right)U_{R}+c_{R1}\left(t\right)\left(f_{\varepsilon}'\right)^{+}\mathscr{Z}_{1}\left(q_{\lambda}\right)\nonumber\\
		&+c_{R2}\left(t\right)\left(f_{\varepsilon}'\right)^{+}\mathscr{Z}_{2}\left(q_{\lambda}\right)\ \ \text{in}\ \mathbb{H}_{-q'_{\lambda}}\times \left[T_{0},T\right],\\
		&\tilde{\phi}_{*R}(\blank,T)=0\ \ \text{in}\ \mathbb{H}_{-q'_{\lambda}},\quad \tilde{\phi}_{*R}\left(z_{\lambda},t\right)=0\ \ \text{on}\ \dell\mathbb{H}_{-q_{\lambda}'}\times\left[T_{0},T\right],
	\end{align*}
	with $\left(\lambda\alpha,\psi^{out}_{*},\lambda\Tilde{\xi}\right)$ given, and sufficiently small in the sense of \eqref{deformed-open-ball-definition} below, and $\tilde{\phi}_{*R}$ satisfying the orthogonality conditions \eqref{tilde-phi-star-R-0-mass-condition}, which correspond to satisfying assumptions \eqref{linear-transport-operator-phi-mass-conditions} for solutions to our linearized operator.
	
	\medskip
	From the computations \eqref{linear-operator-structure-motivation-1}--\eqref{linear-operator-structure-motivation-8}, we know that we can write \eqref{homotopic-operator-inner-error-main-definition} as
	\begin{align}
		&\mathscr{E}_{R,\lambda}\left(\tilde{\phi}_{*R}, \lambda\alpha,\psi^{out}_{*},\lambda\Tilde{\xi}\right)=\varepsilon^{2}\dell_{t}\tilde{\phi}_{*R}\left(z_{\lambda},t\right)+\lambda\tilde{\mathscr{E}}_{R}\left(\tilde{\psi}_{*R}, \lambda\alpha,\psi^{out}_{*},\lambda\tilde{\xi}\right)\label{homotopic-operators-a-priori-estimates-1}\\
		&+\gradperp_{z_{\lambda}}\left(\mathscr{V}\left(q_{\lambda}\right)+\lambda\mathfrak{a}_{R,\lambda}\right)\cdot\nabla_{z_{\lambda}}\left(\tilde{\phi}_{*R}-\mathscr{W}_{1}\left(\mathscr{V}\left(q_{\lambda}\right)+\lambda\mathfrak{A}_{R,\lambda,*}\right)\tilde{\psi}_{*R}\right)\nonumber\\
		&+\mathfrak{R}_{1}+\varepsilon^{2}\dot{\alpha}\left(t\right)U_{R}\left(z_{\lambda},t\right)-\varepsilon\left(\mathcal{N}\left(\xi_{0}+\xi_{1}\right)\left[\lambda\tilde{\xi}\right]+\lambda\dot{\tilde{\xi}}\right)\cdot\grad_{z_{\lambda}}U_{R}\nonumber
	\end{align}
	where $\mathfrak{a}_{R,\lambda}$ is defined in \eqref{homotopic-operator-inner-error-E-star-R-definition-10}. Further to \eqref{a-star-extra-assumption-justification-2} $\mathfrak{A}_{R,\lambda,*}$, and therefore $\mathfrak{A}_{R,\lambda}$ are defined as:
	\begin{align}
		\mathfrak{A}_{R,\lambda,*}&\coloneqq-\mathcal{E}_{2}\left(z,\xi_{0}\right)+\psi_{R1,2}\left(z,\xi_{0}\right)\nonumber\\
		\mathfrak{A}_{R,\lambda}\left(\tilde{\psi}_{*R}, \lambda\alpha,\psi^{out}_{*},\lambda\tilde{\xi}\right)&\coloneqq\mathfrak{a}_{R,\lambda}\left(\tilde{\psi}_{*R}, \lambda\alpha,\psi^{out}_{*},\lambda\tilde{\xi}\right)-\mathfrak{A}_{R,\lambda,*}\left(\xi_{0}+\xi_{1}\right).\label{homotopic-operators-a-priori-estimates-3}
	\end{align}
	We can also use \eqref{linear-operator-structure-motivation-1}--\eqref{linear-operator-structure-motivation-8} to compare the first two terms of \eqref{homotopic-operators-a-priori-estimates-1} to the first three terms on the right hand side of \eqref{homotopic-operator-inner-error-main-definition} and obtain that
	\begin{align}
		&\mathfrak{R}_{1}=\gradperp_{z}\left(\mathscr{V}\left(q_{\lambda}\right)+\lambda\mathfrak{a}_{R,\lambda}\right)\cdot\nabla_{z}\left(\left(\mathscr{W}_{1}\left(\mathscr{V}\left(q_{\lambda}\right)+\lambda\mathfrak{A}_{R,\lambda,*}\right)\right)\tilde{\psi}_{*R}\right)\label{homotopic-operators-a-priori-estimates-3a}\\
		&-\gradperp_{z}\left(\mathscr{V}\left(q_{\lambda}\right)+\lambda\mathfrak{a}_{R,\lambda}\right)\cdot\nabla_{z}\left(\left(\left(f_{\varepsilon}'\right)^{+}+\left(f_{\varepsilon}''\right)^{+}\left(\frac{\lambda\phi_{R1}}{\left(f_{\varepsilon}'\right)^{+}}\right)\right)\tilde{\psi}_{*R}\right)\nonumber\\
		&+\lambda\gradperp_{z}\left(\mathfrak{a}_{R,\lambda}-\frac{\phi_{R1}}{\left(f_{\varepsilon}'\right)^{+}}\right)\cdot\nabla_{z}\left(\left(f_{\varepsilon}'\right)^{+}\tilde{\psi}_{*R}\right)\nonumber\\
		&+\lambda\gradperp_{z}\left(\mathfrak{a}_{R,\lambda}\right)\cdot\nabla_{z}\left(\left(f_{\varepsilon}''\right)^{+}\left(\frac{\phi_{R1}}{\left(f_{\varepsilon}'\right)^{+}}\right)\tilde{\psi}_{*R}\right)+\lambda\gradperp_{z}\tilde{\psi}_{*R}\cdot\nabla_{z}\phi_{R2}.\nonumber
	\end{align}
	It is important to note that upon Taylor expanding the term multiplying $\tilde{\psi}_{*R}$ under the gradient on the first term of the right hand side of \eqref{homotopic-operators-a-priori-estimates-3a}, we obtain cancellation that lets us say
	\begin{align}
		\mathfrak{R}_{1}=\lambda\tilde{\mathfrak{R}}_{1}\label{homotopic-operators-a-priori-estimates-3b}
	\end{align}
	for $\tilde{\mathfrak{R}}_{1}$ bounded as $\lambda\to0$. It is also important to note that using the compact support of the integrand, integration by parts, and specifically statements \eqref{poisson-equation-zero-mass-estimates-lemma-statement-1}--\eqref{poisson-equation-zero-mass-estimates-lemma-statement-2} in Lemma \ref{poisson-equation-zero-mass-estimates-lemma}, we have, for integers $\left(i,j\right)$ with $0\leq i+j\leq 1$, and for $\sigma>0$ arbitrarily small, the estimate
	\begin{align}
		\int_{\mathbb{H}_{-q'}}z_{1}^{i}z_{2}^{j}\lambda\tilde{\mathfrak{R}}_{1}dz=\lambda\bigO{\left(\|\phi\|_{L^{2}\left(\mathbb{H}_{-q'}\right)}^{2}+\frac{\varepsilon^{5-\sigma}}{t^{4}}\right)}.\label{homotopic-operators-a-priori-estimates-3c}
	\end{align}
	By \eqref{a-star-extra-assumption-justification-2} and \eqref{homotopic-operators-a-priori-estimates-3}, we have $ a_{*}=\lambda\mathfrak{A}_{R,\lambda,*}$ and $a=\lambda\mathfrak{A}_{R,\lambda}$, and if we work in the parameter space where $a$ satisfies \eqref{linear-transport-operator-a-bounds-0} and \eqref{linear-transport-operator-a-log-lipschitz-bound}, we have that $\left(a,a_{*}\right)$ satisfy \eqref{linear-transport-operator-Laplacian-a-star-a-bounded}--\eqref{linear-transport-operator-a-log-lipschitz-bound} as well \eqref{a-star-extra-assumption-statement-1} via the computation \eqref{a-star-extra-assumption-justification-3}. With $b=\lambda\tilde{\xi}$ with $\tilde{\xi}$ satisfying \eqref{p-tilde-bound-2}, we have that $b$ satisfies \eqref{b-1-bound}.
	
	\medskip
	If we define the error $\mathfrak{E}_{\lambda}$ as being 
	\begin{align}
		\mathfrak{E}_{\lambda}&=\lambda\tilde{\mathfrak{R}}_{1}+\lambda\varepsilon^{2}\dot{\alpha}\left(t\right)U_{R}\left(z_{\lambda},t\right)-\varepsilon\left(\mathcal{N}\left(\xi_{0}+\xi_{1}\right)\left[\lambda\tilde{\xi}\right]+\lambda\dot{\tilde{\xi}}\right)\cdot\grad_{z_{\lambda}}U_{R}\label{homotopic-operators-a-priori-estimates-6}\\
		&+\lambda\tilde{\mathscr{E}}_{R}\left(\tilde{\psi}_{*R}, \lambda\alpha,\psi^{out}_{*},\lambda\tilde{\xi}\right)\nonumber
	\end{align}
	then $\mathfrak{E}_{\lambda}$ satisfies the assumption on its support \eqref{linear-transport-operator-error-solution-supports}, as well as the bound \eqref{L2-interior-a-priori-estimate-lemma-error-bound-statement} for $\left(\tilde{\psi}_{*R},\alpha,\psi^{out}_{*},\lambda\Tilde{\xi}\right)$ sufficiently small. 
	
	\medskip
	Now, using \eqref{homotopic-operator-inner-error-main-definition} and \eqref{homotopic-operator-inner-error-main-definition-lower-order-terms}, \eqref{tilde-phi-star-R-initial-value-problem-equation} can be written in the form, for $a_{*}$ and $a$ defined in \eqref{homotopic-operators-a-priori-estimates-3}
	\begin{align}
		&\varepsilon^{2}\dell_{t}\tilde{\phi}_{*R}\left(z_{\lambda},t\right)+\gradperp_{z_{\lambda}}\left(\mathscr{V}\left(q_{\lambda}\right)+a_{*}+a\right)\cdot\nabla_{z_{\lambda}}\tilde{\phi}_{*R}\nonumber\\
		&=-\tilde{\mathfrak{E}}_{\lambda}+c_{R0}\left(t\right)U_{R}+c_{R1}\left(t\right)\left(f_{\varepsilon}'\right)^{+}\mathscr{Z}_{1}(q_{\lambda})+c_{R2}\left(t\right)\left(f_{\varepsilon}'\right)^{+}\mathscr{Z}_{2}(q_{\lambda})=\mathfrak{E}_{\lambda}',\label{homotopic-operators-a-priori-estimates-7}
	\end{align}
	with error term $\tilde{\mathfrak{E}}_{\lambda}$ being the sum of the second, third, fifth, and sixth terms on the right hand side of \eqref{homotopic-operator-inner-error-main-definition}. Thus, from \eqref{first-approximation-construction-12}, \eqref{first-approximation-construction-40a0}, $\mathfrak{E}_{\lambda}'$ can be written as
	\begin{align*}
		\mathfrak{E}_{\lambda}'=U_{R}\left(z_{\lambda},t\right)\mathfrak{E}_{0}+\left(f_{\varepsilon}'\right)^{+}\left(z_{\lambda},t\right)\mathfrak{E}_{1}+\left(f_{\varepsilon}''\right)^{+}\left(z_{\lambda},t\right)\mathfrak{E}_{2}+\left(f_{\varepsilon}'''\right)^{+}\left(z_{\lambda},t\right)\mathfrak{E}_{3}
	\end{align*}
	Then by \eqref{f-prime-q-variation-1}--\eqref{tilde-W-k-def}, this becomes
	\begin{align}
		&\mathfrak{E}_{\lambda}'=\left(\mathscr{V}\left(z_{\lambda},q_{\lambda}\right)\right)^{\gamma}_{+}\chi_{B_{\frac{s}{\varepsilon}}\left(0\right)}(z)\mathfrak{E}_{0}+\gamma\left(\mathscr{V}\left(z_{\lambda},q_{\lambda}\right)\right)^{\gamma-1}_{+}\chi_{B_{\frac{s}{\varepsilon}}\left(0\right)}(z)\mathfrak{E}_{1}\nonumber\\
		&+\gamma\left(\gamma-1\right)\left(\mathscr{V}\left(z_{\lambda},q_{\lambda}\right)\right)^{\gamma-2}_{+}\chi_{B_{\frac{s}{\varepsilon}}\left(0\right)}(z)\mathfrak{E}_{2}\nonumber\\
		&+\gamma\left(\gamma-1\right)\left(\gamma-2\right)\left(\mathscr{V}\left(z_{\lambda},q_{\lambda}\right)\right)^{\gamma-3}_{+}\chi_{B_{\frac{s}{\varepsilon}}\left(0\right)}(z)\mathfrak{E}_{3}.\label{E-prime-form}
	\end{align}
	We now show that $\tilde{\phi}_{*R}$ satisfies the support assumption \eqref{linear-transport-operator-error-solution-supports}. 
	\begin{lemma}\label{phi-support-lemma}
		Suppose $\tilde{\phi}_{*R}$ solves \eqref{tilde-phi-star-R-initial-value-problem} with $\left(\lambda\alpha,\psi^{out}_{*},\lambda\Tilde{\xi}\right)$ given, and sufficiently small in the sense of \eqref{deformed-open-ball-definition}, and $\tilde{\phi}_{*R}$ further satisfies the orthogonality conditions \eqref{tilde-phi-star-R-0-mass-condition}. Then for all $\varepsilon>0$ small enough, $\tilde{\phi}_{*R}$ satisfies \eqref{linear-transport-operator-error-solution-supports}.
	\end{lemma}
	\begin{proof}
		We use a continuity argument similar to \eqref{bar-z-5-rho-bound}--\eqref{5-rho-proof-5}. Making use of the representation formula related to the characteristics defined in \eqref{characteristics-ode}, we have
		\begin{align}
			\tilde{\phi}_{*R}\left(z_{\lambda},t\right)=\varepsilon^{-2}\int_{t}^{T}\mathfrak{E}_{\lambda}'\left(\bar{z}(\tau,t,z_{\lambda}),\tau\right)d\tau.\label{tilde-phi-representation}
		\end{align}
		We note that by \eqref{E-prime-form} and \eqref{tilde-phi-representation},
		\begin{align*}
			z_{\lambda}&\in\bigcap_{\tau\in\left[t,T\right]}\left(\left\{z\colon \mathscr{V}\left(\bar{z}(\tau,t,z),q_{\lambda}\left(\tau\right)\right)<0\right\}\cup\left\{z\colon \left|\bar{z}(\tau,t,z)\right|>\frac{s}{\varepsilon}\right\}\right)\\
			&\implies \tilde{\phi}_{*R}\left(z_{\lambda},t\right)=0.
		\end{align*}
		Thus
		\begin{align}
			&\left\{z_{\lambda}\colon \tilde{\phi}_{*R}\left(z_{\lambda},t\right)\neq0\right\}\nonumber\\
			&\subset\bigcup_{\tau\in\left[t,T\right]}\left(\left\{z\colon \mathscr{V}\left(\bar{z}(\tau,t,z),q_{\lambda}\left(\tau\right)\right)\geq0\right\}\cap\left\{z\colon \left|\bar{z}(\tau,t,z)\right|\leq\frac{s}{\varepsilon}\right\}\right).\label{set-inclusion-0}
		\end{align}
		We claim that for all $\tau\in\left[t,T\right]$, 
		\begin{align}
			&\left\{z\colon \mathscr{V}\left(\bar{z}(\tau,t,z),q_{\lambda}\left(\tau\right)\right)\geq0\right\}\cap\left\{z\colon \left|\bar{z}(\tau,t,z)\right|\leq\frac{s}{\varepsilon}\right\}\nonumber\\
			&\subset\bigcap_{\tau'\in\left[t,T\right]}\left\{z\colon \left|\bar{z}\left(\tau',t,z\right)\right|\leq 3\rho\right\}.\label{set-inclusion}
		\end{align}
		So fix $t\in\left[T_{0},T\right]$. Then fix $\tau_{0}\in\left[t,T\right]$, and choose $z_{0}\in\mathbb{H}_{-q'}$ so that 
		\begin{align*}
			z_{0}\in\left\{z\colon \mathscr{V}\left(\bar{z}(\tau_{0},t,z),q_{\lambda}\left(\tau_{0}\right)\right)\geq0\right\}\cap\left\{z\colon \left|\bar{z}(\tau_{0},t,z)\right|\leq\frac{s}{\varepsilon}\right\}.
		\end{align*}
		Then, since $\mathscr{V}\left(\bar{z}(\tau_{0},t,z_{0}),q_{\lambda}\left(s\right)\right)\geq0$ in the region $\left|\bar{z}(\tau_{0},t,z_{0})\right|\leq\frac{s}{\varepsilon}$, by Remark \ref{support-of-nonlinearity-remark}, and \eqref{vortex-pair-variation-generalisation}, we have that $\left|\bar{z}(\tau_{0},t,z_{0})\right|\leq \rho$, $\rho>0$ the absolute constant defined in \eqref{vortex-pair-support-true}. So given that $\left|\bar{z}(\tau_{0},t,z_{0})\right|\leq \rho$, let $\left[\tau_{*},\tau^{*}\right]\subseteq\left[t,T\right]\subseteq\left[T_{0},T\right]$ be the largest interval such that 
		\begin{align*}
			\left|\bar{z}(\tau,t,z_{0})\right|\leq 3\rho\ \forall \tau\in\left[\tau_{*},\tau^{*}\right],
		\end{align*}
		an interval which manifestly exists due to \eqref{continuity-of-characteristics-0}. Then if $\tau_{*}\neq t$, we know by continuity that
		\begin{align*}
			\left|\bar{z}(\tau_{*},t,z_{0})\right|= 3\rho.
		\end{align*}
		Arguing as in \eqref{5-rho-proof-1}--\eqref{5-rho-proof-3b}, and noting that $\mathfrak{A}_{R,\lambda,*}$ and $\mathfrak{A}_{R,\lambda}\left(\tilde{\psi}_{*R}, \lambda\alpha,\psi^{out}_{*},\lambda\tilde{\xi}\right)$ defined in \eqref{homotopic-operators-a-priori-estimates-3} satisfy the assumptions in \eqref{linear-transport-operator-a-star-bounds}--\eqref{linear-transport-operator-a-bounds-0} as the quantities they depend on are small in the sense of \eqref{deformed-open-ball-definition}, we obtain that
		\begin{align*}
			\left|\mathscr{V}\left(q\right)\left(\bar{z}\left(\tau_{*},t,z_{0}\right),\tau_{*}\right)-\mathscr{V}\left(q\right)\left(\bar{z}\left(\tau_{0},t,z_{0}\right),\tau_{0}\right)\right|\leq C\left(1+3\rho\right)^{6}\left(\frac{\varepsilon}{T_{0}}\right)^{\frac{3}{4}},
		\end{align*}
		using that $T_{0}\leq t\leq \tau_{*}$. Now $\left|\bar{z}(\tau_{0},t,z_{0})\right|\leq \rho$, so $\mathscr{V}\left(q\right)\left(\bar{z}\left(\tau_{0},t,z_{0}\right),\tau_{0}\right)\geq0$, meaning
		\begin{align*}
			\mathscr{V}\left(q\right)\left(\bar{z}\left(\tau_{*},t,z_{0}\right),\tau_{*}\right)&=\mathscr{V}\left(q\right)\left(\bar{z}\left(\tau_{0},t,z_{0}\right),\tau_{0}\right)\\
			&+\left(\mathscr{V}\left(q\right)\left(\bar{z}\left(\tau_{*},t,z_{0}\right),\tau_{*}\right)-\mathscr{V}\left(q\right)\left(\bar{z}\left(\tau_{0},t,z_{0}\right),\tau_{0}\right)\right)\\
			&\geq -C\left(1+3\rho\right)^{6}\left(\frac{\varepsilon}{T_{0}}\right)^{\frac{3}{4}},
		\end{align*}
		which by continuity implies that for all $\varepsilon>0$ small enough, $\left|\bar{z}\left(\tau_{*},t,z_{0}\right)\right|\leq 2\rho$, a contradiction, implying that $\tau_{*}=t$, and via analogous proof, $\tau^{*}=T$. Thus,
		\begin{align*}
			\left|\bar{z}(\tau,t,z_{0})\right|\leq 3\rho\ \forall \tau\in\left[t,T\right],
		\end{align*}
		and since $z_{0}\in\mathbb{H}_{-q'}$ was arbitrary, we obtain \eqref{set-inclusion}, which along with \eqref{set-inclusion-0} implies \eqref{linear-transport-operator-error-solution-supports} for $\tilde{\phi}_{*R}$ as required.
	\end{proof}
	Finally, with all the assumptions of Lemma \ref{projected-linear-transport-problem-coefficients-of-projection-a-priori-estimates-lemma}, we have the following estimates
	\begin{align}
		\|\tilde{\phi}_{*R}\|_{L^{2}\left(\mathbb{H}_{-q'_{\lambda}}\right)}\leq \frac{C\varepsilon^{3-\sigma_{3}}}{t^{2}},\quad \|\tilde{\phi}_{*R}\|_{L^{\infty}\left(\mathbb{H}_{-q'_{\lambda}}\right)}\leq \frac{C\varepsilon^{1-\sigma_{3}}}{t^{1-\sigma_{3}}},\label{phi-star-R-L2-Linfinity-estimates}
	\end{align}
	for all $t\in\left[T_{0},T\right]$ and some small $\sigma_{3}>0$.
	\subsection{A Priori Estimates on the Poisson Equation}
	Having obtained a priori estimates for $\tilde{\phi}_{*R}$, we now obtain a priori estimates for $\psi^{out}_{*}$ solving
	\begin{align}
		-\Delta_{x}\psi_{*}^{out}=\lambda\left(\sum_{j=R,L}\left(-1\right)^{j}\left(\psi_{*j}\Delta_{x}\eta^{(j)}_{K}+2\nabla_{x}\eta^{(j)}_{K}\cdot\nabla_{x}\psi_{*j}\right)+E_{2}^{out}\right)\coloneqq\mathscr{S}_{*}.\label{psi-out-star-equation}
	\end{align}
	As in \eqref{psi-out-integrand region}, we have 
	\begin{align*}
		\psi_{*R}(y)&=\frac{1}{2\pi}\int_{\mathbb{R}^{2}_{+}}\log{\left(\frac{\left|w-\bar{y}\right|}{\left|w-y\right|}\right)}\phi_{*R}(w)\ dw\\
		&=\frac{1}{2\pi}\int_{\mathbb{R}^{2}_{+}}\left(\mathscr{L}_{1}+\mathscr{L}_{2}\right)\left(\tilde{\phi}_{*R}(w)+\lambda\alpha U_{R}(w)\right)dw,
	\end{align*}
	where we recall the ansatz \eqref{phi-star-psi-star-homotopy}. Using the estimates on $\mathscr{L}_{1}$ and $\mathscr{L}_{2}$ as in \eqref{L2-definition}--\eqref{L1-definition} and \eqref{tilde-phi-star-R-0-mass-condition}, we obtain
	\begin{align*}
		&\frac{1}{2\pi}\int_{\mathbb{R}^{2}_{+}}\left(\mathscr{L}_{1}+\mathscr{L}_{2}\right)\tilde{\phi}_{*R}(w)\ dw\\
		&=\bigO{\left(\lambda\left(t^{-1}\left|\mathcal{J}_{\left(0,0\right)}\right|+\varepsilon t^{-1}\left(\left|\mathcal{J}_{\left(1,0\right)}\right|+\left|\mathcal{J}_{\left(0,1\right)}\right|\right)\right)+\varepsilon^{2}t^{-2}\|\tilde{\phi}_{*R}\|_{L^{2}\left(\mathbb{R}^{2}_{+}\right)}\right)},\\
		&\frac{\lambda\alpha}{2\pi}\int_{\mathbb{R}^{2}_{+}}\left(\mathscr{L}_{1}+\mathscr{L}_{2}\right)U_{R}(w)\ dw=\bigO{\left(\lambda t^{-1}\left|\alpha\right|\right)},\\
		&\nabla_{x}\psi_{*R}=\varepsilon^{-1}\nabla_{y}\psi_{*R}\\
		&=\bigO{\left(\lambda\left(t^{-2}\left|\mathcal{J}_{\left(0,0\right)}\right|+\varepsilon t^{-2}\left(\left|\mathcal{J}_{\left(1,0\right)}\right|+\left|\mathcal{J}_{\left(0,1\right)}\right|\right)+t^{-2}\left|\alpha\right|\right)+\varepsilon^{2}t^{-3}\|\tilde{\phi}_{*R}\|_{L^{2}\left(\mathbb{R}^{2}_{+}\right)}\right)}.
	\end{align*}
	Then, as in \eqref{psi-R1-psi-out-estimate}--\eqref{first-approximation-construction-22}, using \eqref{first-approximation-outer-error-E2out-final-bound} for the bound on $\lambda E^{out}_{2}$, $\mathscr{S}_{*}$ and $\psi^{out}_{*}$ have estimates
	\begin{align*}
		&\mathscr{S}_{*}(x,t)=\bigO{\left(\lambda\left(\varepsilon^{3}t^{-5}+\varepsilon^{2}t^{-4}\|\tilde{\phi}_{*R}\|_{L^{2}\left(\mathbb{R}^{2}_{+}\right)}+t^{-3}\left|\alpha\right|\right)\right)},\nonumber\\
		&\|\psi^{out}_{*}\|_{L^{\infty}\left(\mathbb{R}^{2}_{+}\right)}+\|\nabla\psi^{out}_{*}\|_{L^{\infty}\left(\mathbb{R}^{2}_{+}\right)}\leq\frac{C\lambda\varepsilon^{3}\left|\log{t}\right|}{t^{3}},
	\end{align*}
	where we have used \eqref{tilde-phi-star-R-0-mass-condition}, \eqref{alpha-bound}, and \eqref{phi-star-R-L2-Linfinity-estimates}.
	
	\medskip
	Then analogously to \eqref{psi-out-UR-estimate} and \eqref{psi-out-bound-2}, we also have
	\begin{align*}
		&\left|\varepsilon\gradperp_{y}\psi^{out}_{*}\left(\varepsilon y+pe_{1}\right)\cdot\nabla_{y}\left(U_{R}(y)+\phi_{R}(y)\right)\right|\\
		&\leq C\lambda\left(\varepsilon^{5}t^{-4}+\varepsilon^{4}t^{-3}\|\tilde{\phi}_{*R}\|_{L^{2}\left(\mathbb{R}^{2}_{+}\right)}+\varepsilon
		^{2}t^{-2}\left|\alpha\right|\right)\left(\left(f_{\varepsilon}'\right)^{+}+\left(f_{\varepsilon}''\right)^{+}+\left(f_{\varepsilon}'''\right)^{+}\right).
	\end{align*}
	
	\subsection{Fixed Point Formulation and Solution on $[T_{0},T]$}\label{fixed-point-formulation-section}
	We reformulate the system \eqref{psi-star-r-green-function-representation-upper-half-plane}, \eqref{tilde-phi-star-R-initial-value-problem}, \eqref{cRj-initial-value-problem}, \eqref{psi-out-star-boundary-value-problem} as a fixed point problem of the form \eqref{full-solution-construction-fixed-point-formulation} for a well chosen Banach space $\mathfrak{X}$. Recall that $\tilde{\phi}_{*L}$ is an odd reflection of $\tilde{\phi}_{*R}$ in the $x_{1}$ variable, and thus it is enough to solve the above problem.
	
	\medskip
	Fix a small number $\beta\in\left(0,1\right)$. For $\tilde{\phi}_{*R}\left(z_{\lambda},t\right)$ defined on $\mathbb{H}_{-q'_{\lambda}}\times\left[T_{0},T\right]$, we consider the set of functions with support on $B_{3\rho}\left(0\right)$ in the $z_{\lambda}$ variables defined in \eqref{homotopy-parameter-change-of-coordinates} for $\rho>0$ defined in \eqref{first-approximation-main-order-vorticity-support} with finite $\|\cdot\|_{i}$ norm, where 
	\begin{align*}
		\|\phi\|_{i}\coloneqq \sup_{\left[T_{0},T\right]}\left(t^{2}\|\phi\|_{L^{2}\left(B_{3\rho}\left(0\right)\right)}\right)+\sup_{\left[T_{0},T\right]}\left(\varepsilon^{2}t^{1-\beta}\|\phi\|_{L^{\infty}\left(B_{3\rho}\left(0\right)\right)}\right).
	\end{align*}
	Next we define the $\|\cdot\|_{o}$ norm for $\psi_{*}^{out}\left(x,t\right)$ defined on $\mathbb{R}^{2}\times\left[T_{0},T\right]$:
	\begin{align*}
		\|\psi\|_{o}\coloneqq \sup_{\left[T_{0},T\right]}\left(t^{3}\left|\log{t}\right|^{-1}\|\psi\|_{L^{\infty}\left(\mathbb{R}^{2}\right)}\right)+\sup_{\left[T_{0},T\right]}\left(t^{3}\left|\log{t}\right|^{-1}\|\nabla\psi\|_{L^{\infty}\left(\mathbb{R}^{2}\right)}\right).
	\end{align*}
	For $\tilde{\psi}_{*R}$ related to $\tilde{\phi}_{*R}$ via \eqref{psi-star-r-green-function-representation-upper-half-plane}, we define $\mathfrak{X}$ as the Banach space of all points $\mathfrak{P}=\left(\tilde{\phi}_{*R},\lambda\alpha,\psi_{*}^{out},\lambda\tilde{\xi}\right)$ such that $\nabla_{z_{\lambda}}\tilde{\psi}_{*R}$, $\nabla_{z_{\lambda}}\psi_{*}^{out}$, $\dot{\alpha}$, $\dot{\tilde{\xi}}$ all exist and are continuous, with $\tilde{\xi}=\left(\tilde{p},\tilde{q}\right)$, satisfying
	\begin{align*}
		\|\mathfrak{P}\|_{\mathfrak{X}}&\coloneqq\|\tilde{\phi}_{*R}\|_{i}+\|\psi_{*}^{out}\|_{o}+\|t^{2}\tilde{p}\|_{\left[T_{0},T\right]}+\|t^{3}\dot{\tilde{p}}\|_{\left[T_{0},T\right]}\nonumber\\
		&+\|t^{3}\tilde{q}\|_{\left[T_{0},T\right]}+\|t^{4}\dot{\tilde{q}}\|_{\left[T_{0},T\right]}+\|t^{3}\alpha\|_{\left[T_{0},T\right]}+\|t^{4}\dot{\alpha}\|_{\left[T_{0},T\right]}<\infty,
	\end{align*}
	with $\tilde{\phi}_{*R}$ also satisfying $\supp{\tilde{\phi}_{*R}}\subset B_{3\rho}\left(0\right)$.
	
	\medskip
	Then in $\mathfrak{X}$, we define the open set $\mathscr{O}$ as the ``deformed ball" around $0$ in $\mathscr{X}$, that is all $\mathfrak{P}$ satisfying
	\begin{align}
		&\|\tilde{\phi}_{*R}\|_{i}<\varepsilon^{3-3\beta}, \quad \|\psi_{*}^{out}\|_{o}<\varepsilon^{3-3\beta},\quad \|t^{3}\alpha\|_{\left[T_{0},T\right]}+\|t^{4}\dot{\alpha}\|_{\left[T_{0},T\right]}<\varepsilon^{3-3\beta},\label{deformed-open-ball-definition}\\
		&\|t^{2}\tilde{p}\|_{\left[T_{0},T\right]}+\|t^{3}\dot{\tilde{p}}\|_{\left[T_{0},T\right]}<\varepsilon^{4-3\beta}, \quad \|t^{3}\tilde{q}\|_{\left[T_{0},T\right]}+\|t^{4}\dot{\tilde{q}}\|_{\left[T_{0},T\right]}<\varepsilon^{4-3\beta},\nonumber
	\end{align}
	with $\tilde{\phi}_{*R}$ also satisfying the orthogonality conditions \eqref{tilde-phi-star-R-0-mass-condition}.
	
	\medskip
	Now we rewrite the system \eqref{psi-star-r-green-function-representation-upper-half-plane}, \eqref{tilde-phi-star-R-initial-value-problem}, \eqref{cRj-initial-value-problem}, \eqref{psi-out-star-boundary-value-problem} as a fixed point problem by finding an operator $\mathcal{T}$ such that solving the system is equivalent to solving $\mathcal{T}\left(\mathfrak{P},\lambda\right)=\mathfrak{P}$ for $\mathfrak{P}\in\mathscr{O}$, and for all $\lambda\in\left[0,1\right]$.
	
	\medskip
	We begin with \eqref{tilde-phi-star-R-initial-value-problem}. For any $\hat{a}$ satisfying $\Delta_{z_{\lambda}}\hat{a}\in L^{\infty}\left(\mathbb{H}_{-q_{\lambda}'}\times\left[T_{0},T\right]\right)$, and bounded $E\left(z_{\lambda},t\right)$, consider the system
	\begin{subequations}
		\begin{align*}
			&\varepsilon^{2}\dell_{t}\phi+\gradperp_{z_{\lambda}}\left(\mathscr{V}\left(q_{\lambda}\right)+\hat{a}\right)\cdot\nabla_{z_{\lambda}}\phi+E=0\ \ \text{in}\ \mathbb{H}_{-q'_{\lambda}}\times\left[T_{0},T\right]\\
			&\phi(\blank,T)=0\ \ \text{in}\ \mathbb{H}_{-q'_{\lambda}},\quad \phi\left(z_{\lambda},t\right)=0\ \ \text{on}\ \dell\mathbb{H}_{-q_{\lambda}'}\times\left[T_{0},T\right].
		\end{align*}
	\end{subequations}
	The representation formula gives us a unique solution $\phi=\mathcal{L}^{-1}\left(\hat{a}\right)\left[E\right]$, where $\mathcal{L}\left(\hat{a}\right)$ is the transport operator defined above.
	
	\medskip
	Then recalling \eqref{homotopic-operators-a-priori-estimates-1},\eqref{homotopic-operators-a-priori-estimates-3a}--\eqref{homotopic-operators-a-priori-estimates-3b}, and \eqref{homotopic-operators-a-priori-estimates-6}, \eqref{tilde-phi-star-R-initial-value-problem} can be written in the form 
	\begin{align}
		\tilde{\phi}_{*R}=\mathcal{T}^{in}_{\lambda}\left(\tilde{\phi}_{*R},\lambda\alpha,\psi_{*}^{out},\lambda\tilde{\xi}\right)\coloneqq\mathcal{L}^{-1}\left(\lambda\mathfrak{a}_{R,\lambda}\right)\left[\hat{\mathscr{E}}_{R}\left(\tilde{\phi}_{*R},\lambda\alpha,\psi_{*}^{out},\lambda\tilde{\xi}\right)\right],\label{inner-problem-fixed-point-formulation}
	\end{align}
	and $\hat{\mathscr{E}}_{R}\left(\tilde{\phi}_{*R},\lambda\alpha,\psi_{*}^{out},\lambda\tilde{\xi}\right)$ is given by
	\begin{align*}
		\hat{\mathscr{E}}_{R}\left(\tilde{\phi}_{*R},\lambda\alpha,\psi_{*}^{out},\lambda\tilde{\xi}\right)&=\lambda\tilde{\mathfrak{R}}_{1}+\varepsilon^{2}\dot{\alpha}\left(t\right)U_{R}\left(z_{\lambda},t\right)-\varepsilon\left(\mathcal{N}\left(\xi_{0}+\xi_{1}\right)\left[\lambda\tilde{\xi}\right]+\lambda\dot{\tilde{\xi}}\right)\cdot\grad_{z_{\lambda}}U_{R}\nonumber\\
		&+\lambda\tilde{\mathscr{E}}_{R}\left(\tilde{\psi}_{*R}, \lambda\alpha,\psi^{out}_{*},\lambda\tilde{\xi}\right)
	\end{align*}
	where we recall the definition of $\mathfrak{a}_{R,\lambda}$ from \eqref{homotopic-operator-inner-error-E-star-R-definition-10}, of $\tilde{\mathscr{E}}_{R}$ from \eqref{homotopic-operator-inner-error-main-definition-lower-order-terms}, and of $\tilde{\mathfrak{R}}_{1}$ from \eqref{homotopic-operators-a-priori-estimates-3a}--\eqref{homotopic-operators-a-priori-estimates-3b}. We note that the only dependence of $\hat{\mathscr{E}}_{R}$ on $\tilde{\phi}_{*R}$ comes through $\tilde{\psi}_{*R}$ and $\nabla_{z}\tilde{\psi}_{*R}$, via the relationship $\eqref{psi-star-r-green-function-representation-upper-half-plane}$.
	
	\medskip
	Next we recall that \eqref{psi-out-star-boundary-value-problem}, which, for $(-1)^{R}=1$ and $(-1)^{L}=-1$, is given by
	\begin{align*}
		&\Delta\psi_{*}^{out}+\lambda\sum_{j=R,L}\left(-1\right)^{j}\left(\psi_{*j}\Delta\eta^{(j)}_{K}+2\nabla\eta^{(j)}_{K}\cdot\nabla\psi_{*j}\right)+\lambda E_{2}^{out}\left(\lambda\tilde{\xi}\right)=0\ \ \text{in}\ \mathbb{R}^{2}_{+}\times \left[T_{0},T\right],\\
		&\psi^{out}_{*}\left(x_{1},0,t\right)= 0,\ \ t\in \left[T_{0},T\right],
	\end{align*}
	where is defined in $E_{2}^{out}\left(\tilde{\xi}\right)$ \eqref{first-approximation-outer-error-E2out-final-bound}, and we note that $\psi_{*L}$ is an odd reflection in $x_{1}$ of $\psi_{*R}$. Then \eqref{psi-out-star-boundary-value-problem} is equivalent to
	\begin{align}
		\psi^{out}_{*}&=\mathcal{T}^{out}_{\lambda}\left(\tilde{\phi}_{*R},\lambda\alpha,\psi_{*}^{out},\lambda\tilde{\xi}\right)\nonumber\\
		&\coloneqq\left(-\Delta\right)^{-1}\left(\lambda\sum_{j=R,L}\left(-1\right)^{j}\left(\psi_{*j}\Delta\eta^{(j)}_{K}+2\nabla\eta^{(j)}_{K}\cdot\nabla\psi_{*j}\right)+\lambda E_{2}^{out}\right),\label{outer-problem-fixed-point-formulation}
	\end{align}
	where the inverse Laplacian is on $\mathbb{R}^{2}_{+}$ with $0$ boundary conditions at $\left\{x_{2}=0\right\}$.
	
	\medskip 
	Given \eqref{cRj-initial-value-problem}, we integrate \eqref{homotopic-operator-inner-error-main-definition} on $\mathbb{H}_{-q_{\lambda}'}$. Then, we obtain
	\begin{align}
		\lambda\varepsilon^{2}\left(\int_{\mathbb{H}_{-q_{\lambda}'}}U_{R}\ dz_{\lambda}\right)\dot{\alpha}&=\lambda\mathscr{Q}_{(0,0)}-\lambda\int_{\mathbb{H}_{-q_{\lambda}'}}\left(\left(f_{\varepsilon}'\right)^{+}\mathcal{R}^{*}_{1}+\left(f_{\varepsilon}''\right)^{+}\mathcal{R}^{*}_{2}+\left(f_{\varepsilon}'''\right)^{+}\mathcal{R}^{*}_{3}\right)dz_{\lambda}\nonumber\\
		&+\lambda\varepsilon^{2}\alpha\left(\int_{\mathbb{H}_{-q_{\lambda}'}}\dell_{q_{\lambda}}U_{R}\ dz_{\lambda}\right)\dot{q}.\label{cRj-fixed-point-formulation-1}
	\end{align}
	Next, we integrate \eqref{homotopic-operators-a-priori-estimates-1} on $\mathbb{H}_{-q_{\lambda}'}$ against $z_{\lambda,1}$, and noting that $U_{R}$ and $\dell_{t}U_{R}$ are even in $z_{\lambda,1}$ on $\mathbb{H}_{-q_{\lambda}'}$, we obtain
	\begin{align}
		&\varepsilon\left(\int_{\mathbb{H}_{-q_{\lambda}'}}U_{R}\ dz_{\lambda}\right)\left(\lambda\dot{\tilde{p}}+\left[\mathscr{A}\left(\xi_{*}\right)\left(\lambda\tilde{\xi}\right)\right]_{1}\right)\nonumber\\
		&=\varepsilon\left(\int_{\mathbb{H}_{-q_{\lambda}'}}U_{R}\ dz_{\lambda}\right)\left[\mathscr{A}\left(\xi_{*}\right)\left(\lambda\tilde{\xi}\right)-\mathcal{N}\left(\xi_{0}+\xi_{1}\right)\left(\lambda\tilde{\xi}\right)\right]_{1}\nonumber\\
		&+\lambda\mathscr{Q}_{\left(1,0\right)}\nonumber-\lambda\int_{\mathbb{H}_{-q_{\lambda}'}}z_{\lambda,1}\left(\tilde{\mathfrak{R}}_{1}+\tilde{\mathscr{E}}_{R}\left(\tilde{\psi}_{*R}, \lambda\alpha,\psi^{out}_{*},\lambda\tilde{\xi}\right)\right)dz_{\lambda}\nonumber\\
		&-\int_{\mathbb{H}_{-q_{\lambda}'}}z_{\lambda,1}\left(\left[\gradperp_{z_{\lambda}}\left(\mathscr{V}\left(q_{\lambda}\right)+\lambda\mathfrak{a}_{R,\lambda}\right)\right]\cdot\nabla_{z_{\lambda}}\left(\tilde{\phi}_{*R}-\hat{\mathscr{W}}_{1,\lambda}\tilde{\psi}_{*R}\right)\right)dz_{\lambda},\label{cRj-fixed-point-formulation-2}
	\end{align}
	where we have defined $\hat{\mathscr{W}}_{1,\lambda}=\mathscr{W}_{1}\left(\mathscr{V}\left(q_{\lambda}\right)+\lambda\mathfrak{A}_{R,\lambda}\left(\xi_{0}+\xi_{1}\right)\right)$, and we recall the definition of $\tilde{\mathfrak{R}}_{1}$ from \eqref{homotopic-operators-a-priori-estimates-3b}. We have also used \eqref{tilde-phi-star-R-0-mass-condition}.
	
	\medskip
	Recalling \eqref{tilde-phi-star-R-0-mass-condition}, we can then proceed as in \eqref{projected-linear-transport-problem-coefficients-of-projection-a-priori-estimates-lemma-5}--\eqref{projected-linear-transport-problem-coefficients-of-projection-a-priori-estimates-lemma-6}, and rewrite $\eqref{cRj-fixed-point-formulation-2}$ as
	\begin{align}
		&\varepsilon\left(\int_{\mathbb{H}_{-q_{\lambda}'}}U_{R}\ dz_{\lambda}\right)\left(\lambda\dot{\tilde{p}}+\left[\mathscr{A}\left(\xi_{*}\right)\left(\lambda\tilde{\xi}\right)\right]_{1}\right)\nonumber\\
		&=\varepsilon\left(\int_{\mathbb{H}_{-q_{\lambda}'}}U_{R}\ dz_{\lambda}\right)\left[\mathscr{A}\left(\xi_{*}\right)\left(\lambda\tilde{\xi}\right)-\mathcal{N}\left(\xi_{0}+\xi_{1}\right)\left(\lambda\tilde{\xi}\right)\right]_{1}\nonumber\\
		&+\lambda\mathscr{Q}_{\left(1,0\right)}-\lambda\int_{\mathbb{H}_{-q_{\lambda}'}}z_{\lambda,1}\left(\tilde{\mathfrak{R}}_{1}+\tilde{\mathscr{E}}_{R}\left(\tilde{\psi}_{*R}, \lambda\alpha,\psi^{out}_{*},\lambda\tilde{\xi}\right)\right)dz_{\lambda}-\lambda\mathfrak{R}_{2},\label{cRj-fixed-point-formulation-2a}
	\end{align}
	for $\mathfrak{R}_{2}$ satisfying estimates inferred from Lemma \ref{projected-linear-transport-problem-coefficients-of-projection-a-priori-estimates-lemma}. Moreover, following the proof of this lemma with $a_{*}+a=\lambda \mathfrak{a}_{R,\lambda}$ accounts for the factor of $\lambda$ in front of $\mathfrak{R}_{2}$.
	
	\medskip
	We similarly integrate \eqref{homotopic-operators-a-priori-estimates-1} on $\mathbb{H}_{-q_{\lambda}'}$ against $z_{\lambda,2}$, and obtain
	\begin{align}
		&\varepsilon\left(\int_{\mathbb{H}_{-q_{\lambda}'}}U_{R}\ dz_{\lambda}\right)\left(\lambda\dot{\tilde{q}}+\left[\mathscr{A}\left(\xi_{*}\right)\left(\lambda\tilde{\xi}\right)\right]_{2}\right)\label{cRj-fixed-point-formulation-3}\\
		&=\varepsilon\left(\int_{\mathbb{H}_{-q_{\lambda}'}}U_{R}\ dz_{\lambda}\right)\left[\mathscr{A}\left(\xi_{*}\right)\left(\lambda\tilde{\xi}\right)-\mathcal{N}\left(\xi_{0}+\xi_{1}\right)\left(\lambda\tilde{\xi}\right)\right]_{2}\nonumber\\
		&+\lambda\mathscr{Q}_{\left(0,1\right)}-\lambda\int_{\mathbb{H}_{-q_{\lambda}'}}z_{\lambda,2}\left(\tilde{\mathfrak{R}}_{1}+\tilde{\mathscr{E}}_{R}\right)dz_{\lambda}-\lambda\mathfrak{R}_{3}-\lambda\varepsilon^{2}\left(\int_{\mathbb{H}_{-q_{\lambda}'}}z_{\lambda,2}U_{R}\left(z_{\lambda},\lambda\tilde{\xi}\right)dz_{\lambda}\right)\dot{\alpha},\nonumber
	\end{align}
	where, analogously to \eqref{cRj-fixed-point-formulation-2a}, $\mathfrak{R}_{3}$ satisfies estimates inferred from Lemma \ref{projected-linear-transport-problem-coefficients-of-projection-a-priori-estimates-lemma}.
	
	\medskip
	Then, similarly to Lemma \ref{projected-linear-transport-problem-coefficients-of-projection-a-priori-estimates-lemma}, the system of ODEs defined by \eqref{cRj-fixed-point-formulation-1}, \eqref{cRj-fixed-point-formulation-2a}, and \eqref{cRj-fixed-point-formulation-3}, the proof of Theorem \ref{first-approximation-construction-theorem}, the definitions of the $\mathscr{Q}_{\left(i,j\right)}$ in \eqref{0-orthogonality-rhs}, the orthogonality conditions \eqref{tilde-phi-star-R-0-mass-condition} and \eqref{homotopic-operator-inner-error-main-definition-lower-order-terms} can be written as
	\begin{align}
		\frac{d}{dt}\left(\begin{pmatrix}
			\lambda\alpha \\
			\lambda\tilde{p} \\
			\lambda\tilde{q}
		\end{pmatrix}\right)=-\mathfrak{A}\begin{pmatrix}
			\lambda\alpha \\
			\lambda\tilde{p} \\
			\lambda\tilde{q}
		\end{pmatrix}+\hat{\mathfrak{R}}_{5}\left(\tilde{\phi}_{*R},\lambda\alpha,\psi_{*}^{out},\lambda\tilde{\xi}\right)=\begin{pmatrix}
			\bigO{\left(\lambda\varepsilon^{3}t^{-4}\right)} \\
			\bigO{\left(\lambda\varepsilon^{4}t^{-3}\right)} \\
			\bigO{\left(\lambda\varepsilon^{4}t^{-4}\right)}
		\end{pmatrix},\quad \mathfrak{A}=\begin{pmatrix}
			1 & \textbf{0} \\
			\textbf{0} & \mathscr{A}(\xi_{*})
		\end{pmatrix},\label{cRj-fixed-point-formulation-4b}
	\end{align}
	where $\mathfrak{A}$ is in block matrix form and $\mathscr{A}$ is defined in \eqref{matrix-A-polar-form}. Thus \eqref{cRj-initial-value-problem} is equivalent to the fixed point problem
	\begin{align}
		\begin{pmatrix}
			\lambda\alpha \\
			\lambda\tilde{p} \\
			\lambda\tilde{q}
		\end{pmatrix}=\mathcal{T}^{ode}_{\lambda}\left(\tilde{\phi}_{*R},\lambda\alpha,\psi_{*}^{out}\lambda\tilde{\xi}\right)\coloneqq\int_{t}^{T}\left(-\mathfrak{A}\begin{pmatrix}
			\lambda\alpha \\
			\lambda\tilde{p} \\
			\lambda\tilde{q}
		\end{pmatrix}+\hat{\mathfrak{R}}_{5}\right)ds.\label{cRj-fixed-point-formulation-5}
	\end{align}
	Then if we define the new operators
	\begin{align*}
		\tilde{\mathcal{T}}^{out}_{\lambda}\left(\tilde{\phi}_{*R},\lambda\alpha,\psi_{*}^{out},\lambda\tilde{\xi}\right)&=\mathcal{T}^{out}_{\lambda}\left(\mathcal{T}^{in}_{\lambda}\left(\tilde{\phi}_{*R},\lambda\alpha,\psi_{*}^{out},\lambda\tilde{\xi}\right),\lambda\alpha,\psi_{*}^{out},\lambda\tilde{\xi}\right)\\
		\tilde{\mathcal{T}}^{ode}_{\lambda}\left(\tilde{\phi}_{*R},\lambda\alpha,\psi_{*}^{out},\lambda\tilde{\xi}\right)&=\mathcal{T}^{ode}_{\lambda}\left(\mathcal{T}^{in}_{\lambda}\left(\tilde{\phi}_{*R},\lambda\alpha,\psi_{*}^{out},\lambda\tilde{\xi}\right),\lambda\alpha,\psi_{*}^{out},\lambda\tilde{\xi}\right),
	\end{align*}
	we see that the system consisting of \eqref{inner-problem-fixed-point-formulation}, \eqref{outer-problem-fixed-point-formulation}, and \eqref{cRj-fixed-point-formulation-5} on the deformed open ball $\mathscr{O}$ defined in \eqref{deformed-open-ball-definition} is equivalent to the fixed point problem
	\begin{align}
		\mathfrak{P}=\tilde{\mathcal{T}}_{\lambda}\left(\mathfrak{P}\right)=\left(\mathcal{T}^{in}_{\lambda}\left(\mathfrak{P}\right),\tilde{\mathcal{T}}^{out}_{\lambda}\left(\mathfrak{P}\right),\tilde{\mathcal{T}}^{ode}_{\lambda}\left(\mathfrak{P}\right)\right),\quad \mathfrak{P}\in\mathscr{O}.\label{fixed-point-problem-final}
	\end{align}
	\begin{lemma}
		The operator $\tilde{\mathcal{T}}\colon\mathscr{O}\times\left[0,1\right]\rightarrow\mathfrak{X}$ defined by $\tilde{\mathcal{T}}\left(\cdot,\lambda\right)\coloneqq\tilde{\mathcal{T}}_{\lambda}$ in \eqref{fixed-point-problem-final} is compact, and for all $\lambda\in\left[0,1\right]$, is a degree $1$ operator.
	\end{lemma}
	\begin{proof}
		Since we are trying to construct a solution to \eqref{fixed-point-problem-final} on a finite time interval $\left[T_{0},T\right]$, the proof is almost exactly the proof given in \cite{DDMW2020}. The extra detail is due to the fact that we must obtain a solution for $\left(\lambda\alpha,\lambda\tilde{\xi}\right)$ instead of $\left(\alpha,\tilde{\xi}\right)$. However, the right hand side of \eqref{cRj-fixed-point-formulation-4b} is $\bigO{\left(\lambda\right)}$ as $\lambda\to0$, so we have compactness of the operator as well as the desired information on the degree for all $\lambda\in\left[0,1\right]$. 
	\end{proof}
	The degree of the operator at $\lambda=1$ exactly means we have a unique solution in $\mathscr{O}$ to \eqref{psi-star-r-green-function-representation-upper-half-plane}, \eqref{tilde-phi-star-R-initial-value-problem}, \eqref{cRj-initial-value-problem}, \eqref{psi-out-star-boundary-value-problem} at $\lambda=1$, and therefore to \eqref{2d-euler-vorticity-stream} with the desired estimates in $\varepsilon$ and $t$ on $\left[T_{0},T\right]$, for all $T$ large enough, as desired. For a proof of this fact, see, for example \cite{ambromalch}, Theorem $4.1$.
	
	\section{Proof of Theorem \ref{teo1}}\label{conclusion}
	We have a sequence of solutions to \eqref{2d-euler-vorticity-stream} that is identified with the sequence of solutions with estimates coming from \eqref{deformed-open-ball-definition} on $\left[T_{0},T\right]$, $\mathfrak{P}^{T}=\left(\tilde{\phi}_{*R}^{T},\alpha^{T},\psi_{*}^{out,T},\tilde{\xi}^{T}\right)$ to \eqref{psi-star-r-green-function-representation-upper-half-plane}, \eqref{tilde-phi-star-R-initial-value-problem}, \eqref{cRj-initial-value-problem}, \eqref{psi-out-star-boundary-value-problem} at $\lambda=1$ that we constructed in Section \ref{fixed-point-formulation-section}. These solutions satisfy $\mathfrak{P}^{T}\left(T\right)=0$.
	
	\medskip
	By a diagonal argument applied to $\left[T_{0},T_{0}+1\right]$, $\left[T_{0},T_{0}+2\right]\dots$, we obtain a sequence $(\alpha^{T_{n}},\tilde{\xi}^{T_{n}})\to(\alpha,\tilde{\xi})$, $(\dot{\alpha}^{T_{n}},\dot{\tilde{\xi}}^{T_{n}})\to(\dot{\alpha},\dot{\tilde{\xi}})$ that converges locally uniformly on $\left[T_{0},\infty\right)$, and satisfies bounds 
	\begin{align*}
		&\|t^{3}\alpha\|_{\left[T_{0},\infty\right)}+\|t^{4}\dot{\alpha}\|_{\left[T_{0},\infty\right)}\leq \varepsilon^{3-\sigma},\\
		&\|t^{2}\tilde{p}\|_{\left[T_{0},\infty\right)}+\|t^{3}\dot{\tilde{p}}\|_{\left[T_{0},\infty\right)}+\|t^{3}\tilde{q}\|_{\left[T_{0},\infty\right)}+\|t^{4}\dot{\tilde{q}}\|_{\left[T_{0},\infty\right)}\leq \varepsilon^{4-\sigma}. 
	\end{align*}
	For equicontinuity in $z$ for $\tilde{\phi}_{*R}^{T}$ for $t_{0}\in\left[T_{0},T_{0}+1\right]$, we first recall that $\tilde{\phi}_{*R}^{T}$ solves the equation
	\begin{align}
		\varepsilon^{2}\dell_{t}\tilde{\phi}_{*R}^{T}+\gradperp_{z}\left(\mathscr{V}\left(q^{T}\right)+\mathfrak{a}_{R}^{T}\right)\cdot\nabla_{z}\left(\tilde{\phi}_{*R}^{T}-\mathscr{W}_{1}\left(\mathscr{V}\left(q^{T}\right)+\mathfrak{A}_{R}^{T}\right)\tilde{\psi}_{*R}^{T}\right)+\mathfrak{E}_{1}^{T}=0,\label{phi-star-T-equation}
	\end{align}
	where $z$ is the coordinate coming from \eqref{homotopy-parameter-change-of-coordinates} at $\lambda=1$, $\mathfrak{a}_{R}^{T}$ is defined in \eqref{homotopic-operator-inner-error-E-star-R-definition-10} for $\lambda=1$, $\mathfrak{A}_{R}^{T}$ is defined in \eqref{homotopic-operators-a-priori-estimates-3} for $\lambda=1$, $\mathfrak{E}_{1}^{T}$ is defined in \eqref{homotopic-operators-a-priori-estimates-6} for $\lambda=1$, and $q^{T}=q_{0}+q_{1}+\tilde{q}^{T}$. For ease of notation, define 
	\begin{align*}
		\tilde{\mathfrak{E}}_{1}^{T}=-\gradperp_{z}\left(\mathscr{V}\left(q^{T}\right)+\mathfrak{a}_{R}^{T}\right)\cdot\nabla_{z}\left(\mathscr{W}_{1}\left(\mathscr{V}\left(q^{T}\right)+\mathfrak{A}_{R}^{T}\right)\tilde{\psi}_{*R}^{T}\right)+\mathfrak{E}_{1}^{T}.
	\end{align*}
	Then as in \eqref{characteristics-ode}, let $t\in\left[T_{0},T\right]$, and define characteristics $\bar{z}\left(\tau,t,z\right)$ by the ODE problem
	\begin{align}
		\frac{d\bar{z}^{T}}{d\tau}=\varepsilon^{-2}\gradperp_{\bar{z}^{T}}\left(\mathscr{V}\left(q^{T}\right)+\mathfrak{a}_{R}^{T}\right)\left(\bar{z}^{T}\left(\tau,t,z\right),\tau\right),\quad \tau\in\left[t,T\right],\quad \bar{z}^{T}\left(t,t,z\right)=z.\label{characteristics-ode-T}
	\end{align}
	We note that $\alpha^{T}$ and $\xi^{T}$ are both a uniformly bounded family of constant functions with respect to the space variables. Then Corollary \ref{L-infinity-a-priori-estimate-corollary} gives a Log-Lipschitz bound on the quantity $\gradperp_{\bar{z}}\left(\mathscr{V}\left(q^{T}\right)+\mathfrak{a}_{R}^{T}\right)$, which along with \eqref{characteristics-ode-T}, gives
	\begin{align}
		\left|\bar{z}_{1}^{T}\left(\tau,t,z\right)-\bar{z}_{2}^{T}\left(\tau,t,z\right)\right|\leq C\varepsilon^{-2}e^{e^{\tau}}\left|z_{1}-z_{2}\right|,\quad \tau\in\left[t,T\right],\label{characteristics-double-exponential-bound}
	\end{align}
	some constant $C>0$. We stress here that for the purposes of passage to the limit $T\to\infty$, $\varepsilon>0$ is fixed, and can be considered a constant when checking equicontinuity of these terms.
	
	\medskip
	We define for each pair of distinct points $z_{1},z_{2}\in B_{3\rho}\left(0\right)$, $T_{*}\left(z_{1},z_{2}\right)$ by the identity
	\begin{align}
		\varepsilon^{-4}e^{e^{T_{*}}}\left|z_{1}-z_{2}\right|=\left(\frac{\varepsilon}{T_{*}}\right)^{1-\sigma}.\label{phi-space-equicontinuity-T-star-definition}
	\end{align}
	Then suppose for $\delta>0$ small, $\left|z_{1}-z_{2}\right|<\varepsilon^{5-\sigma}e^{-2e^{\delta^{-\frac{1}{1-\sigma}}}}$. Then using \eqref{phi-star-T-equation} and \eqref{characteristics-ode-T} and \eqref{phi-space-equicontinuity-T-star-definition}, we have
	\begin{align}
		\left|\tilde{\phi}_{*R}^{T_{n}}\left(z_{1},t_{0}\right)-\tilde{\phi}_{*R}^{T_{n}}\left(z_{2},t_{0}\right)\right|\leq 2C\varepsilon^{1-\sigma}\delta.\label{infinite-time-solution-2}
	\end{align}
	which gives equicontinuity in $z$ for $\tilde{\phi}_{*R}^{T_{n}}$ at any $t_{0}\in\left[T_{0},T\right]$, as $\supp{\tilde{\phi}_{*R}^{T_{n}}}\subset B_{3\rho}\left(0\right)$ for any $T_{n}$.
	
	\medskip
	Next, due to the fact that $\tilde{\mathfrak{E}}_{1}^{T_{n}}\left(t\right)=\bigO{\left(\varepsilon^{3-\sigma}t^{-\left(2-\sigma\right)}\right)}$ on $\left[T_{0},T_{n}\right]$, we have
	\begin{align}
		\|\tilde{\phi}_{*R}^{T_{n}}\left(t_{1}\right)-\tilde{\phi}_{*R}^{T_{n}}\left(t_{2}\right)\|_{L^{\infty}\left(\mathbb{H}_{-q'}\right)}=\varepsilon^{-2}\|\int_{t_{1}}^{t_{2}}\tilde{\mathfrak{E}}_{1}^{T_{n}}\|_{L^{\infty}\left(\mathbb{H}_{-q'}\right)}\leq C\varepsilon^{1-\sigma}\left|t_{1}-t_{2}\right|.\label{infinite-time-solution-3}
	\end{align}
	This gives equicontinuity on $\left[T_{0},T_{0}+1\right]$ for $\tilde{\phi}_{*R}^{T_{n}}$ for fixed $z$.
	
	\medskip
	Combining \eqref{phi-space-equicontinuity-T-star-definition}--\eqref{infinite-time-solution-3}, we obtain uniform equicontinuity in $\left(z,t\right)$ for $\tilde{\phi}_{*R}^{T_{n}}$. Uniform boundedness for this family is clear, so we obtain upon possibly relabelling once more, that $\tilde{\phi}_{*R}^{T_{n}}\to\tilde{\phi}_{*R}$ uniformly in space and locally uniformly in time on $B_{3\rho}\left(0\right)\times\left[T_{0},\infty\right)$.
	
	\medskip
	By Corollary \ref{L-infinity-a-priori-estimate-corollary}, compact support,  convergence of $\tilde{\phi}_{*R}^{T_{n}}\to\tilde{\phi}_{*R}$ in $L^{\infty}\left(\mathbb{H}_{-q'}\right)$, and therefore convergence in $L^{2}\left(\mathbb{H}_{-q'}\right)$, $\tilde{\phi}_{*R}$ satisfies the estimates \eqref{phi-star-R-L2-Linfinity-estimates}. The existence, regularity, and bounds for $\tilde{\psi}_{*R}$ and $\psi^{out}_{*}$ follow from the estimates for $\tilde{\phi}_{*R}$. 
	
	\medskip
	Finally, by \eqref{phi-space-equicontinuity-T-star-definition} and \eqref{infinite-time-solution-2}, $\tilde{\phi}_{*R}\left(\cdot,t\right)$, for fixed $\varepsilon>0$ and all $t\geq T_{0}$, has a modulus of continuity given by, for some constant $C>0$.
	\begin{align*}
		\frac{2C\varepsilon^{1-\sigma}}{\left(\log{\left(\log{\left(\frac{\varepsilon^{\frac{5-\sigma}{2}}}{\left|z_{1}-z_{2}\right|^{\frac{1}{2}}}\right)}\right)}\right)^{\left(1-\sigma\right)}},
	\end{align*}
	For $\left(-1\right)^{R}=1$ and $\left(-1\right)^{L}=-1$, $x_{R}=\varepsilon^{-1}\left(x-pe_{1}\right)$, $x_{L}=\varepsilon^{-1}\left(x+pe_{1}\right)$, defining $\phi_{\varepsilon}$ as
	\begin{align*}
		\phi_{\varepsilon}\left(x,t\right)&=\sum_{j=R,L}\left(-1\right)^{j}\left[\phi_{j}\left(x_{j},t\right)+\left(\phi_{*j}\left(x_{j}-q'e_{2},t\right)-\phi_{*j}\left(x_{j}+q'e_{2},t\right)\right)\right],
	\end{align*}
	for $\phi_{j}=\phi_{j1}+\phi_{j2}$, $j=R,L$ constructed in Theorem \ref{first-approximation-construction-theorem} and $\phi_{*j}$, $j=R,L$ constructed in Section \ref{constructing-full-solution-section} and \ref{conclusion}, we have a solution to $\eqref{2d-euler-vorticity-stream}$ with the properties stated in Theorem \ref{teo1}, as desired.
	
	\appendix
	\section{Linearisation Around the Power Nonlinearity Vortex}
	Here we record some useful results about the linearised operator around $\Gamma$ defined in \eqref{power-semilinear-problem-R2}. 
	
	\begin{lemma}\label{dancer-yan-non-degeneracy-lemma}
		Suppose $\varphi\in L^{\infty}$ is a solution to
		\begin{align}
			\Delta\varphi+\gamma \Gamma^{\gamma-1}_{+}\varphi=\kappa \gamma \Gamma^{\gamma-1}_{+}\varphi\nonumber
		\end{align}
		on $\mathbb{R}^{2}$, with $\kappa\in[0,1)$ a constant. If $\kappa=0$, then $\varphi$ lives in the span of $\dell_{1}\Gamma$ and $\dell_{2}\Gamma$. If $\kappa>0$, then $\varphi\equiv0$.
	\end{lemma}
	The proof of this result (with minor modifications) can be found in \cite{danceryan}.
	
	\medskip
	Next, we look at the linear problem given by
	\begin{align}
		\Delta\psi+\gamma \Gamma^{\gamma-1}_{+}\psi=\gamma \Gamma^{\gamma-1}_{+}g,\label{vortex-linearised-equation}
	\end{align}
	where $g$ is a $C^{1}$ function. We move to complex coordinates $y=re^{i\theta}$, so that
	\begin{align}
		\psi\left(r,\theta\right)=\sum_{k\in\mathbb{Z}}\uppsi_{k}\left(r\right)e^{ik\theta},\quad g\left(r,\theta\right)=\sum_{k\in\mathbb{Z}}\mathfrak{g}_{k}\left(r\right)e^{ik\theta}.\label{vortex-linearised-equation-solution-fourier-expansion}
	\end{align}
	We assume that $g$ has no mode $0$ terms so that $\mathfrak{g}_{0}\equiv0$, and we accordingly impose that $\uppsi_{0}\equiv0$. Then \eqref{vortex-linearised-equation} splits into infinitely many ODEs given by
	\begin{align}
		\mathscr{L}_{k}\left[\uppsi_{k}\right]\coloneqq\dell_{r}^{2}\uppsi_{k}+\frac{1}{r}\dell_{r}\uppsi_{k}-\frac{k^{2}}{r^{2}}\uppsi_{k}+\gamma\Gamma^{\gamma-1}_{+}\uppsi_{k}=\gamma\Gamma^{\gamma-1}_{+}\mathfrak{g}_{k}.\label{modal-odes}
	\end{align}
	For $k\neq0$, one can find a positive solution $\zeta_{k}$ to the equation $\mathscr{L}_{k}\left[\zeta_{k}\right]=0$, that behaves like
	\begin{align}
		\zeta_{k}=r^{\left|k\right|}\left(1+\bigO{\left(r^{2}\right)}\right),\ r\to0,\label{zeta-k-origin-behaviour}
	\end{align}
	For $k=\pm1$, we have that $\zeta_{k}=-\Gamma'\left(r\right)$, which by \eqref{Gamma-def}, decays like $r^{-1}$ as $r\to\infty$. For $\left|k\right|\geq2$, we instead have
	\begin{align}
		\zeta_{k}=r^{\left|k\right|}\left(1+\bigO{\left(r^{-2\left|k\right|}\right)}\right).\ r\to\infty,\label{zeta-k-infinity-behaviour}
	\end{align}
	Then, in the case that $\mathfrak{g}_{k}\sim r^{\left|k\right|}$ as $r\to0$, we have the following lemma on how $\uppsi_{k}$ behaves at the origin and as $r\to\infty$.
	\begin{lemma}\label{vortex-linearised-equation-fourier-coefficients-behaviour-lemma}
		Assume we have $\psi$ and $g$ as in \eqref{vortex-linearised-equation}--\eqref{vortex-linearised-equation-solution-fourier-expansion}, with no mode $0$ terms. Also assume that $\mathfrak{g}_{k}\sim r^{\left|k\right|}$ as $r\to0$ for $\left|k\right|\geq 1$. Then for each $\left|k\right|\geq1$, a unique solution $\uppsi_{k}$ to \eqref{modal-odes} exists. Moreover, we have
		\begin{align}
			&\uppsi_{k}\in L^{\infty}\left(\mathbb{R}^{2}\right),\quad \uppsi_{k}=r^{\left|k\right|}\left(1+o\left(1\right)\right),\ r\to0,\quad \left|k\right|\geq2,\nonumber\\
			&\uppsi_{k}=\twopartdef{\bigO{\left(r^{2}\left|k\right|^{-2}\right),}}{r\to0,\ \left|k\right|\to\infty,}{\bigO{\left(r^{-\left|k\right|}\left|k\right|^{-2}\right)},}{r\to\infty,\ \left|k\right|\to\infty,}\nonumber\\
			&\frac{1}{r+1}\uppsi_{k}\in L^{\infty}\left(\mathbb{R}^{2}\right),\quad \uppsi_{k}=\twopartdef{\bigO{\left(r^{3}\right)},}{r\to0,}{\bigO{\left(r\right)},}{r\to\infty,}\quad \left|k\right|=1.\label{vortex-linearised-equation-fourier-coefficients-behaviour-lemma-statement-1}
		\end{align}
	\end{lemma}
	\begin{proof}
		Using variation of parameters, we can write $\uppsi_{k}$ in terms of $\zeta_{k}$ and $\mathfrak{g}_{k}$ explicitly. Taking into account the growth or decay of each $\zeta_{k}$ at either $0$ or $\infty$, we have
		\begin{align}
			\uppsi_{k}=\twopartdef{\zeta_{k}\int_{0}^{r}\frac{1}{s\zeta_{k}\left(s\right)^{2}}\int_{0}^{s}\zeta_{k}\left(\tau\right)\mathfrak{g}_{k}\left(\tau\right)\gamma\Gamma^{\gamma-1}_{+}\left(\tau\right)\tau\ d\tau\ ds,}{\left|k\right|=1,}{\zeta_{k}\int_{r}^{\infty}\frac{1}{s\zeta_{k}\left(s\right)^{2}}\int_{0}^{s}\zeta_{k}\left(\tau\right)\mathfrak{g}_{k}\left(\tau\right)\gamma\Gamma^{\gamma-1}_{+}\left(\tau\right)\tau\ d\tau\ ds,}{\left|k\right|\geq2.}\label{variation-of-parameters}
		\end{align}
		The properties \eqref{vortex-linearised-equation-fourier-coefficients-behaviour-lemma-statement-1} can then be inferred from \eqref{variation-of-parameters} along with the additional fact that $\supp{\Gamma^{\gamma-1}_{+}}\subset B_{1}\left(0\right)$.
	\end{proof}
	Using the above results, we define $\varrho_{k}\left(r\right)$ such that in \eqref{vortex-linearised-equation-solution-fourier-expansion},
	\begin{align}
		\mathfrak{g}_{k}\left(r\right)=r^{\left|k\right|}.\label{vortex-linearised-equation-rk-error-solution}
	\end{align}
	
	\section{Spectral Theory of the Vortex Pair}\label{vortex-pair-spectral-theory-appendix}
	In this appendix we record relevant facts related to the spectral theory of the linearised operator around the vortex pair solution exhibited in Theorem \ref{vortexpair-theorem}. In $y=\varepsilon x$ variables, let
	\begin{align*}
		\mathcal{H}\coloneqq L^{2}\left(\supp{f_{\varepsilon}'},f_{\varepsilon}'\ dy\right)
		=\left\{ u:\supp{f_{\varepsilon}'}\to\R\,\Big|\, \int_{\R^2}f_\varepsilon'\ |u|^2<\infty,\ u \text{ odd in $y_2$}\right\},
	\end{align*}
	with $f_{\varepsilon}'$ defined in \eqref{f-prime-definition}. We recall that by \eqref{first-approximation-main-order-vorticity-support}, $\supp{f_{\varepsilon}'}\subset B_\rho(q'e_2)\cup B_\rho(-q'e_2)$ where $q'=\frac{q}{\varepsilon}$, and $\rho>0$ is the absolute constant defined in \eqref{vortex-pair-support-true}. 
	
	\medskip
	Define the linear operator $\mathfrak{T}$ by
	\begin{align}
		\mathfrak{T}:\mathcal{H}\rightarrow \mathcal{H}\quad h\mapsto \left(-\Delta\right)^{-1}\left(f_{\varepsilon}'h\right)\nonumber
	\end{align}
	where $(-\Delta)^{-1}$ is the Newtonian potential \eqref{newtonian}.
	One can show that, for each fixed $\varepsilon>0$ small enough, $\mathfrak{T}$ is a compact self-adjoint operator on $\mathcal{H}$. Thus, $ \mathcal{H}$ has a basis $(e_{j})_{j\geq 0}$ of eigenvectors of $\mathfrak{T}$ with eigenvalues $\mu_{j}\left(\varepsilon\right)>0$ which, for fixed $\varepsilon$, tend to $0$ as $j\to\infty$. We normalize
	\begin{align}
		\left(-\Delta\right)^{-1}\left( f_{\varepsilon}'e_{j}\right)=\mu_{j}\left(\varepsilon\right)e_{j},\quad \int_{\mathbb{R}^{2}} f_{\varepsilon}'e_{i}e_{j}=\delta_{ij}.\label{spectral-theory-appendix-orthogonality-relations}
	\end{align}
	Let $\mu_0$ be the largest eigenvalue.
	We will prove below that $\mu_{0}\sim\left|\log{\varepsilon}\right|$ as $\varepsilon\to0$.  Let
	\begin{align}
		e_{1}=\kappa_{1}\dell_{1}\Psi_{R},\quad e_{2}=\kappa_{2}\dell_{q'}\left(\Psi_{R}-c\varepsilon y_{2}-\left|\log{\varepsilon}\right|\Omega\right)\label{vortex-pair-1st-eigenfunction}
	\end{align}
	where $\kappa_{1}$ and $\kappa_{2}$ are normalization constants so that $e_{1}$ and $e_{2}$ satisfy \eqref{spectral-theory-appendix-orthogonality-relations}. 
	Since $e_{1}$ and $e_{2}$ are in the kernel of $\Delta+f_{\varepsilon}'$,  $e_1$ and $e_2$ are eigenvectors with $\mu_{1}\left(\varepsilon\right)=\mu_{2}\left(\varepsilon\right)=1$ for all $\varepsilon$.
	
	\medskip
	For all eigenvalues $\mu_{j}$, $j\geq3$, we have $\mu_j(\varepsilon)<1$ with a uniform gap.
	\begin{lemma}\label{eigenvalue-estimates-lemma}
		There is $\varepsilon_0>0$ such that:
		\textup{a)}
		for $0<\varepsilon<\varepsilon_0$ there is $e_0\not=0$ and $\mu_0=c_0|\log\varepsilon|+\bigO{(1)}$ as $\varepsilon\to0$ for some $c_0>0$ such that 
		$$
		\mathfrak{T}(e_0)=\mu_0e_0.
		$$
		After normalization, $e_0=1+\bigO{\left(\frac{1}{|\log\varepsilon|}\right)}$ uniformly in $B_{2\rho}(q'e_2)$, and is unique.
		
		\textup{b)}
		There exists a fixed $\delta\in\left(0,1\right)$ such that for all $0<\varepsilon<\varepsilon_{0}$ and $j\geq3$,
		\begin{align}
			\mu_{j}\left(\varepsilon\right)\in[0,\delta].\nonumber
		\end{align}
	\end{lemma}
	\begin{proof}
		In this proof we consider equations
		$$
		-\Delta \phi = \lambda f_\varepsilon' \phi + h\quad\text{in }\R^2
		$$
		with $h$ odd in $y_2$ and fast decay. The solution is the one given by the Newtonian potential and hence odd in $y_2$. The eigenfunctions satisfy
		$$
		-\Delta e_j = \frac{1}{\mu_j} f_\varepsilon' e_j \quad\text{in }\R^2.
		$$
		a) 
		Here we prove the existence of $e_0$ with eigenvalue $\mu_0\sim |\log\varepsilon|$.
		Let $\phi_0$ be the solution of 
		$$
		-\Delta \phi_0 = \lambda_0 f_\varepsilon'\chi_{\R^2_+}- \lambda_0 f_\varepsilon'\chi_{\R^2_-}  \quad\text{in }\R^2
		$$
		(given by the Newtonian potential). A calculation shows that choosing $\lambda_0$ appropriately with the expansion $\lambda_0=\frac{c_0}{|\log\varepsilon|}+\bigO{\left(\frac{1}{|\log\varepsilon|^2}\right)}$, we get $\phi_0(q'e_2)=1$. 
		From the Newtonian potential we also get
		$$
		\phi_0(y)=1+\bigO{\left(\frac{1}{|\log\varepsilon|}\right)}\quad\text{uniformly in } B_{2\rho}(q'e_2).
		$$
		We look for $\phi_1$ so that $e_0=\phi_0+\phi_1$ is an eigenfunction of $\mathfrak{T}$ with eigenvalue $\mu_0=\frac{1}{\lambda}$, where $\lambda=\bigO{\left(\frac{1}{|\log\varepsilon|}\right)}$. We then get the equation
		\begin{align}
			\label{eqphi1e}
			-\Delta\phi_1=\lambda f_{\varepsilon}'\phi_1+\lambda_0 f_{\varepsilon}'(\phi_0-1)\chi_{\R^2_+}+\lambda_0 f_{\varepsilon}'(\phi_0+1)\chi_{\R^2_-}+(\lambda-\lambda_0)f_{\varepsilon}'\phi_0.
		\end{align}
		So we study 
		\begin{align}
			\label{eqphi1}
			-\Delta \phi=\lambda f_{\varepsilon}'\phi+ f_{\varepsilon}'h+d f_{\varepsilon}'\phi_0
		\end{align}
		with $d\in\mathbb{R}$, $h\in L^\infty$ and odd in $y_2$ and $\lambda=o(1)$ as $\varepsilon\to0$.
		The solution is the one given by the Newtonian potential.
		
		\medskip
		Then there exists $C$ such that for $\varepsilon>0$ small, if $\phi$ solves \eqref{eqphi1} and $\int_{\R^2_+} f_{\varepsilon}'\phi=0$, then
		\begin{align}
			\label{a1}
			\|\phi\|_{L^\infty(\R^2)} +\left|d\right|\leq C \|h\|_{L^\infty(\R^2)},
		\end{align}
		with the proof following very similar lines to Proposition $5.1$ in \cite{ao}.
		
		\medskip
		Estimate \eqref{a1} and a standard argument using the Fredholm alternative give that for any $h\in L^\infty(\R^2)$ odd in $y_2$, there are unique $\phi$ and $d$ satisfying \eqref{eqphi1}. 
		Applying this solvability theory to \eqref{eqphi1e} we find $\lambda = \lambda_0+\bigO{\left(\frac{1}{|\log\varepsilon|^2}\right)}$ and $\phi_1$ with $\|\phi_1\|_{L^\infty(\R^2}\leq \frac{C}{|\log\varepsilon|^2}$ satisfying \eqref{eqphi1e}.
		This proves the existence of $e_0$. 
		The uniqueness is consequence of b) below.
		
		\medskip
		\noindent b) Suppose by contradiction that there  is a sequence $\varepsilon_n\to 0$, such that there is an eigenfunction $v_n = e_{j_n}$ with corresponding eigenvalue $\mu_n=\mu_{j_n}(\varepsilon_n)$ such that $v_n$ satisfies
		\begin{align}
			\label{norma2}
			\int_{\R^2}f_\varepsilon' |v_n|^2=1,\quad 
			\int_{\R^2}f_\varepsilon' v_n e_j =0,\quad j=0,1,2
		\end{align}
		and either $\mu_n \to \mu\geq 1$ or $\mu_n\to\infty$. We claim that 
		\begin{align}
			\label{claim1}
			\|v_n\|_{L^\infty(\R^2)}\leq C.
		\end{align}
		Indeed, from \eqref{norma2} and standard elliptic estimates we have
		$$
		\sup_{B_{\frac{\rho}{2}}(q'e_2)}|v_n|\leq C.
		$$
		From here, estimates for the Newtonian potential applied to $\frac{1}{\mu_n}f_{\varepsilon}' v_n \chi_{\R^2_+}$ and the Newtonian potential applied to $\frac{1}{\mu_n}f_{\varepsilon}' v_n \chi_{\R^2_-}$, the oddness of $v_{n}$, and the fact that $v_n$ is harmonic in the complement of $B_{2\rho}(q'e_2)\cup  B_{2\rho}(-q'e_2)$, and $v_n(y)\to 0$ as $|y|\to 0$, gives us \eqref{claim1}.
		
		\medskip
		Let $\bar v_n(z)=v_n(z_n+q'e_2)$. By elliptic estimates $\bar v_n$ has a subsequence converging locally in $\mathbb{R}^{2}$ to some $\bar v \in L^\infty(\R^2)$, solving
		\begin{align}
			\Delta\bar v+\gamma \Gamma^{\gamma-1}_{+}\bar v=\kappa \gamma \Gamma^{\gamma-1}_{+}\bar v,\quad\text{in }\R^2
			\label{eigenvalue-estimates-lemma-1}
		\end{align}
		for some $\kappa\in[0,1)$ if $\mu=\lim_{n\to\infty}\mu_n<\infty$ and $\kappa=1$ ir $\mu=\infty$.
		From \eqref{norma2}, we have 
		\begin{align}            
			\int_{\mathbb{R}^{2}} \gamma \Gamma^{\gamma-1}_{+}\left|\bar v\right|^{2}=1,\quad \int_{\mathbb{R}^{2}} \Gamma^{\gamma-1}_{+}\bar v=0,\quad 
			\int_{\mathbb{R}^{2}} \Gamma^{\gamma-1}_{+}\bar v\dell_{i}\Gamma =0,\ \ i=1,2.
			\label{eigenvalue-estimates-lemma-2}
		\end{align}
		In particular, this implies that $\bar v$ is not in the span of $\dell_{1}\Gamma,\dell_{2}\Gamma$. 
		
		\medskip
		If $\mu<\infty$, so $\kappa\in[0,1)$, by Lemma~\ref{dancer-yan-non-degeneracy-lemma} we get $\bar v\equiv0$. This directly contradicts \eqref{eigenvalue-estimates-lemma-2}. If $\mu=\infty$ so $\kappa=1$, then $\bar v$ must be constant and hence $0$, again a contradiction.
	\end{proof}
	
	
	{\it Acknowledgement.}
		We would like to thank Michele Dolce for useful conversations regarding this work.
	
	\medskip
	
{\it Funding.}
		J.~D\'avila has been supported  by  a Royal Society  Wolfson Fellowship, UK. M.~del Pino has been supported by the Royal Society Research Professorship grant RP-R1-180114 and by the ERC/UKRI Horizon Europe grant ASYMEVOL, EP/Z000394/1. M. Musso has been supported by EPSRC research Grant EP/T008458/1.  S. Parmeshwar has been supported by EPSRC research Grants EP/T008458/1 and EP/V000586/1.


	
	
	
	
	
	
	\bibliography{main} 
	\bibliographystyle{siam}
	
\end{document}